\documentclass[]{amsart}

\usepackage{amsmath,amsrefs,bbold,mathrsfs,amscd,lineno,hyperref}
\usepackage{mathtools,afterpage}
\usepackage{tikz}
\usepackage{mathrsfs}
\usetikzlibrary{matrix}

\newcommand{\Dom}{\operatorname {Dom}}

\renewcommand{\sp}{{\ \ }}

\renewcommand{\Re}{\operatorname {Re}}
\renewcommand{\Im}{\operatorname {Im}}

\newcommand{\norm}[1]{\left\lVert#1\right\rVert}

\newcommand\LL{{\mathcal L}}

\newcommand\MM{\mathcal{M}}
\newcommand\TT{\mathcal{T}}

\newcommand\HH{\boldsymbol{\Delta}}
\newcommand\Hors{{\mathcal H}} 
\newcommand{\bDelta}{{\mathbb{\Delta}}}
\newcommand{\bLambda}{{\mathbb{\Lambda}}}
\newcommand{\bPi}{{\mathbb{\Pi}}}
\newcommand{\bA}{{\mathbb{A}}}
\newcommand\RR{\mathcal{R}}
\newcommand\cRR{R}
\newcommand\mRR{\boldsymbol R}
\newcommand\prm{{\operatorname{prm}}}
\newcommand\RRc{\mathcal{R}_{\prm}}

\newcommand\cRRc{R_{\prm}}

\newcommand\mRRc{ \boldsymbol R_{\prm}}

\newcommand\FF{\mathcal{F}}
\newcommand\intr{\operatorname{int}}
\newcommand\Mol{{\boldsymbol{\mathcal{M}ol }}}

\newcommand\seq{{\operatorname{\mathbf s}}}
\newcommand\seqo{{\operatorname{\mathbf s}_{\bullet}}}
\newcommand\oseq{{\bullet}}
\newcommand\Jul{{\mathfrak J}}
\newcommand\Kfilled{{\mathfrak K}}
\newcommand{\bF}{{\mathbf F}}
\newcommand{\bX}{{\mathbf X}}
\newcommand{\bW}{{\mathbf W}}
\newcommand{\bD}{{\mathbf D}}
\newcommand{\bS}{{\mathbf S}}
\newcommand{\bU}{{\mathbf U}}
\newcommand{\bH}{{\mathbf H}}
\newcommand{\bG}{{\mathbf G}}
\newcommand{\fH}{{\mathbb H}}
\newcommand{\bK}{{\mathbf K}}
\newcommand{\bbf}{{\mathbf f}}
\newcommand{\bbw}{{\mathbf w}}
\newcommand{\bbh}{{\mathbf h}}
\newcommand{\btau}{{ \boldsymbol {\tau}}}
\newcommand{\bbg}{{\mathbf g}}
\newcommand{\Rect}{{\mathfrak R }}
\newcommand{\bgamma}{{\boldsymbol{ \gamma}}}

\newcommand{\str}{{\star}}

\newcommand{\pp}{\mathfrak{p}}
\newcommand{\ee}{\mathbf e}
\renewcommand{\aa}{{\mathfrak{a}}}
\newcommand{\bb}{{\mathfrak{b}}}

\renewcommand{\qq}{\mathfrak{q}}
\renewcommand{\tt}{\mathfrak{t}}
\newcommand{\rr}{\mathfrak{r}}
\renewcommand{\ss}{\mathfrak{s}}
\newcommand{\kk}{\mathfrak{k}}
\newcommand{\nn}{\mathfrak{n}}

\newcommand{\A}{\mathbf A}
\newcommand{\B}{\mathbf B}
\newcommand\wzero{{\widetilde 0}}
\newcommand\wall{{Q}}
\newcommand\inn{{\Omega}}

\newcommand\wbeta{{\widetilde \beta}}
\newcommand\wgamma{{\widetilde \gamma}}

\newcommand\WW{\mathcal{W}}
\newcommand\UU{\mathcal{U}}
\newcommand\BB{\mathcal{B}}

\newcommand\PP{\mathcal{P}}

\newcommand\ext{{\operatorname{ext}}}
\newcommand\new{{\operatorname{new}}}

\newcommand\frb{{\operatorname{frb}}}
\newcommand\crt{{\mathfrak o}}

\newtheorem{claim}{Claim}
\newtheorem{claim2}{Claim}
\newtheorem{claim3}{Claim}

\newcommand{\C}{\mathbb{C}}
\newcommand{\Q}{\mathbb{Q}}
\newcommand{\R}{\mathbb{R}}
\newcommand{\N}{\mathbb{N}}
\newcommand{\Z}{\mathbb{Z}}

\newcommand{\Disk}{\mathbb{D}^1}
\newcommand{\Lbb}{{\mathbb L}}
\newcommand{\Ibb}{{\mathbb I}}
\newcommand{\Sbb}{{\mathbb X}}
\newcommand{\Sby}{{\mathbb Y}}
\newcommand{\Sbz}{{\mathbb T}}

 \newcommand{\Fol}{{\mathcal F}}
  \newcommand{\bFol}{{\boldsymbol{ \mathcal F}}}

\newtheorem{thm}{Theorem}[section]
\newtheorem{cor}[thm]{Corollary}
\newtheorem{lem}[thm]{Lemma}
\newtheorem{keylem}[thm]{Key Lemma}
\newtheorem{prop}[thm]{Proposition}
\newtheorem{rem}[thm]{Remark}

\theoremstyle{remark}

\newtheorem{conj}[thm]{Conjecture}

\numberwithin{equation}{section}

\theoremstyle{definition}
\newtheorem{defn}[thm]{Definition}
\font\nt=cmr7

\def\be{\begin{equation}}

\def\ssk{\smallskip}

\def\nin{\noindent}

\newcommand{\Mandel}{{\mathcal{M}}}

\newcommand{\per}{{\mathrm{per}}}
\newcommand{\bnd}{{\mathrm{bnd}}}
\newcommand{\fast}{{\mathrm{fast}}}

\newcommand{\irr}{{\mathrm{irr}}}

\newcommand{\cp}{{\mathrm{cp}}}
\newcommand{\cyl}{{\mathrm{cyl}}}

\newcommand{\di}{\partial}

\newcommand{\ra}{\rightarrow}

\def\sm{\smallsetminus}

\newcommand{\dist}{\operatorname{dist}}

\renewcommand{\mod}{\operatorname{mod}}

\newcommand{\Diff}{\operatorname{Diff}}
\newcommand{\orb}{\operatorname{orb}}

\newcommand{\id}{\operatorname{id}}

\newcommand{\const}{\mathrm{const}}
\def\loc{{\mathrm{loc}}}

\newcommand{\la}{{\lambda}}

\newcommand{\CC}{{\mathcal C}}

\def\note#1
{\marginpar
{\nt $\leftarrow$
\par
\hfuzz=20pt \hbadness=9000 \hyphenpenalty=-100 \exhyphenpenalty=-100
\pretolerance=-1 \tolerance=9999 \doublehyphendemerits=-100000
\finalhyphendemerits=-100000 \baselineskip=6pt
#1}\hfuzz=1pt}

\everymath{\displaystyle}
\usepackage{amssymb}


\begin{document}

\title[Pacmen]{Pacman renormalization\\and self-similarity of the Mandelbrot set \\near Siegel parameters}
\author{Dzmitry Dudko}
\author{Mikhail Lyubich}
\author{Nikita Selinger}

\begin{abstract}
 In the 1980s Branner and Douady discovered a surgery relating various limbs of the Mandelbrot set. We put this surgery in the framework of ``Pacman Renormalization Theory'' that combines features of quadratic-like and Siegel renormalizations. We show that Siegel renormalization
periodic points (constructed by McMullen in the 1990s) can be promoted to pacman renormalization periodic points. Then we prove that these periodic points are hyperbolic with one-dimensional unstable manifold resolving a long standing problem. As a consequence, we obtain the scaling laws for the centers of satellite components of the Mandelbrot set near the corresponding Siegel parameters.
\end{abstract}
\maketitle

\setcounter{tocdepth}{1}

\tableofcontents

\section{Introduction}

Renormalization was introduced into Dynamics in the mid 1970s by Feigenbaum,
Coullet and Tresser and since then has established itself as  a powerful tool of penetrating into the small-scale structure of  phase portraits and bifurcation loci.
It turned out to be challenging to develop a rigorous mathematical theory of renormalization (for example, to prove hyperbolicity of the renormalization operator),
but every time when this is achieved, plentiful  deep consequences  reward the effort.

The complex quadratic family provided us with several important renormalization schemes: quadratic-like,  near-parabolic, and Siegel. 
All three are intimately related to the 
observed self-similarity of the Mandelbrot set $\MM$  and to the 
celebrated MLC Conjecture on the local connectivity of $\MM$. 
The conjecture comes in two flavors, ``primitive" and ``satellite".  Development of the quadratic-like renormalization has led to a substantial progress in the primitive case, while the near-parabolic renormalization has given an insight into the satellite situation. 

In this paper we design a new ``pacman'' renormalization  and prove the hyperbolicity of the corresponding renormalization operator. It 
implies the hyperbolicity of the Siegel renormalization for arbitrary periodic combinatorics  (resolving a  problem going back to the early 1980s) and gives an insight into the self-similarity of the Mandelbrot set near the main cardioid. In the 2nd part of this project \cite{DL} we will give further applications  by 
proving local connectivity of the Mandelbrot set at some satellite parameters of bounded type (that had been previously out of reach) and  showing that the corresponding  Julia sets have positive area.

\subsection{Statements of the results}

\begin{figure}[t!]
\[\begin{tikzpicture}[scale=1.3]

\coordinate  (w0) at (-3.9,1.5);

\node (v0) at (-2.9,1.4)  {};

\draw  (v0) ellipse (3.5 and 2);

\node[shift={(0.15,0)}]  at (w0)  {$\alpha$};

\coordinate (v1) at (-2.2,1.3) {};
\coordinate (v2) at (-0.6,2.1) {} {};
\coordinate (v3) at (-0.6,0.6) {} {};

  \draw[ line width=0.5pt,red] 
            (v3) 
            .. controls (v3) and (v1) ..
            (v1)
            .. controls (v1) and (v2) ..
            (v2)
            .. controls (-7.3,4.4) and (-7.1,-1.5) ..
            (v3);

\coordinate (w1) at (-4.7,2.4) {} {};

\draw[line width=0.5pt,red]  (w0)--(w1);

\coordinate (w2) at (-6.1,0.6)  {};
\draw  [line width=0.5pt]  (w0)--(w2);

\draw[line  width=0.5pt,-latex] (-3.1,1.6) .. controls (-2.6,2.1) and (-3.1,3.2) .. (-3.8,2.8);

\node at (-5.5,0.7) {$\gamma_1$};
\node at (-2.8,2.7) {$f$};
\node[red] at (-4.5,1.9) {$\gamma_0$};

\node[red] at (-1.2,1.6) {$\gamma_+$};
\node[red] at (-1.1,0.9) {$\gamma_-$};
\node at (-1.9,4.6) {};
\node[red] at (-3.5,0.8) {$U$};
\node at (-1.6,0.1) {$V$};
\node at (-2.3,1.3) {$\alpha'$};
\end{tikzpicture}\]
\caption{A (full) pacman is a $2:1$ map $f:U
\to V$ such that the critical arc $\gamma_1$ has $3$ preimages: $\gamma_0$, $\gamma_+$ and $\gamma_-$. }
\label{Fg:Pacman}
\end{figure}
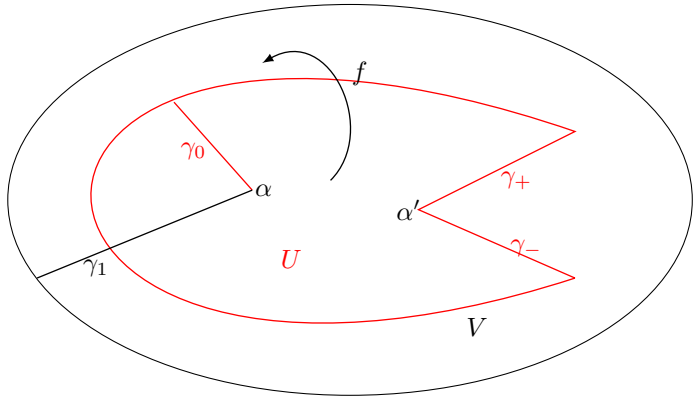

Though the Mandelbrot set $\Mandel$ is highly non-homogeneous, it  possesses some remarkable 
self-similarity features. Most notable is the presence of baby Mandelbrot sets  inside $\Mandel$
which are almost indistinguishable from $\Mandel $ itself. 
The explanation of this phenomenon
is provided by the Renormalization Theory for quadratic-like maps, which has been  a central theme
in Holomorphic Dynamics since the mid-1980s
(see \cites{DH,S,McM2,L:FCT} and references therein).

By exploring the pictures, one can also observe that the Mandelbrot set has self-similarity features near its main cardioid. For instance, as Figure \ref{Fig:GeomScal} indicates, near the (anti-)golden mean  point, 
the $(\pp_n/\pp_{n+2} )$-limbs of $\Mandel$ scale down at rate $\la^{-2n}$,  
where $\la = (1+\sqrt{5} ) /2 $ and $\pp_n$ are the Fibonacci numbers.  The goal of this paper is to develop a renormalization theory responsible 
for this phenomenon.

\begin{figure} 
\centering{\begin{tikzpicture}
  \node at (0,0){\includegraphics[scale=0.6]{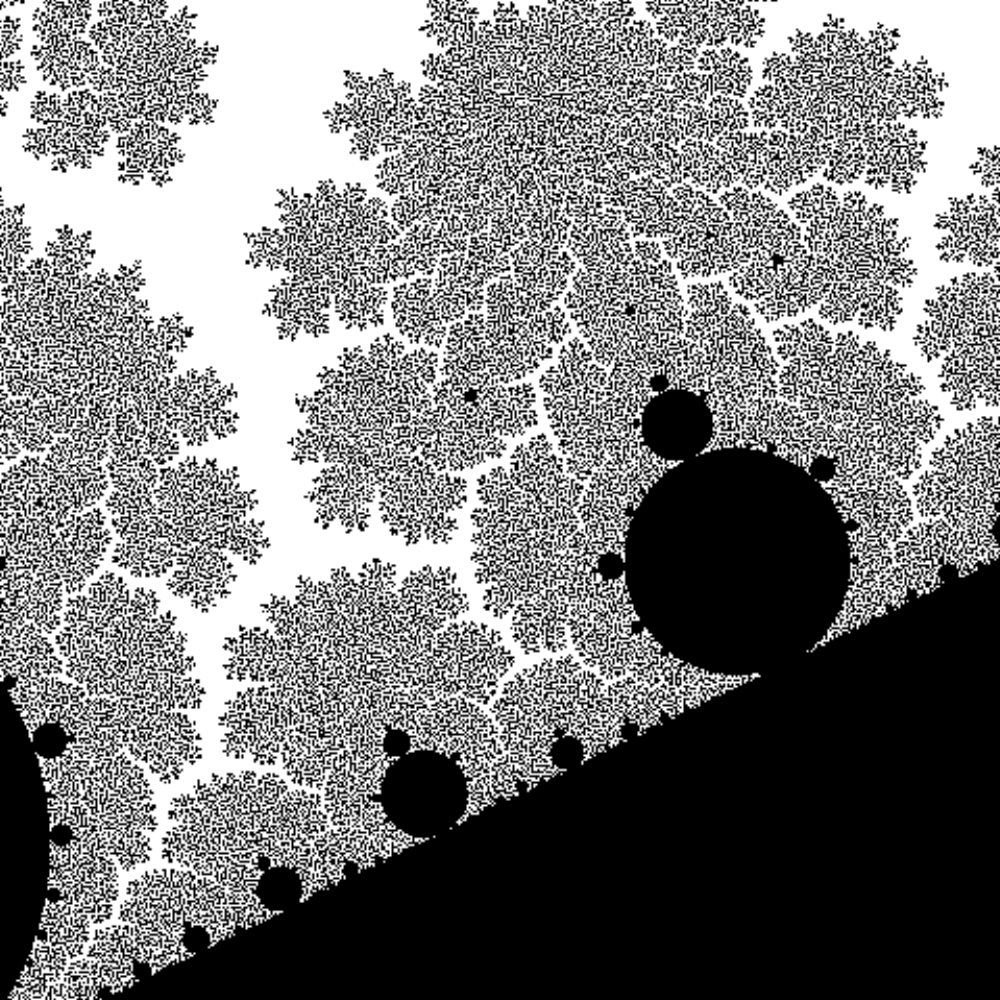}};
   \draw[white] node at (3.1,-2.1) {$21/55$}; 
    \draw[white] node at (1.25,-3) {$55/144$}; 
 \draw[white] node at (1.5,-4) {Golden}; 
      \draw[white] node at (1.5,-4.4) {Siegel map}; 
     \draw[ ->,white] 
            (1.5,-3.72) 
            .. controls (0.75,-3.52) and (0.6,-3.22) ..
            (0.4,-2.92); 
     \draw[white] node at (-0.15,-3.7) {$34/89$};
      \node at (0,-9.2){\includegraphics[scale=0.395]{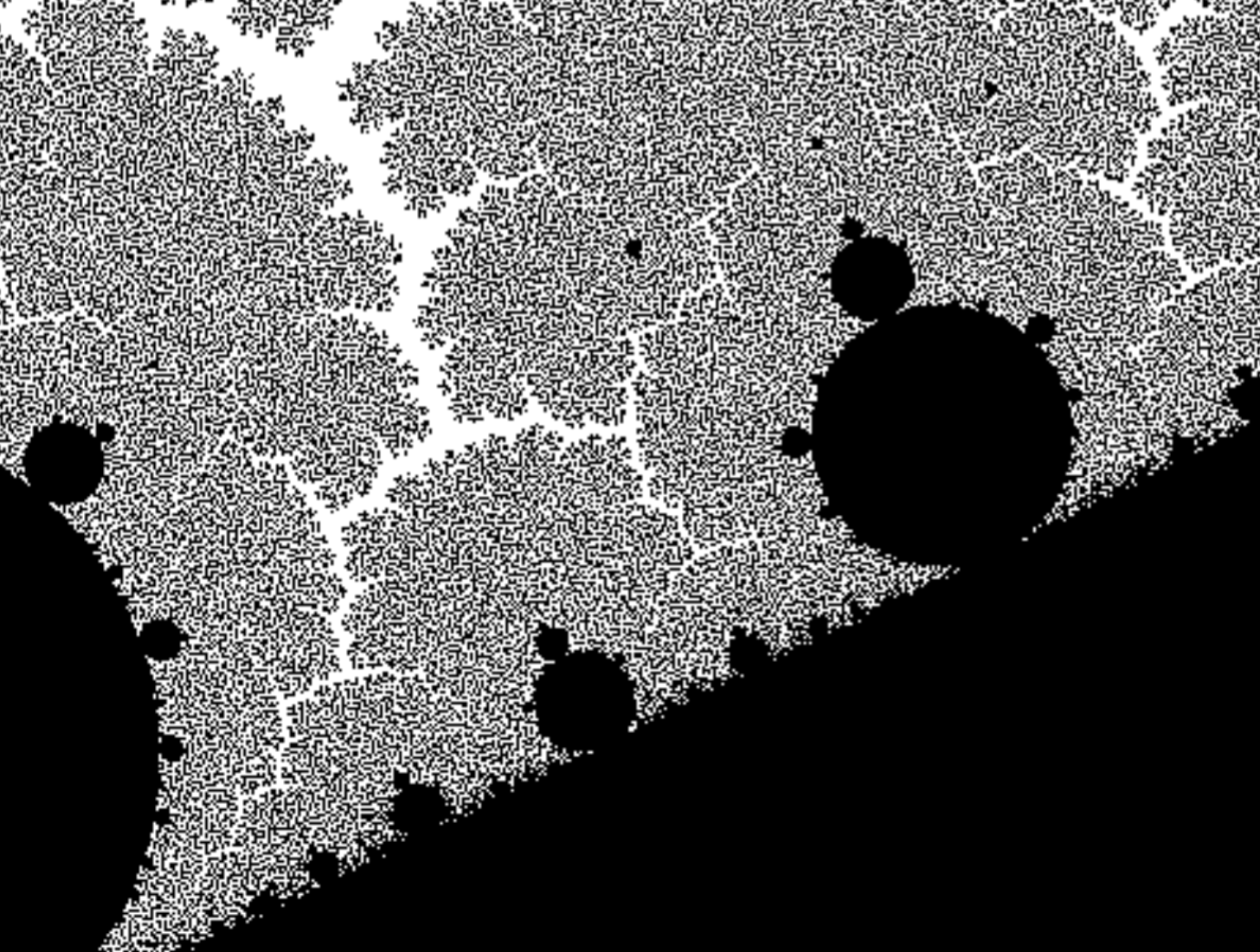}};  
   \draw[white] node at (3.3,-10.2) {$55/144$}; 
 \draw[white] node at (1.7,-12.05)  {Golden}; 
      \draw[white] node at (1.7,-12.45)  {Siegel map}; 
     \draw[ ->,white] 
            (1.7,-11.85) 
            .. controls (0.95,-11.65) and (0.8,-11.25) ..
            (0.6,-11.05);    
   
    \draw[white] node at (1.57,-11.05) {$144/377$}; 
     \draw[white] node at (0.27,-11.7) {$89/233$};
     \draw[white] node at (-4.5,-11.5) {$34/89$};

 \end{tikzpicture}}
\caption{Limbs $8/21,21/55, 55/144, 144/377,\dots $ scale geometrically fast on the right-hand side of the (anti-)golden Siegel parameter, while limbs $5/13,13/34,34/89, 89/233,\dots$ scale geometrically fast on the left-hand side. The bottom picture is a zoom of the top picture.}
\vspace{128in} \label{Fig:GeomScal}
\end{figure}

Our renormalization operator acts on the space of ``pacmen'', 
which are holomorphic  maps $f:  (U,\alpha) \ra (V,\alpha)$ between two nested  domains, see Figure \ref{Fg:Pacman},
such that  $f: U\setminus \gamma_0\ra V\sm \gamma_1$ is a double branched covering,
where $\gamma_1$ is  an arc connecting $\alpha$ to $\di V$. 
The {\em pacman renormalization} $\RR f$ of $f$ (see Figure~\ref{Fg:PacmanRenorm}) is defined by removing the sector $S_1$
 bounded by $\gamma_1$ and its image $\gamma_2$,and taking the first return map to the remaining space; see \S\ref{sec:pacmen} for precise definitions. Note that it acts on the rotation numbers as
\begin{equation}
\label{eq:ActOnRotNumb}
\theta\longrightarrow \frac{\theta}{1-\theta}\sp \mbox{ if }0\le \theta \le \frac{1}{2};\quad \sp \theta\longrightarrow\frac{2\theta-1}{\theta} \sp \mbox{if }\frac{1}{2}\le \theta\le 1; 
\end{equation}
the graph of~\eqref{eq:ActOnRotNumb} is shown on Figure~\ref{Fg:eq:ActOnRotNumb}, see Appendix~\ref{ss:ap:SecRen}, in particular, \eqref{eq:R_prm}.

Let us denote by $\Theta_\per$ the set of {\em  combinatorially periodic}
rotation numbers; i.e., rotation numbers periodic under~\eqref{eq:ActOnRotNumb}. Numbers in $\Theta_\per$ belong to the cycles of numbers with periodic continued fraction expansion.

A pacman is called {\em Siegel}  with rotation number $\theta$ if $\alpha$ is a 
Siegel fixed point with  rotation number $\theta$ whose closed Siegel disk is a quasidisk compactly contained in $U$ (subject of extra technical assumption, see Definition~\ref{dfn:SiegPcm}).

\begin{thm}[Hyperbolicity of the Renormalization] \label{renorm thm}
For any rotation number $\theta\in \Theta_\per$,  
the pacman renormalization operator $\RR$ has a unique periodic  point $f_\str$ which is a
Siegel pacman with rotation number $\theta$.    
This periodic point is hyperbolic with one-dimensional unstable manifold. 
Moreover, the stable manifold of $f_\str$ consists of all Siegel pacmen. 
\end{thm}
\noindent The problem of hyperbolicity goes back to the work of physicists \cites{Wi, MN,MP}, see~\S\ref{ss:hist comments} for the description of the previous progress in the area.

\begin{cor}[Stability of Siegel maps]
Let $f$ be a Siegel pacman with rotation number $\theta\in \Theta_\per$. Consider the space $N_\theta(f)$ of maps sufficiently close to $f$ whose $\alpha$-fixed point has rotation number $\theta$. Then the Siegel disk of $g$ depends continuously on $g\in N_\theta(f)$.
\end{cor}
\noindent In fact, the stability of Siegel disks on the unstable manifold is one of the main steps in the proof of Theorem~\ref{renorm thm}.

\begin{figure}
\begin{tikzpicture}
      \draw[->] (-0.5,0) -- (4.8,0)node[right] {$\theta$} ;%
      \draw[->] (0,-0.5) -- (0,4.8) ;
      \filldraw (0,0) circle (0.04cm);
      \node[ below left] at (0,0) {$0$};
      \filldraw (2,0) circle (0.04cm);
      \node[ below ] at (2,0) {$0.5$};
      \filldraw (4,0) circle (0.04cm);
      \node[ below ] at (4,0) {$1$};
      \filldraw (0,4) circle (0.04cm);
      \node[ left ] at (0,4) {$1$};
      \draw[dashed] (0,4) -- (4,4);
     \draw[dashed] (4,4) -- (4,0);
      \draw[scale =4, domain=0:0.5,smooth,variable=\x,red] plot ({\x},{\x + \x*\x +\x*\x*\x +\x*\x*\x*\x+\x*\x*\x*\x*\x+\x*\x*\x*\x*\x*\x+\x*\x*\x*\x*\x*\x *\x
      +\x*\x*\x*\x*\x*\x *\x*\x});    
      
       \draw[scale =4, domain=0.5:1,smooth,variable=\x,red] plot ({\x},{2-1/\x});

    \end{tikzpicture}
\caption{The map~\eqref{eq:ActOnRotNumb} induces a $2:1$ map on $\R/\Z$.}
\label{Fg:eq:ActOnRotNumb}
\end{figure}

 Let $c(\theta)$, $\theta\in \R/\Z$,  be the parameterization of the main cardioid $\CC$ by the rotation number 
$\theta$. At any parabolic point $c(\pp/\qq)$, there is a satellite hyperbolic component $\HH_{\pp/\qq}$ of $\Mandel$
  attached to $c(\pp/\qq)$.  Let $a_{\pp/\qq}$ be the {\em center} of this component, i.e., the unique superattracting
parameter inside $\HH_{\pp/\qq} $. 

In this paper, 
notation $ \alpha_n \sim \beta_n$  will mean that  $\alpha_n/\beta_n \to \const \not= 0 $. 

\begin{thm}[Scaling Theorem]\label{scaling thm}
  Let $\theta\in \Theta_\per$ be a  rotation number,
and let $\pp_n/\qq_n$ be its continued fraction approximands. 
Then 
\[
  | c(\theta)  - a_{\pp_n/\qq_n}  | \sim \frac 1 {\qq_n^2 } . 
\]  
\end{thm}   
\noindent See Theorem~\ref{thm:ScalThm} for a more precise version of the Scaling Theorem. In particular, Theorem~\ref{thm:ScalThm} is stated for any one-parameter space witnessing the bifurcation of a Siegel map.

 Self-similarity of the Mandelbrot set near the (anti-)golden parameter is illustrated on Figure~\ref{Fig:GeomScal}. Theorem~\ref{thm:ScalThm} says that the centers of satellite hyperbolic components indeed scale as the picture predicts. In~\cite{DL} we combine methods of Theorem~\ref{scaling thm} with methods and ideas from transcendental dynamics to obtain a scaling law for a much larger class of parameters. The self-similarity of the whole limbs is still an open question. This problem is closely related to the problem of the realization of parameter rays for the transcendental family on the unstable manifold.

We believe that our methods allow to prove Theorems~\ref{renorm thm} and~\ref{scaling thm} in the case of rotation numbers of
{\em bounded} type, the details will appear elsewhere.  We  conjecture that an analogous statement is true for arbitrary combinatorics, which  would provide us with a good geometric control 
of the {\em molecule} of the Mandelbrot set (see Appendix~\ref{S:ApMolec}).

\subsection{Outline of the proof}
\label{ss:out of the proof}
We let:
\newline \ssk\nin -- $\ee(z)=e^ {2\pi i z};$ 
\newline \ssk\nin 
-- $p_\theta : z\mapsto \ee(\theta) z + z^2$;
\newline\ssk\nin
-- $\PP_\theta$ be the set of pacmen with rotation number $\theta\in \R/\Z$;
\newline \ssk\nin
--  $\Theta_\bnd$ be the set of {\em  combinatorially bounded}
rotation numbers (i.e., rotation numbers with  continued
fraction expansion where all its coefficients are bounded).

Let us first review Siegel Renormalization theory which is the most relevant to our results; for extra historical comments for the progress in this program see~\S\ref{ss:hist comments}.

Any holomorphic map  $f : (U_f, \alpha) \ra (\C, \alpha)$  whose fixed
point $\alpha$ is neutral with rotation
number $\theta\in \Theta_\per$ is locally linearizable near $\alpha$.
Its maximal completely invariant linearization domain $Z_f$ is
called the {\em Siegel disk} of $f$.  
If $\overline Z_f$ is a quasidisk compactly contained in $U_f$ whose boundary
contains exactly one  critical point, then $f$ is called  
a  (unicritical) {\em Siegel map}. 
For any $\theta\in \Theta_\per $, the  quadratic polynomial $p_\theta $ and
any Siegel pacman give examples of Siegel maps (see \S\ref{s:SiegPacm}, in particular Theorem~\ref{thm:SiegJulSet P is LocCon}).

There are two versions of the Siegel Renormalization theory: 
 {\em holomorphic  commuting pairs} renormalization and 
the {\em cylinder renormalization}.  The former was developed by McMullen \cite{McM3} (see also an earlier work by Stirnemann~\cite{St})
who proved, for any rotation number $\theta\in \Theta_\per$,  the existence of a renormalization periodic point $f_\str$
and the exponential convergence  of the renormalizations $\RR_\cp^n (p_\theta) $  to the orbit of $ f_\str$.  
McMullen has also studied the maximal domain of analyticity for $f_\str$.

The cylinder renormalization $\RR_\cyl$  was introduced by
Yampolsky who showed that $f_\str$ can be transformed into a periodic point for $\RR_\cyl$ 
with a {\em finite codimension} stable manifold $\WW^s(f_\str)$  
and {\em at least one-dimensional} unstable manifold $\WW^u(f_\str)$ 
\cite{Ya-posmeas}. Inou and Shishikura \cite{IS} established hyperbolicity of cylinder renormalization for high type Siegel parameters and Gaidashev and Yampolsky \cite{GY} proved it for the golden rotation number (see~\S\ref{ss:hist comments}). However, the general conjecture that $f_\str$ is hyperbolic 
with $\dim \WW^u(f_\str)=1 $ remained open.

Let us now select our favorite $\theta\in \Theta_\per$;
it is fixed under some iterate of~\eqref{eq:ActOnRotNumb}.  
Then  the  corresponding iterate of the Siegel
renormalization fixes $f_\str$, so below we will refer to the
$f_\str$ as ``renormalization fixed points''.

We start our paper (\S \ref{sec:pacmen}) by discussing an interplay between 
a ``pacman'' and a ``prepacman''. The latter (see Figure~\ref{Fg:Prepacman}) is a piecewise
holomorphic map with two branches $f_\pm : U_\pm \ra S$, one of which
is univalent while the other has ``degree $1.5$'', with a single
critical point. Such an object can be obtained from a pacman by
cutting along the critical
arc $\gamma_1$. For a technical reason, we ``truncate''  both pacmen and prepacmen
by removing a small disk around the co-$\alpha$ point, see Figure~\ref{Fg:TruncPacman}.

Then we define, in three steps,  the pacman renormalization.
 First we define a ``pre-renormalization'' (Definition~\ref{den:PacmRenorm}) as a prepacman obtained 
as the first return map to an appropriate sector $S$. 
Then, by gluing the boundary arcs of $S$,
 we obtain an ``abstract'' pacman. 
Finally, we embed this pacman back to the complex plane.

There are some choices involved in this definition. We proceed to show that near any
renormalizable pacman $f$, the  choices can be made  so that we
obtain a holomorphic operator $\RR$  in a Banach ball (Theorem~\ref{thm:RenOper}).

In Section~\ref{s:SiegPacm}
we analyze the structure of Siegel pacmen $f$. 
The key result is  that any Siegel  map 
can be renormalized (in an appropriate sense) to a Siegel pacman (Corollary~\ref{cor:SP embeds in SM}), where the rotation number changes as an iterate of~\eqref{eq:ActOnRotNumb}, see Lemma~\ref{lem:eq:S3:R_prm}.

In case when  $f=f_\str$  is the Siegel renormalization fixed point,
this provides us with the pacman renormalization fixed point (\S \ref{ss:FixedREnormPts}). 
Moreover, the pacman  renormalization $\RR$  becomes a compact holomorphic operator in a
Banach neighborhood of $f_\str$, with at least one-dimensional  unstable
manifolds $\WW^u(f_\str)$, see Theorem~\ref{thm: An Pacm  self ope}. 

Along the lines, we  introduce and discuss the  associated  geometric objects (\S\ref{ss:LocConnK_p}):
the  pacman ``Julia sets'' $\Kfilled(f)$ and  $\Jul(f)$,
``bubble chains'',  and ``external rays''. 
We also use them to show, via the Pullback Argument,
that any two combinatorially equivalent Siegel pacmen are hybrid equivalent (Theorem~\ref{thm:HyrbrEquivOfPacm}),
i.e.~there is a qc conjugacy between them which is 
conformal on the Siegel disk.

For a Siegel pacman $f_\str$, any renormalization prepacman can be
``spread around'' to provide us with a dynamical tiling of a
neighborhood of the Siegel disk, see~\S\ref{ss:RenTiling} and Figures~\ref{Fg:Prepacman:OrbitOfUi} and~\ref{Fig:RenormTiling:2levels}. Moreover, this tiling is robust under perturbations of $f_\str$, even when the rotation number gets changed, see Theorem~\ref{thm:ComPsConj}.
In this case, the domain filled with the tiles can be used as  the central
``bubble'' for  the perturbed  map $ f$, replacing for many purposes the original Siegel 
disk $Z_\str$  of $f_\str$.  In particular, it allows us to control long-term
$ f^n$-pullbacks of small disks $D$ centered at $\di Z_\str$
(making sure that these pullbacks are not   ``bitten''  
 by the pacman mouth). This is the crucial technical result of this paper (Key Lemma~\ref{keylem:ContrOfPullbacks}).

When $f_\str$ is the renormalization fixed point and the perturbed
map $ f$ belongs to its unstable
manifold $\WW^u(f_\str)$ then we can apply this construction to the
anti-renormalizations $\RR^{-n} f$ . This allows us to
show that  the maximal holomorphic extension of the associated
prepacman is a $\sigma$-proper map $ \bF=(\bbf_\pm: \bX_\pm \ra \C)$,
where $\bX_\pm$ are plane domains  (Theorem~\ref{thm:MaxCommPair:short}).

Applying this result to a parabolic map $f\in W^u(f_\str)$,
we conclude that its attracting Leau-Fatou flower  
contains the critical point, so the critical point is non-escaping 
under the dynamics  (Corollary~\ref{cor:wHp:has:0}).

After this preparation,  we are ready for proving Theorem~\ref{renorm thm}, see \S\ref{s:HypThm}.
Assuming for the sake of contradiction that  $\dim \WW^u (f_\str) >1$,
we can find a holomorphic curve $\Gamma_\str \subset  \WW^u (f_\str) $
through $f_\str$ 
consisting of Siegel pacmen with the same rotation number. 
 Approximating  this curve  with  parabolic curves $\Gamma_n\subset \WW^u(f_\str)$,
we conclude that the critical point is non-escaping for $f\in
\Gamma_\str$.
This allows us to apply Yampolsky's holomorphic motions argument \cite{Ya-posmeas} to show that
$\dim \WW^u (f_\str) =1$.

Finally, using the Small Orbits argument of \cite{L:FCT},
we prove that $f_\str$ is hyperbolic under the pacman renormalization,
completing the proof. 

Along the lines we prove the stability of Siegel maps (see Corollary~\ref{cor:Sieg disk depends cont}): if a small perturbation of a Siegel map $f$ fixes the multiplier of the $\alpha$-fixed point, then the new map $g$ is again a Siegel map. Moreover, the Siegel quasidisk $\overline Z_g$ is in a small neighborhood of $\overline Z_f$.

To derive Theorem \ref{scaling thm} from Theorem \ref{renorm thm}, 
we need to show that the centers of the hyperbolic components in question are represented on the
unstable manifold $\WW^u(f_\str)$. We first show that the roots of these components
are represented on  $\WW^u(f_\str)$
 which requires good control  of the corresponding pacmen Julia sets
 (see \S \ref{ss:ValFlower}), and robustness of the renormalization with respect to a particular choice of cutting arcs, see Appendix~\ref{s:SectRenAndAntiren}.
Then we use quasiconformal deformation techniques to reach the
desired centers from the parabolic points,
see  \S\ref{s:ScalinThm}.

Throughout the paper we use Appendix~\ref{s:SectRenAndAntiren} containing a topological preparation justifying robustness
 of the anti-renormalizations with respect to the choice of cutting
 arcs.

In Appendix~\ref{S:ApMolec} we formulate the Molecule Conjecture on existence of a pacman hyperbolic operator with the one-dimension unstable foliation whose horseshoe is parametrized by the boundary of the main molecule of the Mandelbrot set. This conjecture would imply the MLC for all infinitely renormalizable parameters of satellite type.

\subsection{More historical comments}
\label{ss:hist comments}

 Renormalization of Siegel  maps appeared first in the work by physicists (see \cites{Wi, MN,MP})
as a mechanism for self-similarity of the golden mean Siegel disk near the critical point.
A few years later, Douady and Ghys  discovered a surgery that reduces
previously unaccessible  geometric problems for Siegel disks%
\footnote{The original  surgery applies to Siegel  polynomials only. Its
  extension to general Siegel maps leads to {\em quasicritical} circle maps,
  see \cite{AL-posmeas}.}  
 of bounded type to much better understood problems for critical circle maps. 
This led, in particular, to the local connectivity result for Siegel Julia sets of bounded type (Petersen \cite{Pe}) and
 also  became a  key to the mathematical study of the Siegel renormalization.
In particular,  McMullen-Yampolsky theory mentioned above (see~\S\ref{ss:out of the proof}) is based upon this machinery. 

Holomorphic commuting pairs (as well as almost commuting holomorphic pairs)
were studied by Stirnemann. In~\cite{St}, he gave a computer-assisted proof of the existence of a renormalization fixed point with a golden-mean Siegel disk and showed that the  quadratic
polynomial with a golden-mean Siegel disk converges to that fixed point. Recently, Gaidashev and Yampolsky gave a computer-assisted proof of the hyperbolicity of the renormalization for the golden mean rotation number~\cite{GY}.

On the other hand, in the mid 2000's, Inou and Shishikura proved the existence and hyperbolicity
 of Siegel renormalization fixed points {\em of  sufficiently high combinatorial type} using a completely
different approach, based upon the parabolic perturbation theory \cite{IS}. For a different viewpoint on this result see~\cite{Ya-posmeas}. 

The proof in~\cite{IS} involves certain computer estimates. A computer-free proof of hyperbolicity in high type was presented by Cheritat~\cite{Che:nearpar}. His approach also gives a proof of hyperbolicity for high type in the unicritical case $z^d+c$.

The Siegel renormalization theory achieved further prominence when it was used for 
constructing examples of Julia sets of positive area 
(see Buff-Cheritat \cite{BC} and Avila-Lyubich \cite{AL-posmeas}). 

A different line of research emerged in the 1980s in the work of Branner and Douady
who discovered a {\em surgery} that embeds the $1/2$-limb of the Mandelbrot set into the
$1/3$-limb \cite{BD}. This surgery is the prototype for the pacman renormalization 
that we are developing in this paper. 

Note also that according to the Yoccoz inequality, the $\pp/\qq$-limb of the Mandelbrot set has size
 $O(1/\qq)$. It is believed, though, that  $ 1/\qq^2 $  is the right scaling.
 The pacman renormalization can eventually provide an insight into this problem.  

\begin{rem}
Genadi Levin has informed us about his unpublished work where it is proven, by different methods, that 
\begin{equation}
\label{eq:est:apq and cpq} |a_{\pp/\qq}-c(\pp/\qq)| \le \frac{C}{ \qq^2}, \sp\sp\sp\sp  C>0 
\end{equation} where $a_{\pp/\qq}$ is the center of the $\pp/\qq$-satellite hyperbolic component and $c(\pp/\qq)$ is its root. He has also informed us that~\eqref{eq:est:apq and cpq} has been independently established by Mitsuhiro Shishikura. Note that Theorem~\ref{scaling thm} gives a precise asymptotics for $|a_{\pp/\qq}-c(\pp/\qq)|$.
\end{rem}

\subsection{Notation}
\label{ss:Notations}
We often write a partial map as $f\colon W\dashrightarrow W$; this means that $\Dom f\cup \Im f \subset W$.

A \emph{simple arc} is an embedding a closed interval. We often say that a simple arc $\ell\colon [0,1]\to \C$ \emph{connects} $\ell(0)$ and $\ell(1)$. A \emph{simple closed curve} or a \emph{Jordan curve} is an embedding of the unit circle. A \emph{simple curve} is either a simple closed curve or a simple arc.

A \emph{closed topological disk} is a subset of a plane homeomorphic to the closed unite disk. In particular, the boundary of a closed topological disk is a Jordan curve. A \emph{quasidisk} is a closed topological disk qc-homeomorphic to the closed unit disk.

Given a subset $U$ of the plane, we denote by $\intr U $ the interior of $U$.

Let $U$ be a closed topological disk. For simplicity we say that a homeomorphism $f\colon U\to \C$ is \emph{conformal} if $f\mid \intr U $ is conformal. Note that if $U$ is a quasidisk, then such an $f$ admits a qc extension through $\partial U$.

A \emph{closed sector}, or \emph{topological triangle} $S$ is a closed topological disk with two distinguished simple arcs $\gamma_-~,\gamma_+$ in $\partial S$ meeting at the \emph{vertex $v$ of $S$} satisfying $\{v\}=\gamma_-\cap \gamma_+$. Suppose further that $\gamma_-~,\intr S, \gamma_+$ have clockwise orientation at $v$. Then $\gamma_-$ is called the \emph{left boundary of $S$} while $\gamma_+$ is called the \emph{right boundary} of $S$. A closed \emph{topological rectangle} is a closed topological disk with four marked sides.

Let $f : (W, \alpha) \ra (\C, \alpha)$  be a holomorphic map with a distinguished $\alpha$-fixed point. We will usually denote by $\lambda$ the multiplier at the $\alpha$-fixed point. If $\lambda=\ee(\phi)$ with $\phi\in \R$, then $\phi$ is called the \emph{rotation number of $f$}. If, moreover, $\phi=\pp/\qq\in \Q$, then $\pp/\qq$ is also the \emph{combinatorial rotation number}: there are exactly $\qq$ local attracting petals at $\alpha$ and $f$ maps the $i$-th petal to $i+\pp$ counting counterclockwise.

Consider a continuous map $f\colon U\to \C$ and let $S\subset \C$ be a connected set. An \emph{$f$-lift} is a connected component of $f^{-1}(S)$. Let \[x_0,x_1,\dots x_n,\sp x_{i+1}=f(x_i)\]
be an $f$ orbit with $x_n\in S$. The connected component of $f^{-n}(S)$ containing $x_0$ is called the \emph{pullback of $S$ along the orbit $x_0,\dots, x_n$.}

To keep notations simple, we will often suppress indices. For example, we denote a pacman by $f\colon U_f\to V$, however a pacman indexed by $i$ is denoted as $f_i\colon U_i\to V$ instead of $f_i\colon U_{f_i}\to V$.

Consider two partial maps $f\colon  X\dashrightarrow X$ and $g\colon Y\dashrightarrow Y$. A  homeomorphism $h\colon X\to Y$ is \emph{equivariant} if 
\begin{equation}
\label{eq:defn:equiv}
 h\circ f(x)=g\circ h (x)
 \end{equation}
for all $x$ with $x\in \Dom f$ and $h(x)\in \Dom g$. If~\eqref{eq:defn:equiv} holds for all $x\in T$, then we say that $h$ is \emph{equivariant on $T$.} 

We will usually denote an analytic renormalization operator as ``$\RR$'', i.e.~$\RR f$ is a renormalization of $f$ obtained by an analytic change of variables. A renormalization postcomposed with a straightening will be denoted by ``$\mRR$''; for example, $\mRR_{s} \colon \MM_s\to \MM$ is the Douady-Hubbard straightening map from a small copy $\MM_s$ of $\MM$ to the Mandelbrot set. The action of the renormalization operator on the rotation numbers will be denoted by ``$\cRR$''.

Slightly abusing notations, we will often identify a triangulation (or a lamination) with its support.

\subsection*{Acknowledgments} The first author was partially supported by Simons Foundation grant at the IMS, DFG grant BA4197/6-1, and ERC grant ``HOLOGRAM''. The second author thanks the NSF for their continuing support.

The results of this paper were first announced at the North-American Workshop in Holomorphic Dynamics, May 27--June 4, 2016, Canc\'un, M\'exico.

Figures~\ref{Fig:GeomScal},~\ref{Fg:Bubble Chain},~\ref{Fig:SattValFlow},~\ref{Fg:Molec} are made with W. Jung's program \emph{Mandel}.

\section{Pacman renormalization operator}
\label{sec:pacmen}

\begin{defn}[Full pacman]
\label{dfn:Pacman}
Consider a closed topological disk $\overline{V}$ with a simple arc  $\gamma_1$ connecting a boundary point of $V$ to a point $\alpha$ in the interior. We will call $\gamma_1$ the \emph{critical arc} of the pacman.

A \emph{full pacman} is a map \[f:\overline{U}
\to \overline{V}\] such that (see  Figure~\ref{Fg:Pacman})
\begin{itemize}
\item $f(\alpha)=\alpha$;
\item   $\overline{U}$ is a closed topological disks with $\overline{U}\subset V$;

\item the critical arc $\gamma_1$ has exactly $3$ lifts $\gamma_0\subset U$ and $\gamma_-~,\gamma_+ \subset \partial U$ such that $\gamma_0$   start at the fixed point $\alpha$ while $\gamma_-~,\gamma_+$ start at the pre-fixed point $\alpha'$; we assume that $\gamma_1$ does not intersect $\gamma_0,\gamma_-~,$  $\gamma_+$ away from $\alpha$;
\item $f:U\to V$ is analytic and $f:U\setminus \gamma_0\to V\setminus \gamma_1$ is a two-to-one branched covering;  
\item $f$ admits a locally conformal extension through $\partial U\setminus \{\alpha'\}$.
\end{itemize}
\end{defn}

Since $f:U\setminus \gamma_0\to V\setminus \gamma_1$ is a two-to-one branched cover, $f$ has a unique critical point, called $c_0(f)$, in $U\setminus \gamma_0$. We denote by $c_1(f)$ the image of $c_0$.


We will mostly consider  truncated pacmen or simply pacmen defined as follows. Consider first a full pacman $f\colon U\to V$ and let $O$ be a small closed topological disk around $\alpha \in \intr O\not\ni c_1(f)$ and assume that $\gamma_1$ cross-intersects $\partial O$ at single point. Then $f^{-1}(O)$ consists of two connected components, call them $ O_0\ni \alpha$ and $O'_0\ni \alpha'$. We obtain a truncated pacman
\begin{equation}
\label{eq:TrunPacm}
f\colon (U\setminus O'_0, O_0)\to (V,O).
\end{equation}
 A \emph{pacman} is an analytic map as in~\eqref{eq:TrunPacm} admitting a locally conformal extension through $\partial U$ such that $f$ can be topologically extended to a full pacman, see Figure~\ref{Fg:TruncPacman}. In particular, every point in $V\setminus O$ has two preimages while every point in $O$ has a single preimage.

\begin{figure}[t!]
\[\begin{tikzpicture}[scale=1.3]

\coordinate  (w0) at (-3.9,1.5);

\draw (w0) circle (0.2cm);
\draw (-3.5,1.5) node {$O$};

\node (v0) at (-2.9,1.4)  {};

\draw  (v0) ellipse (3.5 and 2);

\coordinate (v1) at (-2.2,1.3) {};
\coordinate (v1a) at (-2.02,1.39) {};
\coordinate (v1b) at (-2.02,1.21) {};

\coordinate (v2) at (-0.6,2.1) {} {};
\coordinate (v3) at (-0.6,0.6) {} {};

  \draw[ line width=0.5pt,orange] 
            (v3) 
            .. controls (v3) and (v1b) ..
            (v1b);
    \draw[orange, line width=0.5pt]           
            (v1a).. controls (v1a) and (v2)..(v2);
    \draw[line width=0.5pt,red]  
            (v2)
            .. controls (-7.3,4.4) and (-7.1,-1.5) ..
            (v3);

\draw [orange, shift={(-2.2,1.3)}, scale=0.2] plot[domain=-2.74468443448354:2.6325217651674095,variable=\t]({-1.*1.0089598604503545*cos(\t r)+0.*1.0089598604503545*sin(\t r)},{-0.*1.0089598604503545*cos(\t r)+1.*1.0089598604503545*sin(\t r)});

\coordinate (w1) at (-4.7,2.4) {} {};

\draw[line width=0.5pt,red]  (w0)--(w1);

\coordinate (w2) at (-6.1,0.6)  {};
\draw  [line width=0.5pt]  (w0)--(w2);

\draw[line  width=0.5pt,-latex] (-3.1,1.6) .. controls (-2.6,2.1) and (-3.1,3.2) .. (-3.8,2.8);

\node at (-5.5,0.5) {$\gamma_1$};
\node at (-2.8,2.7) {$f$};
\node[red] at (-4.5,1.9) {$\gamma_0$};

\node at (-1.9,4.6) {};
\node[red] at (-3.5,0.8) {$U$};
\node at (-1.6,0.1) {$V$};
\node[red] at (-3.6,-0.05) {$\partial ^\ext U$};
\node[orange] at (-1.1,1.3) {$\partial ^\frb U$};
\end{tikzpicture}\]
\caption{A pacman is a truncated version of a full pacman, see Figure~\ref{Fg:Pacman}; it is an almost $2:1$ map $f:(U,O_0)\to (V,O)$ with $f(\partial U)\subset \partial V\cup \gamma_1\cup \partial O$.}
\label{Fg:TruncPacman}
\end{figure}
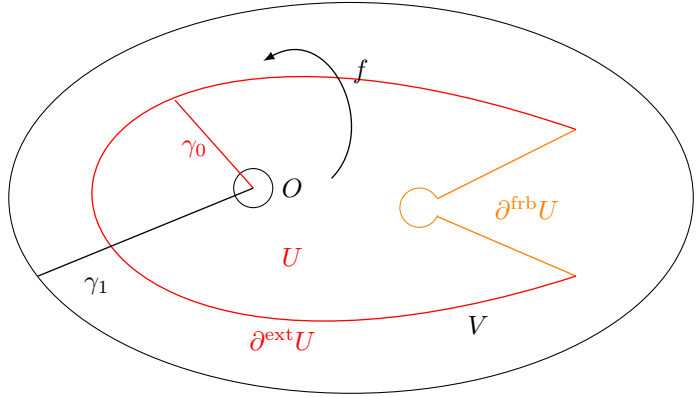

\subsection{Dynamical objects}
\label{ss:DynObj}
Let us fix a pacman $f: U \to V$. Note that  objects below are sensitive to   small deformations of $\partial U$.

The \emph{non-escaping} set of a pacman  is  \[\Kfilled_{f}:=\bigcap_{n\ge 0} f^{-n}(\overline U).\] The \emph{escaping set} is $V\setminus \Kfilled_f$. 

We recognize the following two subsets of the boundary of $U$:  the \emph{external boundary} $\partial^{\ext} U := f^{-1}(\partial V)$ and \emph{the forbidden part of the boundary} $\partial^{\frb} U:= \overline {\partial U\setminus \partial^{\ext} U}$.

Suppose $\ell_0:[0,1]\to \overline V$ is an arc connecting a point in $\Kfilled_f$ to $\partial V$. We define inductively images $\ell_m:[0,1]\to V$ for $m\le M\in \{1,2,\dots, \infty\}$ as follows. Suppose $t_m\le 1$ be the maximal such that the image of $[0,t_m]$ under $\ell_m$ is within $\overline U$. If $\ell_m(t_m)\in \partial^{\ext} U$, then we say $\ell_{m+1}$ is \emph{defined} and we set $\ell_{m+1}(t):=f(\ell_m(t/t_m))$ for $t\le 1$.  Abusing notation, we write \[\ell_m= f(\ell_{m-1}).\]

We define \emph{external rays} of a pacman in the following way. Let us embed a rectangle $\Rect$ in $\overline{V} \setminus U$ so that bottom horizontal side $B$ is equal to $\partial^{\ext} U $ and the top horizontal side $T$ is a subset of $\partial V$. The images of the vertical lines within $\Rect$ form a lamination of $\overline{V} \setminus U$. We pull back this lamination to all iterated preimages $f^{-n}(\Rect)$. Leaves of this lamination that start at $\partial V$ are called \emph{external ray segments} of  $f$; infinite external ray segments are called \emph{external rays} of $f$. Note that if $\gamma$ is an external ray, then $f(\gamma)$, as defined in the previous paragraph, is also an external ray.

We have two maps from $B$ to $T$: one is the natural identification $\pi$ along the vertical lines, the other is the map $f\colon B \dashrightarrow T$ which is defined only on $f^{-1}(T)$. Composition thereof, $\phi = \pi^{-1} \circ f \colon B \dashrightarrow B$ is a partially defined two-to-one map. We consider the set $\mathcal{A}\subset B$ of all the points with whole forward orbits are well defined. Then $\mathcal{A}$ is completely invariant and there is a unique orientation preserving map $\theta \colon \mathcal{A} \to \mathbb{S}^1$   which semi-conjugates  $\phi \colon \mathcal{A} \to \mathcal{A}$ to the doubling map of the circle.
We say that $\theta(a)$ is the \emph{angle} of the  external ray segment passing through the point $a$.

An external ray segment passing through a point $a\in \mathcal{A}$ is infinite (i.e.~it is an external ray) if and only if it hits neither an iterated precritical point nor an iterated lift of $\partial^{\frb} U$. The latter possibility is a major technical issue we have to deal with.

\subsection{Prime pacman renormalization}
\label{sec:PrimePacmenRenorm}  
Let us first give an example of a prime renormalization of full pacmen where we cut out the sector bounded by $\gamma_1$ and $\gamma_2$, see Figure~\ref{Fg:PacmanRenorm}. This renormalization is motivated by the surgery procedure that Branner and Douady \cite{BD} used to construct a map between the Rabbit $\mathcal{L}_{1/3}$ and the Basilica  $\mathcal{L}_{1/2}$ limbs of the Mandelbrot set, see Appendix~\ref{ss:ap:BranDouad}. Pacman renormalization will be defined in~\S\ref{sec:PacmenRenorm}.

Recall that a sector $S$ is a closed topological disk with two distinguished  arcs in $\partial S$ meeting at single point, called the vertex of $S$. Suppose $f:U\to V$ is a full pacman and  
\begin{itemize}
\item[(A)] $\gamma_0$, $\gamma_1$, and $\gamma_2 :=f(\gamma_1)$ are mutually disjoint except for the fixed point $\alpha$.
\end{itemize}
Denote by  $S_{1}$ the closed sector of $V$ bounded by $\gamma_1\cup \gamma_2$ and not containing $\gamma_0$. Let us further assume that 
\begin{itemize}
\item[(B)] $S_{1}$ does not contain the critical value; and
\item[(C)]  $\gamma_-\cup \gamma_+\subset V\setminus S_{1}$. 
\end{itemize}

 Let $\widehat V$ be the Riemann surface with boundary obtained from $\overline{V}\setminus \intr{S_{1}}$ by gluing $\gamma'_1:= f^{-1}(\gamma_2)\cap \gamma_1$ and $\gamma_2$ along $f$. This means that there is a quotient map 
 \[\psi: \overline{V}\setminus\intr{S_{1}}\to \widehat V\] such that $\psi$ is conformal in  $V\setminus S_1$ while $\psi(z)=\psi(f(z))\in \widehat V$ for all $z\in \gamma'_1$. Let us select an embedding  $\widehat V \hookrightarrow \C$.

The sector $S_{1}$ has two $f$-lifts; let $S_0$ be the lift of $S_{1}$ attached to $\alpha$ and let $S'_0$ be the lift of $S_{1}$ attached to $\alpha'$. Condition (B) implies that $\gamma_-\cup \gamma_+\subset V\setminus S_{0}$.
 Define
 \[\bar f(z):=\begin{cases}
f(z), & \text{ if } z\in U\setminus (S_{1}\cup S_{0}\cup  S'_{0})\\
f^2(z)& \text{ if } z\in S_{0}\cap f^{-1}(U).
\end{cases}.\]
Then the map $\bar f$ descends via $\psi$ into a full pacman $\hat f: \widehat U\to \widehat V$ with the critical ray $\hat\gamma_1$.

\begin{figure}[t!]

\centering{\[\begin{tikzpicture}[ scale=0.88]

\coordinate  (w0) at (-3.9,1.5);

\node (v0) at (-2.9,1.4)  {};

\draw  (v0) ellipse (3.5 and 2);

\node[shift={(0.15,0)}]   at (w0)  {$\alpha$};

\coordinate (v1) at (-2.2,1.3) {};
\coordinate (v2) at (-0.6,2.1) {} {};
\coordinate (v3) at (-0.6,0.6) {} {};

  \draw[ line width=0.5pt,red] 
            (v3) 
            .. controls (v3) and (v1) ..
            (v1)
            .. controls (v1) and (v2) ..
            (v2)
            .. controls (-7.3,4.4) and (-7.1,-1.5) ..
            (v3);

\coordinate (w1) at (-4.7,2.4) {} {};

\draw[line width=0.5pt,red]  (w0)--(w1);

\coordinate (w2) at (-5.4,0.9) {}  {};

\draw[line  width=0.5pt,-latex] (-3.1,1.6) .. controls (-2.6,2.1) and (-3.1,3.2) .. (-3.8,2.8);

\node at (-5.7,0.9) {$\gamma_1$};
\node at (-2.8,2.7) {$f$};

\node at (-1.9,4.6) {};
\node[red] at (-3.5,0.8) {$U$};
\node at (-1.6,0.1) {$V$};
\node at (-2.3,1.3) {$\alpha'$};

\coordinate (w5) at (-6.1,0.6) {} {};
\draw  [line width=0.5pt,blue]  (w0)--(w2);
\draw  (w2) edge (w5);
\coordinate (w6) at (-4.2,-0.5) ;
\draw  [line width=0.5pt,blue]  (w0)--(w6);
\coordinate (w7) at (-1.7,2.4) {};
\draw [blue] (v1) edge (w7);
\draw[red]  plot[smooth, tension=.7] coordinates {(-5.1,1) (-5.2,1.5) (-5,2) (-4.5,2.2)};
\node[blue] at (-3.9,0.3) {$\gamma_2$};

\node[red] at (-4.7,1.8) {$f^2$};
\node[blue] at (-4.8,0.2) {(delete)};
\node[blue] at (-4.7,0.6) {$S_{1}$};
\node[blue] at (-1.4,2) {$S'_{0}$};
\node[red] at (-5.15,1.4) {$S_{0}$};

  \begin{scope}[shift={(-0.4,0)}, scale =1]
\coordinate (w8) at (5.9,1.6) {} {} {};
\coordinate (w9) at (3.6,1.5) {} {};
\draw[red] (2.6,2.3) -- (w9);
\draw (w9);
\draw[blue] (w9) -- (2.6,0.2);
\draw (2.2,-0.3) -- (2.6,0.2);
\draw (2.2,-0.3) .. controls (-1.6,3.1) and (7.6,5.4) .. (7.6,1.4);
\draw (2.6,0.2) .. controls (3.5,-0.7) and (7.5,-1.2) .. (7.6,1.4);
\draw[blue] (w8) -- (5.9,2.6);
\draw[red] (w8) -- (6.9,1.2);
\draw[red] (5.9,2.6) .. controls (4.8,3.3) and (3.5,2.9) .. (2.6,2.3);
\draw[red] (2.8,2.1) .. controls (2.4,1.8) and (2.4,1.1) .. (3,0.6);
\draw[red] (3,0.6) .. controls (4.1,-0.2) and (6.1,-0.1) .. (6.9,1.2);
\node[shift={(0.15,0)}]  at (3.6,1.5) {$\alpha$};
\node at (5.9,1.6) {$\alpha'$};
\node[red] at (4.6,1.3) {$\widehat U$};
\draw[line  width=0.5pt,-latex] (4.5,2) .. controls (4.6,2.6) and (4,3) .. (3.5,3.1);
\node at (4.4,2.8) {$\hat f$};
\node at (6.5,0.3) {$\widehat V$};
\node[blue] at (3.4,0.9) {$\hat \gamma_1$};

\end{scope}
\end{tikzpicture}\]}

\caption{Prime renormalization of a pacman: delete the sector $S_{1}$, forget in $U$ the sector $S'_{0}$ attached to $\alpha'$, and iterate $f$ twice on $S_{0}$. By gluing $\gamma_1$ and $\gamma_2$ along $f:\gamma_1\to \gamma_2$ we obtain a new pacman $\hat f:\widehat U\to \widehat V.$
 }
\label{Fg:PacmanRenorm}
\end{figure}
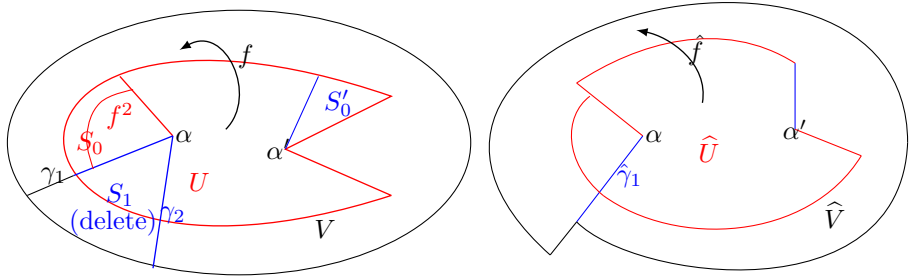

\subsection{Pacman renormalization}
\label{sec:PacmenRenorm} 
 Let us start with defining an analogue of commuting pairs for pacmen.

A map $\psi \colon S \to \overline{V}$ from a closed sector $(S,\beta_-~,\beta_+)$ onto a closed topological disk $\overline{V}\subset \C$ is called a \emph{gluing} if $\psi$ is conformal in the interior of $S$, $\psi(\beta_-)=\psi(\beta_+)$, and  $\psi$ can be conformally extended to a neighborhood of any point in $\beta_-\cup \beta_+$ except the vertex of $S$.


\begin{defn}[Prepacmen, Figure~\ref{Fg:Prepacman}]
\label{defn:Prepacmen} Consider a sector $S$ with boundary rays $\beta_-~, \beta_+$ and with an interior ray $\beta_0$ that divides $S$ into two subsectors $T_-~, T_+$.  Let $f_- \colon U_- \to S,\sp f_+ \colon U_+ \to S$ be a pair of holomorphic maps, defined on $U_- \subset T_-~,U_+ \subset T_+$. We say that $F=(S,f_-~,f_+)$ is a \emph{prepacman} if there exists a gluing $\psi$ of $S$ which projects $(f_-~,f_+)$ onto a pacman $f\colon U \to V$ where $\beta_-~, \beta_+$ are mapped to the critical arc $\gamma_1$ and $\beta_0$ is mapped to $\gamma_0$. 

The map $\psi$ is called a \emph{renormalization change of variables.}
\end{defn}
The definition implies that $f_-$ and $f_+$ commute in a neighborhood of $\beta_0$. Note that every pacman  $f\colon U\to  V$ has a prepacman obtained by cutting $V$ along the critical arc $\gamma_1$.

Dynamical objects (such as the non-escaping set) of a prepacman $F$ are preimages of the corresponding dynamical objects of $f$ under $\psi$.

\begin{defn}[Pacman renormalization, Figure~\ref{Fg:Prepacman:OrbitOfUi}]
\label{den:PacmRenorm} We say that a pacman ${f\colon U\to   V}$ is \emph{renormalizable} if there exists a prepacman \[G=(g_-=f^\aa \colon U_-\to S,\sp g_+=f^\bb \colon U_+\to S)\] defined on a sector  $S \subset V$ with vertex at $\alpha$  such that $g_-~, g_+$ are iterates of $f$ realizing the first return map to $S$ and such that the $f$-orbits of $U_-~,U_+$ before they return to $S$ cover a neighborhood of $\alpha$ compactly contained in $U$. We call $G$ the \emph{pre-renormalization} of $f$ and the pacman $g\colon \widehat{U}\to   \widehat{V} $ is the \emph{renormalization} of $f$.

The numbers $\aa,\bb$ are the \emph{renormalization return times}.

The renormalization of $f$ is called \emph{prime} if $\aa+\bb=3$. 
\end{defn}
\noindent Similarly, a \emph{pacman renormalization} is defined for any map $f\colon U\to V$ with a distinguished fixed point which will be called $\alpha$. For example, we will show in Corollary~\ref{cor:SP embeds in SM} that any Siegel map is pacman renormalizable.

Combinatorially, a general pacman renormalization is an iteration of the prime renormalization -- see details in Appendix~\ref{ss:ap:SecRen}, in particular Lemma~\ref{lem:SectRen is prime power}.

\emph{We define $\bDelta=\bDelta_G$} to be the union of points in the $f$-orbits of $\overline U_-~,\overline U_+$ before they return to $S$. Naturally, $\bDelta$ is a triangulated neighborhood of $\alpha$, see Figure~\ref{Fg:Prepacman:OrbitOfUi}. We call $\bDelta$ a \emph{renormalization triangulation} and we will often say that $\bDelta$ is obtained by \emph{spreading around} $U_-~,U_+$.

\begin{defn}[Conjugacy respecting prepacmen]
\label{defn:ConjBetweenPreRenorm}
Let $f$ and $g$ be any two maps with distinguished $\alpha$-fixed points and let $R$ and $Q$ be two prepacmen in the dynamical plane of $f$ and $g$ defining some pacman renormalizations. Let $h$ be a local conjugacy between $f$ and $g$ restricted to neighborhoods of their $\alpha$-fixed points. Then \emph{$h$ respects $R$ and $Q$} if $h$ maps the triangulation $\bDelta_R$ to $\bDelta_Q$ so that the image of $(S_R, U_{R,\pm})$ is $(S_Q,U_{Q,\pm})$.
\end{defn}

\begin{figure}[t!]

\centering{\begin{tikzpicture}[ scale=1.18]


\draw[ draw opacity=0 ,fill=red, fill opacity=0.2]
  (0,0) --(-1.34,3.54)--  (-1.34+ 1.09,3.54+0.72)-- (-3 + 1.09 ,2.4+0.72)--(-2.7 + 1.09 ,1.38+0.72)--(-4.22 + 1.09, 1.6+0.72)--(-4.9127182890505665,1.0896661385188597) --(0,0);

\draw[blue] (0,0) edge node[below] {$\beta_-$}  (-6.34,1.44)
(-6.34,1.44) edge (0.26,5.82)
(0.26,5.82) edge node[right] {$\beta_+$}  (0,0);

\draw[blue] (0,0) edge node[below] {$\beta$}   (-4.78,2.47);

\draw[red] (-2.7+ 1.09 ,0.72+1.38)--(-4.22+ 1.09 ,0.72+ 1.6);
\draw[red] (-4.9127182890505665,1.0896661385188597)-- (-4.22+ 1.09 ,0.72+ 1.6);

\draw[red,fill=red, fill opacity=0.1] (-2.7+ 1.09 ,0.72+1.38)--(-3+ 1.09 ,0.72+2.4)--(-3.74+ 1.09 ,0.72+1.92)--(-2.7+ 1.09 ,0.72+1.38);

\draw[draw opacity=0,fill=green, fill opacity=0.1] (0,0)--(-1.34,3.54)-- (0.21, 4.64)--(0,0);

\draw[red]  (0,0) --(-1.34,3.54)-- (-3+ 1.09 ,0.72+2.4);

\draw[red] (-1.34,3.54)-- (0.21, 4.64);


\draw[blue] (-4.48, 1.92) node{$S_+$};
\draw[blue] (-1.9, 3.8) node{$S_-$};

\draw[red] (-1.64, 1.28) node{$\Upsilon_-$};
\draw[red] (-3.15+ 1.09 ,0.72+ 1.90) node{$\Upsilon_+$};

\draw (-3.15+ 1.09 ,0.72+ 2.10) edge[->,bend right]  (-4.48, 2.12);


\draw[red] (-0.48, 2.8) node{$U_+$};

\draw(-0.38, 3.1)  edge[->,bend right]  (-4.58, 2.22);

\draw  (-2.34, 1.28)  edge[->,bend left]   node[left ]{$2:1$} (-1.9, 3.6);

\draw[dashed,red] (-1.71,4.51) edge node[right] {$\beta_0$} (-1.34, 3.54);

\draw (-1.74, 4.76) edge[->,bend right] node[above]{gluing} node[below]{ $\beta_\pm\simeq \gamma_1$} (-3.74, 5.2);


\begin{scope}[shift={(-4.5,6.2)} ,yscale=-1,scale=0.7]

\coordinate  (w0) at (-3.9,1.5);

\node (v0) at (-2.9,1.4)  {};

\draw[blue]  (v0) ellipse (3.5 and 2);


\coordinate (v1) at (-2.2,1.3) {};
\coordinate (v2) at (-0.6,2.1) {} {};
\coordinate (v3) at (-0.6,0.6) {} {};

\coordinate (h1) at (-4.7,2.4);
\coordinate (h2) at (-5.44,0.87);
\coordinate (h3) at (-4.04, 0.1);

\coordinate (hh) at (-1.8,0.1); 
\coordinate (hh2) at (-1.7,2.44);

\draw[ draw opacity=0 ,fill=red, fill opacity=0.2]
(hh)
..   controls (-1.4,0.15) and (-1,0.35) ..
(v3)
..   controls (v3) and (v1) ..
(v1);

\draw[ draw opacity=0 ,fill=red, fill opacity=0.2]
(v1)
..   controls (v1) and (w0) ..
(w0)
..   controls (w0) and (h2) ..
(h2)
  ..     controls (-5.24,0.37) and  (-4.64, 0.15)..
              (h3)
        ..   controls (-3.14, 0.0) and  (-2.7,0.)..
           (hh);
\draw[ draw opacity=0 ,fill=red, fill opacity=0.2]
(v1)
..   controls (v1) and (v2) ..
(v2)
   .. controls (-2.1,2.6)  and (-3.2,2.9) ..
        (h1)
..   controls (h1) and (w0) ..
(w0);

\draw[ draw opacity=0 ,fill=green, fill opacity=0.1]
(w0)
..   controls (w0) and (h1) ..
(h1)
    ..   controls (-5.4,2.1) and (-5.64,1.17)..
        (h2);

\filldraw[white] (v1)--(hh2) --(v2);
\filldraw[red, opacity=0.1] (v1)--(hh2) --(v2);

\draw[red, line width=0.5pt]  (v1)--(hh2);

  \draw[ line width=0.5pt,red] 
            (v3) 
            .. controls (v3) and (v1) ..
            (v1)
            .. controls (v1) and (v2) ..
            (v2)
            .. controls (-2.1,2.6)  and (-3.2,2.9) ..
            (h1)
         ..   controls (-5.4,2.1) and (-5.64,1.17)..
            (h2)
        ..     controls (-5.24,0.37) and  (-4.64, 0.15)..
              (h3)
        ..   controls (-3.14, 0.00) and  (-2.7,0.00)..
           (hh)
          ..   controls (-1.4,0.15) and (-1,0.35) ..
           (v3) ;
           



\coordinate (w1) at (-4.7,2.4) {} {};


\draw[red, line width=0.5pt]  (w0)--(w1);

\coordinate (w2) at (-6.1,0.6)  {};
\draw  [blue, line width=0.5pt]  (w0)--(w2);

\draw [blue, line width=0.5pt]  (w0)edge node[right]{$\gamma_2$}(-4.1,-0.5);

           




\node[blue] at (-5.55,0.65) {$\gamma_1$};
\node[red] at (-4.55,1.9) {$\gamma_0$};

\node at (-1.9,4.6) {};
\end{scope}

\end{tikzpicture}}

\caption{A (full) prepacman $(f_-\colon U_-\to S,\sp f_+\colon U_+\to S)$. We have $U_-=\Upsilon_-\cup \Upsilon_+$ and $f_-$ maps $\Upsilon_-$ two-to-one to $S_-$ and $\Upsilon_+$ to $S_+$. The map $f_+$ maps $U_+$ univalently onto $S_+$. After gluing dynamically $\beta_-$ and $\beta_+$ we obtain a full pacman: the arcs $\beta_-$ and $\beta_+$ project to $\gamma_1$, the arc $\beta_0$ projects to $\gamma_0$, and the arc $\beta$ projects to $\gamma_2$.
 }
\label{Fg:Prepacman}
\end{figure}
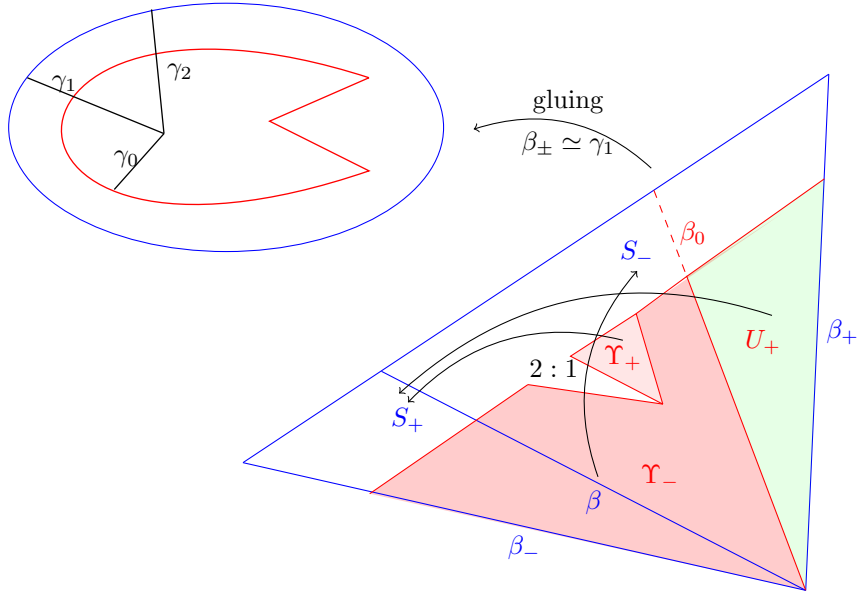

\subsection{Banach neighborhoods}
\label{ss:BanNeig and Ren Oper}
Consider a pacman $f:U_f\to V$ with a non-empty truncation disk $O$. We assume that there is a topological disk $\widetilde U\Supset U_f$ with a piecewise smooth boundary such that $f$ extends analytically to $\widetilde U$ and continuously to its closure.
Choose a small $\varepsilon>0$ and define $N_{\widetilde U}(f, \varepsilon)$ to be the set of analytic maps $g:\widetilde U\to \C$ with continuous extensions to $\partial \widetilde U$ such that
\[\sup_{z\in \widetilde U} |f(z)-g(z)| < \varepsilon.\]
Then $N_{\widetilde U}(f,\varepsilon)$ is a Banach ball.

We say a curve $\gamma$ lands at $\alpha$ at a \emph{well-defined angle} if there exists a tangent line to $\gamma$ at $\alpha$.

\begin{lem}
\label{lmm:PacmN}
Suppose $\gamma_0,\gamma_1$ land at $\alpha$ at distinct well-defined  angles. If $\varepsilon>0$ is sufficiently small, then for every $g\in N_{\widetilde U}(f, \varepsilon)$ there is a domain $U_g\subset \widetilde U$ such that $g:U_g \to V$ is a pacman with the same $V, \gamma_1 , O$ (up to translation). 
\end{lem}
\begin{proof}
For $g\in N_{\widetilde U}(f, \varepsilon)$ with small $\varepsilon$, set \[
\begin{array}{cc}
V(g)&\coloneqq V(f)+(\alpha(g)-\alpha(f)),\\ 
\gamma_1(g)&\coloneqq \gamma_1(f)+(\alpha(g)-\alpha(f)),\\
O(g)&\coloneqq O(f)+(\alpha(g)-\alpha(f)),
 \end{array}
\]  and set $\gamma_0(g)$ to be the lift of $\gamma_1(g)$ landing at $\alpha$. Since $\gamma_0(f),\gamma_1$ land at distinct well-defined  angles, so are $\gamma_0(g),\gamma_1$ if $\varepsilon$ is small; i.e.~$\gamma_0(g),\gamma_1$ are disjoint.

Set $g_\delta = f+\delta(g-f)$ and $T_\delta(z)\coloneqq z+(\alpha(g_\delta)-\alpha(f))$, where $\delta\in[0,1]$. Define $\psi_\delta(z)= g_\delta^{-1} \circ T_\delta \circ f(z)$ on $\partial U_f$ where the inverse branch is chosen so that $\psi_0(z)=z$ and $\psi_\delta(z)$ is continuous with respect to $\delta$. We claim that $\psi_\delta$ is well defined and that $\psi_\delta (\partial U_f)$ is a simple closed curve for all $\delta \in [0,1]$. Indeed, let $A \Subset \widetilde U$ be a closed annular neighborhood of $\partial U_f$ that contains no critical points of $f$. For $\varepsilon$ small enough, the derivative of any $g\in N_{\widetilde U}(f, \varepsilon)$ is uniformly bounded and non-vanishing on a slightly shrunk
 $A$; in particular $g$ has no critical points in $A$.

It follows that 
$\psi_\delta\mid A$ has uniformly bounded derivative and (choosing yet smaller $\varepsilon$, if necessary) is close to the identity map,  hence $\psi_\delta(\partial U_f)\subset A$ is well-defined for all $\delta$. 
Since $f$ has no critical values in $A$, it is locally injective, which implies that $\psi_\delta (x)\neq \psi_\delta (y)$ when $x$ is sufficiently close to $y$. We conclude that $ \psi_\delta$ is injective on $\partial U_f$. Therefore $\psi_1(\partial U_f)$ is a simple closed curve; let $U_g$ be the disk enclosed by $\psi_1(\partial U_f)$.  It is straightforward to check that $g:U_g \to  V $ is a pacman with critical arc $\gamma_1$ and truncation disk $O$.
\end{proof}

Consider a pacman $f\colon U_f\to V$. Applying the $\lambda$-lemma, we can endow all $g\colon U_g\to V$ from a small neighborhood of $f$ with a foliated rectangle $\Rect_g$ as in~\S\ref{ss:DynObj} such that $\Rect_g$ moves holomorphically and the holomorphic motion of $\Rect_g$ is equivariant. As a consequence, an external ray $R(g)$ with a given angle depends holomorphically on $g$ unless $R(g)$ hits an iterated lift of $\partial ^\frb U_g$ or an iterated precritical point.

\begin{lem}[Stability of periodic rays]
\label{lem:PerRaysStable}
Suppose a periodic ray $R(f)$ lands at a repelling periodic point $x$ in the dynamical plane of $f$. Then the ray $R(g)$ lands at $x$  for all $g$ in a small neighborhood of $f$. Moreover, the closure $\overline R(g)$ is contained in a small neighborhood of $\overline R(f)$.
\end{lem}
\begin{proof}
Since $x$ is repelling periodic, it is stable by the implicit function theorem. Present $R(g)$ as a concatenation of arcs $R_1 R_2 R_3\dots$ such that $R_{i+1}(g)$ is an iterated lift of $R_i(g)$. By continuity, $R_i(g)$ is stable for $i\le n$, where $n$ is big if $g$ is sufficiently close to $f$. If $n$ is sufficiently big and $g$ is sufficiently close to $f$, then $R_{n}(g)$ is in a small neighborhood of $x(g)$ and, since $x(g)$ is repelling, $R_{n+1}(g)$ is in an even smaller neighborhood of $x(g)$. Proceeding by induction, we obtain that $R_{n+j}(g)$ shrinks to $x(g)$; i.e.~$R(g)$ lands at $x(g)$. It also follows that $\overline R(g)$ is in a small neighborhood of $\overline R(f)$.
\end{proof}


\subsection{Pacman analytic operator}
\label{ss:Pacm}

Suppose that $\hat f:\widehat U\to  \widehat V $ is a renormalization of $f:U_f\to   V $ via a quotient map  $\psi_f:   S_f\to \widehat V$ that extends analytically through $\partial   S_f\setminus \{\alpha\}$ (this actually follows from the definition of renormalization) where $  S_f\subset V$ is the domain of a prepacman $\widehat F$ such that curves $\beta_0, \beta_+, \beta_-$ all land at $\alpha$ at pairwise distinct well-defined  angles. We claim that there exists an analytic renormalization operator defined on a neighborhood of $f$.

We note that $\beta_\pm =f^{k_\pm}(\beta_0)$ for some integers $k_+,k_-$. For a map $g$ that is sufficiently close to $f$, the fact that the three curves land at different angles implies that $\beta_0, g^{k_+}(\beta_0), g^{k_-}(\beta_0)$ are disjoint. Define $\tau_g \colon \beta_0\cup \beta_- \cup  \beta_+   \to \C$ by \[\tau_g\colon z\mapsto z+\alpha(g)-\alpha(f)\sp \text{ on $\beta_0$}\sp\sp\text{ and }\sp  \sp \tau_g=g^{k_\pm}\circ \tau_g\circ  f^{-k_\pm}  \sp \text{ on $\beta_\pm$}.\] Then $\tau_g$ is an equivariant holomorphic motion of $\beta_0\cup \beta_- \cup  \beta_+ $ over a neighborhood of $f$. By the $\lambda$-lemma \cites{BR,ST} $\tau_g$ extends to a holomorphic motion of $S_f$  over a possibly smaller neighborhood of $f$. Denote by $\mu_g$ the Beltrami differential of $\tau_g$. Define a Beltrami differential $\nu_g$ on $\C$ as $\nu_g=(\psi_f)_* \mu_g$ on $\widehat{V}$ and $\nu_g=0$ outside of $\widehat V$ and let $\phi_g$ be the solution of the Beltrami equation \[ \frac{\partial\phi_g}{\partial \bar z}=\nu_g \frac{\partial\phi_g}{\partial   z}\] that fixes $\alpha$, $\infty$, and the critical value. We see that $\psi_g := \phi_g \circ \psi_f \circ \tau_g^{-1}$ is conformal on $  S_g:=\tau_g( {S})$. It follows, that $\psi_g$ depends analytically on $g$ (see Remark on page 345 of \cite{L:FCT}).

We claim now that $\widehat G = (S_g, g^{k_-}, g^{k_+})$ is a prepacman. Indeed, by definition of $\tau_g$, we have $g^{k_\pm}(\tau_g(\beta_0))=\beta_\pm$ and  $\psi_g$  glues $\widehat G$ to a map $\hat g$ which is close to $\hat f$. By Lemma~\ref{lmm:PacmN}, $\hat g$ restricts to a pacman with the same range as $\hat f$. We have:

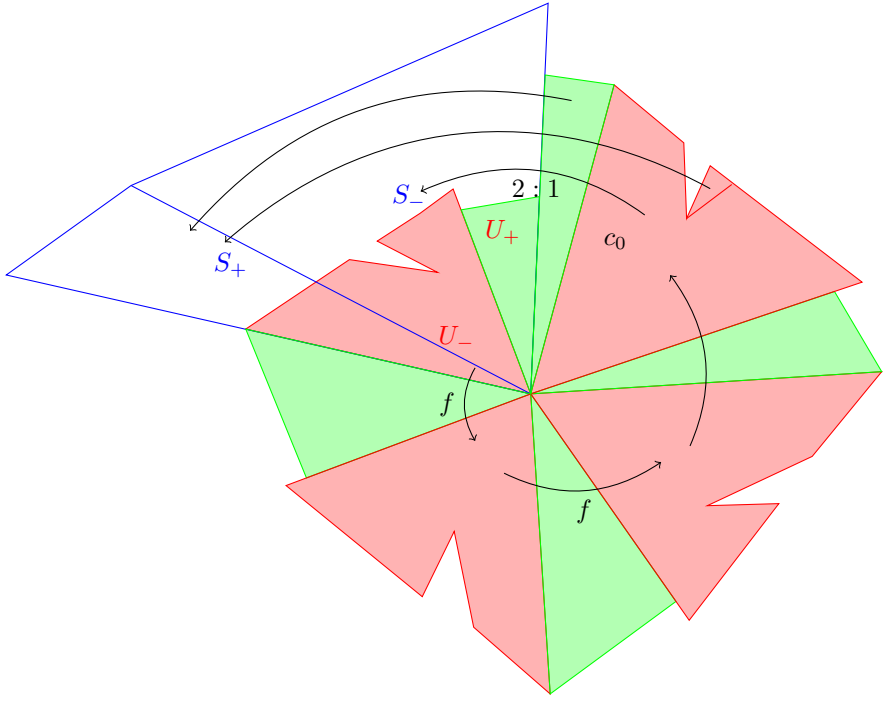
\begin{figure}[t!]

\centering{\begin{tikzpicture}[ scale=0.77]


\draw[red, fill=red, fill opacity=0.3]
  (0,0) --(-1.34,3.54)--  (-3+ 1.09 ,0.72+2.4)--(-3.74+ 1.09 ,0.72+1.92)--(-2.7+ 1.09 ,0.72+1.38)--(-4.22+ 1.09 ,0.72+ 1.6)--(-4.92, 1.12) --(0,0);

\draw[green,fill=green, fill opacity=0.3] (0,0)--(-1.2, 3.18)-- (0.18/1.2, 4.09/1.2)--(0,0);

\draw[blue] (0,0)-- (-9.96/1.1, 2.26/1.1) -- (-6.9, 3.6)--(0.3, 6.75) --(0,0);

\draw[blue] (0,0) edge   (-6.9, 3.6);



\draw[red] (-0.48, 2.8) node{$U_+$};


\draw[green ,fill=green, fill opacity=0.3] (0,0)--(-4.92, 1.12)--(-3.52*1.1,-1.32*1.1)--(0.,0.);
\draw[red, fill=red, fill opacity=0.3,scale=1.2] (0,0)--(-3.52,-1.32)--(-1.06 -0.50062, 0.42016-3.34)--(-0.6 -0.50062, 0.42016-2.4)
--(-0.32 -0.50062, 0.42016-3.78)
--(0.28, -4.32)
--(0.,0.);


\draw[green,fill=green, fill opacity=0.3] (0.,0.)-- (0.28*1.2, -4.32*1.2)-- (2.28*1.1, -3.26*1.1)--(0,0);

\draw[red,fill=red, fill opacity=0.3,scale=1.2] (0,0)--(2.28, -3.26)-- (3.82-0.25,-0.3-1.28)
-- (2.74-0.2,-0.25-1.36)
-- (4.3-0.25,-0.3-0.6)
-- (5.05, 0.32)
-- (0.,0.);

\draw[green,fill=green, fill opacity=0.3] (0,0)--(5.05*1.2, 0.32*1.2)
--(4.7675*1.1, 1.60875*1.1)-- (0.,0.);


\draw[red,fill=red, fill opacity=0.3,scale=1.2] (0,0)--(4.7675, 1.60875)-- (1.58+0.50125*2,-0.6775*2+4.64)-- (1.24+0.50125*2,-0.6775*2+3.88)--(1.2+0.50125*2,-0.6775*2+ 4.97)-- (1.2, 4.45)--(0,0);

\draw[red,scale=1.2] (1.24+0.50125*2,-0.6775*2+3.88)--(1.89+0.50125*2,-0.6775*2+ 4.37);

\draw[green,fill=green, fill opacity=0.3] (0,0)--(1.2*1.2, 4.45*1.2)--(0.25, 5.51)--(0,0);



\draw (0.705, 5.06812)edge[->, bend right]  (-4.68-1.2,+0.5+ 2.32);

\draw[blue] (-4.78-0.4,+0.3+ 1.92) node{$S_+$};

\draw[blue] (-1.7-0.4,-0.3+ 3.7) node{$S_-$};
\draw[red] (-1.58+0.3,+ 0.97) node{$U_-$};

\draw(1.58*1.2+0.50125*2.4,-0.6775*2.4 +4.31*1.2)  edge[->,bend right=35]  (-4.68-0.6,+0.3+ 2.32);

\draw  (-0.96, 0.45) edge[->,bend right]  node[left]{$f$} (-0.56-0.4,+0.3 -1.11); 
\draw  (-0.46, -1.37) edge[->,bend right]  node[below]{$f$} (2.25, -1.18);
\draw  (2.75, -0.9) edge[->,bend right] (2.41, 2.05);
\draw  (1.79*1.1, 2.81*1.1)edge[->,bend right] node[below]{$2:1$} (-1.5-0.4,-0.3+3.8);
\draw[scale=1.1] (1.33, 2.39) node{$c_0$};

\end{tikzpicture}}

\caption{ Pacman renormalization of $f$: the first return map from $U_-\cup U_+$ back to $S=S_-\cup S_+$ is a prepacman. Spreading around $U_\pm$: the orbits of $U_-$ and $U_+$ before returning back to $S$ triangulate a neighborhood $\bDelta$ of $\alpha$; we obtain $f\colon \bDelta\to \bDelta\cup S$, and we require that $\bDelta\cup S$ is compactly contained in $\Dom f$.}
\label{Fg:Prepacman:OrbitOfUi}
\end{figure}

\begin{thm}[Analytic renormalization operator]
\label{thm:RenOper}
Suppose that $\hat f:\widehat U\to  \widehat V $ is a renormalization of $f:U_f\to   V $ via a quotient map  $\psi_f:S_f\to \widehat V$. Assume that the curves $\beta_0, \beta_-,  \beta_+$ (see Definition~\ref{defn:Prepacmen}) land at $\alpha$ at pairwise distinct well-defined angles. Then for every sufficiently small neighborhood $N_{\widetilde U}(f,\varepsilon)$, 
there exists a compact analytic pacman renormalization operator $\mathcal{R}\colon g \mapsto \hat g$ defined on $  N_{\widetilde U}(f,\varepsilon) $ such that $\mathcal{R}(f)=\hat{f}$. Moreover, the gluing map $\psi_g$, used in this renormalization, also depends analytically on $g$.\qed
\end{thm}
\begin{proof}
We have already shown that $\hat g$ depends analytically on $g\in N_{\widetilde U}(f,\varepsilon)$. Choose an intermediate domain $\widetilde U'\Subset \widetilde U$ with $U_f \Subset \widetilde U' \Subset \widetilde U$ so that the operator $\RR$ is the composition of the restriction operator $ N_{\widetilde U}(f,\varepsilon)\to  N_{ \widetilde U'}(f,\varepsilon)$ and the pacman renormalization operator defined on $N_{ \widetilde U'}(f,\varepsilon)$. Since the former is compact, we conclude that $\RR$ is compact.
\end{proof}

\section{Siegel pacmen}
\label{s:SiegPacm}
We say a holomorphic map $f\colon U\to V$ is  \emph{Siegel} if it has a fixed point $\alpha$, a Siegel quasidisk $\overline Z_f\ni \alpha$ compactly contained in $U$, and a unique critical point $c_0\in U$ that is on the boundary of $ Z_f$. Note that in~\cite{AL-posmeas} a Siegel map is assumed to satisfy additional technical requirements; these requirements are satisfied by restricting $f$ to an appropriate small neighborhood of $\overline Z_f$.

 Let us foliate Siegel disk $Z_{f}$ of $f$ by equipotentials parametrized by their heights ranging from $0$ (the height of $\alpha$) to $1$ (the height of $\partial Z_f$). Namely, if $h\colon Z_f\to \Disk$ is a linearizing map conjugating $f\mid Z_f$ and the rotation $z\mapsto \ee(\theta) z$, then the preimage under $h$ of the circle with radius $\eta$ is the equipotential of $Z_f$ at height $\eta$.

\begin{defn}
\label{dfn:SiegPcm}
A pacman $f:U\to V$ is \emph{Siegel} if 
\begin{itemize}
\item $f$ is a Siegel map with Siegel disk $Z_f$ centered at $\alpha$;
\item the critical arc $\gamma_1$ is the concatenation of an external ray $R_1$ followed by an inner ray $I_1$ of $Z_f$ such that the unique point in the intersection $\gamma_1\cap \partial Z_f$ is not precritical;  and
\item writing $f\colon (U\setminus O'_0, O_0)\to (V,O)$ as in~\eqref{eq:TrunPacm}, the disk  $O $ is a subset of $Z_f$ bounded by its equipotential.
\end{itemize}
\end{defn}
The \emph{rotation number} of a Siegel pacman (or a Siegel map)
 is $\theta \in \R/\Z$ so that $\ee(\theta)$ is the multiplier at $\alpha$. It follows (see Theorem~\ref{thm:SiegJulSet P is LocCon}) that the rotation number of Siegel map is in $ \Theta_\bnd$. The level of \emph{truncation} of $f$ is the height of $\partial O$.  
 
 Since $\gamma_1$ is a concatenation of an external ray $R_1$ and an internal ray $I_1$, so is $\gamma_0$: it is a concatenation of an external ray $R_0$ and an internal ray $I_0$ with $f(R_0\cup I_0)=R_1\cup I_1$. Two Siegel pacmen $f\colon U_f\to V_f$ and $g\colon U_g\to V_g$ are combinatorially equivalent if they have the same rotation number and if $R_0(f_1)$ and $R_0(f_2)$ have the same external angles, see~\eqref{ss:DynObj}. Starting from~\S\ref{ss:StandardPacmen} we will normalize $\gamma_0$ so that it passes through the critical value.

A \emph{hybrid conjugacy} between Siegel maps $f_1\colon U_1\to V_1$ and $f_2\colon U_2\to V_2$ is a qc-conjugacy $h\colon U_1\cup V_1\to U_2\cup V_2$ that is conformal on the Siegel disks. A hybrid conjugacy between Siegel pacmen is defined in a similar fashion.  We will show in Theorem~\ref{thm:HyrbrEquivOfPacm} that combinatorially equivalent pacmen are hybrid equivalent.

We will often refer to the connected component $Z'_f$ of $f^{-1}(Z_f)\setminus Z_f$ attached to $c_0$ as \emph{co-Siegel disk}.

\subsection{Local connectivity and bubble chains}
\label{ss:LocConnK_p}
Consider a quadratic polynomial $p_\theta\colon z\mapsto \ee(\theta) z + z^2$. 

\begin{thm}
\label{thm:SiegJulSet P is LocCon}
If $\theta\in \Theta_\bnd$, then the closed Siegel disk $\overline Z$ of $p_\theta$ is a quasidisk containing the critical point of $p_\theta$.

Conversely, suppose a holomorphic map $f\colon U\to V$ with a single critical point has a fixed Siegel quasidisk $\overline Z_f\Subset U\cap V$ containing the critical point of $f$. Then $f$ has a rotation number of bounded type.
\end{thm}
\noindent In particular, $p_\theta$ is a Siegel map. The first part of Theorem~\ref{thm:SiegJulSet P is LocCon} follows essentially from the Douady-Ghys surgery, see~\cite{D-Siegel}. Conversely, if $f\colon U\to V$ is a Siegel map, then applying the inverse Douady-Ghys surjery we obtain a quasicritical circle map $\widetilde f$, see~\cite{AL-posmeas}*{Definitions 3.1}. By~\cites{H,Sw,AL-posmeas}, the restriction of $\widetilde f$ to the unit circle is quasisymmetrically conjugate to the rigid rotation if and only if the rotation number of $\widetilde f$ is bounded. (Compare to~ \cite{GJ}.)

Let us now fix a polynomial $p=p_\theta$ with $\theta\in \Theta_\bnd$. A \emph{bubble} of $p$ is either
\begin{itemize}
\item $Z_0\coloneqq \overline Z_p$, or
\item $Z_1\coloneqq  \overline Z'_p = \overline {p^{-1}(Z_p)\setminus Z_p}$, or
\item an iterated $p$-lift of $\overline Z'_p$ (see~\S\ref{ss:Notations} for the definition of a lift).
\end{itemize}
The \emph{generation} of a bubble $Z_k$ is the smallest $n\ge0$ such that $p^n(Z_k)\subset  Z_0$. In particular, $ Z_0$ has generation $0$ and $ Z_1$ has generation $1$. If the generation of $Z_k$ is at least $2$, then $p\colon Z_k\to p(Z_k)$ admits a conformal extension through $\partial Z_k$ (because  $p(Z_k)\not\ni c_1$).

We say that a bubble $Z_n$ is \emph{attached} to a bubble $Z_{n-1}$ if $Z_{n}\cap Z_{n-1}\not=\emptyset$ and the generation of $Z_{n}$ is greater than the generation of $Z_{n-1}$.

A \emph{limb} of a bubble $Z_k$  is the closure of a connected component of $\Kfilled_p\setminus Z_k$ not containing the $\alpha$-fixed point. A limb of $Z_0=\overline Z_p$ is called \emph{primary}. 

\begin{thm}[\cite{Pe}]
\label{thm:K_P is LocConn}
The filled-in Julia set $\Kfilled_p$ is locally connected. Moreover, for every $\varepsilon>0$ there is an $n\ge 0$ such that every connected component of $\Kfilled_p$ minus all bubbles with generation at most $n$ has diameter less than $\varepsilon$.
\end{thm}
In particular the diameter of bubbles in $\Kfilled_p$ tends to $0$: for every $\varepsilon>0$ there are at most finitely many bubbles with diameter greater than $\varepsilon$. Similarly, the diameter of limbs of any bubble tends to $0$.

An (infinite) \emph{bubble chain of $\Kfilled_p$} is an infinite sequence of bubbles $B=(Z_1,Z_2,\dots )$ such that $Z_1$ is attached to $\overline Z_p$ and $Z_{n+1}$ is attached to $Z_n$, see Figure~\ref{Fg:Bubble Chain}.

As a consequence of Theorem~\ref{thm:K_P is LocConn}, every bubble chain $B=(Z_1,Z_2,\dots )$ \emph{lands}: there is a unique $x\in \Kfilled_p$ such that for every neighborhood $U$ of $x$ there is an $m\ge 0$ such that $\displaystyle\bigcup_{i\ge m}Z_i$ is within $U$. Conversely, if $x\in  \Kfilled_p$ does not belong to any bubble, then there is a bubble chain $B=(Z_1,Z_2,\dots )$ landing at $x$. A point $x$ is periodic if and only if $B$ is periodic: there is an $m>1$ and $q\ge 1$ such that $p^q$ maps $(Z_m,Z_{m+1},\dots)$ to $(Z_1,Z_2,\dots )$.

Let $f\colon U\to V$ be a Siegel pacman. \emph{Limbs, bubbles, and bubble chains for $f$} are defined in the same way as for quadratic polynomials with Siegel quasidisks. In particular, a bubble of $f$ is either $\overline Z_f$, or $\overline Z'_f = \overline {f^{-1}(Z_f)\setminus Z_f}$, or an $f^{n-1}$-lift of $\overline Z'_f$, where $n$ is the \emph{generation} of the bubble. Since $\overline Z_f$ is the only bubble intersecting $\{c_1\}\cup \gamma_1$, all bubbles of positive generation are conformal lifts of $\overline Z'_f$. 

We define the \emph{Julia set} of $f$ as
\begin{equation}
\label{eq:Jul f}
\Jul_f\coloneqq \overline{ \bigcup_{n\ge 0} f^{-n} (\partial Z_f) }.\end{equation}
We will show in Theorem~\ref{thm:LocConnOfJul for StandPacm} that Theorem~\ref{thm:K_P is LocConn} holds for standard Siegel pacmen and that $\Jul_f$ is the closure of repelling periodic points. 

\emph{Limbs, bubbles, and bubble chains of a prepacman $F$} are preimages of the corresponding dynamical objects of $f$.

\begin{figure}[t!]
\centering{\begin{tikzpicture}[ scale=0.88]
  \node at (0,0.1){\includegraphics[scale=0.6]{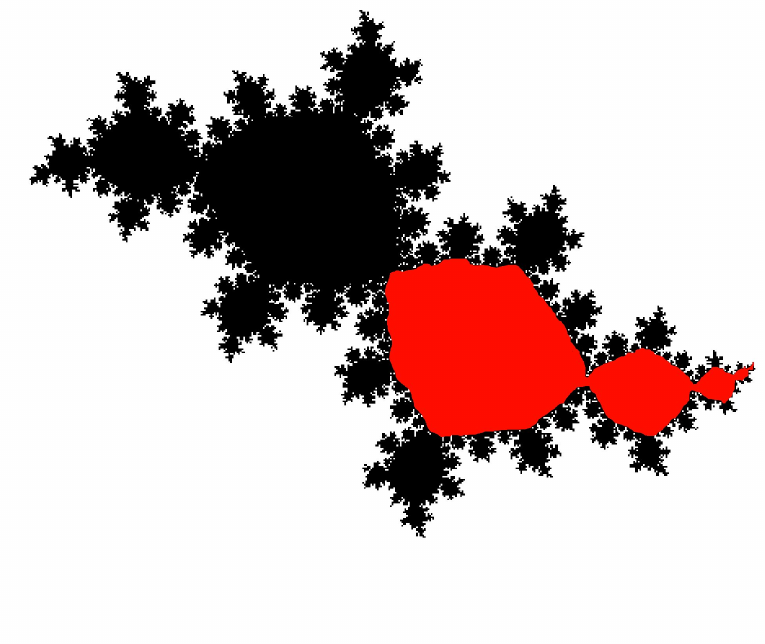}};
  \node[white] at (-0.8,1.5) {$Z_0$};
   \node[white] at (1.1,-0.27) {$Z_1$};
   \node[white] at (3,-0.71) {$Z_2$};
   \node[red,right] at (4.34,-0.5){$\beta$};
\end{tikzpicture}}
\caption{A bubble chain (red) landing at the $\beta$-fixed point.
 }
\label{Fg:Bubble Chain}
\end{figure}
 
\subsection{Siegel prepacmen}
\label{ss:SiegelPrepacmen}
A prepacman $Q$ of a Siegel pacman $q$ is also called \emph{Siegel}; the \emph{rotation number} and \emph{level of truncation} of $Q$ are those of $q$. Recall that $Q$ consists of two commuting maps $q_-\colon U_-\to S_Q,$ $q_+ \colon U_+\to S_Q$ such that $U_{-}$ and $U_+$ are separated by $\beta_0$. Given a Siegel map $f$ we say that $f$ has a \emph{prepacman $Q$ around $x\in \partial Z_f$} if $q_-~,q_+$ are iterates of $f$, the vertex of $S_Q$ is at $\alpha(f)$, and  $\beta_0(Q)$ intersects $\partial Z_{f}$ at $x$.

\begin{lem}
\label{lem:SiegPacmExist}
Suppose that $p$ is a Siegel quadratic polynomial with rotation number $\theta\in \Theta_\bnd$. Consider a point $x\in \partial Z_p$ such that $x$ is neither the critical point of $p$ nor its iterated preimage. Then for every $r\in (0,1)$ and every $\varepsilon>0$, the map $p$ has a Siegel prepacman 
 \begin{equation}
 \label{eq:lem:SiegPacmExist}
   Q=(q_-\colon U_-\to S_Q,\sp q_+ \colon U_+\to S_Q)
\end{equation}
around $x$ such that
\begin{itemize}
\item the rotation number of $Q$ is a renormalization of $\theta$ -- iteration of~\eqref{eq:R_prm};
\item for every $z\in U_{-}\cup U_+$ the orbit $z, p(z), \dots , p^k(z)$ is in the $\varepsilon$-neighborhood of $\overline Z_p$ for $k$ such that $p^k(z)=q_{\pm}(z)$; and  
\item $r$ is the level of truncation of $Q$;
\item every external ray segment (see~\S\ref{ss:DynObj}) of $Q$ is within an external ray of $p$.
\end{itemize}
\end{lem}

\begin{figure}[t!]

\centering{\begin{tikzpicture}[ scale=0.9]

\draw (-4.,0.)-- (0.58,3.04);
\draw (0.58,3.04)-- (4.,0.);
\draw (4.,0.)-- (-0.68,-3.);
\draw (-0.68,-3.)-- (-4.,0.);
\draw (4.,0.)-- (6.,1.56);
\draw (6.02,-1.9)-- (4.,0.);

\draw[ draw opacity=0 ,fill=red, fill opacity=0.1]
  (0,0) --(-1.34,3.54)--  (-1.34,3.54)-- (-3+ 1.09 ,0.72+2.4)--(-2.7+ 1.09 ,0.72+1.38)--(-4.22+ 1.09 ,0.72+ 1.6)--(-4.9127182890505665,1.0896661385188597) -- (-3.047927460454034,0.6319433450261434)--(0,0);

\draw[orange] (0.,0.)-- (-2.24, 1.14);
\draw[orange] (0.,0.)-- (0.15034270964325902,2.754812628234827);
\draw[orange] (0.,0.)-- (-3.047927460454034,0.6319433450261434);

\draw[orange] (0.,0.)-- (-0.8,2.12);

\draw[orange] (-1.98, 0.1) node {$I_-$};

\draw[orange] (-0.36, 1.48) node {$I_0$};
\draw[orange] (-1.5, 1.08) node {$I$};
\draw[orange] (0.34, 1.76) node {$I_+$};


\draw[blue] (-3.047927460454034,0.6319433450261434)-- (-6.34,1.44);
\draw[blue] (-6.34,1.44)-- (0.26,5.82);
\draw[blue] (0.15034270964325902,2.754812628234827)-- (0.26,5.82);

\draw[blue,dashed] (-4.82,2.45) edge  node [above] {$R$}(-6, 3);
\draw[blue]  (-2.24, 1.14) edge   (-4.82,2.45);

\draw[blue, dashed] (-6.34,1.44) edge  node [below] {$R_-$}  (-7.62, 1.76);
\draw[blue,dashed] (0.26,5.82)edge  node [right] {$R_+$} (0.26,6.84);

\draw[red]  (-0.8,2.12) --(-1.34,3.54);
\draw[red, dashed] (-1.34,3.54) edge  node [left] {$R_0$} (-2.28, 5.76);

\draw[red] (-1.34,3.54)-- (-3+ 1.09 ,0.72+2.4);
\draw[red] (-3+ 1.09 ,0.72+2.4)-- (-2.7+ 1.09 ,0.72+1.38);

\draw[red] (-2.7+ 1.09 ,0.72+1.38)--(-4.22+ 1.09 ,0.72+1.6);
\draw[red] (-3+ 1.09 ,0.72+2.4)-- (-3.74+ 1.09 ,0.72+1.92);
\draw[red] (-4.9127182890505665,1.0896661385188597)-- (-4.22+ 1.09 ,0.72+ 1.6);
\draw[red] (-3.74+ 1.09 ,0.72+1.92)-- (-2.7+ 1.09 ,0.72+1.38);
\draw[red] (-1.34,3.54)-- (-1.18, 3.08);
\draw[red] (-1.18, 3.08)-- (0.19, 3.88);

\draw[blue] (-4.58, 1.82) node{$S_+$};
\draw[blue] (-1.9, 3.8) node{$S_-$};

\draw[red] (-1.64- 0.79 ,-0.72+ 2.18) node{$\Upsilon_-$};
\draw[red] (-3.15+ 1.09 ,0.72+ 1.90) node{$\Upsilon_+$};

\draw (-3.25+ 1.0 ,0.72+ 2.10) edge[->,bend right]  (-4.38, 2.12);


\draw[blue] (-0.86, 5.62) node{$E$};

\draw[red] (-0.48, 2.8) node{$U_+$};

\draw(-0.38, 3.1)  edge[->,bend right=80]  (-4.78, 2.22);

\draw  (-2.44, 1.88)  edge[->,bend left]   (-1.9, 3.6);

\draw (-0.5,-1) node{$Z_\str$};

\draw (5,-0) node{$Z'_\str$};

\draw (0,-0.15) node{$\alpha$};
\draw (-2.7+ 1.09 ,0.72+1.23) node{$\omega$};
\end{tikzpicture}}

\caption{A full Siegel prepacman, compare with Figure~\ref{Fg:Prepacman}. In the dynamical plane of a quadratic polynomial $p$, the sector $S_Q=S_-\cup S_+$ is bounded by $R_-\cup I_-\cup I_+\cup R_+$ and truncated by an equipotential at small height. Pulling back $S_-~, S_+$ along appropriate branches of $p^\aa,p^\bb$ we obtain $U_-=\Upsilon_- \cup \Upsilon_+$ and $U_+$ so that $(p^\aa\mid U_-~,$ $p^\bb\mid U_+)$ is a full prepacman\eqref{eq:lem:prf:FullSiegPacmExist}. Truncating $(p^\aa\mid U_-~,$ $p^\bb\mid U_+)$ at $\omega$ and at the vertex where $R_+$ meets $E$ (see Figure~\ref{Fg:ConstrOfSiegPrepac Axil:part 2}) we obtain a required prepacman~\eqref{eq:lem:SiegPacmExist}.}
\label{Fg:SiegelPrepacman}
\end{figure}

Before proceeding with the proof let us define a sector renormalization of $p\mid \overline Z_p$. Consider the rotation of $\Lbb_\theta\colon z\mapsto \ee(z)$ of the unit disk $\overline \Disk$. Let $h\colon \overline Z_p\to \overline \Disk$ be the unique conformal conjugacy between $p\mid \overline Z_g$ and $\Lbb_\theta\mid \overline \Disk$ normalized such that $h(x)=1$. A sector \emph{pre-renormalization} of $\Lbb_\theta$ is a commuting pair ${(\Lbb^\aa\mid \Sbb_-~,\sp  \Lbb^\bb \mid \Sbb_+)}$ realizing the first return map to $\Sbb_-\cup  \Sbb_+$ (see Figure~\ref{Fg:RenOfRotDisc}), where $\Sbb_-,\ \Sbb_+$ are closed sectors of $\overline \Disk$ such that $\Sbb_-\cap \Sbb_+$ is the internal ray towards $1$; see details in~\S\ref{ss:RenRotUnitDisc}. Denote by $\delta$ the angle of $\Sbb=\Sbb_-\cup \Sbb_+$ at $0$. The gluing map  $z\to z^{1/\delta}$ projects  $(\Lbb^\aa\mid \Sbb_-~,\sp  \Lbb^\bb \mid \Sbb_+)$ to a new rotation.

\begin{defn}
\label{defn:SectRenormZp}
A \emph{sector pre-renormalization of $p\mid \overline Z_p$ around $x\in \partial Z_p$} is a commuting pair ${(p^\aa\mid X_-~,\sp  p^\bb \mid X_+)}$ obtained by pulling back a sector pre-renormalization $(\Lbb^\aa_\theta\mid \Sbb_-~,\sp  \Lbb_\theta^\bb \mid \Sbb_+)$ by $h$ where
\begin{itemize}
\item $X_-\coloneqq h^{-1}(\Sbb_-),$ $X_+\coloneqq h^{-1}(\Sbb_+)$ and $X\coloneqq h^{-1}(\Sbb)=X_-\cup X_+$ are closed sectors of $\overline Z_f$,
\item the internal ray $I_0\coloneqq X_-\cap X_+$  lands at $x$.
\end{itemize}
The gluing map $z\mapsto z^{1/\delta}$ descents to \[\psi_x\coloneqq h^{-1}\circ[z\to z^{1/\delta}] \circ h\] with $\psi_x(X)=\overline Z_f$. 
\end{defn}
\begin{proof}[Proof of Lemma~\ref{lem:SiegPacmExist}]
Consider the sector renormalization  $(p^\aa\mid X_-~,\sp  p^\bb \mid X_+)$ from Definition~\ref{defn:SectRenormZp} and assume that the angle $\delta$ of $\Sbb$ is small.  We will now extend $(p^\aa\mid X_-~,\sp  p^\bb \mid X_+)$ beyond $\overline Z_p$ to obtain a prepacman~\eqref{eq:lem:SiegPacmExist}, see Figure~\ref{Fg:SiegelPrepacman}. 
Set \[I_-\coloneqq p^{\bb}(I_0),\sp I_+\coloneqq p^{\aa}(I_0),\sp I\coloneqq f^{\aa+\bb}(I_0).\]  Then the sector $X_-$ is bounded by $I_-,I_0$ and $X_+$ is bounded by $I_0$, $I_+$.

 Since $x$ is not precritical, there are unique external rays $R_-~, R_+~, R$ extending $I_-~, I_+~, I$ beyond $\overline Z_p$. Let $S_Q$ be the closed sector bounded by $R_-\cup I_-\cup I_+\cup R_+$ and truncated by an external equipotential $E$ at a small height $\sigma>0$. The curve $R\cup I$ divides $S$ into two closed sectors  $S_+$ and $S_-$ such that  $S_+$ is between $R_-\cup I_-$ and $R\cup I$ while $S_-$ is between $R\cup I$ and $R_+\cup I_+$. We note that $p^\aa(X_-)\subset S_-$ and $p^\bb(X_+)\subset S_+$.

Let us next specify $U_-\supset X_-~,\sp U_+\supset X_+$ such that  \begin{equation}
 \label{eq:lem:prf:FullSiegPacmExist}Q= (q_-~,q_+)= (p^\aa\mid U_-~, \sp p^\bb\mid U_+)
 \end{equation}
is a full prepacman. Since the $p$-orbits of $X_-~,X_+$ cover $\overline Z_\str$ before they return back to $X$, we see that $\partial X\cap \partial Z_p$ has a unique precritical point, call it $c'_0$, that travels though the critical point of $p$ before it returns to $X$. Below we assume that $c'_0\in X_-$; the case  $c'_0\in X_+$ is analogous. Then $S_+$ has a conformal pullback $U_+$ along $p^\bb\colon X_+\to  S_+$. We have $U_+\subset S_Q$ because rays and equipotentials bounding $S_Q$ enclose $U_+$.

\begin{figure}[t!]

\centering{\begin{tikzpicture}[ scale=0.88]

\draw (-4.,0.)-- (0.58,3.04);
\draw (0.58,3.04)-- (4.,0.);
\draw (4.,0.)-- (-0.68,-3.);
\draw (-0.68,-3.)-- (-4.,0.);
\draw (4.,0.)-- (6.4,1.76);
\draw (6.02,-1.9)-- (4.,0.);

\draw (-0.5,-1) node{$Z_\str$};

\draw (6.5,-1.5) node{$Z'_\str$};

\draw (0,-0.25) node{$\alpha$};
\draw (7.64,-0.05) node{$\alpha'$};
\draw[red] (-2.04,2.2) node{$S_Q$};
\draw[red] (4.1,1.) node {$W'$};
\draw (4.05,-0.4) node {$c_0$};

\draw[red,draw opacity=0 ,fill=red, fill opacity=0.3] (0.08,-0.08)-- (-2.34,1.5)-- (-1.32,2.14)-- (0.08,-0.08);
\draw (0.08,-0.08)-- (-2.34,1.5)-- (-1.32,2.14)-- (0.08,-0.08);

\draw (0.08,-0.08)-- (3.6201379310344826,0.3376551724137933)-- (3.76,0.7);
\draw (3.76,0.7)-- (4.28,0.68);
\draw (4.28,0.68)-- (4.393485734290215,0.3024554524768074);
\draw (4.393485734290215,0.3024554524768074)-- (6.72,0.1);
\draw (6.74,-0.32)-- (4.527983634757543,-0.4655855688316517);
\draw (4.527983634757543,-0.4655855688316517)-- (4.42,-0.84);
\draw (4.42,-0.84)-- (3.56,-0.84);
\draw (3.56,-0.84)-- (3.3606616961789375,-0.4098322460391426);
\draw (3.3606616961789375,-0.4098322460391426)-- (0.08,-0.08);
\draw (4.84, 0.26) --(4.84,-0.45);

\draw[red, draw opacity=0.1 ,fill=red, fill opacity=0.3] (0.08,-0.08)-- (3.6201379310344826,0.3376551724137933)-- (3.76,0.7)-- (4.28,0.68)--(4.393485734290215,0.3024554524768074)-- (4.84, 0.26) --(4.84,-0.45)--(4.527983634757543,-0.4655855688316517)  -- (4.42,-0.84)-- (3.56,-0.84)-- (3.3606616961789375,-0.4098322460391426)-- (0.08,-0.08);


\draw (3.14, -0.04) edge [->,bend right]  node[above] {$f^{k-j}$} (-1.24, 1.18); 
\end{tikzpicture}}

\caption{Since $W'$ is truncated by an equipotential of $Z'_p$ at small height, the point $p^j(z)\in W'\setminus \overline Z_p$ is in a small neighborhood of $c_0$.
 }
\label{Fg:ConstrOfSiegPrepac Axil}
\end{figure}
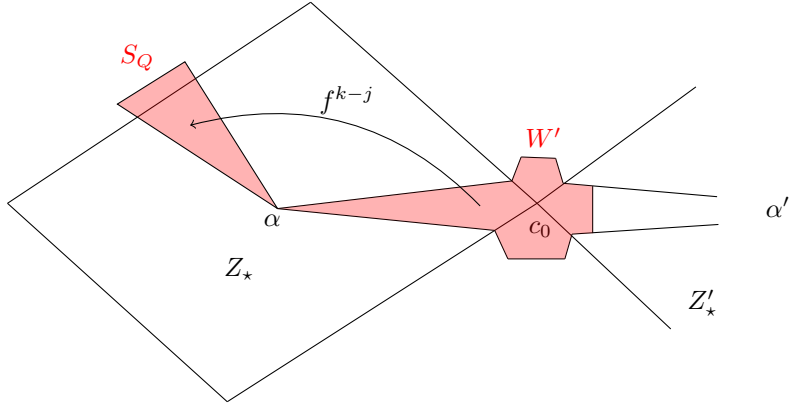

The sector $S_-$ has a degree two pullback $\Upsilon_-$ along $p^\aa\colon X_-\to  S_-$. Under $p^\aa\colon \Upsilon_-\to S_-$ the fixed point $\alpha$ has two preimages, one of them is $\alpha$, we denote the other preimage by $\omega$. Let $\Upsilon_+$ be the conformal pullback of $S_+$ along the orbit $p^\aa\colon \{\omega\}\to \{\alpha\}$. We define $U_-\coloneqq \Upsilon_-\cup \Upsilon_+\subset S_Q$ and we observe that $Q$ in~\eqref{eq:lem:prf:FullSiegPacmExist} is a full prepacman. 

By Theorem~\ref{thm:K_P is LocConn}, primary limbs of $\Kfilled_p$ intersecting a small neighborhood of $x$ have small diameters. By choosing $\delta$ and $\sigma$ sufficiently small we can guarantee that $S_Q\setminus \overline Z_p$ is in a small neighborhood of $x$.

Let us now truncate $Q$ at level $r$ and let us show that the orbit \[z, p(z), \dots ,  p^k(z)=q_{\pm}(z),\sp\sp k\in \{\aa,\bb\} \] of any $z\in U_{\pm}$ is in a small neighborhood of $\overline Z_p$. The truncation of $Q$ at level $r$ removes points in $U_-=\Upsilon_-\cup \Upsilon_+$ with $p^\aa$- images in the subdisk of $Z_p$ bounded by the equipotential at height $t\coloneqq r^\delta$. Since $\delta$ is small, we obtain that $t$ is close to $1$.

Since $\Kfilled_p$ is locally connected (Theorem~\ref{thm:K_P is LocConn}), all the external rays of $p$ land. For $z\in U_{\pm} \setminus \Kfilled_f$, define $\rho(z)\in \Kfilled_p$ to be the landing point of the external ray passing through $z$. Since $S_Q$ is truncated by an equipotential at a small height, the orbit of $z$ stays close to that of $\rho(z)$. This reduces the claim to the case  $z\in \Kfilled_p\cap U_{\pm}$. 

By Theorem~\ref{thm:SiegJulSet P is LocCon}, there is an $\ell\ge 0$ such that all of the \emph{big} (with diameter at least $\varepsilon$) primary limbs of $p$ are attached to one of $c_0, c_{-1},\dots, c_{-\ell}$, where $c_0$ is the critical point of $p$ and $c_{-i}$ is the unique preimage of $c_0$ under $p^i\mid \overline Z_p$. Since $\delta $ is assumed to be small, the orbit of $c'_0$ travels through all $c_{-\ell}, \dots, c_{0}$ before it returns to $S_Q$.

 Let us denote by $L$ the primary limb of $p$ containing $z$ (the case $z\in \overline Z_p$ is trivial). If $L$ is not attached to $c'_0$, then by the above discussion all $L, p(L), \dots , p^k(L)=q_{\pm}(L)$ are small and the claim follows.

 Suppose that $L$ is attached to $c'_0$. Denote by $L_{-i}$ the connected component of $\Kfilled_p\setminus \overline Z_p$ attached to $c_{-i}$. Since $c'_0$ travels through a critical point, we have $L=L_{-j}$ for some $j< k$. 
 
 Let $W$ be the pullback of $S_Q$ along \[p^{k-j}\colon c_0=p^j(c'_{0})\to p^k(c'_0)\]
and let $W'$ be $W$ truncated by the equipotential of $Z'_p$ at height $t$, see~Figure \ref{Fg:ConstrOfSiegPrepac Axil}. Since $t= r^\delta$ is close to $1$, we obtain that $W'\cap L'_0$ is in a small neighborhood of $c_0$ because the angle of $W$ at $\alpha'$ (the non-fixed preimage of $\alpha$) is small -- it is equal to $\delta$. Therefore, $p^j(z)$ is close to $c_0$, and by continuity all $p^{j-1}(z),p^{j-2}(z),\dots , p^{j-\ell}(z)$ are close to $c_0, c_{-1},\dots , c_{-\ell}$. Recall that $p^{j-i}(z)\in L_{-i}$. Since $L_0,L_{-1}, \dots , L_{-\ell}$ are the only big limbs, we obtain that the orbit $z,p(z), \dots, p^k(z)$ is in a small neighborhood of $\overline Z_p$.

It remains to specify external rays for $Q$. As it shown on the Figure~\ref{Fg:ConstrOfSiegPrepac Axil:part 2} we slightly truncate $S_Q$ at the vertex where $R_+$ meets the equipotential $E$ and we slightly truncate $U_\pm$ such the truncations are respected dynamically and such that the preimage of the $\partial S_Q\setminus (R_-\cup R_+)$ under $Q$  consists of exactly two curves that project to $\partial^\ext U_q$, where $q\colon U_q\to V_q$ is the pacman of $Q$. We now can embed in $S_Q\setminus (U_-\cup U_+)$ two rectangles $\Rect_-$ and $\Rect_+$ that define external rays of $Q$ as in~\S\ref{ss:DynObj}.
\end{proof}

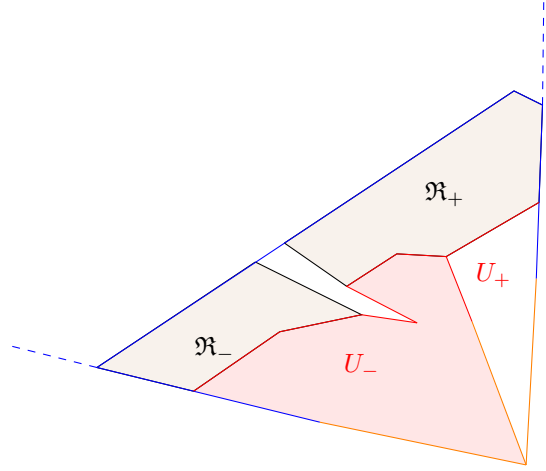
\begin{figure}[t!]

\centering{\begin{tikzpicture}[ scale=0.9]

\draw[ draw opacity=0 ,fill=red, fill opacity=0.1]
  (0,0) --(-1.18, 3.08)--   (-3+ 1.09 ,0.72+2.4)-- (-3.74+ 1.09 ,0.72+1.92)--(-2.7+ 1.09 ,0.72+1.38)--(-2.43,2.22)--(-3.6395543480712296,1.9683336865125984)--(-4.9127182890505665,1.0896661385188597) -- (-3.047927460454034,0.6319433450261434)--(0,0);

\draw[ fill=brown, fill opacity=0.1] (-3.57, 3.28)-- (-2.65,2.64)-- (-1.91,3.12)-- (-1.18,3.08)-- (0.19,3.88)--  (0.2419562160010615,5.31992941488656)--(-0.17901543404670395,5.528653393769006)--(-3.57, 3.28);

\draw[ fill=brown, fill opacity=0.1] (-4,3)--  (-6.34,1.44)-- (-4.9127182890505665,1.0896661385188597)-- (-3.6395543480712296,1.9683336865125984)--(-2.43,2.22)--(-4,3);

\draw (-1.2, 4.) node  {$\Rect_+$}
(-4.6, 1.72) node{$\Rect_-$};

\draw[red] (-1.18, 3.08)--   (-3+ 1.09 ,0.72+2.4);

\draw[orange] (0.,0.)-- (0.15034270964325902,2.754812628234827);
\draw[orange] (0.,0.)-- (-3.047927460454034,0.6319433450261434);

\draw[orange] (0.,0.)-- (-0.8,2.12);


\draw[blue] (-3.047927460454034,0.6319433450261434)-- (-6.34,1.44);
\draw[blue] (-6.34,1.44)--(-0.17901543404670395,5.528653393769006)-- (0.2419562160010615,5.31992941488656);
\draw[blue] (0.15034270964325902,2.754812628234827)-- (0.2419562160010615,5.31992941488656);

\draw[blue, dashed] (-6.34,1.44) edge  (-7.62, 1.76);
\draw[blue,dashed] (0.2419562160010615,5.31992941488656) edge (0.26,6.84);

\draw[red]  (-0.8,2.12) --(-1.18, 3.08);


\draw[red] (-2.7+ 1.09 ,0.72+1.38)--(-2.43,2.22)--(-3.6395543480712296,1.9683336865125984);

\draw[red] (-4.9127182890505665,1.0896661385188597)-- (-3.6395543480712296,1.9683336865125984);
\draw[red] (-3+ 1.09 ,0.72+2.4)-- (-3.74+ 1.09 ,0.72+1.92)-- (-2.7+ 1.09 ,0.72+1.38);
\draw[red] (-1.18, 3.08)-- (0.19, 3.88);

\draw[red] (-0.48, 2.8) node{$U_+$};
\draw[red] (-1.64- 0.79 ,-0.72+ 2.18) node{$U_-$};


\end{tikzpicture}}
\caption{By truncating the prepacman from Figure~\ref{Fg:SiegelPrepacman} and embedding rectangles $\Rect_\pm$, we endow the prepacman with external rays.
 }
\label{Fg:ConstrOfSiegPrepac Axil:part 2}
\end{figure}

\subsection{Pacman renormalization of Siegel maps}An immediate consequence of~\cite[Theorem 3.19, Proposition 4.3]{AL-posmeas} is 
\begin{thm}
\label{thm:HybrCongSiegMaps}
Any two Siegel maps with the same rotation number are hybrid conjugate on neighborhoods of their closed Siegel disks.
\end{thm}
\begin{proof} By~\cite[Proposition 4.3]{AL-posmeas} any Siegel maps $f,g$ can be obtained  by performing the Douady-Ghys surgery on quasicritical circle maps $\tilde f, \tilde g$. By \cite[Theorem 3.19]{AL-posmeas}, there is a qc map $\tilde h $ conjugating $\tilde f$ and $\tilde g$ in a small neighborhood of the unit circle. Then $\tilde h$ descents into a qc map $h$ conjugating $ f$ and $ g$ in small neighborhoods of the boundaries of their Siegel disks. A hybrid conjugacy between $f$ and $g$ is obtained by setting  $h\mid Z_f$ to be the canonical conformal conjugacy between $f\mid Z_f$ and $g\mid Z_g$ and running the pullback argument.
\end{proof}

As a corollary Theorem~\ref{thm:HybrCongSiegMaps} and Lemma~\ref{lem:SiegPacmExist}  we obtain. 
\begin{cor}
\label{cor:SP embeds in SM}
Every  Siegel map $f\colon U\to V$ is pacman renormalizable. 

Moreover the following holds. Let $f$ be a Siegel map and let $p$ be the unique quadratic polynomial with the same rotation number as $f$. Let $h$ be a hybrid conjugacy from a neighborhood of $\overline Z_f$ to a neighborhood of $\overline Z_p$ respectively. Then there are prepacmen $R$ and $Q$ in the dynamical planes of $f$ and $p$ respectively such that $h$ respects $R$ and $Q$ in the sense of Definition~\ref{defn:ConjBetweenPreRenorm}.
\end{cor}
\begin{proof}

 Choose a small $\varepsilon>0$ such that the  $\varepsilon$-neighborhood of $Z_f$ is in the domain of $h$. Then $h$ pullbacks a prepacman $Q$ from Lemma~\ref{lem:SiegPacmExist} to a prepacman $R$ in the dynamical plane of $f$. This shows that $f$ is pacman renormalizable.   
\end{proof}

\begin{lem}
\label{lem:thm:K_P is LocConn}
Suppose that a Siegel pacman $f$ is a renormalization of a quadratic polynomial. Then the non-escaping set $\Kfilled_f$ is locally connected. 

Moreover, for every $\varepsilon>0$ there is an $n\ge 0$ such that every connected component of $\Kfilled_f$ minus all the bubbles with generation at most $n$ is less than $\varepsilon$. All the external rays of $f$ land and the landing point in $\Jul_f$. Conversely, every point in $\Jul_f$ is the landing point of an external ray. The Julia set $\Jul_f$ is the closure of repelling periodic points.
\end{lem}
\begin{proof}
Follows from Theorem \ref{thm:K_P is LocConn}. Suppose that $f$ is obtained from a quadratic polynomial $p$. Then every bubble $Z_\alpha$ of $f$ is obtained from a bubble $\widetilde Z_\alpha$  of $p$ by removing an open sector. All of the limbs of $\widetilde Z_\alpha$ attached to the removed sector are also removed.  It follows from Theorem \ref{thm:K_P is LocConn} that for $\varepsilon>0$ there is an $n\ge 0$ such that every connected component of $\Kfilled_f$ minus all of the bubbles with generation at most $n$ is less than $\varepsilon$. Since bubbles of $f$ are locally connected, so is $\Kfilled_f$. The landing property of external rays is straightforward.  
\end{proof}

\begin{figure}[t!]

\centering{\begin{tikzpicture}[ scale=0.88]

\draw (-4.,0.)-- (0.58,3.04);
\draw (0.58,3.04)-- (4.,0.);
\draw (4.,0.)-- (-0.68,-3.);
\draw (-0.68,-3.)-- (-4.,0.);
\draw (4.,0.)-- (6.4,1.76);
\draw (6.02,-1.9)-- (4.,0.);

\draw (-0.5,-1) node{$Z_\str$};

\draw (6.5,-1.5) node{$Z'_\str$};

\draw (0,-0.15) node[right]{$\alpha$};


\draw (-1.8371026540472566,1.4356349195843536)-- (-3.06,2.08);
\draw (-3.06,2.08)-- (-2.88,3.);
\draw (-2.88,3.)-- (-2.26,2.36);
\draw (-2.26,2.36)-- (-2.,3.04);
\draw (-2.,3.04)-- (-1.48,2.64);
\draw (-1.48,2.64)-- (-1.8371026540472566,1.4356349195843536);
\draw (-3.06,2.08)-- (-3.76,2.02);
\draw (-3.76,2.02)-- (-3.92,2.74);
\draw (-3.92,2.74)-- (-3.74,2.92)--(-3.6,2.52)--(-3.44,2.82);
\draw (-3.44,2.82)-- (-3.06,2.08);
\draw (-3.92,2.74)-- (-4.24,2.82);
\draw (-4.24,2.82)-- (-4.36,3.26);
\draw (-4.36,3.26)-- (-4.14,3.22)--(-4.11,3.02)--(-4.0,3.13)--(-3.92,2.74);


\draw[red,fill=red, fill opacity=0.1]  (-1.8371026540472566,1.4356349195843536).. controls (-2.35, 1.7) ..
(-2.35, 1.7) ..controls (-3.04, 1.12) and (-5.66, 1.14)..
 (-6.66, 1.38) .. controls   (-6.34, 3.12) and (-4.62, 6.18)..
 (-2.48, 7.4) .. controls (-0.94, 6.06) and (-0.32, 3.74)..
 (-1.72, 1.83).. controls (-1.72, 1.83)..
 (-1.8371026540472566,1.4356349195843536);

\draw[red,,fill=red, fill opacity=0.3] 
 (-3.06,2.08) .. controls (-3.26, 2.5)..
 (-3.26, 2.5)..controls (-3.18, 3.46) and (-3.32, 4.32)..
 (-3.72,5.04) .. controls (-4.58, 4.44) and  (-5.34, 2.38)..
 (-5.32,1.82) .. controls  (-4.84, 1.72) and  (-3.52, 1.8).. 
  (-3.44, 2.02) ..controls  (-3.44, 2.02)..
 (-3.06,2.08);
 
 \draw (-3.84, 3.98) edge [->,bend left]  node[above] {$f^p$} (-3.28, 5.88); 

\draw[red] (-4.12, -2.74) edge node[above] {$R_1$} (-2.32, -1.54);

\draw  (0,-0.15) edge node[above]{$I_1$} (-2.32, -1.54);

\draw[red,fill=red] (-2.32, -1.54) circle (0.05cm);
\node[above,red] at  (-2.32, -1.54) {$x$};

\draw (-2.8, 5.78) node {$D$}
(-4.54, 0.82) node[red] {$R_\ell$}
(-0.7, 3.6) node[red] {$R_\rho$}
(-4.56, 3.26) node {$y$};
\draw (-3.74, 3.7) node {$D_p$};
\draw (-2.05,1.9) node {$Z_1$};
\end{tikzpicture}}
\caption{Illustration to the proof of Lemma~\ref{lem:SP ExternalRays}. The ray $R_1$ has preimages $R_\ell$ and $R_\rho$ that land at $Z_1$ such that $R_\ell \cup R_\rho$ together with together with appropriate arcs in $\partial V\cup Z_1$ bound a disk $D$ containing $y$. The disk $D$ has a univalent lift $D_p\Subset D$. By Schwarz lemma, $f^p\colon D_p \to D$ is expanding, which implies that there is an external ray landing at $y$.}
 
\label{Fg:RatRaysForSiegPacm}
\end{figure}
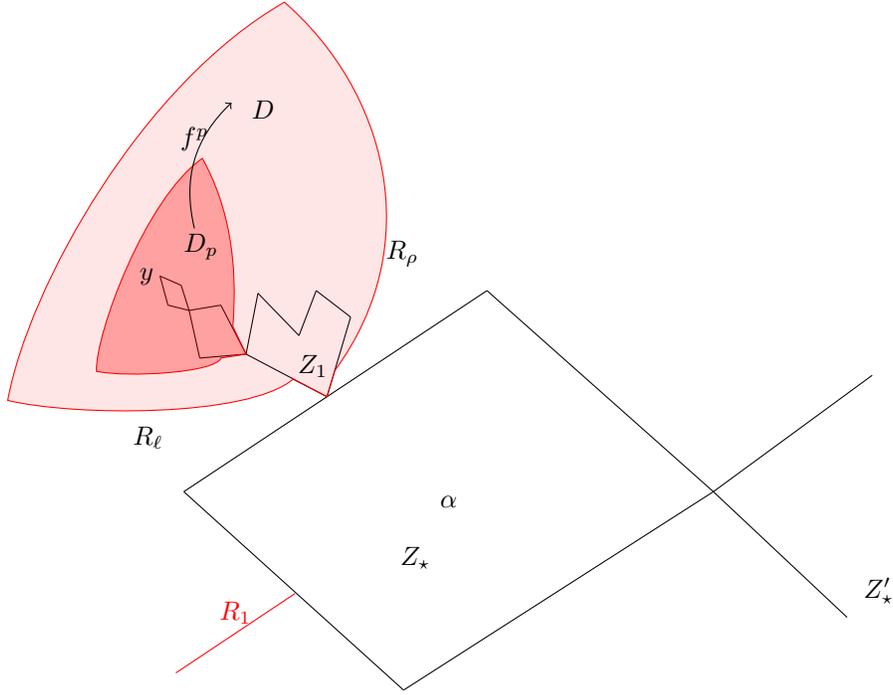

\subsection{Rational rays of Siegel pacmen}
\label{ss:Ret rays of SP}
By a \emph{rational point}  we mean either a periodic or preperiodic point. Similarly, a periodic or preperiodic  ray is \emph{rational.} 

Let us fix pacmen $f,p$ and prepacmen $R,Q$ as in Corollary~\ref{cor:SP embeds in SM}. Let $\Kfilled_R$ be the non-escaping set of $R$. By definition, $\Kfilled_R\subset \Kfilled_f$; spreading around $\Kfilled_R$ we define the \emph{local non-escaping set of $f$:}
\begin{equation}
\label{eq:K f loc}
\Kfilled^\loc_f\coloneqq \bigcup_{n\ge 0} f^{n}(\Kfilled_R).
\end{equation}
This is the set of point that do not escape $\bDelta_R$ under $f\colon \bDelta_R\to \bDelta_R\cup S_R$, see Figure~\ref{Fg:Prepacman:OrbitOfUi}. Similarly we define  
\[ \Kfilled^\loc_p\coloneqq \bigcup_{n\ge 0} p^{n}(\Kfilled_Q).\]
It is immediate that $h$ conjugates $f\mid \Kfilled^\loc_f$ and $p\mid \Kfilled^\loc_p$. As a consequence, the \emph{local Julia set}
\[\Jul^\loc_f\coloneqq \overline{\bigcup_{n\ge 0} \left[f\mid \Kfilled^\loc_{f}\right]^{-n} (\partial Z_f) }\]
is the closure of repelling periodic points because so is $\Jul^\loc_p$. (Indeed, every $y\in \Jul^\loc_p$ is the landing point of an external ray $R_y$ because $\Jul_p$ is locally connected. Since external rays in a pacman are parametrized by angles in  $\mathbb{S}^1$ (see \S\ref{ss:DynObj}), $p$ has a periodic external ray $R_x$ landing at $x\in \Jul^\loc_p$.) Moreover, for every periodic point $y\in \Kfilled^\loc_f$ there is a unique periodic bubble chain $B_y$ of $\Kfilled^\loc_f$ landing at $y$.

\begin{lem}[External rays]
\label{lem:SP ExternalRays}
Let $y\in \Jul^\loc_f$ be a periodic point. Then there is a periodic external ray $R_y$ landing at $y$ with the same period as $y$. 
\end{lem}
\begin{proof}
Let $B_y=(Z_1,Z_2,\dots)$ be the bubble chain in $\Kfilled^\loc_f$ landing at $y$. Denote by $x$ the unique point in the intersection of $\gamma_1\cap \partial Z_\str$. By Definition~\ref{dfn:SiegPcm}, the external ray $R_1$ lands at $x$. There are two iterated preimages $x_\ell, x_{\rho}\in \partial Z_1$ of $x$ (by density of those) such that the rays $R_\ell$, $R_\rho$ (iterated lifts of $R_1$) landing at $x_\ell,x_\rho$ 
 together with $Z_1$ separate $y$ from $\overline Z_f$, see Figure~\ref{Fg:RatRaysForSiegPacm}. We denote by $D$ the open subdisk of $V$  bounded by $R_\ell, R_\rho$, $Z_1$ and containing $y$. 
 Let $D_p$ be the (univalent) pullback of $D$ along $f^p \colon \{y\}\to \{y\}$. Then $D_p\Subset D$. By the Schwarz lemma, $f^p\colon D_p\to D$ expands the hyperbolic metric of $D$. 
 
There is a unique periodic external ray  $R_y$ in $D$ with period $p$. We claim that $R_y$ lands at $y$. Indeed, parametrize $R$ as $R\colon \R_{>0}\to V$ with $f^p(R_y(t+p))= R_y(t)$. Since all the points in $D$ away from $y$ escape in finite time under $f^p\colon D_p\to D$, the Euclidean distance between $R(t)$ and $y$ goes to $0$ as $t\to +\infty$.
\end{proof}

The next lemma is a preparation for a Pullback Argument.
\begin{lem}[Rational approximation of $\gamma_1$]
\label{lem:SP per pts and rays}
For every $\varepsilon>0$, there are 
\begin{itemize}
\item periodic points $x_\ell,x_\rho\in \Jul^\loc_f$,
\item external rays $R_\ell$ and $R_\rho$ landing at $x_\ell, x_\rho$ respectively,
\item periodic bubble chains $B_\ell$ and $B_\rho$ in $\Kfilled^\loc_f$ landing at $x_\ell, x_\rho$ respectively, and  
\item internal rays $I_\ell$ and $I_\rho$ of $\overline Z_f$ landing at the points at which $B_\ell$ and $B_\rho$ are attached  
\end{itemize}
such that $ R_\ell \cup B_\ell\cup I_\ell$ and $R_\rho \cup B_\rho\cup I_\rho$ are in the $\varepsilon$-neighborhood of $\gamma_1$ and such that $R_\ell\cup B_\ell\cup I_\ell$ is on the left of $\gamma_1$ while $R_\rho\cup B_\rho\cup I_\rho$ is on the right of $\gamma_1$.
\end{lem}
\begin{proof}
Consider a finite set of periodic points $y_1,y_2,\dots , y_p\in \Jul^\loc_f$. By Lemma~\ref{lem:SP ExternalRays} each $y_i$ is the landing point of an external periodic ray, call it $R_{y,i}$, and the landing point of a periodic bubble chain, call it $B_{y,i}$. Let $\{W_1,W_2,\dots, W_p\}$ be the set of connected components of 
\[U\setminus \overline {Z_f\bigcup_i (B_{y,i}\cup R_{y,i})};\]
we assume that $\partial W_p$ contains $\partial^\frb U_f$.
By adding more periodic points we can also assume that $c_0\not\in \partial W_p$. Set 
\[W\coloneqq W_1\cup W_2\cup \dots \cup W_{p-1}.\]

By Schwarz lemma, $f\mid W$ is expanding with respect to the hyperbolic metric of $W$. Since $c_0\not\in \partial W_p$, there is a sequence of periodic points $x_{\ell, j}\in \Jul^\loc_f$ such that the orbit of $x_\ell$ is in $\overline W$ and such that  $x_{\ell, j}$ converges from the left to the unique point $x_1$ in $\gamma_1\cap \partial Z_f$. 

We claim that the external rays  $R_{\ell,j}$ landing at $x_{\ell,j}$ converge to the external ray landing at $x_1$. Indeed, since $x_{\ell,j}\to x_1$, the external angle of $R_{\ell,j}$ (see~\S\ref{ss:DynObj}) converges to the external angle of $R_1$. By continuity, $R_{\ell,j}([0,T])$ converges to $R_{1}([0,T])$ for any $T\in  \R_{>0}$. Since $f\mid W$ is expanding, $R_{\ell,j}([T,+\infty))$ is in a small neighborhood of $x_{\ell,j}$ which converges to $x_1$.

The bubble chains $B_{\ell,j}$ of $\Kfilled^\loc_f$ landing at $x_{\ell,j}$ shrink because there are no big limbs in a neighborhood of $x_1$. Define $I_{\ell, j}$ to be the internal ray of $Z_f$ landing at the point where $B_{\ell,j}$  is attached. Then $R_{\ell,j}\cup B_{\ell,j}\cup I_{\ell,j}$ is a required approximation for sufficiently big $j$. Similarly, $R_{\rho}\cup B_{\rho}\cup I_{\rho}$ is constructed.
\end{proof}

\subsection{Hybrid equivalence}
Recall~\S\ref{sec:pacmen} that a pacman $f\colon U_f\to V_f$ is required to have a locally analytic extension through $\partial U_f$. By means of the Pullback Argument, we will now show the following.

\begin{thm}
\label{thm:HyrbrEquivOfPacm}
Let $f\colon U_f\to V_f$ and $g\colon U_g\to V_g$  be two combinatorially equivalent Siegel pacmen and suppose that $f$ and $g$ have the same truncation level.  
  Then $f$ and $g$ are hybrid equivalent.    
\end{thm}
\begin{proof}
Let $p$ be the unique quadratic polynomial with the same rotation number as $f$ and $g$. Let $h_f$ and $h_g$ be hybrid conjugacies from neighborhoods of $\overline Z_f$ and $\overline Z_g$ to a neighborhood of $\overline Z_p$ respectively. As in Corollary~\ref{cor:SP embeds in SM}, there are prepacmen $Q$, $R$, $S$ in the dynamical planes of $p$, $f$, and $g$ respectively such that $h_f$ and $h_g$ are conjugacies respecting prepacmen $R,Q$ and  $S,Q$ respectively, see Definition~\ref{defn:ConjBetweenPreRenorm}. The composition $h\coloneqq h_g^{-1}\circ h_f$ is a conjugacy respecting $R, S$.

As in \eqref{eq:K f loc} we define $\Kfilled^\loc_f$, similarly $\Kfilled^\loc_g$ and $\Kfilled^\loc_p$ are defined. Then $h$ conjugates $f\mid \Kfilled^\loc_f$ and $g\mid \Kfilled^\loc_g$. 

As in Lemma~\ref{lem:SP per pts and rays} let $ R_\ell(f) \cup B_\ell(f)\cup I_\ell(f)$ and $R_\rho (f)\cup B_\rho(f)\cup I_\rho(f)$ be approximations of $\gamma_1(f)$ from the left and  from the right respectively. Similarly, let $ R_\ell(g) \cup B_\ell(g)\cup I_\ell(g)$ and $R_\rho (g)\cup B_\rho(g)\cup I_\rho(g)$ be approximations of $\gamma_1(g)$. We choose the approximations in the compatible ways:
\begin{itemize}
\item  $B_\ell(g),I_\ell(g), B_\rho(g),I_\rho (g)$ are the image of $B_\ell(f),I_\ell(f), B_\rho(f),I_\rho (f)$ under $h$;
\item $R_\ell(g),R_\rho(g)$ have the same external angles as $R_\ell(f),R_\rho(f)$.
\end{itemize}

Write \[T_f\coloneqq \Kfilled^{\loc}_f \bigcup_{n\ge 0} f^n(R_\rho\cup R_\ell )\sp \text{ and }\sp T_g\coloneqq \Kfilled^{\loc}_g \bigcup_{n\ge 0} g^n(R_\rho\cup R_\ell ).\]
Then $T_f$ and $T_g$ are forward invariant sets such that $V_f\setminus T_f$ and $V_g\setminus T_g$ consist of finitely many connected components. Since $R_\ell(g),R_\rho(g)$ have the same external angles, we can extend $h$ to a qc map $h\colon V_f\to V_g$ such that $h$ is equivariant on $ T_f\cup \partial^\ext U_f$.

We now slightly increase $U_f $ by moving $\partial^\frb U_f$ so that the new disk $\mathfrak U_f$ satisfies \[f(\partial^\frb \mathfrak U_f)\subset Z_f \cup \overline {B_\ell\cup R_\ell \cup B_\rho \cup R_\rho}.\]
(Indeed, we can slightly move $\gamma_-\subset \partial^\frb U_f$ so that its image is within $R_\ell\cup B_\ell\cup Z_f$ and we can slightly move $\gamma_+\subset \partial^\frb U_f$ so that its image is within $R_\ell\cup B_\ell\cup Z_f$.) 
Similarly, we slightly increase $U_g$ by moving $\partial^\frb U_g$ so that the new disk $\mathfrak U_g$ satisfies \[g(\partial^\frb \mathfrak U_g)\subset Z_g \cup \overline {B_\ell\cup R_\ell \cup B_\rho \cup R_\rho}\] and such that $h\mid T_f$ lifts to a conjugacy between $f \mid \partial \mathfrak U_f$ and $g \mid \partial \mathfrak U_g$. This allows us to apply the \emph{Pullback Argument}: we set $h_0\coloneqq h$ and we construct qc maps 
\[ h_n:V_f\to V_g \mbox{ by }h_n(x)\coloneqq \begin{cases} g^{-1}\circ h_{n-1}\circ f(x) &\mbox{if } x\in \mathfrak  U_f,\\
  h_{n-1} (x) &\mbox{if } x\notin  \mathfrak  U_f.
  \end{cases} \]
  We can choose $h$ such that $h_0$ and $h_1$ are connected by an isotopy $\tilde h\colon [0,1]\times V_f\to V_g$ that is constant on $T_f$ and uniformly continuous in the hyperbolic metrics of $V_f\setminus T_f$, $V_g\setminus T_g$. This implies that the Euclidean distance between $h_n$ and $h_{n+1}$ tends to $0$. Since the space of qc maps with uniformly bounded dilatation is compact, we may pass to the limit and construct a hybrid conjugacy between $f$ and $g$.
\end{proof}

\subsection{Standard Siegel pacmen}
\label{ss:StandardPacmen}
We say a Siegel pacman is \emph{standard} if $\gamma_0$ passes through the critical value; equivalently if $\gamma_1$ passes through the image of the critical value.

  A \emph{standard prepacman} $R$ in the dynamical plane of a Siegel map $g$ is a prepacman around (see~\S\ref{ss:SiegelPrepacmen}) the critical value of $g$. Then the pacman $r$ obtained from $R$ is a standard and the renormalization  change of variables $\psi_R$ respects the internal ray towards the critical value: 
  \begin{equation}
  \label{eq:I_1 is respect}
  \psi_R(I_1(g))=I_1(r).
  \end{equation} The pacman renormalization associated with $R$ is called a \emph{standard pacman renormalization of $g$}.

By Theorem~\ref{thm:HyrbrEquivOfPacm}, two standard Siegel pacmen are hybrid equivalent if and only if they have the same rotation number. 

\begin{thm}
\label{thm:LocConnOfJul for StandPacm}
Let $f$ be a standard Siegel pacman. Then $\Kfilled_f$ is locally connected. Moreover, for every $\varepsilon>0$ there is an $n\ge 0$ such that every connected component of $\Kfilled_f$ minus all of the bubbles with generation at most $n$ is less than $\varepsilon$.
\end{thm}

\noindent As a consequence, every periodic point of $\Jul_f$ is the landing point of a bubble chain.

\begin{proof}
For every $\theta\in \Theta_\bnd$, there is a Standard pacman $g$ with rotation number $\theta$ such that is $g$ is a renormalization of a quadratic polynomial.  The statement now follows from Theorem~\ref{thm:HyrbrEquivOfPacm} combined with Lemma~\ref{lem:thm:K_P is LocConn}.
\end{proof}

\subsection{A fixed point under renormalization}
\label{ss:FixedREnormPts}
Consider a Siegel map $f$ with rotation number $\theta\in \Theta_\per$ and consider $x\in \partial Z_f$ such that $x$ is neither the critical point nor its iterated preimage. Let ${(f^\aa\mid X_{-,x}~,\sp  f^\bb \mid X_{+,x})}$ be the sector pre-renormalization of $f\mid \overline Z_f$ as in Definition~\ref{defn:SectRenormZp}. Since $\theta\in \Theta_\per$, we can assume (see~\S\ref{sss:ap:SectRen:PerCase}) that the renormalization fixes $f\mid \overline Z_f$: the gluing map $\psi_x:X_x\to \overline Z_f$ projects ${(f^\aa\mid X_{-,x}~,\sp  f^\bb \mid X_{+,x})}$ back to $f\mid \overline Z_f.$ For $x\in\{c_0,c_1\}$ we write \[\psi_0=\psi_{c_0}, \sp X_0=X_{c_0}, \sp X_{\pm,0}=X_{\pm,c_0}\] and \[\psi_1=\psi_{c_1}, \sp X_1=X_{c_1}, \sp X_{\pm,1}=X_{\pm,c_1}.\]

\begin{thm}[\cite{McM3}]
\label{thm:RenFixedPoint}
For every $\theta\in \Theta_\per$, there is a Siegel map $g_\str\colon U_\str\to V_\str$ with rotation number $\theta$ such that for a certain sector pre-renormalization of $g_\str\mid \overline Z_{g_\str}$ as above the gluing map $\psi_0$ extends analytically through $\partial Z_\str\cap \partial X_0$ to a gluing map $\psi_0$ projecting $(g^\aa_\str\mid S_{-,0}~,\sp  g_\str^\bb \mid S_{+,0})$ back to $g_\str\colon U_\str\to V_\str$, where $S_{\pm,0}\subset U_\str$. Moreover, there is an \emph{improvement of the domain}: the forward orbits \[\bigcup_{i\in\{0,1,\dots, \aa\}} g^{i}_\str (S_{-,0})\bigcup_{j\in\{0,1,\dots, \bb\}} g^{j}_\str (S_{+,0})\] are compactly contained in $U_\str\cap V_\str$.
\end{thm}
\noindent We note that the improvement of the domain follows from complex \textit{a priori} bounds for quasicritical circle maps~\cite{AL-posmeas}*{\S 3.3} after applying the inverse Douady-Ghys surjery; see also~\cite{Ya-posmeas}*{Proposition 3.2}. Up to conformal conjugacy, $g_\str$ is unique in a neighborhood of $\overline Z_{g_\str}$. The improvement of the domain will allow us in Theorem~\ref{thm: An Pacm  self ope} to construct a pacman analytic self-operator $\RR\colon\BB\dashrightarrow \BB$.

\begin{cor}
\label{cor:psi1:thm:RenFixedPoint}
The gluing map $\psi_1$ extends analytically through $\partial Z_{g_\str}\cap \partial X_1$ and, up to replacing $\psi_1$ with its iterate, satisfies the same properties as $\psi_0$ in Theorem~\ref{thm:RenFixedPoint}; in particular, the improvement of the domain holds for $\psi_1$.
\end{cor}
\begin{proof}
We need to check that $\psi_1\coloneqq g_\str \circ \psi_0\circ g_\str^{-1}$ is well defined. Since $\psi_0$ projects $(g_\str^{\aa}, g_\str^{\bb})$ to $g_\str$ and since the maps $g_\str^{\aa}, g_\str^{\bb}$ are two-to-one in a neighborhood of $c_0$, we obtain that for $z$ close to $c_1$ the gluing map $\psi_0$ maps $g_\str^{-1}(z)$ to a pair of points that have the same $g_\str$-image. This shows that $\psi_1$ is well defined. Up to replacing $\psi_1$ with its iterate we can guarantee that the improvement of the domain holds for $\psi_1$.
\end{proof}

Note that $\psi_1$ is expanding on $ \overline Z_{g_\str}\cap \partial X_1$ because $\psi_1\mid \overline Z_{g_\str}$ is conjugate to \[\overline \Disk\to\overline \Disk,\sp  z\to z^{1/t},\sp\sp t>1.\]

\begin{lem}[Fixed Siegel pacman]
\label{lem:FixedSiegelPacman}
For any $\theta\in \Theta_\per$ there is a standard Siegel pacman $f_\str\colon U_\str\to V_\str$ that has a standard Siegel prepacman \[F_\str = (f^\aa_\str\mid U_-\to S_\str, \sp f^\bb_\str\mid U_+ \to S_\str)\] together with a gluing map $\psi_\str\colon S_\str\to \overline V_\str$ projecting $F_\str$ back to $f_\str.$ Moreover, the improvement of the domain holds for the renormalization:  
\begin{equation}
\label{eq:lem:FixedSiegelPacman}
\bDelta_{F_\str}\Subset f_\str^{-1}(U_\str).
\end{equation} (See~\rm{\S\ref{sec:PacmenRenorm}} for the definition of $\bDelta_{F_\str}$.)
\end{lem}
\noindent The pacman $f_\str$ is conformally conjugate to $g_{\str}$ in a neighborhood of $Z_{\str}\coloneqq Z_{f_\str}.$ 
\begin{proof}
Consider a Siegel map $g_\str$ from Theorem~\ref{thm:RenFixedPoint} and $\psi_1$ from Corollary~\ref{cor:psi1:thm:RenFixedPoint}.

By Corollary~\ref{cor:SP embeds in SM},  $g_\str$ has a standard prepacman $Q\colon U_{Q, \pm}\to S_Q$ such that $S_Q\setminus Z_{g_\str}$ is in a small neighborhood of $c_1$. Since $\theta$ is of periodic type, we can prescribe $Q$ to have rotation number $\theta$. Since $\psi_1$ is expanding on $\partial \overline Z_{g_\str}$, for a sufficiently big integer $t\ge 1$ the prepacman \[(\psi_1^t)^* (Q)\coloneqq (\psi_1^{-t} \circ q_{\pm}\circ \psi_1^{ t}\colon \psi_1^{-t}(U_{Q,\pm})\to \psi_1^{-t} (S_{Q}) )\]
 has the property that $\psi_1^{-t} (S_{Q})\setminus Z_{g_\str}$ is in a much smaller neighborhood of $c_1$.
 
 Let $f_\str\colon U_\str\to V_\str$ be a pacman obtained from $Q$. The prepacman $(\psi_1^t)_* (Q)$ projects to the standard prepacman, call it 
 \[F_\str\colon (f_{\str,\pm}\colon U_{\str,\pm}\to S_{\str})\]
 such that $S_\str\setminus Z_{f_\str}$ is in a small neighborhood of $c_1$. The map $\psi_1^{t}$ descents to a gluing map, call it $\psi_\str$, projecting $F_\str$ back to $f_\str$.

If $t$ is sufficiently big, then $\bDelta_{F_\str}$ is compactly contained in $f_\str^{-1}(U_\str)$.
\end{proof}

\subsection{Analytic renormalization self-operator} 
 Applying Theorem~\ref{thm:RenOper} to $f_\str$ from Lemma~\ref{lem:FixedSiegelPacman} we obtain 
\begin{thm}[Analytic operator $\RR\colon \BB\dashrightarrow \BB$]
\label{thm: An Pacm  self ope}
Let $f_\str\colon U_\str\to V_\str$ be a pacman and $F_\str$ be a prepacman from Lemma~\ref{lem:FixedSiegelPacman}. Then there are small neighborhoods $N_{\widetilde U}(f_\str,\varepsilon), N_{\widetilde U}(f_\str,\delta)$ of $f_\str$ with $\varepsilon<\delta$ and there is an analytic pacman renormalization operator $\RR\colon N_{\widetilde U}(f_\str,\varepsilon)\to N_{\widetilde U}(f_\str,\delta)$ such that $\RR f_\str =f_\str$. Moreover, the operator $\RR$ is compact, so its spectrum is a sequence converging to $0$. The pre-renormalization of $\RR f_\str$ is $F_\str$.
\qed
\end{thm}
\begin{proof}
Let $f_\str\colon U'\to V'$ be a pacman obtained form $f_\str\colon U_\str\to V_\str$ by slightly decreasing $U_\str$ so that $U'\Subset U_\str$ and $\bDelta_{F_\str}\Subset  f_\str^{-1}(U')$. Since the renormalization is defined on $\bDelta_{F_\str}$, by Theorem~\ref{thm:RenOper} there is a compact analytic pacman renormalization operator $N_{\widetilde U'}(f_\str,\varepsilon)\to N_{\widetilde U}(f_\str,\delta)$, where $\widetilde U'$ and $\widetilde U$ are small neighborhoods of the closures of $U'$ and $U_\str$. Precomposing with the restriction operator $N_{\widetilde U}(f_\str,\varepsilon)\to N_{\widetilde U'}(f_\str,\varepsilon)$, we obtain the required operator $\RR$.
\end{proof}

 To simplify notations, we will often write an operator in Theorem~\ref{thm: An Pacm  self ope} as $\RR\colon \BB\dashrightarrow \BB$ with $\BB=N_{\widetilde U}(f_\str,\delta)$. We can assume (by Lemma~\ref{lem:SiegPacmExist}) that $f_\str$ has any given truncation level between $0$ and $1$.

\begin{cor}
\label{cor:UnstManExist}
In a small neighborhood of $f_\str$, the operator $\RR\colon \BB\dashrightarrow \BB$ has an analytic finite-dimensional unstable submanifold $\WW^u$ tangent to the unstable direction of $\RR$.
\end{cor}
\noindent We will show in Theorem~\ref{thm:RR is hyper} that $\WW^u$ has dimension $1$.
\begin{proof}
Since $\RR$ is compact, it has a finite-dimensional unstable direction.

\cite{HPS}*{Corollary (5.4)} asserts that $\WW^u$ exists as a $C^\infty$-smooth submanifold. The corollary is proven by showing that the graph transform on the space of submanifolds in a sufficiently small cone-neighborhood of the unstable direction of $\RR$ (i.e.~``candidates'' to $\WW^u$) has the unique fixed point $\WW^u$. In our analytic setup, the graph transform iteratively applied to an analytic submanifold gives a sequence of analytic submanifolds converging exponentially fast and uniformly to $\WW^u$ (see the estimate on page 55 of~\cite{HPS}). Therefore, $\WW^u$ is an analytic submanifold.
\end{proof}

An \emph{indifferent pacman} is a pacman with indifferent $\alpha$-fixed point. The \emph{rotation number} of an indifferent pacman $f$ is $\theta \in \R/\Z$ so that $\ee(\theta)$ is the multiplier at $\alpha(f)$. If, moreover, $\theta\in \Q$, then $f$ is \emph{parabolic}.

We denote by $\theta_\str$  the multiplier of $f_\str$.

\begin{lem}
\label{lem:eq:S3:R_prm}
Let $\cRRc\colon \R/\Z\to \R/\Z$ be the map defined by
\begin{equation}
\label{eq:S3:R_prm}
 \cRRc (\theta)  = \begin{cases} \frac{\theta}{1-\theta}& \mbox{if }0\le \theta \le \frac{1}{2} \\
\frac{2\theta-1}{\theta} & \mbox{if }\frac{1}{2}\le \theta\le 1, \end{cases}
\end{equation} see~\eqref{eq:R_prm}. Then there is a $\kk\ge 1$ such that the following holds. Let $f\in \BB$ be an indifferent pacman with rotation number $\theta$. Then $\RR f$ is again an indifferent pacman with rotation number $\cRRc^\kk (\theta)$. 
\end{lem}
\noindent In particular, $\cRRc^\kk (\theta_\str)=\theta_\str$.
\begin{proof}
Recall that the renormalization $\RR$ of $ f_\str$ is an extension of a sector renormalization of $f\mid \overline Z_\str$, see Definition~\ref{defn:SectRenormZp} and Appendix~\ref{ss:ap:SecRen}. By Lemma~\ref{lem:SectRen is prime power}, a sector renormalization is an iteration of the prime renormalization. Therefore, $\RR$ is an iteration of the prime pacman renormalization $\RRc$, see Definition~\ref{den:PacmRenorm}. We need to check that if $f$ is an indifferent pacman with rotation number $\theta$, then $\RRc f$ is again an indifferent pacman with rotation number $\cRRc (\theta)$. By continuity, it is sufficient to assume that $f$ is a parabolic pacman with rotation number $\pp/\qq$. Then $f$ has $\qq$ local attracting petals in a small neighborhood of $\alpha$. If $\pp\le \qq/2$, then $\RRc$ deletes $\pp$ local attracting petals; otherwise $\RRc$ deletes $\qq-\pp$ local attracting petals. We see that $\RRc f$ has rotation number $\cRRc(\pp/\qq)$.
\end{proof}

\begin{rem}
\label{lem:min ren per}
We will show in~\cite{DL} that $\RR:\BB \dashrightarrow\BB$ can be constructed so that $\kk$ is the minimal period of $\theta_\str$ under $\RRc$. 
\end{rem}

\section{Control of pullbacks}

Let us fix the renormalization operator \[\RR\colon \BB\dashrightarrow \BB, \sp\sp \RR f_\str=f_\str\] from Theorem~\ref{thm: An Pacm  self ope} around a fixed Siegel pacman $f_\str$. By Corollary~\ref{cor:UnstManExist} $\RR$ has an unstable manifold $\WW^u$ at $f_\str$.

\subsection{Renormalization triangulations}
\label{ss:RenormalTriangul}

Suppose that $f_0\in \BB$ is renormalizable $n\ge 0$ times (this is always the case if $f_0$ is sufficiently close to $f_\str$) and anti-renormalizable $-m\ge 0$ times.
We write $[f_k\colon U_k\to V]\coloneqq \RR^k f_0$ for the $k$th (anti-) renormalization of $f_0$, where $m \le k\le n$. We denote by $\psi_{k}\colon S_k\to V$ the renormalization change of variables realizing the renormalization of $f_{k-1}$ (compare with the left side of Figure~\ref{Fig:Sf1dash}). We write
\[\phi_k\coloneqq \psi_k^{-1}.\]

Let us cut the dynamical plane of $f_k\colon U_k\to V$, with $ k\in \{m, \dots,   n\}$, along $\gamma_1$; we denote the resulting 
prepacman by
\begin{equation}
\label{eq:defn:F_k}
F_{k} = \left(f_{k, \pm } \colon U_{k,\pm}\to V\setminus \gamma_1\right). 
\end{equation}

\begin{lem}
\label{lem:psi is adjusted}
By restricting $\RR$ to a smaller neighborhood of $f_\str$, the following is true. Suppose $f_0$ is renormalizable $n\ge 1$ times. Then the map 
\[ \Phi_{n} \coloneqq \phi_{1}\circ \phi_{2}\circ \dots\circ \phi_{n}\] 
admits a conformal extension from a neighborhood of $c_1(f_n)$ (where $\Phi_{n}$ is defined canonically) to $V\setminus \gamma_1$. The map $\Phi_{n}\colon V\setminus \gamma_1 \to V$ embeds the prepacman $F_n$~\eqref{eq:defn:F_k} to the dynamical plane of $f_0$; we denote the embedding by 
\begin{align*}
 F^{(0)}_n= &\left(f^{(0)}_{n,\pm }\colon U^{(0)}_{n,\pm}\to S_{n}^{(0)}\right)\\=& \left(f^{\aa_{n}}_{0}\colon U^{(0)}_{n,-}\to S_{n}^{(0)}, \sp f^{\bb_{n}}_{0}\colon U^{(0)}_{n,+}\to S_{n}^{(0)}\right),
\end{align*}
where the numbers $\aa_{n}, \bb_{n}$ are the renormalization return times satisfying~\eqref{eq:RernReturns}.

Let $\bDelta_n$ be the triangulation obtained by spreading around $U^{(0)}_{n,-}$ and $U^{(0)}_{n,+}$, see{\rm{~\S\ref{sec:PacmenRenorm}}} and Figure~{\rm\ref{Fg:Prepacman:OrbitOfUi}}. In the dynamical plane of $f_0$ we have 
\[\bDelta_0\coloneqq \overline U_0\Supset \bDelta_1\Supset \bDelta_2 \Supset \dots \Supset \bDelta_n,\]
$\bDelta_1(f_0)$ is close in Hausdorff topology to $\bDelta_1(f_\str)$, and moreover $f_0(\bDelta_n)\Subset \bDelta_{n-1}$.
\end{lem}
\noindent We call $\bDelta_n$ the \emph{$n$th renormalization triangulation.} Examples of $\bDelta_0,\bDelta_1,\bDelta_2$ are shown in Figures~\ref{Fig:DfnDelta} and~\ref{Fig:RenormTiling:2levels}. We say that $\bDelta_n(f_0)$ is the \emph{full lift} of $\bDelta_{n-1}(f_1)$. Similarly (i.e.~by lifting and then spreading around), a \emph{full lift} will be defined for other objects.

In the proof of Lemma~\ref{lem:psi is adjusted} we need to deal with the fact that $\psi_1(\gamma_1)$ can spiral around $\alpha$, see Figure~\ref{Fig:SpirAt0} for illustration. We will first show in Lemma~\ref{lem:SiegWall} that Lemma~\ref{lem:psi is adjusted} holds in a neighborhood of $\partial Z_\str$. By the topological robustness of anti-renormalization (Theorem~\ref{thm:RobAntiRen}), Lemma~\ref{lem:psi is adjusted} holds also inside $Z_\str$.

\subsubsection{Combinatorics of triangles}
Before giving the proof of Lemma~\ref{lem:psi is adjusted}, let us introduce additional notations. For consistency, we set $\Phi_0\coloneqq \id$; then $\bDelta_0=\overline U_0$ is a triangulation consisting of two closed triangles -- the closures of the connected components of $U_0\setminus (\gamma_0\cup \gamma_1)$. We denote these triangles by $\Delta_0(0)$ and $\Delta_0(1)$ so that $\intr(\Delta_0(0)), \gamma_0, \intr(\Delta_0(1))$ have counterclockwise orientation around $\alpha$, see Figure~\ref{Fig:DfnDelta}. The triangulation $\bDelta_0(f_n)$ is defined similarly.

 Let $\Delta_n(0,f_0), \Delta_n(1,f_0)$ be the images of $\Delta_0(0,f_n), \Delta_0(1,f_n)$ via the map $\Phi_{n}$ from Lemma~\ref{lem:psi is adjusted}. By definition, $\bDelta_n$ is a triangulated neighborhood of  $\alpha$ obtained by spreading around $\Delta_n(0,f_0), \Delta_n(1,f_0)$. We enumerate in counterclockwise order these triangles as $\Delta_n(i)$ with $i\in \{0,1,\dots, \qq_n-1\}$. By construction, $\Delta_n(0)\cup \Delta_n(1)\ni c_1(f_n)$.

We remark that $f_0\mid \bDelta_n$ is an anti-renormalization of $f_n\colon U_n\to V$ in the sense of Appendix~\ref{ss:ap:SecRen}. Moreover, there is a $\pp_n$ such that \begin{equation}
 \label{eq:f_0 resp triangl}
 f_0\colon  \Delta_n(i)\to\Delta_n(i+\pp_n)
 \end{equation}
is conformal for all $i\not\in \{-\pp_n, -\pp_n+1 \}$ with the index taken modulo $\qq_n$. For the exceptional triangles, we have an almost two-to-one map  \begin{equation}
 \label{eq:f_0 viol triangl}f_0\colon \Delta_n(-\pp_n)\cup \Delta_n(-\pp_n+1) \to S_0^{(n)}\supset \Delta_n(0)\cup \Delta_n(1).
 \end{equation} 
We will show in Theorem~\ref{thm:ComPsConj} that if $f_0$ is close to $f_\str$, then $\displaystyle{\bDelta_n= \bigcup_i \Delta_n(i)}$ approximates $\overline Z_\str$ dynamically and geometrically.

By construction, for every triangle $\Delta_n(i,f_0)$ there is a $t\ge 0$ and $j\in \{0,1\}$ such that a certain branch of $f^{-t}_0$ maps conformally $\Delta_n(i,f_0)$ to $\Delta_n(j,f_0)$. We define $\Psi_{n,i}$ on $\Delta_n(i,f_0)$ by
\begin{equation}
\label{eq:PsiDefn}
 \Psi_{n,i} \coloneqq \Phi_{n}^{-1}\circ f_0^{-t} \colon \Delta_n(i, f_0)\to \Delta_0(j,f_n). 
\end{equation}

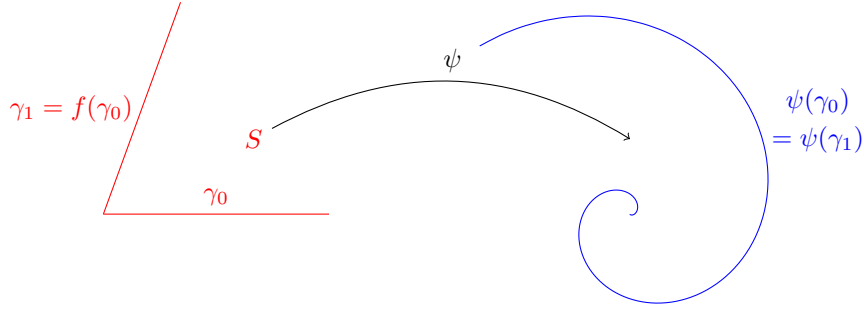
\begin{figure}
\centering{\begin{tikzpicture}
        \draw [scale=1.2,blue,domain=0:10,variable=\t,smooth,samples=100]
        plot ({3\t r}: {0.025*\t*\t});
   
 \begin{scope}[shift={(-7,0)}]     
        \draw[red]
          (0,0) edge node[above]{$\gamma_0$} (3,0);
       \draw[red,rotate=70]
          (0,0) edge node[left]{$\gamma_1=f(\gamma_0)$}  (3,0); 
 \end{scope}
     \draw (-5,1) node[red]{$S$} edge[bend left,->] node[above]{$\psi$} (0,1);
          \node[blue] at (2.5,1.5) {$\psi(\gamma_0)$};
     \node[blue] at (2.5,1) {$=\psi(\gamma_1)$};
 \end{tikzpicture}}       
 \caption{Suppose $f(z)=\lambda z$ with $\lambda\not \in \R_{+}$ and let $S$ be the sector between $ \gamma_0$ and $\gamma_1=f(\gamma_0)$. Let $\psi \colon S\to \C$ be the gluing map identifying dynamically $\gamma_0$ and $\gamma_1$. If $|\lambda|\not=1$, then $\psi(\gamma_0)$ does not land at $0$ at a well defined angle.}
 \label{Fig:SpirAt0}
\end{figure}

\subsubsection{Walls}
Let $A$ be a closed annulus, and let $O$ be the bounded component of $\C\setminus A$. We say that $A$ is a \emph{univalent $N$-wall} if the restriction  $f_0\mid  O\cup  A$ is univalent and for all $z\in O$ and all $j$ with $|j|\le N$ we have \[(f_0\mid  O\cup  A)^{j}(z)\subset O\cup A.\]
More generally, we say that $A$ is an \emph{$N$-wall} if $A$ contains a univalent $N$-wall $A'$ such that $O$ is in the bounded component of $\C\setminus A'$.

Fix a small $r>0$ and denote by $Z_\str^{ r}$ the open subdisk of $Z_\str$ bounded by the equipotential at height $r$. Set $\bPi_0\coloneqq \overline U_0\setminus Z^{r}_\str$. It is a closed annulus enclosing $\alpha$. We decompose $\bPi_0$ into two closed rectangles $\Pi_0(0)=\bPi_0\cap \Delta_0(0)$ and $\Pi_0(1)=\bPi_0\cap \Delta_0(1)$; they are the closures of the connected components of $\bPi_0\setminus (\gamma_0\cap \gamma_1)$. The following lemma proves that the \emph{wall of $\bDelta_n$} exists.

\begin{lem}[The wall of $\bDelta_n$]
\label{lem:SiegWall}
Suppose all $f_0,f_1,\dots, f_n$ are in a small neighborhood of $f_\str$. Then there exists a wall $\bPi_n(f_0)$ with the following properties.
\begin{enumerate}
\item The map $\Phi_n$ extends from a neighborhood of $c_1(f_n)$ to $\bPi_0\setminus \gamma_1$;
\label{Claim:Axil:on expansion:1}
\item Let $\Pi_n(0,f_0)$ and $\Pi_n(1,f_0)$ be the images of $\Pi_0(0,f_n)$ and $\Pi_0(1,f_n)$ under $\Phi_n$. Then, by spreading around $\Pi_n(0,f_0)$ and $\Pi_n(1,f_1)$, we obtain an annulus $\bPi_n$ enclosing $\alpha$. We enumerate counterclockwise rectangles in $\bPi_n$ as $\Pi_n(i)$ with $i\in \{0,1,\dots, \qq_n-1\}$. \label{Claim:Axil:on expansion:2}
\item We have $\bPi_0\Supset \bPi_1\Supset \dots \Supset \bPi{}_n$ with $\bPi_0(f_0)$ close to $\bPi_0(f_\str)$.\label{Claim:Axil:on expansion:3}
\item For every $\Pi_n(i)$, there is a $t\ge0$ such that a certain branch of $f^{-t}_n$ maps $\Pi_n(i)$ onto $\Pi_n(j)$ with $j\in\{0,1\}$. Then 
\begin{equation}
\label{eq:PsiDefn:2}
 \Psi_{n,i} \coloneqq \Phi_{n}^{-1}\circ f_0^{-t} \colon \Pi_n(i, f_0)\to \Pi_0(j,f_n). 
\end{equation}
is conformal. If $n$ is sufficiently big, then all $\Psi_{n,i}$ expand the Euclidean metric and the expanding constant is at least $\eta^n$ for a fixed $\eta>1$. In particular, the diameters of the rectangles in $\bPi_n$ tend to $0$. 
\label{Claim:Axil:on expansion:4}
\item The wall $\bPi_n(f_0)$ approximates $\partial Z_\str$ in the following sense: $\partial Z_\str$ is a concatenation of arcs $J_0J_1\dots J_{\qq_n-1}$ such that $\Pi_n(i)$ and $J_i$ are close in the Hausdorff topology.
\label{Claim:Axil:on expansion:5}
\end{enumerate}
\end{lem}
\noindent As with renormalization triangulation, we say that  $\bPi_n(f_0)$ is the \emph{full lift} of $\bPi_{n-1}(f_1)$.  
\begin{proof}
The proof follows from the robustness of the renormalization change of variables in a neighborhood of $\partial Z_\str$. Such change of variables is eventually expanding.

Consider first the case $f_0=f_1=\dots= f_n= f_\str$. It follows from the improvement of the domain that the wall $\bPi_n(f_\str)$ is well defined and, moreover, the diameters of the rectangles in $\bPi_n(f_\str)$ tend to $0$ as $n$ increases. Choosing a sufficiently big $k$ and applying the Schwarz lemma (after a slight enlargement of the rectangles), we obtain that all the $\Psi_{k,i}  \colon \Pi_k(i, f_\str)\to \Pi_0(j,f_\str)$ expand the Euclidean metric. 

By continuity and the assumption that $f_0,f_1,\dots, f_n$ are sufficiently close to $f_\str$, the maps $\Psi_{k,i}  \colon \Pi_k(i, f_{s})\to \Pi_0(j,f_{s+k})$ also expand the Euclidean metric. Decomposing $\Psi_{n,i}  \colon \Pi_n(i, f_{0})\to \Pi_0(j,f_{n})$ into a composition of $\lfloor \frac n k \rfloor$ maps of the form $\Psi_{k,t}$ and one remaining map, we see that $\Psi_{n,i}  \colon \Pi_n(i, f_{0})\to \Pi_0(j,f_{n})$ is a required expanding map. This implies Claim~\eqref{Claim:Axil:on expansion:4}; other claims are consequences of Claim~\eqref{Claim:Axil:on expansion:4}.
\end{proof}

\subsubsection{Proof of Lemma~\ref{lem:psi is adjusted}}
We will now apply Theorem~\ref{thm:RobAntiRen} to show that the full lift $\bDelta_{n}(f_0)$ of $\bDelta_{0}(f_n)$ exists.

Let $\wall_0\subset Z_\str$ be the closed annulus bounded by the equipotentials at heights $r$ and $2r$. Then $\wall_0\subset \bPi_0$ and we decompose $\wall_0$ into two rectangles $\wall_0(0)=\Pi_0(0)\cap \wall_0$ and  $\wall_0(1)=\Pi_0(1)\cap \wall_0$. Let $\wall_n(0,f_0)$ and $\wall_n(1,f_0)$ be the images of $\wall_0(0,f_n)$ and $\wall_0(1,f_n)$ under $\Phi_n$. By spreading around $\wall_n(0,f_0)$ and $\wall_n(1,f_0)$, we obtain (by  Lemma~\ref{lem:SiegWall}) an annulus $\wall_n$ enclosing $\alpha$. We enumerate counterclockwise rectangles in $\wall_n$ as $\wall_n(i)$ with $i\in \{0,1,\dots, \qq_n-1\}$. We have $\wall_n(i)\subset \Pi_n(i)$. 

Denote by $\inn_n$ the open topological disk enclosed by $\wall_n$. Then $f_0\mid \inn_n \cup \wall_n$ is an anti-renormalization of $f_1\mid \inn_{n-1} \cup \wall_{n-1}$ (in the sense of Appendix~\ref{s:SectRenAndAntiren}) with respect to the dividing pair of curves $\gamma_0,\gamma_1$.

By induction, we will now extend the wall $\bPi_n$ to the triangulation $\bDelta_n$. Suppose the statement is verified for $n-1$.  In the dynamical plane of $f_1$, we denote by $\gamma_0^{(n-1)}$ the lift of $\gamma_0(f_n)$ under the $(n-1)$-anti-renormalization specified so that $\gamma_0^{(n-1)}$ crosses $\wall_{n-1}$ at $\wall_{n-1}(0)\cap \wall_{n-1}(1)$. Note that $\wall_n(0)\cup \wall_n(1)$ is in a small neighborhood of $c_1$ because $\Phi_n$ is contracting. Therefore,  $\gamma_0^{(n-1)}\cap \wall_{n-1}$ is uniformly close to $\gamma_0\cap \wall_{n-1}$. We can sightly adjust $\gamma_0$ in a neighborhood of $\wall_{n-1}$, such that the new $\gamma_0^\new$ crosses $\wall_{n-1}$ at $\wall_{n-1}(0)\cap \wall_{n-1}(1)$. Let $\gamma_1^{(n-1)}$ and $\gamma_1^\new$ be the images of $\gamma_0^{(n-1)}$ and $\gamma_0^\new$ respectively. Since a wall contains a fence (see Remark~\ref{rem:wall has a fence}), by Theorem~\ref{thm:RobAntiRen} the anti-renormalization of $f_1\mid  \inn_{n-1} \cup \wall_{n-1}$ with respect to $\gamma_0^{(n-1)}, \gamma_1^{(n-1)}$ is naturally conjugate to the corresponding anti-renormalization of $f_1\mid  \inn_{n-1} \cup \wall_{n-1}$ with respect to  $\gamma_0^\new, \gamma_1^\new$. Therefore,  the full lift $\bDelta_{n}(f_0)$ of $\bDelta_{n-1}(f_1)$ exists; $\bDelta_{n}(f_0)$ is a  required triangulation of $\bPi_n\cup \inn_n$.

  By~\eqref{eq:lem:FixedSiegelPacman} combined with continuity, we have $f_0(\bDelta_1)\subset \bDelta_0$.  Applying induction on $n$, we obtain $f_0(\bDelta_{n+1})\Subset \bDelta_n$.

We can now define $F_n^{(0)}=\left(f^{(0)}_{n,\pm }\colon U^{(0)}_{n,\pm}\to S_{n}^{(0)}\right)$ as the lift of $F_n$ (see~\eqref{eq:defn:F_k}) to the dynamical plane of $f_0$, where \[S_n^{(0)}\coloneqq f_0\left( \Delta_n(-\pp_n)\cup\Delta_n(-\pp_n+1) \right)\]  
(compare with~\eqref{eq:f_0 viol triangl}). \qed

\subsubsection{Changing $\gamma_1$}
In fact, the exact behavior of $\gamma_1$ in a small neighborhood of $\alpha$ is irrelevant in the proof of Lemma~\ref{lem:psi is adjusted}. We have

\begin{lem}
\label{lem:gam1 rotate}
 Let $\gamma_0^\new, \gamma_1^\new=f_n(\gamma_0^\new)$ be a new pair of curves in the dynamical plane of $f_n$ such that
\begin{itemize}
\item $\gamma_0\setminus Z^r_\str=\gamma_0^\new \setminus Z^r_\str$ and $\gamma_1\setminus Z^r_\str=\gamma_1^\new \setminus Z^r_\str$; and
\item $\gamma_0^\new$ and $\gamma_1^\new$ are disjoint away from $\alpha$. 
\end{itemize}

Then Lemma~\ref{lem:psi is adjusted} still holds after replacing $\gamma_0,\gamma_1$ with $\gamma^\new_0,\gamma_1^\new$. More precisely, let $\Delta^\new_0(0,f_n),\Delta_0^\new(1,f_n)$ be the closures of the connected components of $U_0\setminus (\gamma^\new_0\cup \gamma^\new_1)$ in the dynamical plane of $f_n$. As in Lemma~\ref{lem:psi is adjusted} the map $\Phi_n$ extends from a neighborhood of $c_1(f_n)$ to $V\setminus \gamma_1^\new$; let $\Delta^\new_n(0,f_0),\Delta_n^\new(1,f_0)$ be the images of $\Delta^\new_0(0,f_n),\Delta_0^\new(1,f_n)$ under the new $\Phi_n$. By spreading around $\Delta^\new_n(0,f_0),\Delta_n^\new(1,f_0)$  we obtain a new triangulated neighborhood $\bDelta_n^\new$ of $\alpha$. 
\end{lem}
\noindent Note that $\bDelta_n^\new$ and $\bDelta_n$ triangulate the same neighborhood of $\alpha$. 
\begin{proof}
Since $\gamma^\new_1,\gamma^\new_0$ coincide with $\gamma_1,\gamma_0$ outside $Z^r$, the wall $\bPi_n$ is unaffected; thus we can repeat the proof of Lemma~\ref{lem:psi is adjusted} for $\gamma^\new_1$.
\end{proof}

\subsubsection{Siegel triangulations}
\label{ss:SiegTriang} We will also consider triangulations that are perturbations of $\bDelta_n$. Let us introduce appropriate notations.
Consider a pacman $f\in \BB$. A \emph{Siegel triangulation} $\bDelta$ is a triangulated neighborhood of $\alpha$ consisting of closed triangles, each has a vertex at $\alpha$, such that
\begin{itemize}
\item triangles of $\bDelta$ are $\{\Delta(i)\}_{i\in \{0,\dots \qq\}}$ enumerated counterclockwise around $\alpha$ so that $\Delta(i)$ is attached to $\Delta(i-1)$ (on the right) and to $\Delta(i+1)$ (on the left); all other pairs of triangles are disjoint away from $\alpha$;
\item there is a $\pp>0$ such that $f$ maps $\Delta(i)$ to $\Delta(i+\pp)$ for all $i\not\in \{-\pp, -\pp+1\}$, while $f(\Delta(-\pp, -\pp+1))\cap \bDelta=\Delta(0,1)$;
\item $\bDelta$ has a distinguished $2$-wall $\bPi$ enclosing $\alpha$ and containing $\partial \bDelta$ such that each $\Pi(i)\coloneqq \bPi\cap \Delta(i)$ is connected and $f$ maps $\Pi(i)$ to $\Pi(i+\pp)$ for all $i\not\in \{-\pp, -\pp+1\}$; and 
\item  $\bPi$ contains a univalent $2$-wall $\wall$ such that each $\wall(i)\coloneqq \wall\cap \Pi(i)$ is connected and $f$ maps $\wall(i)$ to $\wall(i+\pp)$ for all $i\not\in \{-\pp, -\pp+1\}$.
 \end{itemize}
The $n$th renormalization triangulation is an example of a Siegel triangulation.

Similar to Lemma~\ref{lem:SiegWall}, Part~\eqref{Claim:Axil:on expansion:5}, we say that $\bPi$ \emph{approximates} $\partial Z_\str$ if $\partial Z_\str$ is a concatenation of arcs $J_0J_1\dots J_{\qq-1}$ such that $\Pi(i)$ and $J_i$ are close in the Hausdorff topology.

\begin{lem}
\label{lem:SiegTriangLifting}
Let $f\in \BB$ be a pacman such that all $f,\RR f,\dots, \RR^n f$ are in  a small neighborhood of $f_\str$. Let $\bDelta(\RR^n f)$ be a Siegel triangulation in the dynamical plane of $\RR^n f$ such that $\bPi(\RR^n f)$ approximates $\partial Z_\str$. Then $\bDelta(\RR^n f)$ has a full lift $\bDelta(f)$ which is again a Siegel triangulation. Moreover, $\bPi(f)$ also approximates $\partial Z_\str$.
\end{lem}
\begin{proof}
It is similar to the proof of  Lemma~\ref{lem:psi is adjusted}. Suppose first ${n=1}$. Since all $\Pi(i, \RR f)$ are small, the arc $\gamma_0$ can be slightly adjusted\footnote{The lift of the triangulation will depend on this adjustment.} in a neighborhood of $\bPi$ so that $\gamma_0$ crosses $\bPi$ along $\Pi(i, \RR f)\cap \Pi(i+1, \RR f)$ with $i\not\in \{-\pp,-\pp+1\}$. This allows us to construct a full lift $\bPi(f)$ of $\bPi(\RR f)$. By Corollary~\ref{cor:lift of 2wall is 2wall}, the annuli $\bPi(f)$ and $\wall(f)$ are again $2$-walls. Applying Theorem~\ref{thm:RobAntiRen} from Appendix~\ref{s:SectRenAndAntiren} we construct a full lift $\bDelta(f)$ of $\bDelta(\RR f)$. Lemma~\ref{lem:SiegWall} Part~\eqref{Claim:Axil:on expansion:4} allows to apply induction on $n$: for big $n$, the wall $\bPi(f)$ approximates $\partial Z_\str$ better than $\bPi(\RR^n f)$ approximates $\partial Z_\str$.
\end{proof}

\subsection{Renormalization tilings}
\label{ss:RenTiling}

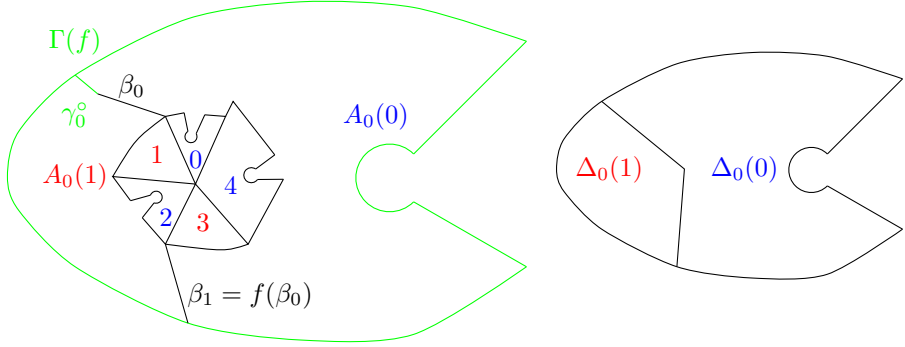
\begin{figure}[t!]
\centering{\begin{tikzpicture}[scale=1.5]

\begin{scope}[shift={(2,0)},rotate=45,scale =0.3]
\draw[green] (0,0) arc (0:270:1);
\coordinate (A1) at (0,0);
\coordinate (B1) at (-1,-1);
\end{scope}
\draw[green] (A1) -- (3,1);
\draw [green] (B1)-- (3,-1);
\draw[green] plot[smooth,tension=0.5] 
coordinates{ (3,1) (2,1.3) (1,1.35) (0,1.2)(-1,0.7) (-1.5, 0.2) (-1.6,-0.2)
 (-1.5, -0.5) (-1,-1) (0,-1.5) (1,-1.65) (2,-1.6) (3,-1)};
 \draw (0,-1.5)-- (0.1,-0.2);
 \draw (0.1,-0.2) --  (-0.78, 0.52); 
 \draw[green] (-0.78, 0.52)--(-1,0.7);

\node[below,green] at (-1,0.6){$\gamma_0^\circ$}; 
\node[above, green] at (-0.6,1) {$\Gamma(f)$};
 
 \draw[blue] (0.9,-0.2) node {$\Delta_0(0)$};
  \draw[red] (-0.8,-0.4) node {$\Delta_0(1)$};
  
  \begin{scope}[scale=0.666,shift={(-0.3,0)}]
  
  \begin{scope}[shift={(-6.5,0)}, scale =1.5,green]
\begin{scope}[shift={(2,0)},rotate=45,scale =0.3]
\draw (0,0) arc (0:270:1);
\coordinate (A2) at (0,0);
\coordinate (B2) at (-1,-1);
\end{scope}
\draw (A2) -- (3,1);
\draw (B2) -- (3,-1);
\draw plot[smooth,tension=0.5] 
coordinates{ (3,1) (2,1.3) (1,1.35) (0,1.2)(-1,0.7) (-1.5, 0.2) (-1.6,-0.2)
 (-1.5, -0.5) (-1,-1) (0,-1.5) (1,-1.65) (2,-1.6) (3,-1)};
 \coordinate (C) at (0,-1.5); 
\coordinate (D) at (-1,0.7);
\draw (-1,1) node{$\Gamma(f)$};
\draw (-1,0.35) node{$\gamma_0^\circ$};
\draw[black] (-0.5,0.6) node{$\beta_0$};
\draw[black] (0.55,-1.25) node{$\beta_1=f(\beta_0)$};
  \end{scope}

 \begin{scope}[shift={(-6.5,0.3)},rotate=120,scale =0.08]
\draw (0,0) arc (0:270:1);
\coordinate (A3) at (3,-1);
\draw (0,0)--(A3);
\coordinate (B3) at (1,-3.8);
\draw (B3) --(-1,-1);
\end{scope} 
\draw (A3)-- (-6.8,0.5) -- (-6.4,-0.4)--(-6,0.5)--(B3); 

\begin{scope}[ scale =0.7]
\draw[red] (-9.87, -0) node {$1$}
(-9., -1.3)  node {$3$}; 
\draw[blue] (-9.14, -0.1) node {$0$}
(-9.7, -1.2) node {$2$}
(-8.47, -0.6) node {$4$};
\end{scope}

 \begin{scope}[shift={(-6.95,-0.65)},rotate=250,scale =0.08]
\draw (0,0) arc (0:270:1);
\coordinate (A4) at (3,-1);
\draw (0,0)--(A4);
\coordinate (B4) at (1,-3.8);
\draw (B4) --(-1,-1);
\end{scope} 
\draw (B4)-- (-7.5,-0.3)-- (-6.4,-0.4)--(-6.8,-1.2)--(A4);
\draw(-7.5,-0.3) ..controls (-7.2,0.2)..(-6.8,0.5);

\draw (-6.8,-1.2) .. controls (-6,-1.3) .. (-5.7,-1.2);

\begin{scope}[shift={(-5.6,-0.2)},rotate=55,scale =0.1]
\draw (0,0) arc (0:270:1);
\coordinate (A4) at (3,-1);
\draw (0,0)--(A4);
\coordinate (B4) at (1,-3.8);
\draw (B4) --(-1,-1);
\end{scope} 

\draw (B4) --(-5.7,-1.2)-- (-6.4,-0.4);
\draw (A4) -- (-5.9,0.7) --(-6,0.5);

\draw[green] (D)--(-6.8-0.9,0.5+0.3);
\draw (-6.8-0.9,0.5+0.3)-- (-6.8,0.5);
\draw (C)--(-6.8,-1.2);

\draw[blue] (-4,0.5) node {$A_{0}(0)$};
\draw[red] (-8,-0.3) node {$A_{0}(1)$};  
\end{scope}
\end{tikzpicture}}
\caption{Renormalization tiling. Right: triangles $\Delta_0(0),\Delta_0(1)$ are the closures of the connected components of $U_f\setminus (\gamma_0\cup \gamma_1)$. They form a renormalization tiling of level $0$. Left: the triangles labeled by $0$ and $1$, i.e.~$\Delta_1(0,f)$ and $\Delta_1(1,f)$ respectively, are anti-renormalization embeddings of $\Delta_0(0,f_1),\Delta_0(1,f_1)$; the forward orbit of $\Delta_1(0,f),\Delta_1(1,f)$ triangulates a neighborhood of $\alpha$. Together with $A_0(0)\cup A_0(1)$, this gives a tiling of $U_f$ of level $1$.} \label{Fig:DfnDelta}
\end{figure}
In this subsection we will show the robustness of renormalization triangulations. Along the lines we will also extend $\bDelta_n(f)$ to a tiling of $U_f$.

Let $\gamma_0^\circ $ be the preimage of $\overline{\gamma_1\setminus U_f}$ under $f\colon \bar \gamma_0\to \bar \gamma_1$, see Figure~\ref{Fig:DfnDelta}. In other words, $\gamma_0^\circ$ is the subcurve of $\bar\gamma_0$ consisting of points that escape $U_f$ after one iteration. Set $\Gamma(f)\coloneqq \partial U_f\cup \gamma_0^\circ (f)$.

\begin{lem}
\label{lem:InitHolMotion}
For every $i$ we have 
\[  \Psi_{1,i}(\partial \bDelta_1(f_0)\cap \partial \Delta_1(i,f_0))\subset \Gamma(f).\]
Moreover, there is an $i$ such that $\gamma_0^\circ\subset  \Psi_{1,i}(\partial \bDelta_1(f_0)\cap \partial \Delta_1(i,f_0))$. The set $\Gamma(f)$ is disjoint from $\bDelta_0(f)$.

There are disjoint arcs $\beta_0$ and $\beta_1=f(\beta_0)$ such that
\begin{itemize}
\item the concatenation of $\gamma^\circ_0$ and $\beta_0$ connects $\partial \bDelta_0$ to $\partial \bDelta_1$; and
\item $\beta_1$ connects $\partial \bDelta_0$ to $\partial \bDelta_1$.
\end{itemize}
In a small neighborhood of $f_\str$ the curves $\beta_0,\beta_1$ can be chosen so that there is a holomorphic motion of 
\begin{equation}
\label{eq:0:lem:InitHolMotion}
\left[\bDelta_1\cup \partial \bDelta_0\cup \gamma^\circ_0\cup \beta_0\cup\beta_1\right] (f_0)
\end{equation} that is \emph{equivariant} with the following maps
\begin{enumerate}
\item $f_0\colon \beta_0(f_0)\to \beta_1(f_0)$;\label{eq:1:lem:InitHolMotion}
\item $f_0\colon \Delta_1(i,f_0)\to \Delta_1(i+\pp_1,f_0)$ for $i\not \in \{-\pp_1,-\pp_1+1\}$;\label{eq:2:lem:InitHolMotion}
\item $\Psi_{1,i}\colon \partial \bDelta_1(f_0)\cap \Delta_1(i,f_0)\to \Gamma(f_1)$.\label{eq:3:lem:InitHolMotion}
\end{enumerate}
\end{lem}

\begin{proof}
Each triangle $\Delta_{1}(i)$ has three distinguished closed sides; we denote them by $\lambda(i)$, $\rho(i)$, and $\ell(i)$ such that $\lambda(i)$ and $\rho(i)$ are the left and right sides meeting at the $\alpha$-fixed point while $\ell(i)$ is the opposite to $\alpha$ side. We have:
\[\partial \bDelta_1(f_0)\cap \partial \Delta_1(i,f_0))=\ell_i\cup \big(\overline{\lambda(i)\triangle \rho(i+1)}\big)\cup \big(\overline{\rho(i)\triangle \lambda(i-1)}\big), \]
where $\lambda(i)\triangle \rho(i+1)$ is the symmetric difference between $\lambda(i)$ and $\rho(i+1).$
Note that  $\Psi_{1,i}( \ell (i))\subset \partial \bDelta_0$ and, moreover, $\displaystyle \bigcup_i\Psi_{1,i}(\ell(i)) =\partial \bDelta_0$.

 Let us analyze $ \overline{\lambda(i)\triangle \rho(i+1)}$. We assume that $\lambda(i)\not= \rho(i+1)$. Then one of the curves in $\{\lambda(i), \rho(i+1)\}$ is a $\Psi$-preimage of $\gamma_0(f_1)$ while the other is a preimage of $\gamma_1(f_1)$. We have: 
 \[ \Psi_{1,i}\big(\overline{\lambda(i)\triangle \rho(i+1)}\big)=\gamma_0^\circ.\]
It is clear (see Appendix~\ref{ss:ap:SecRen}) that $\lambda(i)\not= \rho(i+1)$ for at least one $i$.

The property  $\Gamma(f)\cap \partial \bDelta_1(f)=\emptyset$ follows from $\partial \bDelta_0\cap f(\bDelta_1)=\emptyset$, see Lemma~\ref{lem:psi is adjusted}. Since $\Gamma(f)\cap \partial \bDelta_1(f)=\emptyset$, we can find $\beta_0$ such that $\gamma_0^\circ\cup \beta_0$ is in a small neighborhood of $\gamma_0$ and $\gamma_0^\circ\cup \beta_0$  connects $\partial \bDelta_0$ to $\partial (\bDelta_1\setminus (\Delta_1(-\pp_1)\cup  \Delta_1(-\pp_1-1)))$. Then $\beta_1=f(\beta_0)$ is disjoint from $\gamma^\circ_0\cup \beta_0$ and $\beta_1$ connects $\partial \bDelta_0$ to $\partial \bDelta_1$.

In a small neighborhood of $f_\str$ we have a holomorphic motion of $\partial \bDelta_0(f_0).$ Applying the $\lambda$-lemma, we obtain a holomorphic motion of the triangulation $\bDelta_0$ that is equivariant with $f_0\mid \gamma_0$. Lifting this motion via $\Psi_{1,i}$, we obtain a holomorphic motion of $\bDelta_1\cup \Gamma$ equivariant with~\eqref{eq:2:lem:InitHolMotion} and~\eqref{eq:3:lem:InitHolMotion}. Applying again the $\lambda$-lemma, we extend the latter motion to the motion of~\eqref{eq:0:lem:InitHolMotion} that is also equivariant with~\eqref{eq:1:lem:InitHolMotion}. 
\end{proof}

 Let $\bA_0$ be the closed annulus between $\partial \bDelta_0$  and $\partial  \bDelta_1$. The arcs $\gamma^\circ_0\cup \beta_0,\beta_1$ split $\bA_0$ into two closed \emph{rectangles $A_0(0), A_0(1)$} (see Figure~\ref{Fig:DfnDelta}) enumerated such that  $\intr(A_0(0)),\gamma_0^\circ \cup \beta_0, \intr( A_0(1)), \beta_1$ have counterclockwise orientation.

Let $\bA_n$ be the closed annulus between $\partial \bDelta_n$ and $\partial \bDelta_{n+1}$. Define \[A_n(0,f_0)\coloneqq \Phi_n(A_0(0,f_n))\sp\text{ and }\sp  A_n(1,f_0)\coloneqq \Phi_n(A_0(1,f_n))\] and spread dynamically $A_n(0,f_0),A_n(1,f_0)$ (compare with the definition of $\Delta_n(i)$ in~\S\ref{ss:RenormalTriangul}); we obtain the partition of $\bA_n(f_0)$ by rectangles $\{A_n(i,f_0)\}_{0\le i<\qq_n}
$ enumerated counterclockwise. Similar to~\eqref{eq:PsiDefn:2} we define the map $ \Psi_{n,i}\colon A_n(i,f_0)\to A_0(j,f_n)$ with $j\in \{0,1\}$.

\emph{The $n$th renormalization tiling} is the union of all the triangles of $\bDelta_{n}$ and the union of all the rectangles of all $\bA_m$ for all $m<n$. The $n$th renormalization tiling is defined as long as $f_0,\dots, f_n$ are in a small neighborhood of $f_\str$.

A \emph{qc combinatorial pseudo-conjugacy of level $n$ between $f_0$ and $f_\str$} is a qc map $h\colon \overline U_{0}\to \overline U_{\str}$ that is compatible with the $n$th renormalization tilings as follows:

\begin{itemize}
\item $h$ maps $\Delta_n(i,f_0)$ to $\Delta_n(i,f_\str)$ for all $i$;
\item $h$ maps $A_m(i,f_0)$ to $A_m(i,f_\str)$ for all $i$ and $m<n$;
\item $h$ is equivariant on $ \Delta_n(i,f_0)$ for all $i\not \in \{-\pp_n, -\pp_n+1\}$; and
\item  $h$ is equivariant on $A_m(i,f_0)$ for all $i\not \in \{-\pp_m, -\pp_m+1\}$ and $m<n$.
\end{itemize}

\begin{figure}[t!]
\centering{\begin{tikzpicture}[scale=2.2]

\draw (0.29125,0.29) .. controls  (1.16,1.04)..
(1.16,1.04).. controls  (1.16-0.5,1.04+0.9) and (-0.6025+0.5,2.17125-0.1)..
(-0.6025,2.17125).. controls (-0.6025-0.3,2.17125+0.1) and (-2.5781103515625037+0.3,2.1729345703124996+0.1)..
(-2.5781103515625037,2.1729345703124996).. controls
(-2.5781103515625037-0.25,2.1729345703124996) and (-3.804916992187504+0.1,1.1383886718749987+0.25)
..  (-3.804916992187504,1.1383886718749987)
.. controls 
(-3.804916992187504-0.07,1.1383886718749987) and
(-4.19625,0.115+0.1)..
(-4.19625,0.115).. controls
(-4.19625,0.115-0.25) and
 (-3.612656250000004-0.1,-0.9856347656250029)..
 (-3.612656250000004,-0.9856347656250029).. controls
 (-3.612656250000004+0.1,-0.9856347656250029-0.1) and (-2.5323339843750037-0.3,-1.7272119140625035)..
(-2.5323339843750037,-1.7272119140625035).. controls (-0.463242187500003-0.8,-1.9103173828125035).. (-0.463242187500003,-1.9103173828125035)
..controls 
(-0.463242187500003+0.5,-1.9103173828125035) and
(1.20375,-0.4053125000000028-0.5)..
(1.20375,-0.4053125000000028)..
 controls (0.29125,0.29)..(0.29125,0.29);
\draw (-2.2689062500000015,1.4228124999999963)-- (-3.073593750000002,0.8290624999999966);
\draw (-3.073593750000002,0.8290624999999966)-- (-3.366562500000002,-0.014687500000003066);
\draw (-3.343125000000002,-0.206093750000003)-- (-2.675156250000002,-0.7783593750000036);
\draw (-2.675156250000002,-0.7783593750000036)-- (-1.309921875000001,-0.9658593750000037);
\draw (-1.309921875000001,-0.9658593750000037)-- (-0.426621093750003,-0.22574707031250227);
\draw (-0.426621093750003,-0.22574707031250227)-- (-0.7845312500000009,-0.05375);
\draw (-0.7845312500000009,-0.05375)-- (-0.380844726562503,0.08553222656249797);
\draw (-2.2689062500000015,1.4228124999999963)-- (-1.5306250000000012,1.4228124999999963);
\draw (-1.5306250000000012,1.4228124999999963)-- (-1.5345312500000012,1.1376562499999965);
\draw (-1.5345312500000012,1.1376562499999965)-- (-1.3157812500000012,1.3485937499999963);
\draw (-1.3157812500000012,1.3485937499999963)-- (-0.9993750000000009,0.9442968749999974);
\draw (-0.9993750000000009,0.9442968749999974)-- (-0.7836767578125031,1.1109228515624987);
\draw (-0.380844726562503,0.08553222656249797)-- (-0.7836767578125031,1.1109228515624987);
\draw (-2.2689062500000015,1.4228124999999963)-- (-2.1185156250000023,0.8930273437499984);
\draw (-3.073593750000002,0.8290624999999966)-- (-2.624375000000002,0.5009374999999966);
\draw (-2.675156250000002,-0.7783593750000036)-- (-2.2845312500000015,-0.4014062500000029);
\draw (-1.309921875000001,-0.9658593750000037)-- (-1.5032812500000012,-0.44046875000000285);
\draw (-0.9993750000000009,0.9442968749999974)-- (-1.342148437500002,0.6513281249999983);

\draw[red,fill=red, fill opacity=0.3]  (-2.1185156250000023,0.8930273437499984)-- (-1.9056250000000015,0.2040624999999968)--
(-2.4095312500000015,0.7548437499999966)--
(-2.1185156250000023,0.8930273437499984);
\draw[blue,fill=blue, fill opacity=0.3] (-1.9056250000000015,0.2040624999999968)-- (-2.624375000000002,0.5009374999999966)--
(-2.4095312500000015,0.7548437499999966)--
(-1.9056250000000015,0.2040624999999968);

\draw[red,fill=red, fill opacity=0.3] 
(-1.9056250000000015,0.2040624999999968)--
(-2.624375000000002,0.5009374999999966)-- (-2.647812500000002,0.09468749999999689)--
(-1.9056250000000015,0.2040624999999968);

\draw[blue,fill=blue, fill opacity=0.3] (-1.9056250000000015,0.2040624999999968)-- (-2.538437500000002,-0.17875)-- (-2.647812500000002,0.09468749999999689)--
(-1.9056250000000015,0.2040624999999968);

\draw[blue,fill=blue, fill opacity=0.3] (-1.9056250000000015,0.2040624999999968)-- (-2.2845312500000015,-0.4014062500000029)-- (-2.4533836490637815,-0.12729767042129922)--
(-1.9056250000000015,0.2040624999999968);

\draw[red,fill=red, fill opacity=0.3] 
(-1.9056250000000015,0.2040624999999968)-- (-2.2845312500000015,-0.4014062500000029)-- (-1.8939062500000015,-0.5342187500000029)--
(-1.9056250000000015,0.2040624999999968);

\draw[blue,fill=blue, fill opacity=0.3] (-1.9056250000000015,0.2040624999999968)-- (-1.5032812500000012,-0.44046875000000285)-- (-1.8939062500000015,-0.5342187500000029)--
(-1.9056250000000015,0.2040624999999968);

\draw[red,fill=red, fill opacity=0.3] 
(-1.9056250000000015,0.2040624999999968)-- (-1.5032812500000012,-0.44046875000000285)-- (-1.206406250000001,-0.13968750000000302)--
(-1.9056250000000015,0.2040624999999968);

\draw[blue,fill=blue, fill opacity=0.3] (-1.9056250000000015,0.2040624999999968)-- (-1.206406250000001,-0.13968750000000302)-- (-1.061875000000001,0.37593749999999665)--
(-1.9056250000000015,0.2040624999999968);

\draw[blue,fill=blue, fill opacity=0.3] (-1.9056250000000015,0.2040624999999968)-- (-1.342148437500002,0.6513281249999983)-- (-1.2060654017122159,0.34656538113269364)--
 (-1.9056250000000015,0.2040624999999968);

\draw[red,fill=red, fill opacity=0.3] 
(-1.9056250000000015,0.2040624999999968)-- (-1.342148437500002,0.6513281249999983)-- (-1.5423437500000012,0.8485937499999966)--
(-1.9056250000000015,0.2040624999999968);

\draw[blue,fill=blue, fill opacity=0.3] (-1.9056250000000015,0.2040624999999968)-- (-1.5423437500000012,0.8485937499999966)-- (-1.8768164062500021,0.9662695312499984)--
(-1.9056250000000015,0.2040624999999968);

\draw[blue,fill=blue, fill opacity=0.3] (-1.9056250000000015,0.2040624999999968)-- (-2.1185156250000023,0.8930273437499984)-- (-1.8807868631382758,0.8612206634093975)--
(-1.9056250000000015,0.2040624999999968);

\draw (-3.366562500000002,-0.014687500000003066)-- (-3.007187500000002,-0.09671875000000303);
\draw (-3.007187500000002,-0.09671875000000303)-- (-3.343125000000002,-0.206093750000003);
\draw (-2.2689062500000015,1.4228124999999963)-- (-2.5781103515625037,2.1729345703124996);
\draw (-2.675156250000002,-0.7783593750000036)-- (-2.5323339843750037,-1.7272119140625035);

\draw[red] (-3.70421, 0.28695) node{$A_0(1)$}
(-2.57811, 0.90951) node{$A_1(1)$}
(-1.69005, -0.75675) node{$A_1(3)$};
\draw[blue](0.07692, -0.5828) node {$A_0(0)$}
(-1.87315, 1.16585) node {$A_1(0)$}
(-2.77953, -0.34477) node {$A_1(2)$}
(-0.7379, 0.26864) node{$A_1(4)$};

\draw (-1.97386, 0.7203) node{$0$}
(-2.16917, 0.68978) node{$1$}
(-2.37669, 0.52499) node{$2$}
(-2.48045, 0.29305) node{$3$}
(-2.47435, 0.00619) node{$4$}
(-2.27293, -0.16471) node{$5$}
(-2.02879, -0.3234) node{$6$}
(-1.74803, -0.3173) node{$7$}
(-1.49779, -0.1525) node{$8$}
(-1.3513, 0.13436) node{$9$}
(-1.47948, 0.40292) node{$10$}
(-1.58934+0.03, 0.61654) node{$11$}
(-1.79076+0.05, 0.75692) node{$12$};

\draw (-1.90672, -0.76286) edge[<-,bend left=35] (-2.76122-0.05, 0.8912);

\end{tikzpicture}}
\caption{Renormalization tiling of level $2$; tilings of smaller levels are depicted on Figure~\ref{Fig:DfnDelta}. There are $\qq_2=12$ triangles in $\bDelta_2$ with rotation number $\pp_2/\qq_2=5/12$. Geometry of triangles in $\bDelta_2$ is simplified. The image of $\Delta_2(8)\cup \Delta_2(9)$ is a sector slightly bigger than $\Delta_2(0)\cup \Delta_2(1)$ -- compare with Figure~\ref{Fg:Prepacman:OrbitOfUi}.} \label{Fig:RenormTiling:2levels}
\end{figure}
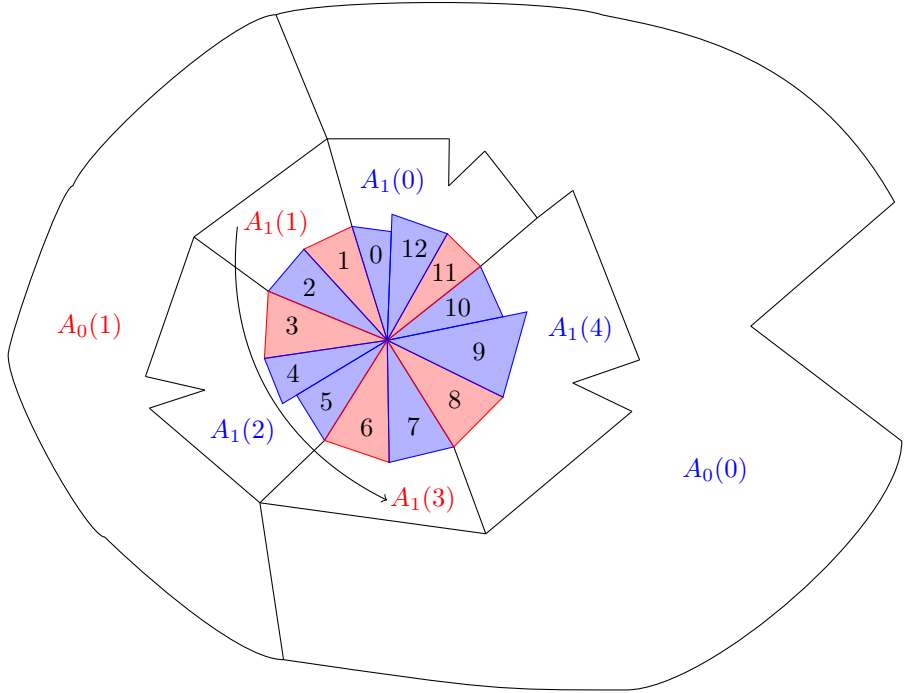

The following theorem says that $f\mid \bDelta_n(f)$ \emph{approximates} $f_\str\mid \overline Z_\str$ both dynamically and geometrically.
\begin{thm}[Combinatorial pseudo-conjugacy]
\label{thm:ComPsConj}
Consider an $n$th renormalizable pacman $f$ and set
\[d\coloneqq \max_{i\in\{0,1,\dots, n\}} \dist(\RR^i f,f_\str).\]

If $d$ is sufficiently small, then there is a qc combinatorial pseudo-conjugacy $h$ of level $n$ between $f$ and $f_\str$ and, moreover, the following properties hold. The qc dilatation and the distance between $h\mid \bDelta_n(f)$ and the identity on $\bDelta_n(f)$ are bounded by constants $K(d), M(d)$ respectively with $K(d)\to 1$ and $M(d)\to 0$ as $d\to 0$.
\end{thm}
\begin{proof}
By Lemma~\ref{lem:InitHolMotion}, the set~\eqref{eq:0:lem:InitHolMotion} moves holomorphically with $f$ in a small neighborhood of $f_\str$. Applying the $\lambda$-lemma, we obtain a holomorphic motion $\tau$ of the first renormalization tiling with $f$ in a small neighborhood $\UU$ of $f_\str$.

Suppose now that $d$ is so small that all $f_i\coloneqq \RR^i f$ are in $\UU$. For every $\Delta_n(i)$ of $f_0$ or of $f_\str$ consider the map $ \Psi_{n,i}\colon \Delta_n(i)\to \Delta_0(j)$, where $\Delta_0(j)$ is the corresponding triangle of $f_n$ or of $f_\str$. Then $h$ on $\Delta_n(i,f_0)$ is defined by applying first $\Psi_{n,i}\colon \Delta_n(i,f_0)\to \Delta_0(j,f_n)$ (see~\eqref{eq:PsiDefn}), then applying the motion $\tau$ from $ \Delta_0(j,f_n)$ to $ \Delta_0(j,f_\str)$, and then applying $\Psi_{n,i}^{-1}\colon  \Delta_0(j,f_\str)\to  \Delta_n(i,f_\str)$.

Similarly, for every $A_m(i)$ of $f_0$ or of $f_\str$ consider the map $\Psi_{m,i}\colon A_m(i)\to A_0(j)$, where $A_0(j)$ is the corresponding rectangle of $f_{m}$ or of $f_\str$. Then $h$ on $A_m(i,f_0)$ is defined by applying first $\Psi_{m,i}\colon A_m(i,f_0)\to A_0(j,f_m)$, then applying the motion $\tau$ from $A_m(j,f_0)$ to $A_m(j,f_\str)$, and then applying $\Psi_{m,i}^{-1}\colon  A_0(j,f_\str)\to  A_m(i,f_\str)$. 

Observe now that $h$ is well defined for all the points on the boundaries of all the rectangles and all the triangles because $\tau$ is equivariant with~\eqref{eq:1:lem:InitHolMotion},~\eqref{eq:2:lem:InitHolMotion},
~\eqref{eq:3:lem:InitHolMotion} -- see Lemma~\ref{lem:InitHolMotion}. Therefore, all points have well defined images under $h$.

The qc dilatation of $h$ is bounded by the qc dilatation of $\tau$ at $f_i$ with $i\in\{0,1,\dots, n\}$. This bounds the qc dilatation of $h$ by $K(d)$ as above with $K(d)\to 1$ as $d\to 0$.

If $n=1$, then since $\tau$ is continuous, the distance between  $h\mid \bDelta_1(f_0)$ and the identity on $\bDelta_1(f_0)$ is bounded by $M(d)$ as required. If $n>1$, then $\bDelta_n(f_0)\Subset U_{0}$ and the claim follows from the compactness of qc maps with bounded dilatation
\end{proof}

\begin{cor} 
\label{cor:inf ren are hybr conj}
There is an $\varepsilon>0$ with the following property. Suppose that $f\in \BB$ is infinitely renormalizable and that all $\RR^n  f$ for $n\ge 0$ are in the $\varepsilon$-neighborhood of $f_\str$. Then there is a qc map $h\colon U_f\to U_\str$ such that $h^{-1}$ is a conjugacy on $\overline Z_\str$. Therefore, a certain restriction of $f$ is a Siegel map and $f,f_\str$ are hybrid conjugate on neighborhoods of their Siegel disks.
\end{cor}
\begin{proof}
If $\varepsilon$ is sufficiently small, then by Theorem~\ref{thm:ComPsConj}, for every $n\ge 0$ there exists qc combinatorial pseudo-conjugacy $h_n$ of level $n$ between $f$ and $f_\str$ such that the dilation of $h_n$ is uniformly bounded for all $n$. By compactness of qc maps, we may pass to the limit and construct a qc map $h\colon U_f\to U_\str$ such that $h^{-1}$ is a conjugacy on $\overline Z_\str$ and $c_0(f)\in h^{-1}(\overline Z_\str)$. It follows, in particular, that $f$ is a Siegel map. By Theorem~\ref{thm:HybrCongSiegMaps}, the maps $f,f_\str$ are hybrid conjugate on neighborhoods of their Siegel disks.
\end{proof}
\subsection{Control of pullbacks}
Recall from Lemma~\ref{lem:psi is adjusted} that  $\aa_n, \bb_n$ denote the closest renormalization return times computed by~\eqref{eq:RernReturns}. By definition, $\aa_n+ \bb_n=\qq_n$. We now restrict out attention to $f\in \WW^u$.

\begin{keylem}
\label{keylem:ContrOfPullbacks}
There is a small open topological disk $D$ around $c_1(f_\str)$ and there is a small neighborhood $\UU\subset \WW^u$ of $f_\str$ such that the following property holds. For every sufficiently big $n\ge 1$, for each $\tt\in \{\aa_n,\bb_n\}$, and for all $f\in \RR^{-n} (\UU)$, we have ${c_{1+\tt}(f)\coloneqq f^\tt(c_1)\in D
}$ and $D$ can be pulled back along the orbit $c_1(f), c_2(f),\dots , c_{1+\tt}(f)\in D$ to a disk $D_{0}$ such that $f^{\tt}\colon D_{0} \to D$ is a branched covering; moreover, $D_0\subset U_{f}\setminus \gamma_1$.
\end{keylem}

\begin{proof}
The main idea of the proof is to block the forbidden part of the boundary $\partial^{\frb} U_f$ from the backward orbit of $D$.
The proof is split into short subsections. We start the proof by introducing conventions and additional terminology. The central argument will be presented in Claim~\ref{cl:ContrD_k}, Part~\eqref{cl:ContrD_k:4}.

\begin{figure}[t!]
\centering{\begin{tikzpicture}

\draw (-3,0)--(0,2)--(3,0)--(0,-2)--(-3,0);

\draw (0,0)-- (-1.5,1.4)--(-1,1.75)--(0,0);
\draw[fill, opacity=0.1] (0,0)-- (-1.5,1.4)--(-1,1.75)--(0,0);
\node[right] at (-1,1.75) {$S$};

\node[below] at (0,-1) {$\bDelta$};

\draw (0,0)-- (2.85,-0.1);
\draw (0,0)-- (2.4,0.4);

\node[above,red] at (-2.45,2.0) {$D_{j}$};

\node[above,red] at (2.45,2.0) {$D_{j-1}$};

\draw (2.15,2.0) edge[->, bend right] node[above] {$f$}(-2.05,2.0);

\draw (1.5,0.1) edge[->, bend left=80] node[above] {} (-0.7,0.8);

\draw[red](-2.2,0.56)
.. controls (-2.1,1.85) and (-1.8,1.65)..
 (-1.5,1.45) 
.. controls (-1.4,1.35) and (-1.3,1.35)..
(-1.2,1.55) 
.. controls (-1.7,1.95) and (-2.6,2.45)
..
(-2.6,0.3).. controls (-2.4,0.05) and (-2.2,0.3)
..(-2.2,0.56);

\draw[red,fill=red, opacity=0.2](-2.2,0.56)
.. controls (-2.1,1.85) and (-1.8,1.65)..
 (-1.5,1.45) 
.. controls (-1.4,1.35) and (-1.3,1.35)..
(-1.2,1.55) 
.. controls (-1.7,1.95) and (-2.6,2.45)
..
(-2.6,0.3).. controls (-2.4,0.05) and (-2.2,0.3)
..(-2.2,0.56);

\begin{scope}[shift={(1.1,-1.07)}, rotate=-90 ]

\draw[red](-2.3,-0.2)
.. controls (-3.65,1.7) and (-1.8,1.65)..
 (-1.5,1.45) 
.. controls (-1.4,1.35) and (-1.3,1.35)..
(-1.2,1.55) 
.. controls (-1.7,1.95) and (-3.9,2.05)
..
(-2.6,-0.5).. controls (-2.4,-0.65) and (-2.3,-0.5)
..(-2.3,-0.2);

\draw[red,fill=red, opacity=0.2](-2.3,-0.2)
.. controls (-3.65,1.7) and (-1.8,1.65)..
 (-1.5,1.45) 
.. controls (-1.4,1.35) and (-1.3,1.35)..
(-1.2,1.55) 
.. controls (-1.7,1.95) and (-3.9,2.05)
..
(-2.6,-0.5).. controls (-2.4,-0.65) and (-2.3,-0.5)
..(-2.3,-0.2);
\end{scope}

\draw[line width=0.6mm] (-3+0.3,0.2)--(-3+0.9,0.6);
\node[below right] at (-3+0.5,0.35){$I$};
\draw[line width=0.6mm] (0+0.51,2-0.4+0.06)--(1.2-0.03,2-0.8+0.02);
\node[below left] at (1.2-0.03,2-0.8+0.02) {$f^{-1}(I)$};


\end{tikzpicture}}
\caption{If $D_{j}$ intersects $S=\Delta(0)\cup \Delta(1)$ and $S$ is disjoint from $\Delta(I)$, then $D_{j-1}$ may intersect $\bDelta\setminus \Delta(f^{-1}(I))$ because $f(\bDelta)=\bDelta\cup S$.} \label{Fig:DfnLambda}
\end{figure}

\subsubsection{The triangulated disk $\bDelta$ approximates $\overline Z_\str$}
\label{sss:KeyLmmConvent}
Throughout the proof we will often say that a certain object is \emph{small} if it has  small size independently of $n$. Choose a big $s\gg 0$ and choose a small neighborhood $\UU$ of $f_\str$ such that every $f\in \RR^{-n}(\UU)$ is at least $m\coloneqq n+s$ renormalizable and each $f_i\coloneqq \RR^i f$ with $i\in  \{0,1,\dots ,m\}$ is close to $f_\str.$

Consider the $m$-th renormalization triangulation $\Delta_{m}(i)$ of $f$. Let $h$ be a qc combinatorial pseudo-conjugacy of level $m$ as in Theorem~\ref{thm:ComPsConj}. To keep notation simple, we sometimes drop the subindex $m$ and write $\Delta(i), \bDelta, \qq, \pp$ for $\Delta_{m}(i),\bDelta_{m}, \qq_{m}, \pp_{m}$.

Since $f_i$ with $i\in  \{0,1,\dots ,m\}$ are close to $f_\str$, the map $h\mid \bDelta$ is close to the identity (by Theorem~\ref{thm:ComPsConj}). In particular, $\bDelta(f)=h^{-1}(\bDelta(f_\str))$ approximates $\overline Z_\str$. Since $s$ is big and since $\aa_i,\bb_i$ have exponential growth with the same exponent~\eqref{eq:RernReturns}, we have
 \begin{equation}
 \label{eq:t over q is small}
 \tt/\qq_m\in \{\aa_n/\qq_{n+s}~,\sp \bb_{n}/\qq_{n+s}\} \sp \text{ is arbitrary small.}
\end{equation}

\subsubsection{Disks $D_k\ni f^{k}(c_1)$}  For convenience, we will write $c^\str_0 = c_0(f_\str)$ and $c^\str_1 = c_1(f_\str)$. Let us show that $D\ni f^{\tt}(c_1)$. Consider first the dynamical plane of $f_\str$. Since $n$ is big, we see that $f_\str^{\aa_n}(c^\str_1), f_\str^{\bb_n}(c^\str_1)$ are arbitrary close to $c^\str_1$; i.e.~$D\ni f^{\tt}_\str(c_1)$. It follows from~\eqref{eq:t over q is small} that  
\begin{equation}
\label{eq:a_n less than a_m}
\min\{\aa_m, \bb_m\}-1>  \max\{\aa_n, \bb_n\}\ge \tt.
\end{equation}
This shows that $c^\str_1, \dots , f_\str^{\tt}(c^\str_1)$ do not visit triangles $ \Delta( -\pp_m,f_\str) \cup  \Delta( -\pp_m+1,f_\str)$ as it takes either $\aa_m-1$ or $\bb_m-1$ iterations for a point in $ \Delta( 0,f_\str) \cup  \Delta(1,f_\str)$ to visit them. Since $h$ is a conjugacy away from $ \Delta( -\pp) \cup  \Delta( -\pp+1)$, we obtain that $h^{-1}$ maps $c^\str_1, \dots , f_\str^{\tt}(c^\str_1)$ to  $c_1, \dots , f^{\tt}(c_1)$. Since $h$ is close to the identity, $f^{\tt}(c_1)$ is close to $f_\str^{\tt}(c^\str_1)$; thus $f^\tt(c_1)\in D$.

 Let $D_{0}, D_{1},\dots  D_{\tt}=D$ be the pullbacks of $D$ along the orbit $c_1,\dots,$ $f^\tt(c_1)\in D$; i.e.~$D_\tt\coloneqq D\ni f^\tt(c_1)$ and $D_{i}$ is the connected component of $f^{-1}(D_{i+1})$ containing $f^{i}(c_1)$. Our main objective is to show that the $D_{i}$ do not intersect $\partial ^\frb U_f$; this will imply that the maps $f\colon D_{i}\to D_{i+1}$ are branched coverings for all $i\in \{0,\dots, \tt-1\}$.

\subsubsection{Sectors $ \Delta(I)$ and $ \Lambda(I)$}
\label{sss:Delta Lambda}
An \emph{interval} $I$ of $\Z/\qq\Z$ is a set of consecutive numbers $i,i+1,\dots , i+j $ taken modulo $\qq$. We define the \emph{sector parametrized by $I$} as $\displaystyle{\Delta(I)\coloneqq \bigcup _{i\in I}\Delta(i)}$. Furthermore, we let
\begin{equation}
\label{eq:DefnI_k}
f^{-1}(I)\coloneqq \begin{cases} I-\pp&\text{if } I\cap \{0,1,\pp,\pp+1\} =\emptyset \\
      (I -\pp)\cup  \{-\pp,-\pp+1\} & \text{if } I\cap \{0,1\} \not =\emptyset\\
      (I -\pp)\cup  \{0,1\} &\text{if } I\cap \{\pp,\pp+1\} \not =\emptyset.
    \end{cases}
\end{equation}

    In other words, we require that if $I-\pp$ contains one of $ -\pp, -\pp+1$, then it also contains the other number; and similarly with pair $ 0, 1$. By~\eqref{eq:f_0 resp triangl} and~\eqref{eq:f_0 viol triangl}
\begin{claim}
\label{cl:B1} The preimage of $\Delta(I)$ under $f\mid \bDelta$ is within $\Delta(f^{-1}(I))$.\qed
\end{claim}

 Unfortunately, we do not have the property that \[\text{if }D_j\cap \bDelta\subset \Delta(I),\sp \text{ then }D_{j-1}\cap\bDelta \subset \Delta( f^{-1}(I))\] because the image of $\Delta(-\pp)\cup \Delta(-\pp+1)$ is slightly bigger than $\Delta(0)\cup \Delta(1)$, see~\eqref{eq:f_0 viol triangl}. To handle this issue, we will adjust $\bDelta$ to a slightly smaller triangulated neighborhood $\bLambda$ such that 
 \begin{equation}
 \label{eq:prop bLambda}
   \bLambda \subseteq f^i(\bLambda)\subseteq \bDelta 
\end{equation}
for all $i\in  \{0,1,\dots, \min\{\aa_m,\bb_m\}\}.$  

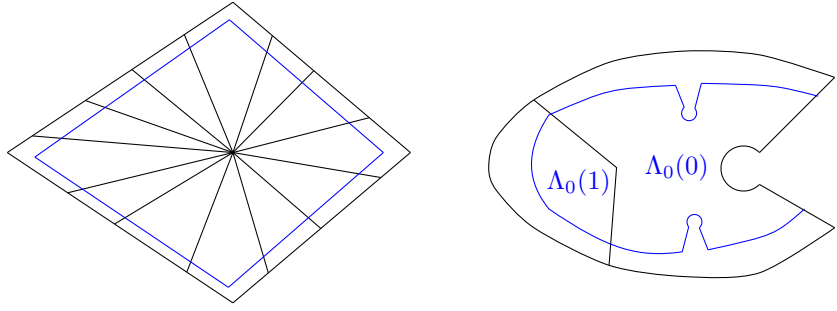
\begin{figure}[t!]
\centering{\begin{tikzpicture}

\begin{scope}[shift={(2,0)},rotate=45,scale =0.3]
\draw (0,0) arc (0:270:1);
\coordinate (A1) at (0,0);
\coordinate (B1) at (-1,-1);
\end{scope}
\draw (A1) -- (3,1);
\draw (B1) -- (3,-1);
\draw plot[smooth,tension=0.5] 
coordinates{ (3,1) (2,1.3) (1,1.35) (0,1.2)(-1,0.7) (-1.5, 0.2) (-1.6,-0.2)
 (-1.5, -0.5) (-1,-1) (0,-1.5) (1,-1.65) (2,-1.6) (3,-1)};
 \draw (0,-1.5)-- (0.1,-0.2) --(-1,0.7);
 
 \draw[blue] (0.9,-0.2) node {$\Lambda_0(0)$};
  \draw[blue] (-0.4,-0.4) node {$\Lambda_0(1)$};

 \begin{scope}[blue, shift={(1,0.6)},rotate=130,scale =0.1]
\draw (0,0) arc (0:270:1);
\coordinate (AA1) at (3,-1);
\draw (0,0)--(AA1);
\coordinate (BB1) at (1,-3.8);
\draw (BB1) --(-1,-1);
\end{scope}   
\draw[blue] (BB1) .. controls (2.2,0.9).. (2.78,0.75);

 \begin{scope}[blue, shift={(1.2,-1)},rotate=310,scale =0.1]
\draw (0,0) arc (0:270:1);
\coordinate (AA2) at (3,-1);
\draw (0,0)--(AA2);
\coordinate (BB2) at (1,-3.8);
\draw (BB2) --(-1,-1);
\end{scope}

\draw[blue] (AA2) .. controls (2.2,-1.1).. (2.6,-0.75);

\draw[blue] (BB2) ..controls (0.4,-1.4) and (0,-1.3).. (-0.8,-0.75) ..  controls (-1,-0.5) and (-1.2,-0.1) ..  (-0.8,0.5) .. controls (0,0.85) .. (AA1);
 
\begin{scope}[shift={(-5,0)}]
\draw (-3.,0.)-- (0.,2.);
\draw (0.,2.)-- (2.36,0.);
\draw (2.36,0.)-- (0.,-2.);
\draw (0.,-2.)-- (-3.,0.);
\draw (0.,0.)-- (0.5242033104831968,1.5557599063701721);
\draw (0.,0.)-- (-0.6461538461538461,1.5692307692307692);
\draw (0.,0.)-- (-1.3430769230769233,1.1046153846153846);
\draw (0.,0.)-- (-1.9615384615384615,0.6923076923076924);
\draw (0.,0.)-- (-2.667692307692308,0.22153846153846146);
\draw (0.,0.)-- (-2.196923076923077,-0.5353846153846151);
\draw (0.,0.)-- (-1.5738461538461537,-0.9507692307692308);
\draw (0.,0.)-- (-0.6138461538461539,-1.5907692307692307);
\draw (0.,0.)-- (0.47034275204815246,-1.6014044474168199);
\draw (0.,0.)-- (1.1683598060524996,-1.0098645711419494);
\draw (0.,0.)-- (1.9675873599732485,-0.33255308476843337);
\draw (0.,0.)-- (1.8125163016218024,0.46396923591372663);
\draw (0.,0.)-- (1.075238254472496,1.0887811402775456);
\draw[blue] (-0.06,-1.793333333333333)-- (2.,0.);
\draw[blue] (-0.046666666666666294,1.7666666666666664)-- (2.,0.);
\draw[blue] (-0.046666666666666294,1.7666666666666664)-- (-2.6333333333333333,-0.06);
\draw[blue] (-2.6333333333333333,-0.06)-- (-0.06,-1.793333333333333);
\end{scope}

\end{tikzpicture}}
\caption{Right: $\Lambda_0(0,f_m)$ and $\Lambda_0(1,f_m)$ are shrunk versions of $\Delta_0(0,f_m)$ and $\Delta_0(1,f_m)$. Left: by transferring $\Lambda_0(0,f_m)$ and $\Lambda_0(1,f_m)$ to $\Lambda_m(0,f_0)$ and $\Lambda_m(1,f_0)$ by $\Phi_{m,0}$, and spreading these triangles dynamically, we obtain the triangulated neighborhood $\bLambda_m$ of $\alpha$ such that $\bLambda_m$ is a slightly shrunk version of $\bDelta_m$; compare with Figures~\ref{Fig:DfnDelta} and~\ref{Fig:RenormTiling:2levels}.} \label{Fig:DfnLambda}
\end{figure}

Consider the dynamical plane of $f_m=\RR^m f$ and let $\Lambda_0(0,f_{m})$ and $\Lambda_0(1,f_{m})$ be the closures of the connected components of $f_{m}^{-1}(U_{m})\setminus (\gamma_1\cup \gamma_{0})$ attached to $\alpha$ such that $\Lambda_0(0,f_{m})\subset \Delta_0(0,f_{m})$ and  $ \Lambda_0(1,f_{m})\subset \Delta_0(1,f_{m})$, see Figure~\ref{Fig:DfnLambda}. Writing $\bLambda_0(f_m)= \Lambda_0(0,f_{m})\cup  \Lambda_1(1,f_{m})$ we obtain a shrunk version of $\bDelta_0(f_m)$. The map $\Phi_{m}$ embeds $\Lambda_0(0,f_{m})$ and $\Lambda_1(1,f_{m})$ to the dynamical plane of $f_0$; spreading around the embedded triangles, we obtain a triangulated neighborhood $\bLambda$ of $\alpha$ such that~\eqref{eq:prop bLambda} holds. 

Let us also give a slightly different description of $\bLambda$.    Recall~\eqref{eq:PsiDefn} that $\Psi_{0,m}$ maps each $\Delta_m(i,f_{0})$ conformally to some $\Delta_0(j,f_{m})$. Then $\Lambda(i)=\Lambda_m(i,f_{0})\subset \Delta_m(i,f_{0})$ is the preimage of $\Lambda_0(j,f_{m})$ under that map. We define \[\bLambda \coloneqq \bigcup_{0\le i< \qq} \Lambda(i) \sp\text{ and }\sp \Lambda(I)=\bigcup_{i\in I} \Lambda(i).\] For the same reason as for $\bDelta(f_\str)$, the triangulation $\bLambda(f_\str)$ approximates $\overline Z_\str$. And since $h\mid \bDelta$ is close to the identity, $\bLambda(f)$ also approximates $\overline Z_\str$ in the sense of Theorem~\ref{thm:ComPsConj}. For the same reason as for Claim~\ref{cl:B1}, we have:
\begin{claim}
\label{cl:B2} We have $\Lambda(i)=\bLambda\cap \Delta(i)$ for every $i$. The preimage of $\Lambda(I)$ under $f\mid \bLambda$ is within $\Lambda(f^{-1}(I))$.\qed
\end{claim}

The following claim is a refinement of~\eqref{eq:prop bLambda}. This will help us to control the intersections of $D_k$ with $\bLambda$.

\begin{claim}
\label{cl:B3} Let $I$ be an interval. Consider $z\in \bLambda$. If $f^i(z)\in \Delta(I)$ for $i<\min\{\aa, \bb\}$, then $z\in \Lambda(f^{-i}(I)).$

As a consequence, if $T\cap \bDelta\subset \Delta(I)$ for an interval $I$ and a set $T\subset V$,  then \[f^{-i}(T)\cap \bLambda \subset \Lambda(f^{-i}(I))\]
for all $i<\min\{\aa, \bb\}$.
\end{claim}
\begin{proof}
Since $f^i(z)\in \Delta(I)$, every preimage of $f^i(z)$ under the $i$-th iterate of $f\mid \bDelta$ is within $\Delta(f^{-1}(I))$ by Claim~\ref{cl:B1}. By Claim~\ref{cl:B2}, $z\in\Delta(f^{-i}(I))\cap \bLambda \subset\Lambda(f^{-i}(I))$.

The second statement follows from the first because points in $\bLambda$ do not escape $\bDelta$ under $f^i$ for all $i<\min\{\aa, \bb\}$, see~\eqref{eq:prop bLambda}.
\end{proof}

\subsubsection{Truncated sectors $S_k$ and disks $\mathfrak D_k\supset \mathfrak D'_k\supset D_k$}
\label{sss:wideD_j}
Let $I_{\tt}$ be the smallest interval containing $\{0,1\}$ such that $\Delta(I_{\tt},f)\supset D_{\tt}\cap \bDelta(f)$ for all $f$ subject to the condition of Key Lemma. Set $I_{\tt-j}\coloneqq f^{-j}(I_{\tt})$. By Claim~\ref{cl:B3} we have $D_k\cap \bLambda\subset \Lambda(I_k)$. 

 Recall that the intersection of each $\Delta(i, f_\str)$ with $\overline Z_\str$ is a closed sector of $\overline Z_\str$ bounded by two closed internal rays of $\overline Z_\str$. Let us fix $p>1$ which will be specified in \S\ref{sss:BubbleChains} as a certain period.  Since $D_0$ is small, we obtain:
\begin{claim}
\label{cl:C}  \begin{enumerate}
\item All $|I_k| /\qq$ are small. All $\Delta( I_k,f_\str)$ have a small angle at the $\alpha$-fixed point. \label{cl:C:1} 
\item For every $j\le  \tt-3-p$, the intervals $I_j, I_{j+1},\dots ,I_{j+p+3}$ are pairwise disjoint. \label{cl:C:2} 
\item  Moreover, the intervals $I_{0}, I_{1}, \dots, I_{p+1}$  are disjoint from $\{-\pp, -\pp+1\}$.  \label{cl:C:3} 
\end{enumerate}
\end{claim}
\begin{proof}
It is sufficient to prove the statement for $f_\str$; the map $h$ transfers the result to the dynamical plane of $f$.

 All $\Delta(i,f_\str)$ have comparable angles (see Lemma~\ref{lem:ap:CompTiling}): there are $x< y$ independent of $n$ such that the angle of $\Delta(i)$ at $\alpha$ is between $x/\qq$ and $y/\qq$.

Let $\chi$ be the angle of $\Delta(I_{\tt})$ at $\alpha$. The angle $\chi$ is small because $D=D_{\tt}$ is small. By definition of $I_{k}=f^{-1}(I_{k+1})$ (see~\eqref{eq:DefnI_k}) the angle of $\Delta(I_{k+1})$ at $\alpha$ is bounded by the angle of  $\Delta(I_{k})$ at $\alpha$ plus $y/\qq$. Therefore, the angle at $\alpha$ of every $\Delta(I_{k})$ is bounded by $\chi+(2+\tt)y/\qq$, where the number $(2+\tt)y/\qq$ is still small by~\eqref{eq:t over q is small}. We obtain that all $\Delta(I_{k})$ have small  angles.

Since $f_\str\mid \overline Z_\str$ is an irrational rotation and $|I_k|/ \qq$ are small, we see that $I_j, I_{j+1},\dots ,I_{j+p+3}$ are disjoint. Since $I_0$ contains $\{0,1\}$ we see that  $I_{0}, I_{1}, \dots, I_{p+1}$ do not intersect $\{-\pp, -\pp+1\}\subset f^{-1}(I_0)$. 
\end{proof}

Recall from~\S\ref{s:SiegPacm} that the Siegel disk $Z_\str$ of $f_\str$ is foliated by equipotentials parametrized by their heights ranging from $0$ (the height of $\alpha$) to $1$ (the height of $\partial Z_\str$). We denote by $Z_\str^{r}$ the open subdisk of $Z_\str$ bounded by the equipotential at height $r$.

\begin{figure}[t!]
\centering{\begin{tikzpicture}

\node[above] at (0,0) {$D_k$};

\draw[] (0,0) circle (0.5cm);
\draw[fill, opacity=0.1] (0,0) circle (0.5cm);

\draw[blue] (0,-3)--(1.2,0)--(-1.2,0)--(0,-3)
(-0.8, -1) --(0.8, -1);

\draw[blue, fill=blue, opacity=0.1] (1.2,0)--(-1.2,0) -- (-0.8, -1) --(0.8, -1) -- (1.2,0);
\node[above, blue] at (0,-1) {$S_k$}; 

\node[below, blue] at (0,-1.3) {$\Lambda(I_k)$}; 

\end{tikzpicture}}
\caption{The rectangle $S_k$ is an appropriate truncation of $\Lambda(I_k)$ such that $S_k\supset D_k \cap \bLambda$ and $S_k\supset (f\mid \bLambda)^{-1}(S_{k+1})$.} \label{Fig:DfnS_k}
\end{figure}

Next we will construct a rectangle $S_k$ by truncating $\Lambda(I_k)$ by a curve in ${h^{-1}(Z_\str^{r}\setminus Z_\str^{r-\varepsilon})}$ such that the family $S_k$ is backward invariant in the following sense: ${S_k\supset (f\mid \bLambda)^{-1}(S_{k+1})}$ for all $k\in \{0,1,\dots, \tt-1\}$, see Figure~\ref{Fig:DfnS_k}.  Assume that $r <1$ is close to $1$ and choose $\varepsilon>0$ such that $1-r $ is much bigger than $\varepsilon$. Consider an interval $I_k$ for $k\le \tt$ and consider $i\in I_k$. 
\begin{itemize}
\item  If for all $\ell\in \{ k,k+1,\dots, \tt\}$ we have $i+\pp (\ell-k) \not \in \{-\pp,-\pp+1\}$, then define \[S_k(i)\coloneqq \Lambda(i)\setminus h^{-1}(Z_\str^{r  });\]
\item  otherwise define \[S_k(i) \coloneqq \Lambda(i)\setminus h^{-1}(Z_\str^{r  -\varepsilon}).\] 
\end{itemize}

Set $\displaystyle S_k\coloneqq \bigcup _{i\in I_k} S_k(i)$. Since $\Lambda(I_k, f_\str)$ has a small angle at $\alpha$ (see Claim~\ref{cl:C}) and the truncation level $r$ is close to $1$, we have

\begin{claim}
\label{cl:S_k is small} All $S_k$ are small.\qed
\end{claim}

\begin{claim}
\label{cl:D1} For every $k\le \tt$, the preimage of $S_{k+1}$ under $f\mid \bLambda$ is within $S_{k}$.
\end{claim}
\begin{proof}
By Claim~\ref{cl:B2} we only need to check that the truncation is respected by backward dynamics. The proof is based on the fact that points in $S_k$ pass at most once through the critical sector $\bLambda(-\pp)\cup \bLambda(-\pp+1)$ under the iteration of $f, f^2,\dots, f^{\tt-k}$. 

The sector $S_k$ consists of subsectors $S_k(i)$, where $i\in I_k$. If $i\not \in \{-\pp,-\pp+1\}$, then \[f\colon S_k(i)\to S_{k+1}(i+\pp)\] is a homeomorphism. Suppose $i\in  \{-\pp,-\pp+1\}$. Then $S_{k+1}(i+\pp)=\Lambda(i)\setminus h^{-1}(Z_\str^{r  })$ because $i+\pp,\dots, i+\pp (\tt-k)$ are disjoint from $\{-\pp,-\pp+1\}$ by~\eqref{eq:a_n less than a_m}. On the other hand, by definition of $S_k$,
 \[S_{k}\supset  h^{-1}\left((\Lambda(-\pp,f_\str)\cup \Lambda(-\pp+1,f_\str))\setminus Z_\str^{r-\varepsilon}\right).\]
Since $h$ is close to identity (see~\S\ref{sss:KeyLmmConvent}), the preimage of $S_{k+1}(i+\pp)$ under $f\mid \bLambda$ is within $S_{k}(i)\subset S_k$.
\end{proof}

We can assume that $D_{\tt}$ is so small that it does not intersect $h^{-1}(Z_\str^{r})$. Then $D_{\tt}\cap \bLambda\subset S_{\tt}$; using Claims~\ref{cl:B3} and~\ref{cl:D1} we obtain $D_k\cap \bLambda \subset S_k$.

Next let us inductively enlarge $D_k$ as $\mathfrak D_{k}\supset \mathfrak D'_{k}\supset D_k$. Set \[\mathfrak D_{\tt}=\mathfrak D'_{\tt}\coloneqq  D_{\tt}\]  and define $\mathfrak D'_{k}$ to be the connected component of $f^{-1}(\mathfrak D_{k+1})$ containing $D_{k}$. We define $\mathfrak D_{k}$ to be the filled-in $\mathfrak D'_{k}\cup  \intr S_{k}$; i.e.~ $\mathfrak D_{k}$ is $\mathfrak D'_{k}\cup  \intr S_{k} $ plus all of the bounded components of $\C\setminus (\mathfrak D'_{k}\cup    \intr S_{k})$. 

\begin{claim}
\label{cl:D2} For all $k\le \tt$ the intersection $\mathfrak D_k\cap\bLambda$ is connected and we have $ {S_k= \overline{ \mathfrak D_k}\cap\bLambda}$.
\end{claim}
\begin{proof}
Follows from $D_k\cap \bLambda \subset S_k$, Claim~\ref{cl:D1}, and the definition of $\mathfrak D_k$.
\end{proof}

\subsubsection{Bubble chains}
\label{sss:BubbleChains} 
Below we will separate the forbidden part of the boundary $\partial ^\frb U_f$ from all $\mathfrak D_j$ by external rays and bubble chains (see Figure~\ref{Fig:SepForbSector}). Recall from~\S\ref{ss:LocConnK_p} that for $f_\str$ a bubble chain is a sequence of iterated lifts of $\overline Z_\str$; for $f$ the role of $\overline Z_\str$ will be played by $\bLambda$.

\begin{figure}[t!]
{\usetikzlibrary{arrows}
\pagestyle{empty}
\begin{tikzpicture}
\begin{scope}[scale =2]

\begin{scope}[shift={(0.66,0.088)},scale =0.2]
\draw (-2.,0.)-- (0.,2.);
\draw (0.,2.)-- (2.08,0.);
\draw (2.08,0.0)-- (2.08,-0.3);
\draw (2.08,-0.3)-- (0.,-2.);
\draw (0.,-2.)-- (-2.,0.);
\draw (2.08,0)-- (3.72,2.18);
\draw (3.72,2.18)-- (5.,1.16);
\draw (5.,1.16)-- (4.,0.);
\draw (4.,0.)-- (5.02,-1.16);
\draw (5.02,-1.16)-- (4.,-2.);
\draw (4.,-2.)-- (2.08,-0.3);
\draw (3.72,2.18)-- (3.24,2.84);
\draw (3.24,2.84)-- (3.82,3.4);
\draw (3.82,3.4)-- (4.26,2.98);
\draw (4.26,2.98)-- (3.9,2.82);
\draw (3.9,2.82)-- (4.28,2.68);
\draw (4.28,2.68)-- (3.72,2.18);
\draw (4.,-2.)-- (3.42,-2.6);
\draw (3.42,-2.6)-- (3.82,-3.02);
\draw (4.34,-2.38)-- (4.,-2.);
\draw (3.82,-3.02)-- (4.32,-2.75);
\draw (4.32,-2.75)-- (4.09,-2.55);
\draw (4.09,-2.55)-- (4.34,-2.38);
\draw (3.82,-3.02)-- (3.67,-3.22);
\draw (3.67,-3.22)-- (3.84,-3.39);
\draw (3.84,-3.39)-- (4.,-3.2);
\draw (4.,-3.2)-- (3.82,-3.02);
\draw (3.82,3.4)-- (3.64,3.66);
\draw (3.64,3.66)-- (3.86,3.93);
\draw (3.86,3.93)-- (4.1,3.68);
\draw (4.1,3.68)-- (3.81064,3.41352);
\draw (3.86,3.93)-- (3.85,4.19);

\draw[blue]  (4.58,6.728)-- (3.85,4.19);

\draw[blue] (4.25,-6.682)-- (3.8319859383222643,-3.6859085160679514);
\draw (3.8319859383222643,-3.6859085160679514)-- (3.84,-3.39);
\draw [red](-1.8, 1.8) node{$D_0$};
\draw[red,fill=red, fill opacity=0.4]  (-1,1) circle (0.4cm);

\end{scope}

\draw [shift={(0.8200819142056478,0.11146554573543112)}] plot[domain=1.0540413778662818:5.179936858384536,variable=\t]({1.*7.610224188310985*cos(\t r)/5+0.*7.610224188310985*sin(\t r)/5},{0.*7.610224188310985*cos(\t r)/5+1.*7.610224188310985*sin(\t r)/5});

\draw[blue] (0.67,0) node{$\bLambda$};
\draw[blue] (1.34,0.1) node{$\bLambda'$};
\draw[blue] (2.34,0.1) node{$\partial ^\frb U_\str$};
\draw[blue] (1.7, 1) node {$R_x$};
\draw[blue] (1.7,-1) node {$R_y$};
\end{scope}
\end{tikzpicture}
\caption{Separation of $\partial ^\frb U_f$ from $\alpha$. Disks $\bLambda$ and $\bLambda'$ approximate $\overline Z_\str$ and $\overline Z'_*$. Iterated lifts of $\bLambda'$ form periodic bubble chains $B_x$ and $B_y$ landing at periodic points $x$ and $y$. Together with external rays $R_x,R_y$ the bubble chains $B_x,B_y$ separate $\partial ^\frb U_f$ from the critical value. The configuration is stable because of the stability of local dynamics at $x$ and $y$. Disks $D_k$ may intersect $\bLambda'$ but, by Claim~\ref{cl:ContrD_k}, they do not intersect $B_x\cup B_y\setminus \bLambda'$.} \label{Fig:SepForbSector}}
\end{figure}
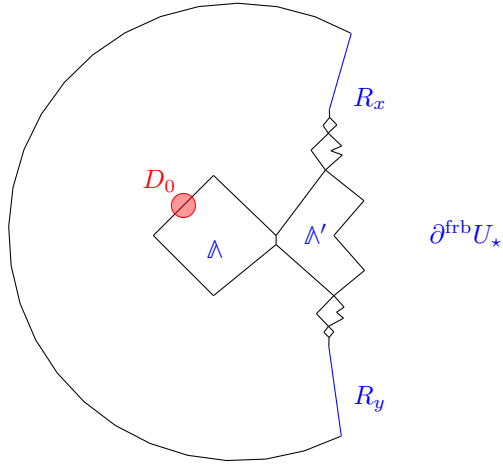
 
Consider first the dynamical plane of $f_\str$. Recall from Theorem~\ref{thm:LocConnOfJul for StandPacm} that the non-escaping set $\Kfilled_\str$ of $f_\str$ is locally connected and that $\Kfilled_\str$ has at most finitely many limbs (and bubbles) with diameter at least $\varepsilon$ for every $\varepsilon>0$. Let $Z_a$ and $Z_b$ be two bubbles attached to $\overline Z'_\str$ such that $Z_a$ and $Z_b$ are close to $\gamma_-$ and $\gamma_+$ respectively.  Let $f_\str^{p_a}$ be the iterate of $f_\str$ such that $f_\str^{p_a}(Z_a)=\overline Z'_\str$. Recall from Definition~\ref{dfn:SiegPcm} that $\gamma_1$, as well as  $\gamma_0$, is a concatenation of an external ray and an internal ray. We also recall from~\S\ref{ss:StandardPacmen} that we normalized $\gamma_0$ to pass through the critical value. Therefore, the critical point $c_0^\str$ is the landing point of two external rays $R_{-}, R_+$; we denote by $W'$ the open wake of $\overline Z'_\str$: the connected component of $V\setminus (R_-\cup R_+)$ containing $Z'_\str$. Let $W_a$ be the univalent pullback of $W'$ under ${f_\str^{p_a}\colon \intr(Z_a)\to Z'_\str}$. Then $W_a$ is the wake of $Z_a$ and we have $W_a\Subset W'$. By the Schwarz lemma, the map $f^{p_a}\colon W_a\to W'$ has a unique fixed point; we call it $x$.

Set $Z_2\coloneqq Z_a$ and for $i\ge 2$ define $Z_{i+1}$ to be the unique preimage of $Z_i$ under $f^{p_a}\colon W_a\to W'$. By the expansion of $f_\str^{p_a}\colon W_a\to W'$, the $Z_i$ shrink to $x$. We have constructed the bubble chain $B_x$
\begin{equation}
\label{eq:B_x}
 Z_{1}=  \overline Z'_{*},  Z_{2},  Z_{3}, \dots \end{equation} 
 landing at $x$. Write $p_x\coloneqq p_a$ and denote by $B'_x\subset B_x$ the subchain obtained by removing $\overline Z'$ in $B_x$. Then $p_x$ is the minimal period of $B_x$ because ${f^{p_x}(B'_x)=B_x}$.

 Similarly, we define $y$ to be the unique fixed point of $f_\str^{p_b}\colon W_b\to W'\Supset W_b$ and $B_y$ to be the bubble chain landing at $y$. The minimal period of $B_y$ is $p_y$. We have 
 \begin{equation}
\label{eq:B'_x}
B_x=f^{p_x}_* (B'_x)\;  \;\text{ and }\; B_y=f^{p_y}_*(B'_y),\sp\sp p_x,\ p_y\ge 2.
\end{equation}
We denote by $p$ the least common period of $B_x$ and $B_y$.

Since $f$ is close to $f_\str$, by Lemma~\ref{lem:PerRaysStable}, periodic rays $ R_x,R_y$ exist in the dynamical plane of $f$ and are close to the corresponding rays in the dynamical plane of $f_\str$.

Set $\bLambda'$ to be the closure of the connected component of $f^{-1}(\bLambda)\setminus \bLambda$ that has a non-empty intersection with $\bLambda$. Then $\bLambda'$ is connected and 
\begin{equation}
\label{eq:f of Lambda'}\bLambda\cap \bLambda'\subset \Lambda(-\pp)\cup  \Lambda(-\pp+1)\sp \text{ and }\sp f(\bLambda')\subset \bLambda.
\end{equation}
We say that $\bLambda'$ is \emph{attached} to $\bLambda$, or more specifically that $\bLambda'$ is \emph{attached} to $\Lambda(-\pp)\cup \Lambda(-\pp+1)$.

 Observe also that $\bLambda'$ approximates $\overline Z'_*$ because $\bLambda$ is close to $\overline Z_*$ and $f$ is close to $f_\str$. 

A \emph{bubble of generation $e+1\ge 1$} for $f$ is an $f^e$-lift of $\bLambda'$. Fix a big $M\gg 1$. We assume that the neighborhood $\UU\subset \WW^u$ in the statement of Key Lemma~\ref{keylem:ContrOfPullbacks} is selected sufficiently small, depending in particular on $M$. Since $\bLambda'$ is close to $\overline Z'_*$, the map $f$ is close to $f_\str$, and $\partial\bLambda \cap (\Lambda(0)\cup \Lambda(1))$ is small, we obtain:
\begin{claim}
\label{cl:ContOfBubbles}
Every bubble $Z_\delta$ of $f_\str$ of generation up to $M$ is approximated by a bubble $\bLambda_\delta$ of $f$ such that 
\begin{enumerate}
\item $\bLambda_\delta$ is close to $Z_\delta$ and $f\mid \bLambda_\delta$ is close to $f_\str \mid Z_\delta$;
\item if $Z_\delta$ is attached to $Z_{\gamma}$, then $\bLambda_\delta$ is attached to $\bLambda_\gamma$; and 
\item if $Z_\delta$ is attached to $Z_\str$, then  $\bLambda_\delta$ is attached to $\bLambda\setminus (\Lambda(0)\cup \Lambda(1))$. \qed.\label{cl:ContOfBubbles:part:3}
\end{enumerate}
\end{claim}

Using Claim~\ref{cl:ContOfBubbles}, we approximate the bubbles $Z_k$ in $B_x$  (see~\eqref{eq:B_x}) with $k\le M$ by the corresponding bubbles $\bLambda_k$. We can assume that the remaining $ Z_{M+j}$ are within the linearization domain of $x$. Taking pullbacks within the linearization domain of $x$, we construct the \emph{bubble chain $B_x(f)$} landing at $x$ as a sequence $\bLambda'=\bLambda_1,\bLambda_2,\dots$  Similarly, $B_y(f)$ is constructed. The chains $B_x\cup B_y(f)$ are close to $B_x\cup B_y(f_\str)$ in the following sense: there are continuous maps \[g_1\colon B_x\cup B_y(f) \to B_x\cup B_y(f_\str) \sp \text{ and }\sp g_2\colon B_x\cup B_y(f_\str) \to B_x\cup B_y(f) \] close to the identities. Equation~\eqref{eq:B'_x} holds in the dynamical planes of $f$. Thus we have constructed $\left(R_x\cup B_x\cup B_y \cup R_y\right)(f)$ that is close to  $\left(R_x\cup B_x\cup B_y \cup R_y\right)(f_\str)$.

Assume that $D$ is so small that it is disjoint from the forward orbit of $R_x\cup R_y$. As a consequence, we obtain: 
\begin{claim}
\label{cl:Separ} All $\mathfrak D_k$ are disjoint from $R_x\cup R_y$.
\end{claim}
\begin{proof}
We proceed by induction on $k\in \{\tt, \tt-1,\dots, 0\}$. Since $\mathfrak D_{k+1}$ is disjoint from the forward orbit of $R_x\cup R_y$, so is $\mathfrak D'_{k}\cup \intr(S_k)$; the latter surrounds $\mathfrak D_k$.
\end{proof}

\subsubsection{Control of $\mathfrak D_k$}
\begin{claim}
\label{cl:ContrD_k}
 For all $k\in  \{0,1,\dots , \tt \}$ the following holds
\begin{enumerate}
\item $\mathfrak D_{k}$ intersects $\bLambda'$ if and only if $I_k\supset \{-\pp,-\pp+1\}$;
\label{cl:ContrD_k:1}
\item If $\mathfrak D_{k}$ intersects $\bLambda'$, then
\[\mathfrak D_{k}\cap \bLambda' \subset f^{-1}(S_{k+1})\] is contained in a small neighborhood of $c_0$; 
\label{cl:ContrD_k:2}
\item If $\mathfrak D_{k}$ intersects $\bLambda'$ for $k<\tt-1$, then $k<\tt-1-p$ and $\mathfrak D_{k+1},$ $\mathfrak D_{k+2},\dots , \mathfrak D_{k+p+1}$ are disjoint from $\bLambda'$;
\label{cl:ContrD_k:3}
\item  If $\mathfrak D_{k}$ intersects $B_x\cup B_y$, then the intersection is within $\bLambda'$ and, in particular, $I_k\supset \{-\pp,-\pp+1\}$;
\label{cl:ContrD_k:4}
\item $\mathfrak D_{k}$ is an open disk disjoint from $\partial^\frb U_f$; in particular, $f\colon \mathfrak D'_{k}\to \mathfrak D_{k+1}$ is a branched covering (of degree one or two) for $k<\tt$.
\label{cl:ContrD_k:5}
\end{enumerate}
\end{claim}

\begin{figure}[t!]
{\usetikzlibrary{arrows}
\pagestyle{empty}
\begin{tikzpicture}
\begin{scope}[scale =2]

\begin{scope}[shift={(0.66,0.088)},scale =0.5]

\filldraw[white, fill=orange, fill opacity=1]  (3.73,2.175) -- (3.98,2.41)  arc (45:132.4:0.31cm);


 \draw[orange, line width =0.15cm ] 
 (-1.3,-1)
 .. controls (-1.6,-1) and (-3,-0.3)..
  (-3,0)
  .. controls  (-3,0.5) and (-2.6,1.9)..
  (-2.5,2)
    .. controls   (-2,2.7) and (-1.2,2.9)..
  (-1,3)
      .. controls   (-0.5,3.2) and (0.5,3.2)..
  (1,3)
        .. controls   (1.3,2.75)  and  (2.0,2.85)..
 (2.3,2.8)
     .. controls    (2.7,2.7)  and  (3.3,2.55)..
(3.64,2.34);

\draw (-2.,0.)-- (0.,2.);
\draw (0.,2.)-- (2.08,0.);
\draw (2.08,0.0)-- (2.08,-0.15);
\draw (2.08,-0.15)-- (0.,-2.);
\draw (0.,-2.)-- (-2.,0.);
\draw (2.08,0)-- (3.72,2.18);
\draw (3.72,2.18)-- (5.,1.16);
\draw (5.,1.16)-- (4.,0.);
\draw (4.,0.)-- (5.02,-1.16);
\draw (5.02,-1.16)-- (4.,-2.);
\draw (4.,-2.)-- (2.08,-0.15);
\draw (3.72,2.18)-- (3.24,2.84);
\draw (3.24,2.84)-- (3.82,3.4);
\draw (3.82,3.4)-- (4.26,2.98);
\draw (4.26,2.98)-- (3.9,2.82);
\draw (3.9,2.82)-- (4.28,2.68);
\draw (4.28,2.68)-- (3.72,2.18);
\draw (4.,-2.)-- (3.42,-2.6);
\draw (3.42,-2.6)-- (3.82,-3.02);
\draw (4.34,-2.38)-- (4.,-2.);
\draw (3.82,-3.02)-- (4.32,-2.75);
\draw (4.32,-2.75)-- (4.09,-2.55);
\draw (4.09,-2.55)-- (4.34,-2.38);
\draw (3.82,-3.02)-- (3.67,-3.22);
\draw (3.67,-3.22)-- (3.84,-3.39);
\draw (3.84,-3.39)-- (4.,-3.2);
\draw (4.,-3.2)-- (3.82,-3.02);
\draw (3.82,3.4)-- (3.64,3.66);
\draw (3.64,3.66)-- (3.86,3.93);
\draw (3.86,3.93)-- (4.1,3.68);
\draw (4.1,3.68)-- (3.81064,3.41352);
\draw (3.86,3.93)-- (3.85,4.19);

\draw (3.8319859383222643,-3.6859085160679514)-- (3.84,-3.39);

\draw[red,fill=red, fill opacity=0.5]  (2.08,-0.075) circle (0.3cm);

\draw[red,fill=red, fill opacity=0.5]  (-1,-1) circle (0.3cm)
 (1.02, 1.02) circle (0.3cm)
 (-1.43, 0.58)circle (0.3cm);


\draw[red] (2.08,-0.65) node {$\mathfrak D_{k+3}$};
\draw[red] (-2.,-1.1) node {$\mathfrak D_{k}$};
\draw[orange]  (4.4,2.1 )node {$\mathfrak D_k {\color{black}  \cap} {\color{blue} \bLambda_2}$};

\draw[blue]  (3.799,3.1) node {$\bLambda_2$}
(-1.87, 1.7) node {$f^2(\bLambda_2)$}
(1, 2) node {$f(\bLambda_2)$};

\draw (-0.6,-1) edge[->,bend right] (0.84, 0.63)
(0.48, 1.15)  edge[->,bend right]  (-0.97, 0.68)
(-1.17, 0.18) edge[->,bend right] (1.65, -0.16);

\draw (3.35,2.2) edge[->] node[above] {$f$}(1.25, 1.3);

\draw[blue] (-1.23, -1.94) node {$\bLambda'_2$};

\draw (-1.425,0.575)-- (-1.94,0.68);
\draw (-1.94,0.68)-- (-2.04,1.26);
\draw (-2.04,1.26)-- (-1.5,1.34);
\draw (-1.5,1.34)-- (-1.55,1.08);
\draw (-1.55,1.08)-- (-1.38,1.09);
\draw (-1.38,1.09)-- (-1.425,0.575);
\draw (1.019215987701768,1.0199846272098385)-- (1.06,1.61);
\draw (1.06,1.61)-- (1.61,1.73);
\draw (1.61,1.73)-- (1.66,1.33);
\draw (1.66,1.33)-- (1.45,1.3);
\draw (1.45,1.3)-- (1.67,1.07);
\draw (1.67,1.07)-- (1.019215987701768,1.0199846272098385);
\draw (-1.54,-1.14)-- (-1.,-1.);
\draw (-1.,-1.)-- (-0.99,-1.36);
\draw (-0.99,-1.36)-- (-1.28,-1.41);
\draw (-1.17,-1.65)-- (-1.28,-1.41);
\draw (-1.17,-1.65)-- (-1.55,-1.71);
\draw (-1.55,-1.71)-- (-1.54,-1.14);
\end{scope}

\draw[blue] (0.67,0.25) node{$\bLambda$};
\draw[blue] (2.34,0.05) node{$\bLambda'$}
(2.34,-0.15) node{$=f^3(\bLambda_2)$}
(2.59,-0.38) node{$=f^3(\bLambda'_2)$};

\end{scope}
\end{tikzpicture}
\caption{Illustration to the proof of Claim~\ref{cl:ContrD_k}, Part~\eqref{cl:ContrD_k:4} in case $p_x=3$. Suppose that $f^3$ maps the bubble $\bLambda_2$ (in $B_x$) to $\bLambda'$ and suppose that $ \mathfrak D_{k}\cap \bLambda_2\not=\emptyset$. Let $\bLambda'_2$ be the lift of $f(\bLambda_2)$ attached to $\bLambda$. Since $f(\bLambda_2)$ is attached to $S_{k+1}=\overline {\mathfrak D}_{k+1}\cap \bLambda$ and since ${\mathfrak D}_{k+1}\cup S_{k+1}$ does not surround the critical value, we obtain that the pullback of $f(\bLambda_2)$ along $f\colon \mathfrak D'_{k}\to {\mathfrak D}_{k+1}$ is attached to $S_{k}$ contradicting $ {\mathfrak D}_{k}\cap \bLambda_2\not=\emptyset$.} \label{Fig:KeyLemma:MainArgum}}
\end{figure}

\begin{proof}
We proceed by induction. Suppose that all of the statements are proven for moments $\{\tt, \dots, k+2,k+1\}$; let us prove them for $k$.

If $I_k\supset \{-\pp,-\pp+1\}$, then $\mathfrak \mathfrak D_{k+1}\supset S_{k+1}(0)\cup S_{k+1}(1)\ni c_1$ (see~\S\ref{sss:wideD_j}). Since  $D_{k+1}$ contains either $S_{k+1}(-1)$ or $ S_{k+1}(2)$, we see that $\mathfrak D'_{k}=f^{-1}(\mathfrak D_{k+1})$ intersects $\bLambda'$. 

Suppose $I_k\cap  \{-\pp,-\pp+1\}= \emptyset$. Then $ \mathfrak D_{k+1}$ does not contain $c_1$. Hence every point in $ \mathfrak D_{k+1}$ has at most one preimage under $f\mid \mathfrak D'_k$. Since $\mathfrak D_{k+1}\cap \bLambda$ is connected, every preimage of $\mathfrak D_{k+1}\cap \bLambda$ under $f\mid \mathfrak D'_k$ is in $\bLambda$. By~\eqref{eq:f of Lambda'}, $\mathfrak D'_k$ has empty intersection with $\bLambda'$. Since $\mathfrak D'_k\cup S_k$ does not surround $\bLambda'$, we also obtain $\mathfrak D_k\cap \bLambda'=\emptyset$. This proves Part~\eqref{cl:ContrD_k:1}.

 Part~\eqref{cl:ContrD_k:2} follows from $\mathfrak D_{k+1}\cap \bLambda\subset S_{k+1}$ (see Claim~\ref{cl:D2}), the definition of $\mathfrak D_{k}$, and the fact that $S_{k+1}$ is a small neighborhood of $c_1$, see Claim~\ref{cl:S_k is small}. 
 
 Part~\eqref{cl:ContrD_k:3} follows from Part~\eqref{cl:ContrD_k:1} combined with Claim~\ref{cl:C} (Part \eqref{cl:C:2}).

Let us now prove Part~\eqref{cl:ContrD_k:4}, see Figure~\ref{Fig:KeyLemma:MainArgum} for illustration. By continuity, Part~\eqref{cl:ContrD_k:4} holds for $k\in \{\tt,\dots, \tt-p\}$. Below we assume that $k<\tt-p$. Assume that Part~\eqref{cl:ContrD_k:4} does not hold; let $\mathfrak D_k\cap (B_x\setminus \Lambda')\not=\emptyset$; the case $\mathfrak D_k\cap (B_y\setminus \Lambda')\not=\emptyset$ is similar. Write \[B_x=(\bLambda'=\bLambda_1, \bLambda_2,\bLambda_3,\dots),\]
where $\bLambda_i$ is a bubble attached to $\bLambda_{i-1}$. Then there is a $\bLambda_i$  with $i\ge 2$ such that 
\[ \mathfrak D_k \cap \bLambda_i\not=\emptyset.\]
We assume that $i\ge 2$ is minimal and we claim that $i=2$. Recall from \S\ref{sss:BubbleChains} that $f^{p_x}$ maps $\Lambda_{k+1}$ to $\Lambda_k$, where $p_x\le p$ is the minimal period of $x$. Suppose $i>2$ and consider $B_x^{(i)}=\bigcup_{j\ge i} \bLambda_j$. Since $\mathfrak D_k\cap B_x^{(i)}\not=\emptyset$ and $(\mathfrak D'_k\cup S_k)\cap R_x=\emptyset$ (the latter follows from Claim~\ref{cl:Separ}), we obtain that $\mathfrak D'_k\cap B_x^{(i)}\not=\emptyset$; hence $\mathfrak D_{k+1}\cap f(B_x^{(i)})\not=\emptyset$. Applying induction we obtain $\mathfrak D_{k+p_x}\cap B_x^{(i-1)}=\mathfrak D_{k+p_x}\cap f^{p_x}(B_x^{(i)})\not=\emptyset$ contradicting the induction assumption that Part~\eqref{cl:ContrD_k:4} holds for $k+p_x$.

Consider the bubbles \[f(\bLambda_2),f^2(\bLambda_2), \dots , f^{p_x}(\bLambda_2)=\bLambda'.\]
By Claim~\ref{cl:ContOfBubbles} Part~\eqref{cl:ContOfBubbles:part:3} they are attached to $\bLambda\setminus (\Lambda(0)\cup \Lambda(1))$. Observe that $\mathfrak D_{k+p_x}$ intersects $\bLambda'$. Indeed, since $\mathfrak D'_k\cup S_k$ is disjoint from $R_x$, the disk $\mathfrak D'_k$ intersects $B'_x=\bigcup_{j\ge 2} \bLambda_j$; hence $\mathfrak D_{k+1}\cap f(B'_x)\not=\emptyset$. Applying induction we obtain $\mathfrak D_{k+p_x}\cap B_x={\mathfrak D_{k+p_x}\cap f^{p_x}(B'_x)\not=\emptyset}$. Therefore, $\mathfrak D_{k+p_x}\cap \bLambda'\not=\emptyset$ because $\mathfrak D_{k+p_x}$ is disjoint from $B'_x$ by the induction assumption that Part~\eqref{cl:ContrD_k:4} holds for $k+p_x$.

Since $\mathfrak D_{k+p_x}$ intersects $\bLambda'$ we have $I_{k+p_x}\supset \{-\pp, -\pp+1\}$. Therefore, each $f^j(\bLambda_2)$ with $j\in \{1,\dots, p_x\}$ is attached to $S_{k+j}\subset \overline {\mathfrak D_{k+j}}$. Moreover, every point in $\mathfrak D_{k+j}$ has at most one preimage under $f\mid \mathfrak D'_{k+j-1}$ for $j\in \{1,\dots, p_x\}$ because ${\mathfrak D_{k+j}\cap \bLambda  \subset S_{k+j}}$ does not contain $c_1$.

Let $\bLambda'_2$ be the lift of $f(\bLambda_2)$ attached to $S_{k}$. We note that $\bLambda'_2\not=\bLambda_2$ and $f(\bLambda_2)\not=\bLambda'$ (by~\eqref{eq:B'_x}). Recall that every point in $\mathfrak D_{k+1}$ has at most one preimage under $f\mid \mathfrak D'_k$. We claim that the lift of $f(\mathfrak D_k \cap \bLambda_2)$ under $f\mid \mathfrak D'_k$ is in $\bLambda'_2$ and not in $\bLambda_2$. Indeed, since $\mathfrak D_{k+1} \cup f(\bLambda_2)$ neither contains nor surrounds the critical value, the lift of $f(\mathfrak D_k \cap \bLambda_2)$ under $f\mid \mathfrak D'_k$ agrees with the lift of $f(\mathfrak D_k \cap \bLambda_2)$ under $f\mid \bLambda'_2$. This proves Part~\eqref{cl:ContrD_k:4}.

By Part~\eqref{cl:ContrD_k:2} $\mathfrak D_{k+1}$ is disjoint from $\partial^\frb U_f$ because $\mathfrak D_{k+1}$ can intersect $B_x\cup B_y$ only in a small neighborhood of $c_0$. Therefore, $f\colon \mathfrak D'_{k}\to \mathfrak D_{k+1}$ is a branched covering.
\end{proof}

This shows  $f^{\tt}\colon D_{0} \to D_{\tt}$ is a branched covering. Observe next that $D_0\cap \bLambda\subset S_0$ is a small neighborhood of $c_1$ that is disjoint from $\gamma_1$. We can easily separate $D_0\setminus \bLambda$ from $\gamma_1\setminus \bLambda$ using $\bLambda$ and finitely many backward iterated lifts of $B_x\cup B_y\cup R_x\cup R_y$. This finishes the proof of the Key Lemma.
\end{proof}

\section{Maximal prepacmen}
\label{s:MaxComPair}

Let $g : X \to Y$ be a holomorphic map between Riemann surfaces.  
Recall that $g$ is:
\begin{itemize}
\item proper, if $g^{-1}(K)$ is compact for each compact $K \subset Y$;
\item $\sigma$-proper (see \cite[\S 8]{McM3}) if each component of $g^{-1}(K)$ is compact for each compact $K \subset Y$; or equivalently if $X$ and $Y$ can be expressed as increasing unions of subsurfaces $X_i$, $Y_i$ such that $g : X_i \to Y_i$ is proper.
\end{itemize}
\noindent A proper map is clearly $\sigma$-proper.

A prepacman $\bF=(\bbf_-~,\bbf_+)$ of a pacman $f$ is called \emph{maximal} if both $\bbf_-$ and $\bbf_+$ extend to $\sigma$-proper maps $\bbf_-\colon \bX_-\to \C$ and $\bbf_+\colon \bX_+\to \C$. We will usually normalize $\bF$ such that $0=\psi^{-1}_\bF(\text{critical value})$, where $\psi_\bF$ is a quotient map from $\bF$ to $\bbf$, see~\S\ref{sec:PacmenRenorm}. Under this assumption $\bF$ is defined uniquely up to rescaling.

\begin{thm}[Existence of  maximal prepacmen]
\label{thm:MaxCommPair:short}
Every $f\in \WW^u$ sufficiently close to $f_\str$ has a maximal prepacman $\bF$ that depends analytically on $f$. 
 \end{thm}
 \noindent A refined statement will be proven as Theorem~\ref{thm:MaxCommPair:long}. The analytic dependence means that the restriction of a map to a disk compactly contained in the domain depends analytically in the associated Banach space. In the proof, we will show that $\bF$ is obtained from $f$ by an analytic change of variables. Note that analytic dependence is sufficient to check for one-parameter families.

\subsection{Linearization of $\psi$-coordinates}
Consider again $[f_0\colon U_0\to V]\in \WW^u$ close to $f_\str$. By definition of $\WW^u$, the map $f_0$ can be anti-renormalized infinitely many times. We define the \emph{tower of anti-renormalizations} as 
\[\TT(f_0) =  \left(F_k\right)_{k\le 0}.\]
Each $f_k$ embeds to the dynamical pane of $f_{k-1}$ as a prepacman  $F_{k}^{(k-1)}$ such that $f_{k,\pm}^{(k-1)}$ are iterates of $f_{k-1}$.

Let us specify the following translation \[ T_k\colon z\to z-c_1(f_k).\]
Let us now translate each $f_k$ so that $c_1(f_k)=0$. \emph{We mark the translated objects with ``$\bullet$.''} For $k\le0$, set \[\phi^{\bullet}_{k}(z):= T_{k-1} \circ \phi_{k}\circ T_k^{-1} \]
so that $\phi^{\bullet}_{k}(0)=0$. Similarly, define $U^{\bullet}_{k}\coloneqq T_k (U_k)$ and $V^{\bullet}_{k}\coloneqq T_k(V)$; and conjugate all $f_k$ and all $F_{k}=F^{(k)}_{k}$ by $T_k$; the resulting maps are denoted by $f_k^\bullet \colon U_k^\bullet \to V^\bullet_k$ and by
\[F^{\bullet }_{k}= (f^{\bullet }_{k,\pm }\colon U^{\bullet }_{k,\pm}\to S_{k}^{\bullet }).\]
We also write $\gamma^{\bullet}_1(f_k)\coloneqq T_k(\gamma_1)$. The tower  $\left(F^\bullet_k\right)_{k\le 0}$ is illustrated on Figure~\ref{Fig:Sf1dash}.

\begin{figure}[p]

{\begin{tikzpicture}

\begin{scope}[shift={(0,5)},scale =0.8]
\draw[blue, shift ={(0,0.6)},yscale =-1] (-1.1,4.3) edge[<-,bend left=8] node [above right] {$\phi^\bullet_{-2}$} (0.7,6.7);
\end{scope}

\begin{scope}[shift={(0,0)},scale =0.8]
\draw [rotate around={0.:(0.,0.)}] (0.,0.) ellipse (2.246099127742333cm and 1.0222334819623482cm);
\draw (-0.7,-0.04)-- (-1.4589248333679499,-0.7772346338620598);

\draw[red] (-0.7,-0.04)-- (-1.46,0.28);
\draw[red] (-1.46,0.28)-- (-1.06,0.64);
\draw[red] (-1.06,0.64)-- (-0.7,-0.04);
\draw (-1.5,-0.5) node {$\gamma^\bullet_1$};
\draw[red] (-.5,0.4) node {$S^\bullet_{-1}$};
\draw[blue, shift ={(0,0.6)},yscale =-1] (-1.1,0.3) edge[<-,bend left] node [ right] {$\phi^\bullet_{-1}$} (0.7,6.7);
\draw[shift= {(0,-1)}] (0.8,0.5) node{$f^\bullet_{-2}$};

\draw[shift={(1,-3)},blue]  (0,3.2) edge[->,bend left] node [above ] {$h^{\bullet}_{-2}$} (6,3);

{\begin{scope}[shift={(8,-1.)},scale =1.5,red]
\draw (-1.,-1.)-- (-1.,2.)-- (1.,2.)-- (1.,-1.);

\draw[scale=0.4] (-1.,-1.)-- (-1.,2.)-- (1.,2.)-- (1.,-1.);
\draw[scale=0.16] (-1.,-1.)-- (-1.,2.)-- (1.,2.)-- (1.,-1.);

\draw (0,1.4) node{$\bS_{-2}$};

\draw[blue] (-0.6,1.2) edge[->] node[above right] {$z\to \mu_\str z$}(-0.47,2.6);

\end{scope}} 

\end{scope}


\begin{scope}[shift={(0,-5)},scale =0.8]
\draw [rotate around={0.:(0.,0.)}] (0.,0.) ellipse (2.246099127742333cm and 1.0222334819623482cm);
\draw (-0.7,-0.04)-- (-1.4589248333679499,-0.7772346338620598);

\draw[red] (-0.7,-0.04)-- (-1.46,0.28);
\draw[red] (-1.46,0.28)-- (-1.06,0.64);
\draw[red] (-1.06,0.64)-- (-0.7,-0.04);
\draw (-1.5,-0.5) node {$\gamma^\bullet_1$};
\draw[red] (-.5,0.4) node {$S^\bullet_{0}$};
\draw[blue, shift ={(0,0.6)},yscale =-1] (-1.1,0.3) edge[<-,bend left] node [ right] {$\phi^\bullet_{0}$} (0.7,6.7);
\draw[shift= {(0,-1)}] (0.8,0.5) node{$f^\bullet_{-1}$};

\draw[shift={(1,-3)},blue]  (0,3.2) edge[->,bend left] node [above ] {$h^{\bullet}_{-1}$} (6,3);

{\begin{scope}[shift={(8,-1.)},scale =1.5,red]
\draw (-1.,-1.)-- (-1.,2.)-- (1.,2.)-- (1.,-1.);

\draw[scale=0.4] (-1.,-1.)-- (-1.,2.)-- (1.,2.)-- (1.,-1.);

\draw (0,1.4) node{$\bS_{-1}$};
\draw [blue](-0.6,1.2) edge[->] node[right] {$z\to \mu_\str z$}(-0.25,4.6);
\end{scope}}

\end{scope}


\begin{scope}[shift={(0,-10)},scale =0.8]
\draw [rotate around={0.:(0.,0.)}] (0.,0.) ellipse (2.246099127742333cm and 1.0222334819623482cm);
\draw (-0.7,-0.04)-- (-1.4589248333679499,-0.7772346338620598);

\draw (-1.5,-0.5) node {$\gamma^\bullet_1$};
\draw[shift= {(0,-1)}] (0.8,0.5) node{$f^\bullet_{0}$};

\draw[shift={(1,-3)},blue]  (0,3.2) edge[->,bend left] node [above ] {$h^{\bullet }_{0}$} (6,3);

{\begin{scope}[shift={(8,-1.)},scale =1.5,red]
\draw (-1.,-1.)-- (-1.,2.)-- (1.,2.)-- (1.,-1.);


\draw (0,1.4) node{$\bS_{0}$};

\draw [blue](-0.6,1.2) edge[->] node[right] {$z\to \mu_\str z$}(-0.25,4.6);

\end{scope}}

\end{scope}

\end{tikzpicture}}

\caption{Left: each pacman $f^\bullet_i$ embeds as a prepacman to the dynamical plane of $f^\bullet_{i-1}$ via $\phi^\bullet_{i}$. Right: sectors $S^\bullet_i$ after linearization of $\psi$-coordinates. Note that $S^\bullet_i$ can intersect $\gamma^\bullet_1$ in a small neighborhood of $\alpha^\bullet=T_i(\alpha)$.} \label{Fig:Sf1dash}
\vspace{128in}
\end{figure}
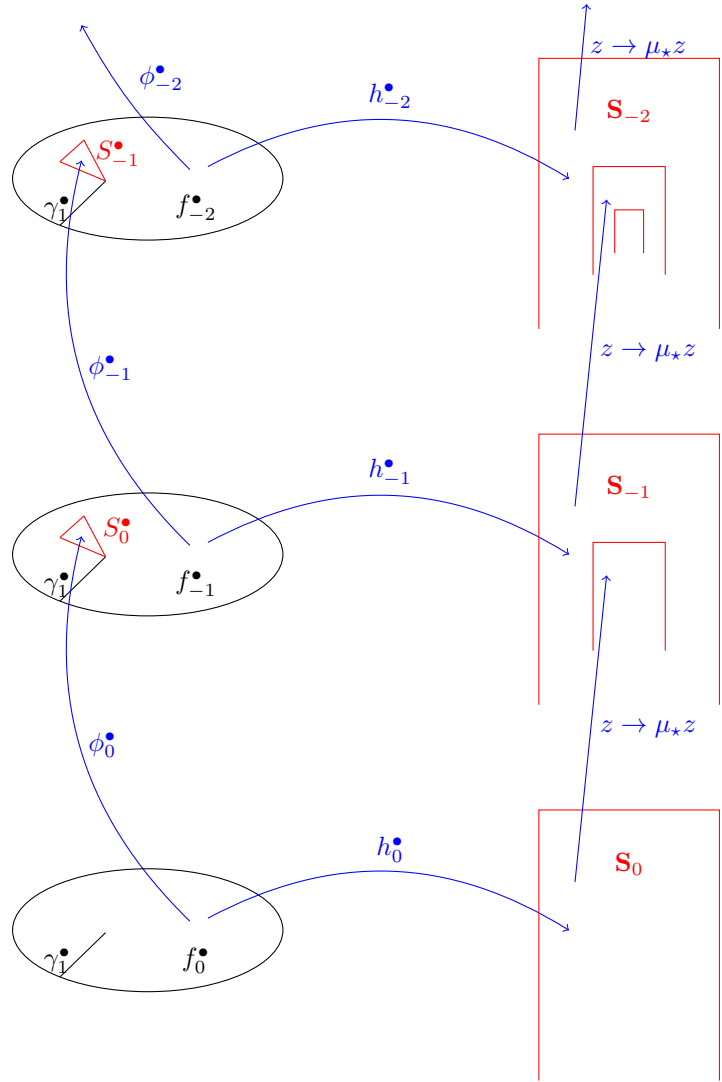

Denote by \[\mu_{*}:= \phi'_{*}(c_1(f_\str))= \left(\phi^{\bullet}_{*}\right)'(0), \sp\sp |\mu_\str |<1\]  the self-similarity coefficient of $\overline Z_\str$.

\begin{lem}[Linearization]
For every $f_0\in \WW^u$ sufficiently close to $f_\str$, the limit 
\begin{equation}
\label{eq:h_f_0}
h^\bullet_{0} (z)=h^\bullet_{f_0}\coloneqq \lim_{i\to - \infty} \frac{\phi^{\bullet}_{i+1}\circ \dots \circ \phi^{\bullet}_{0} (z)}{\mu_*^{-i}}
\end{equation}
is a univalent map on a certain neighborhood of $0$ (independent on $f_0$).
\end{lem}
\noindent We remark that the linearization is normalized in such a way that $(h^\bullet_0)'(0)=1$ if $f_0=f_\str$.

\begin{proof}
Follows from a standard linearization argument. Write $\phi_{i}^{\bullet}(z)=\mu_iz +O(z^2)$; since $\phi^{\bullet}_{i}$ tends exponentially fast to $\phi^{\bullet}_{*}$ we see that $\mu_i$ tends exponentially fast to $\mu_*$ and that the constant in the error term does not depend on $i$. For $z$ in a small neighborhood of $0$, we have
\[\left|\phi^{\bullet}_{i+1}\circ \dots \circ\phi^{\bullet}_{0} (z)\right|\le C (|\mu_*| +\varepsilon )^i|z|\] for some constants $C$ and $\varepsilon>0$ such that $|\mu_*| +2\varepsilon<1$.
 Write 
\[h^{(i)}(z) \coloneqq  \frac{\phi^{\bullet}_{i+1}\circ \dots \circ \phi^{\bullet}_{0} (z)}{\mu_*^{-i}}.\]
Then
\[\frac{h^{(i-1)}(z)}{h^{(i)}(z)}= \frac{\phi^{\bullet}_{i}\left( \phi^{\bullet}_{i+1}\circ \dots \phi^{\bullet}_{0} (z) \right )}{\mu_*  \phi^{\bullet}_{i+1}\circ \dots \phi^{\bullet}_{0} (z)}= \frac{\mu_{-i+1}}{\mu_*}+O\left(\phi^{\bullet}_{i+1}\circ \dots \phi^{\bullet}_{0} (z) \right)\]
tends exponentially fast to $1$ in some neighborhood of $0$. This implies that $h^{(i)}(z)$ converges to a univalent map in some neighborhood of $0$.
\end{proof}

Let us write $h^\bullet_i=h_{f_i}$ and we set $\bS_i\coloneqq h_i^\bullet (S_i)$. \emph{We will usually use bold symbols for object in linear coordinates.} By construction~\eqref{eq:h_f_0}, the maps $h^\bullet_i$ satisfy the linearization equation (see Figure~\ref{Fig:Sf1dash})
\begin{equation}
\label{eq:hf0:LinEq}
h^{\bullet}_{i-1} \circ \phi_{i}^{\bullet}= [z\to \mu_\str  z] \circ h^{\bullet}_{i}. 
\end{equation}

For $i\le 0$, set
\begin{equation}
\label{eq:defn:h harsh}
h_{i}^{\#} (z) \coloneqq   \mu^{-i}_* h^\bullet_{0}(z).
\end{equation}
It follows from \eqref{eq:hf0:LinEq} that 
\begin{equation}
\label{eq:hf0:FuncEq}
 h^\bullet_{0}(z)= h^\#_{-1}\left(\phi^{\bullet}_{0}(z) \right)=\dots=h^\#_{i} \left(   \phi^{\bullet}_{i+1}  \circ\dots \circ \phi^{\bullet}_{0}(z)\right).  \end{equation}
\emph{We will usually use ``$\#$'' to mark linearized objects rescaled by $\mu_\str^{-i}$.}

\begin{lem}[Extension of $h^\bullet_{0}$]
Under the above assumptions $h_0^{\bullet }$ 
extends to a univalent map $h_0^{\bullet }\colon \intr(V^\bullet_0\setminus \gamma^{\bullet}_1)\to \C$.
\end{lem}

\begin{proof}
 By Lemma~\ref{lem:psi is adjusted} the map $\phi^{\bullet}_{i-1}\circ \dots\circ \phi^{\bullet}_{0}$ extends to a conformal map defined on $\intr (V^\bullet_0\setminus \gamma^\bullet_1)$. Since $\phi^{\bullet}_{i-1}\circ \dots\circ  \phi^{\bullet}_{0}$  is contracting, for every $z\in \intr (V^\bullet_0\setminus \gamma^\bullet_1)$ there is an $i< 0$ such that $ \phi^{\bullet}_{i-1}  \circ\dots \circ \phi^{\bullet}_{0}(z)$ is within a neighborhood of $0$ where $h_{i} ^{\bullet }$ is defined (this is easily true if $f^\bullet_0=f^\bullet_\str$; applying Theorem~\ref{thm:ComPsConj} we obtain this property for all $f^\bullet_0$). Therefore,~\eqref{eq:hf0:FuncEq} extends dynamically $h^\bullet_{0}$ to $ \intr(V_0^\bullet \setminus \gamma^\bullet_1)$.
\end{proof}

Let us now conjugate every map $F_k^{\bullet }$ by $h_k^{\#}$; we define $\bF_k^{\# } \coloneqq h^{\# }_k \circ F_k^\bullet \circ \left(h^{\#}_k \right)^{-1}$.  We construct the \emph{tower in linear coordinates}
\begin{equation}
\label{eq:TowerTf0}
\TT^\# (\bF_0)= \left( \bF_{k}^{\#}\right)_{k\le 0} =\left( \bbf^{\# }_{k,\pm }\colon \bU^{\# }_{k,\pm}\to \bS_{k}^{\# }   \right)_{k\le 0},
\end{equation}
where
\begin{equation}
\label{eq:dfn:Ssharp} 
\intr\left(\bS^{\#}_{k}\right)= h^{\#}_{k}\left(V^{\bullet }_k \setminus \gamma^\bullet_1  \right)=h^{\#}_{k}\circ T_k\left(V_k \setminus \gamma_1  \right),
\end{equation}
and similarly other objects marked by ``$\#$'' are defined.

It follows from~\eqref{eq:RernReturns} that 
\begin{lem}
\label{lem:F_0 is iter of F_n}
There are numbers $m_{1,1}, m_{1,2},m_{2,1},m_{2,2}$ such that for $k< 0$ we have
\begin{align*}
\bbf^{\#}_{k+1,-} =  (\bbf^{\#}_{k,-})^{m_{1,1}}\circ  (\bbf^{\#}_{k,+})^{m_{1,2}},\\ 
\bbf^{\#}_{k+1,+} = (\bbf^{\#}_{k,-})^{m_{2,1}}\circ  (\bbf^{\#}_{k,+})^{m_{2,2}}.
\end{align*}
 \qed
\end{lem}

Note also that
\begin{equation}
\label{eq:defn:f hash}
\bbf^{\#}_{k,\pm} = \frac {1}{\mu_*^{k}}\bbf_{k,\pm}\left(\mu_*^{k} z \right).
\end{equation}

\subsection{Global extension of prepacmen in $\WW^u$}
Using Key Lemma~\ref{keylem:ContrOfPullbacks} we deduce
\begin{thm}[Existence of a maximal prepacman]
\label{thm:MaxCommPair:long} 
If $f_0\in \WW^u$  is sufficiently close to $f_\str$, then every pair $\bF_i^{\#}=\left( \bbf^{\#}_{k,\pm}\right)$ in the tower $\TT^{\#}(\bF_0)$ (see~\eqref{eq:TowerTf0}) extends to $\sigma$-proper branched coverings
\[\bbf^{\#}_{k,\pm}\colon \bX^{\#}_{k,\pm}\to \C,\] 
where $\bX^{\#}_{k,\pm}$ are open connected subsets of $\C$.
\end{thm}
\noindent Note that the case $f_0=f_\str$ follows from~\cite[Theorem 8.1]{McM3}.
\begin{proof}Let 
\[\mathfrak F_0 =\left(f_{0, \pm } \colon \mathfrak U_{0,\pm} \to \mathfrak S \coloneqq V\setminus (\gamma_1\cup O)\right)\] 
be a commuting pair obtained from $F_{0} = \left(f_{0, \pm } \colon U_{0,\pm}\to V\setminus \gamma_1\right)$ by removing a small neighborhood $O$ of $\alpha$ from $ V\setminus \gamma_1$ and by removing $f_{0,\pm}^{-1}(O)$ from $  U_{0,\pm}$. By Lemma~\ref{lem:psi is adjusted} the map
$\phi_{k-1}\circ   \dots\circ \phi_{0}$ embeds $\mathfrak F_0$ to the dynamical plane of $f_k$ as commuting pair denoted by
\begin{equation}
\label{eq:CommPaif:n}
\mathfrak F^{(k)}_0 =\left( f_k^{\aa_k},f^{\bb_k}_k \right)\colon \mathfrak U^{(k)}_{0,-} \cup \mathfrak U^{(k)}_{0,+} \to  \mathfrak S_0^{(k)}.
\end{equation}

Since $\phi_{k}$ is contracting at the critical value the diameter of $U^{(k)}_{0,-} \cup U^{(k)}_{0,+}\cup  \mathfrak S^{(k)}_{0}\ni c_1(f_n)$ tends to $0$. By Key Lemma~\ref{keylem:ContrOfPullbacks}, for a sufficiently big $k<0$
there is a small open topological disk $D$ around the critical value of $f_k$ such that  the pair~\eqref{eq:CommPaif:n} extends into a pair of commuting branched coverings
\begin{equation}
\label{eq:CommPaif:n:ext}
F^{(k)}_0 =\left( f_{k}^{\aa_k},f^{\bb_k}_{k} \right)\colon  W_{-}^{(k)} \cup   W^{(k)}_{+} \to D,
\end{equation} with $W_{-}^{(k)} \cup   W^{(k)}_{+} \cup D\subset V\setminus \gamma_1$.

Conjugating~\eqref{eq:CommPaif:n:ext} by $h_k^\#\circ T_k$ we obtain the commuting pair  
\[(\bbf_{0,-},\bbf_{0,+}) \colon  \bW^{(k)}_{-} \cup   \bW^{(k)}_{+} \to \bD^{(k)}.\]
 Since for a sufficiently big $t$ and all $m\le 0$ the modulus of the annulus $\bD^{(tm-t)} \setminus \bD^{(tm)}$ is uniformly bounded from $0$ we obtain $\displaystyle\bigcup_{k\ll 0} \bD^{(k)} =\C$. Setting
\begin{equation}
\label{eq:Xtm}
\bX_{0,-}\coloneqq \bigcup_{k\ll 0}  \bW^{(k)}_{-},  \;\;\; \bX_{0,+}\coloneqq \bigcup_{n\ll 0} \bW^{(k)}_{+}
\end{equation}
we obtain $\sigma$-proper maps $\bbf_{0,\pm}\colon \bX_{0,\pm}\to \C $, where $\bX_{0,\pm}$ are connected. Similarly, $(\bbf_{k,\pm}^{\#})$ extends to a pair of $\sigma$-proper maps.
\end{proof}

\section{Maximal parabolic prepacmen}
\label{s:MaxParPrep}
Since the multiplier of the $\alpha$-fixed point is in the unstable direction of $\RR$ at $f_\str$ (by Lemma~\ref{lem:eq:S3:R_prm}), we can 
 consider a parabolic pacman $f_0\in \WW^u$ close to $f_\str$ such that  Theorem~\ref{thm:MaxCommPair:long} applies for $f_0$. As in~\S\ref{s:MaxComPair}  we denote by $\bF_n=(\bbf_{n,\pm})$ the maximal prepacmen of $f_n=\RR^{n}f_0$ with $n\le 0$ and by $\bF_n^{\#}$ the rescaled version of $\bF_n$ so that $\bF_{0}$ is an iteration of $\bF^{\#}_n$, see Lemma~\ref{lem:F_0 is iter of F_n}.

\begin{figure}

\centering{\begin{tikzpicture}[red]

\begin{scope}[shift={(-1,0)},rotate=0,scale =0.4]
\draw (0,0) .. controls (1,1.2) and (3,1.5).. (5,1.5) .. controls (5.5,1) and (5.5,-1).. (5,-1.5) .. controls (3,-1.5) and (1,-1.2).. (0,0);

\draw[black] (3,0) edge[->,bend right] node [above] {$f_{0}$} (-1,2);
\end{scope}

\begin{scope}[shift={(-1,0)},rotate=120,scale =0.4]
\draw (0,0) .. controls (1,1.2) and (3,1.5).. (5,1.5) .. controls (5.5,1) and (5.5,-1).. (5,-1.5) .. controls (3,-1.5) and (1,-1.2).. (0,0);
\end{scope}

\begin{scope}[shift={(-1,0)},rotate=240,scale =0.4]
\draw (0,0) .. controls (1,1.2) and (3,1.5).. (5,1.5) .. controls (5.5,1) and (5.5,-1).. (5,-1.5) .. controls (3,-1.5) and (1,-1.2).. (0,0);
\end{scope}

\begin{scope}[shift={(5,0)},rotate=0,scale =0.4]
\draw (0,0) .. controls (1,0.4) and (3,0.7).. (5,0.7) .. controls (5.5,0.3) and (5.5,-0.3).. (5,-0.7) .. controls (3,-0.7) and (1,-0.4).. (0,0);

\draw[black] (2.5,0) edge[->,bend right] node [above] {$f_{-1}$} (-1.8,1.8);

\draw[black]  (-2,2) edge[<-,bend left] node [above] {$f_{0,-}$}  (0,3);
\draw[black]  (-4.5,0) edge[->,bend left] node [above left] {$f_{0,+}$}  (0,5);

\end{scope}
\begin{scope}[shift={(5,0)},rotate=45,scale =0.4]
\draw (0,0) .. controls (1,0.4) and (3,0.7).. (5,0.7) .. controls (5.5,0.3) and (5.5,-0.3).. (5,-0.7) .. controls (3,-0.7) and (1,-0.4).. (0,0);
\end{scope}
\begin{scope}[shift={(5,0)},rotate=90,scale =0.4]
\draw (0,0) .. controls (1,0.4) and (3,0.7).. (5,0.7) .. controls (5.5,0.3) and (5.5,-0.3).. (5,-0.7) .. controls (3,-0.7) and (1,-0.4).. (0,0);
\end{scope}
\begin{scope}[shift={(5,0)},rotate=135,scale =0.4]
\draw (0,0) .. controls (1,0.4) and (3,0.7).. (5,0.7) .. controls (5.5,0.3) and (5.5,-0.3).. (5,-0.7) .. controls (3,-0.7) and (1,-0.4).. (0,0);
\end{scope}
\begin{scope}[shift={(5,0)},rotate=180,scale =0.4]
\draw (0,0) .. controls (1,0.4) and (3,0.7).. (5,0.7) .. controls (5.5,0.3) and (5.5,-0.3).. (5,-0.7) .. controls (3,-0.7) and (1,-0.4).. (0,0);
\end{scope}
\begin{scope}[shift={(5,0)},rotate=225,scale =0.4]
\draw (0,0) .. controls (1,0.4) and (3,0.7).. (5,0.7) .. controls (5.5,0.3) and (5.5,-0.3).. (5,-0.7) .. controls (3,-0.7) and (1,-0.4).. (0,0);
\end{scope}
\begin{scope}[shift={(5,0)},rotate=270,scale =0.4]
\draw (0,0) .. controls (1,0.4) and (3,0.7).. (5,0.7) .. controls (5.5,0.3) and (5.5,-0.3).. (5,-0.7) .. controls (3,-0.7) and (1,-0.4).. (0,0);
\end{scope}
\begin{scope}[shift={(5,0)},rotate=315,scale =0.4]
\draw (0,0) .. controls (1,0.4) and (3,0.7).. (5,0.7) .. controls (5.5,0.3) and (5.5,-0.3).. (5,-0.7) .. controls (3,-0.7) and (1,-0.4).. (0,0);
\end{scope}
\begin{scope}[blue, shift={(5,0)},rotate=67.5,scale =0.4]
\draw (0,0) --(7,0);
\end{scope}
\begin{scope}[blue, shift={(5,0)},rotate=202.5,scale =0.4]
\draw (0,0) --(7,0);
\end{scope}

\end{tikzpicture}}
\caption{A parabolic pacman $f_0$ with rotation number $1/3$ embeds as a prepacman into the dynamical plane of a parabolic pacman $f_{-1}$ with rotation number $3/8$. We have $f_{0,-}=f_{-1}^3$ and $f_{0,+}=f_{-1}^2$.} \label{Fig:Ren3over8}
\end{figure}
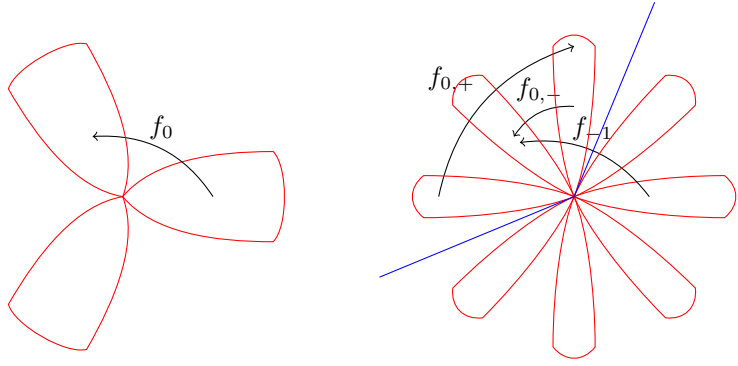

\subsection{The post-critical set of a maximal prepacman}
The \emph{forward orbit of $z\in \C$} under $\bF_n$ is
\[\orb_{z}(\bF_n):= \{\bbf_{n,-}^{s}\circ \bbf^{r}_{n,+}(z)\mid s,r \ge 0\};\]
  we do not require that $\bbf_{n,-}^{s}\circ \bbf^{r}_{n,+}(z)$ is defined for all pairs $ s,r$. A \emph{finite orbit of $z$} is 
 \[\orb_{z}^{\le q}(\bF_{n}):= \left\{\bbf_{n,-}^{s}\circ \bbf^{r}_{n,+}(z)\mid s,r \in \{ 0,1,\dots ,q\}\right\}.\]
 Similarly, $\orb_{z}(\bF^{\#}_n)$ and $\orb_{z}^{\le q}(\bF^{\#}_{n})$ are defined. Since $\bF_0$ is an iteration of $\bF^{\#}_n$, there is a $k>1$ such that \[\orb_{z}^{\le q}(\bF_{0})\subseteq \orb_{z}^{\le k^{-n}q}\left(\bF^{\#}_{n}\right)\]
for all $n\le 0$ and $z\in \C$.

An \emph{orbit path} of $\bF_{m}$ is a sequence $x_0, x_1,\dots , x_n$ such that either $x_{i+1}=\bbf_{m,-}(x_i)$ or $x_{i+1}=\bbf_{m,+}(x_i)$. Since $\bF_0$ is an iteration of $\bF^{\#}_n$, an orbit path of $\bF_{0}$ is a ``sub-orbit'' path of $\bF^{\#}_{n}$.

Let us denote by \[C(\bF_{k})\coloneqq \{z\in \C\mid \bbf'_{k,-}(z)=0 \text{ or }\bbf'_{k,+}(z)=0\}\] the set of critical points of $\bF_{k}$; its  \emph{post-critical set} is
\[P(\bF_{k})= \bigcup_{\substack{n+m> 0\\ n,m\ge 0}} \bbf_{k,-}^n\circ \bbf_{k,+}^m (C(\bF_{k})).\]  Similarly $P(\bF^{\#}_{n})$ is defined. Clearly,
\[P(\bF_{0}) \subset P\left(\bF^{\#}_{n}\right) = \mu_*^{n} P(\bF_{n}). \]

 Recall that $0$ is a critical value of $\bF^{\#}_{n}$ for all $n\le 0$; we denote by $\crt_n^{\#}$ the critical point of $\bF^{\#}_n$ such that $\crt_n^{\#}$ is identified with the critical point $c_0(f_n)$ under $\intr \bS^{\#}_n\simeq V\setminus \gamma_1$, see~\eqref{eq:dfn:Ssharp}.

\begin{lem}[Every critical orbit ``passes'' through $0$]
\label{lem:CritOrbPasses0}
For any critical point $x_0$ of $\bbf_{0,\iota}$ with $\iota\in\{-~,+\}$ the following holds. For all sufficiently big $n< 0$ there is an orbit path of $\bF_n^{\#}$ 
\begin{equation}
\label{eq:lem:CritOrbPasses0:OrbPath}
 x_0, x_1,x_2,\dots x_k;\;\; \sp x_{i}=\bbf^{\#}_{n, j(i)} (x_{i-1})
\end{equation}
such that 
\begin{itemize}
\item  $ \bbf_{0,\iota}= \bbf^{\#}_{n, j(k)}\circ \bbf^{\#}_{n, j(k-1)}\circ \dots \circ \bbf^{\#}_{n, j(1)}$, in particular $x_{k}= \bbf_{0,\iota}(x_0)$; 
\item $x_i=\crt_n^{\#}$ and $x_{i+1}=0$  for some $i$. 
\end{itemize}

Therefore,
\[P(\bF_{0})\subset \bigcup_{n\le 0} \orb_0\left( \bF^{\#}_{n}\right).\]
\end{lem}

\begin{proof}
Clearly, the second statement follows from the first. We will use notations from the proof of Theorem~\ref{thm:MaxCommPair:long}. Suppose for definiteness $\iota=``-''$. Recall~\eqref{eq:Xtm} that $\displaystyle\Dom \bbf_{0,-} =\bigcup_{n\ll 0}  \bW^{(n)}_{-} $; thus $x_0\in  \bW^{(n)}_{-} $ for some $n< 0$. The map $\bbf_{0,-}\mid  \bW^{(n)}_{-}$ is conformally conjugate to $f_{n}^{\aa_n}\mid W^{(n)}_-\to D$  (see~\eqref{eq:CommPaif:n:ext}) after identifying $ \bW^{(n)}_{-}$ with $W^{(n)}_{-}$. 
This shows that $x_0,  \bbf_{0,-}(x_0)$ us within an actual orbit $x_0, x_1,\dots, x_k$ of 
\[
(\bbf^\#_{n,\pm}\colon \bU^{\#}_{n,\pm}\to \bS^{\#}_n).
\]which is a prepacman of $f_n$.
We deduce that one of $x_i$ is $\crt_n^{\#}$ and  $x_{i+1}=0$.
\end{proof}

\subsection{Global attracting basin of a parabolic pacman}
\label{ss:AttrFlow}
Since $\bbf_{0,\pm}\colon \Dom\bbf_{0,\pm}\to \C$ are $\sigma$-proper commuting maps with maximal domain we have
\begin{equation}
\label{eq: dom f_0 pm}
\Dom(\bbf_{0,-}\circ \bbf_{0,+})=\Dom(\bbf_{0,+}\circ \bbf_{0,-})\subset \Dom \bF_0\coloneqq  \Dom \bbf_{0,-} \cap  \Dom \bbf_{0,+}.
\end{equation}

Note that for every $q\ge 1$ we have $f^q_0\not= \id$ in any small neighborhood of $\alpha(f_0)$ because, otherwise, considering a lift $\bbf^\rr_{0,-}\circ \bbf_{0,+}^\ss$ of $f_0^q$ we would obtain $\bbf^\rr_{0,-}\circ \bbf_{0,+}^\ss=\id$ in $\C$ which is impossible. Therefore, there is a small open attracting parabolic flower $H_{0}$ around the $\alpha$-fixed point of $f_0$. Each petal of $H_0$ lands at $\alpha$ at a well-defined angle. Assume $H_{0}$ is small enough so that $H_{0}\subset V\setminus \gamma_1$, possibly up to a slight rotation of $\gamma_1$. By Lemma~\ref{lem:gam1 rotate} the flower $H_0$ lifts to the dynamical plane of $\bF_0$ via the identification $V\setminus \gamma_1\simeq \intr \bS_{0}$; we denote by $\bH_{{0}}$ the lift.

 Let $\ee (\pp_0/\qq_0)\not=1$ be the multiplier of the $\alpha$-fixed point of $f_0$. Since $f_0$ is close to $f_\str$, we have $\qq_0>1$. By replacing $H_0$ with its sub-flower we can assume that there are exactly $\qq_0$ connected components of $H_{0}$ with combinatorial rotation number $\pp_0/\qq_0$. We enumerate them counterclockwise as $H_0^0,H_0^1,\dots, H_0^{\qq_0-1}$. Then $f_0$ maps $H_0^i$ to $H_0^{i+\pp_0}$. We will show in Corollary~\ref{cor:wHp:has:0} that $H_0$ is in fact unique; i.e.~$f_0$ has exactly $\qq_0$ attracting directions at $\alpha$. Denote by $\bH_0^i$ the lift of $H_0^i$ to the dynamical plane of $\bF_0$. 
 
 \begin{lem}
 \label{lem: wH is forw invar}
 There are $\rr,\ss\ge 1$ with $\rr+\ss=\qq_0$ such that 
 \[\bbf_{0,-}^\rr\circ\bbf_{0,+}^\ss (\bH_0^i)\subset \bH_0^i.\] 
 The set $\bH_0$ is in $\Dom(\bbf^a_{0,-}\circ \bbf^b_{0,+})$ for all $a,b\ge 0$.
 \end{lem}
 \noindent It will follow from Proposition~\ref{prop:ClassOfPerComp} that $\bbf_{0,-}^\rr\circ\bbf_{0,+}^\ss\colon \bH_0^i\to \bH_0^i$ is the first return map.
\begin{proof}
We have $f_0^{\qq_0}(H_0^i)\subset H_0^i$. Cutting the prepacman $f_0$ along $\gamma_1$ we see that there are $\rr,\ss\ge 1$ with $\rr+\ss=\qq_0$ such that 
$f_{0,-}^\rr\circ f_{0,+}^\ss (H_0^i)\subset H_0^i.$ This implies the first claim. As a consequence, $\bH_0$ is in $\Dom(\bbf^{\rr j}_{0,-}\circ \bbf^{\ss j}_{0,+})$ for all $j\ge 0$. Combined with~\eqref{eq: dom f_0 pm}, we obtain the second claim. 
\end{proof} 
 As a consequence, all of the branches of $\bbf^a_{0,-}\circ \bbf^b_{0,+}$ with $a,b\in \Z$ are well defined for points in $\bH_0$. Set \[\bH\coloneqq \bigcup_{a,b\in \Z} (\bbf_{0,-})^a\circ (\bbf_{0,+})^{b} \left(\bH_{0} \right)\]
 to be the full orbit of $\bH_{{0}}$. Since $\bbf_{0,-}$, $\bbf_{0,+}$ commute and $\bH_0$ is forward invariant under $\bbf_{0,-}^\rr\circ\bbf_{0,+}^\ss$, the set $\bH$ is an open fully invariant subset of $\C$ within ${\Dom \bbf_{0,-}  \cap \Dom \bbf_{0,+} }$. We call $\bH$ the \emph{global attracting basin} of the $\alpha$-fixed point.

A connected component $\bH'$ of $\bH$ is \emph{periodic} if there are $s,r\in \N_{> 0}$ such that $\bbf_{0,-}^s\circ \bbf_{0,+}^r(\bH')=\bH'$. A pair $(s,r)$ is called a period of $\bH'$. We will show in Corollary~\ref{cor:no ghost per comp} that there is no component $\bH'$ of $\bH$ such that $\bbf^r_{0,-}(\bH') =\bH'$ or $\bbf^r_{0,+}(\bH') =\bH'$ for some $r>0$. 

 By Lemma~\ref{lem: wH is forw invar}, the components of $ \bH$ intersecting $\bH_0$ are $(\rr,\ss)$-periodic. Observe next that for any periodic component $\bH'$ and any component $\bH''$ of $\bH$ there are $a,b\ge 1$ with $\bbf_{0,-}^a\circ \bbf_{0,+}^b(\bH'')=\bH'$;~i.e.~$\bH'$ and $\bH''$ are dynamically related. Indeed, by definition there are $a',b'\in \Z$ such that a certain branch of $\bbf_{0,-}^{a'}\circ \bbf_{0,+}^{b'}$ maps $\bH''$  to $\bH'$. Applying $\bbf_{0,-}^{\ss t}\circ \bbf_{0,+}^{\rr t}$ with $t\gg 1$, we obtain required $a,b\ge 1$. As a consequence, all the periodic components have the same periods; in particular they are $(\rr,\ss)$-periodic.

\subsection{Attracting Fatou coordinates}
\label{ss:AttrFatCoord}
 It is classical that $f_0^{\qq_i}\colon H_{0}^{0}\to H_{0}^{0}$ admits \emph{attracting Fatou coordinates:} a univalent map $h:  H_{0}^{0} \to \C$ such that
\begin{itemize}
\item $h\circ f^{\qq_0}_0(z)= h(z)+1$; and
\item there is an $L>1$  such that 
\begin{equation}
\label{eq:dfn:L} h(H_{0}^{0})\supset  \{z\in \C \mid \Re(z)> L\}.
\end{equation}
\end{itemize}

There is a unique dynamical extension $h\colon H_{0}\to \C$ such that
\begin{equation}
\label{eq:defn:h}h\circ f_0(z)= h(z)+1/\qq_0.
\end{equation}

Lifting $h$ to the dynamical plane of $\bF_{0}$ we obtain $\bbh\colon  \bH_{0}\to \C$.

\begin{lem}[Fatou coordinates of $ \bH$]
\label{lem:FatCoordPar}
The map $\bbh\colon  \bH_{0}\to \C$ extends uniquely to a  map $\bbh\colon \bH\to \C$ satisfying 
\begin{equation}
\label{eq:defn:wid:h}
\bbh\circ \bbf_{0,\pm}(z)= \bbh(z)+1/\qq_0.
\end{equation}
 for any choice of ``$\pm$''. For every component $\bH'$ of $ \bH$, the map $\bbh\mid \bH'$ is $\sigma$-proper. The singular values of $\bbh$ are exactly the $\bbh$-images of the critical points of $\bF_0$ and their iterated preimages. 
 
 Moreover, components of $\bH_0$ are in different components of $ \bH$. The set $\bH$ is a proper subset of $\C$. By postcomposing $\bbh$ with a translation we can assume that
 \begin{equation}
 \label{eq:Norm of bbh}
 \bbh(0)=0. 
 \end{equation}
\end{lem}
\begin{proof}
On $\bH_{0}$ Equation~\eqref{eq:defn:wid:h} is just a lift of\eqref{eq:defn:h}. Applying $\bbf_{0,\pm}^{-1}$ and using commutativity of $\bbf_{0,-},\bbf_{0,+}$, we obtain a unique extension of $\bbh$ to $\bH$ such that~\eqref{eq:defn:wid:h} holds.

Since $\bbf_{0,-},\bbf_{0,+}$ are $\sigma$-proper maps, so is $\bbh\mid \bH$. Indeed, suppose that $\bH'\subset  \bH$ is a periodic component intersecting $\bH_0$; the other cases follow by applying a certain branch of $\bbf^a_{0,-}\circ \bbf^b_{0,+}$, where $a,b\in\Z$. Recall from Lemma~\ref{lem: wH is forw invar} that  $\bH'$ is $(\rr,\ss)$-periodic. Consider a compact set $K\subset \C$. We denote by $\bK$ a connected component of the preimage of $K$ under $\bbh\mid \bH'$. Then for a sufficiently big 
$i\gg 1$, we have $\Re(K+ i)> L$ and $\bK_2\coloneqq  \bbf_{0,-}^{\rr i}\circ \bbf_{0,+}^{\ss i}( \bK)$ intersects $\bH_0$, where $L$ is a constant from~\eqref{eq:dfn:L}. Then $\bK_2\subset \bH_0$ and $\bK_2$ is compact as a connected component of the preimage of $K+ i$ under $\bbh\mid \bH' \cap  \bH_0$. We obtain that $\bK\subset \bbf_{0,-}^{-\rr i}\circ \bbf_{0,+}^{-\ss i}(\bK_2)$ is compact because $\bbf_{0,-}^{\rr i}\circ \bbf_{0,+}^{\ss i}$ is $\sigma$-proper. This also shows that singular values of  $\bbh$ are the $\bbh$-images of either critical points of $\bF_{0}$ or their iterated preimages. (We recall a $\sigma$-proper map has no asymptotic values.)

Let $\bH^x_0$ and $\bH^y_0$ be two different components of $\bH_0$ and let $\bH^x$ and $\bH^y$ be the periodic components of $\bH$ containing $\bH^x_0$ and $\bH^y_0$. Since all points in  $\bH^x$ and $\bH^y$, escape eventually to $\bH^x_0$ and $\bH^y_0$ under the iteration of $\bbf^\rr_{0,-}\circ \bbf_{0,+}^\ss$ we have  $\bH^x\not=\bH^y$. As a consequence $\bH\not=\C$. The claim concerning~\eqref{eq:Norm of bbh} is immediate. 
\end{proof}

From now on we assume that~\eqref{eq:Norm of bbh} holds. Denote by $\bH^\per\subset \bH$ the union of periodic components of $\bH$.

 \begin{cor}[Critical point]
 \label{cor:wHp:has:0}
 The set $ \bH^\per$ contains $P(\bF_0)$ and at least one critical point. In particular, $0\in  \bH^\per$. All the critical points of $\bF_0$ are within $\bH$. In the dynamical plane of $f_0$ the flower $H_0$ is unique: $f_0$ has exactly $\qq_0$ attracting direction at $\alpha$ cyclically permuted by $f_0$.
 \end{cor}
\begin{proof}
Since $ \bbh\colon \bH^\per\to \C$ is not a covering map, $ \bH^\per$ contains at least one critical point of $\bF_0$. Since $ \bH^\per$ is forward invariant, $ \bH^\per$ contains $\crt_{n}^{\#}$ for all sufficiently big $n<0$, see Lemma~\ref{lem:CritOrbPasses0}. Therefore, $ \bH^\per$ contains all of the critical values of $\bF_0$. Since $\bH$ is fully invariant, it contains all of the critical points. As a consequence, $H_0$ is unique because the global attracting basin of another flower would also contain $0$.
\end{proof}

\subsection{Dynamics of periodic components}
\label{ss:Enumer of bH}
It follows from Lemma~\ref{lem: wH is forw invar} that 
\[\bH= \bigcup_{a,b\in \Z} (\bbf^{\#}_{n,-})^a\circ (\bbf^{\#}_{n,+})^{b} \left(\bH_{{0}} \right)\]
 for all $n\le 0$. It is also clear that $\bH^\per$ is the union of $\bF_n^{\#}$-periodic components.

\begin{figure}[t!]
\centering{\begin{tikzpicture}
\begin{scope}[shift={(0,0)},scale =0.2]
\draw[red] (-1.5,-15)  .. controls (-1.5,0.5).. (-1,1)..  controls (0,1.2) ..(1,1).. controls (1.5,0.5)..(1.5,-15);
\draw[<-] (0.5,-7) to node [above] {$\bbf_{0,-}$} (4.5,-7);
\draw[->] (0,-12) to node [above] {$\bbf_{0,+}$} (10,-12);
\draw (0,-3) edge[->,bend right] node [above] {$\bbf^\#_{-1,-}$} (-15,-3);
\end{scope}

\begin{scope}[shift={(-1,0)},scale =0.2]
\draw[red] (-1.5,-15)  .. controls (-1.5,0.5).. (-1,1)..  controls (0,1.2) ..(1,1).. controls (1.5,0.5)..(1.5,-15);
\draw[<-] (0.5,-8) to node [above] {$\bbf_{0,-}$} (4.5,-8);
\end{scope}

\begin{scope}[shift={(-2,0)},scale =0.2]
\draw[red]  (-1.5,-15)  .. controls (-1.5,0.5).. (-1,1)..  controls (0,1.2) ..(1,1).. controls (1.5,0.5)..(1.5,-15);
\draw[<-] (0.5,-7) to node [above] {$\bbf_{0,-}$} (4.5,-7);
\draw[->] (0,-14) to node [above] {$\bbf_{0,+}$} (10,-14);
\end{scope}
\begin{scope}[shift={(-3,0)},scale =0.2]
\draw[red]  (-1.5,-15)  .. controls (-1.5,0.5).. (-1,1)..  controls (0,1.2) ..(1,1).. controls (1.5,0.5)..(1.5,-15);
\end{scope}
\begin{scope}[shift={(-4,0)},scale =0.2]
\draw[red]  (-1.5,-15)  .. controls (-1.5,0.5).. (-1,1)..  controls (0,1.2) ..(1,1).. controls (1.5,0.5)..(1.5,-15);
\end{scope}
\begin{scope}[shift={(1,0)},scale =0.2]
\draw[red]  (-1.5,-15)  .. controls (-1.5,0.5).. (-1,1)..  controls (0,1.2) ..(1,1).. controls (1.5,0.5)..(1.5,-15);
\draw[<-] (0.5,-8) to node [above] {$\bbf_{0,-}$} (4.5,-8);
\end{scope}
\begin{scope}[shift={(2,0)},scale =0.2]
\draw[red]  (-1.5,-15)  .. controls (-1.5,0.5).. (-1,1)..  controls (0,1.2) ..(1,1).. controls (1.5,0.5)..(1.5,-15);
\draw (0,-4) edge [->,bend right] node [above] {$\bbf^{\#}_{-1,-}$} (-15,-4);
\draw (0,-1) edge [<-,bend right] node [above] {$\bbf^{\#}_{-1,+}$} (-25,-1);
\end{scope}
\begin{scope}[shift={(3,0)},scale =0.2]
\draw[red]  (-1.5,-15)  .. controls (-1.5,0.5).. (-1,1)..  controls (0,1.2) ..(1,1).. controls (1.5,0.5)..(1.5,-15);
\end{scope}
\begin{scope}[shift={(4,0)},scale =0.2]
\draw[red]  (-1.5,-15)  .. controls (-1.5,0.5).. (-1,1)..  controls (0,1.2) ..(1,1).. controls (1.5,0.5)..(1.5,-15);
\end{scope}
\draw (4,0.1) edge [->,bend right] node [above] {$\bbf^{\#}_{-2,-}$} (-4,0.1);
\end{tikzpicture}}
\caption{The maximal prepacman $\bF_0=(\bbf_{0,\pm})$ of a parabolic pacman $f_0$ with rotation number $1/3$, see Figure~\ref{Fig:Ren3over8}. The map $\bbf_{0,-}$ shifts periodic components of $\bH$ to the left while $\bbf_{0,+}$ shifts the periodic components of $\bH$ to the right. We have $\bbf_{n,-}= \bbf_{n-1,-}^2  \circ \bbf_{n-1,+}$ and  $\bbf_{n,+}=\bbf_{n-1,-}  \circ \bbf_{n- 1,+}$ for all $n$ (obtained from $f_{0,-}=f_1^3$ and $f_{0,+}=f_1^2$ in  Figure~\ref{Fig:Ren3over8}).} \label{Fig:ParPetOfMCommPair}
\end{figure}
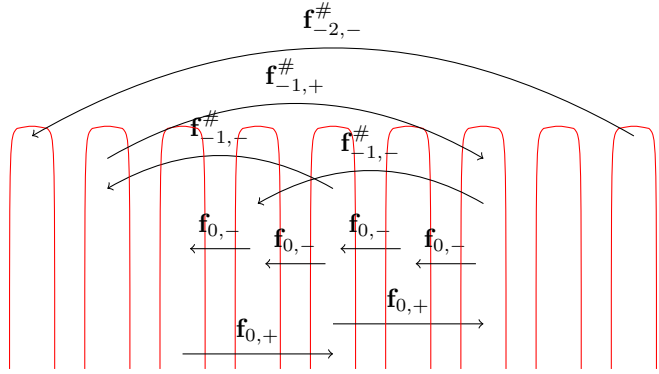

Let $H_{n}$ be a small parabolic attracting flower of $f_n$ admitting a lift to the dynamical plane of $\bF_{n}^{\#}$; we denote this lift by $\bH^\#_{n}\to H_{n}$. We denote by $\pp_n/\qq_n$ the combinatorial rotation number of $f_n$.

Let $I_n$ be an index set enumerating clockwise the connected components of $H_{n}$ starting with the component closest to $\gamma_1$. Since $H_{n}$ embeds naturally to the dynamical plane of $f_{n-1}$ (see Figure~\ref{Fig:Ren3over8}), we have a natural embedding of $I_n$ to $ I_{n-1}$.

Let us write 
\[I_0=\{-a_0,-a_0+1 , \dots b_0-1, b_0\}\]
with $a_0,b_0>0$ and $a_0+b_0+1=\qq_0$. The component of $H_{0}$ indexed by $i+1$ follows in the clockwise order the  component of $H_{0}$ indexed by $i$.
Then $f_0$ maps the component of $H_{0}$ indexed by $i$ to the component of $H_{0}$ indexed by either $i-\pp_0$ or $i+\qq_0-\pp_0$ depending on whether $i-\pp_0\ge -a_0$.

For every $n<0$, choose a parameterization  $I_n=\{-a_n,-a_n+1 , \dots b_n-1, b_n\}$ so that the natural embedding of $I_{n}$ to $I_{n-1}$ is viewed as $I_{n}\subset I_{n-1}$. Set $I_{-\infty}\coloneqq \bigcup_{n\le 0} I_n=\Z$.

Recall (see \S\ref{ss:AttrFlow}) that a connected component $\bH'$ of $\bH$ is periodic if $\bbf_{0,-}^s\circ \bbf_{0,+}^r(\bH')=\bH'$ for some $s,r\in \N_{> 0}$.
\begin{prop}[Parameterization of $\bH^\per$]
\label{prop:ClassOfPerComp}
The connected components of $\bH^\per$ are uniquely enumerated as $(\bH^i)_{i\in \Z}$ so that for every sufficiently big $n\ll0$ the component $\bH^i$ contains the image of the component of $H_{n}$ indexed by $i$ under $H_{n}\simeq \bH_{n}^{\#}\subset \bH^\per$.

The actions of $\bbf^{\#}_{n,\pm}$ on $(\bH^i)_{i\in \Z}$ are given (up to interchanging $\bbf^{\#}_{n,-}$ and $\bbf^{\#}_{n,+}$) by
\begin{equation}
\label{eq:ActionOnHp}
 \bbf^{\#}_{n,-}(\bH^i) = ( \bH^{i-\pp_n}) \text{ and } \bbf^{\#}_{n,+}(\bH^i) = \bH^{i+\qq_n-\pp_n}.
\end{equation}

Moreover, by re-enumerating components of $\bH_0$ we can assume that $\bH^0$ contains $0$.
\end{prop}
\begin{proof} 
By construction, $I_{-\infty}  \simeq \Z$ enumerates all of the periodic components of $\bH$ intersecting $\bigcup_{n\le 0}\bH_{n}^{\#}$ with actions given by~\eqref{eq:ActionOnHp}. Since $\displaystyle\bigcup_{i\in \Z} \bH^i$ is forward invariant and since every periodic component is in the forward orbit of $\bH^0$  (see~\S\ref{ss:AttrFlow}), we obtain  $\displaystyle\bigcup_{i\in \Z} \bH^i= \bH^\per$. We can re-enumerate $(\bH^i)_{i\in \Z}$ in a unique way so that $\bH^0\ni 0$.
 \end{proof}

\begin{cor}
\label{cor:no ghost per comp}
There is no component $\bH'$ of $\bH$ such that $\bbf^r_{0,-}(\bH') =\bH'$ or $\bbf^r_{0,+}(\bH') =\bH'$ for some $r>0$.
\end{cor}

\begin{proof}
Suppose converse and consider such $\bH'$, say $\bbf^r_{0,-}(\bH') =\bH'$. Choose $a,b\in \Z$ such that a certain branch of $\bbf_{0,-}^{a}\circ \bbf_{0,+}^{b}$ maps $\bH'$  to $\bH^0$. Recall that $(\rr,\ss)$ is a period of $\bH^0$. By postcomposing 
 $\bbf_{0,-}^{a}\circ \bbf_{0,+}^{b}$ with an iterate of $\bbf_{0,-}^{\rr }\circ \bbf_{0,+}^{\ss }$ we can assume that $a,b\ge0$. It now follows from Proposition~\ref{prop:ClassOfPerComp} that applying first $ \bbf^r_{0,-}\mid \bH'$ and then $\bbf_{0,-}^{a}\circ \bbf_{0,+}^{b}$ is different from applying $\bbf_{0,-}^{a}\circ \bbf_{0,+}^{b}\mid \bH'$ and then $\bbf^r_{0,-}$. This is a contradiction.
\end{proof}

\begin{cor} 
\label{cor:contr of 0 orb}
For $a,b,c,d\ge 0$ and $n\le 0$,
\[\left(\bbf_{n,-}^\#\right)^a\circ \left(\bbf_{n,+}^\#\right)^b(0)=\left( \bbf_{n,-}^\#\right)^c\circ \left(\bbf_{n,+}^\#\right)^d(0)\]
if and only if $a=c$ and $b=d$.
 \end{cor}
\begin{proof}
It is sufficient to prove it for $n=0$. Suppose $\bbf_{0,-}^a\circ\bbf_{0,+}^b(0)=\bbf_{0,-}^c\circ\bbf_{0,+}^d(0)$. It follows from~\eqref{eq:defn:wid:h} that $a+b=c+d$. If $(a,b)$ is not equal to $(c,d)$, then $\bbf_{0,-}^a\circ\bbf_{0,+}^b(0)$, $\bbf_{0,-}^c\circ\bbf_{0,+}^d(0)$ are in different connected components of $\bH^\per$, see~\eqref{eq:ActionOnHp}. Therefore, $a=c$ and $b=d$. 
\end{proof}

\subsection{Valuable flowers of parabolic pacmen}
\label{ss:ValFlower}
This subsection is a preparation for proving the Scaling Theorem (\S\ref{s:ScalinThm}); it will not be used in proving the Hyperbolicity Theorem (\S\ref{s:HypThm}).

\begin{defn}[Valuable flowers]
\label{dfn:ParValFlow}
Let $f$ be a parabolic pacman with rotation number $\pp/\qq$.
 A \emph{valuable  flower} (see Figure~\ref{Fig:SattValFlow}) is an open forward invariant set $\fH$ such that
\begin{itemize}
\item[(A)] $\fH\cup \{\alpha(f)\}$ is connected;
\item[(B)] $\fH$ has $\qq$ connected components $\fH^{0}, \fH^{1}\dots ,\fH^{\qq-1}$, called \emph{petals}, enumerated counterclockwise at $\alpha$; every $\fH^{i}$ is an open topological disk with a single access to $\alpha$;
\item[(C)] $f(\fH^{i})\subset \fH^{i + \pp}$;
\item[(D)] all of the points in $\fH$ are attracted by $\alpha$;
\item[(E)] $\fH^{-\pp}$ contains the critical point of $f$.
\end{itemize}
\end{defn}
\noindent We remark that a local flower (see~\S\ref{ss:AttrFlow}) satisfies (A)--(D).

We say a Siegel triangulation (see~\S\ref{ss:SiegTriang}) $\bDelta$ \emph{respects} a flower $\fH$ if every petal of $\fH$ is within a triangle of $\bDelta$.

\begin{thm}[Parabolic valuable flowers]
\label{thm:GlobParabFlower}
Let $f_0\in \WW^u$ be a parabolic pacman. Then for all sufficiently big $n\ll 0$ the pacman $f_n=\RR^{n}f_0$ has a valuable flower $\fH_n$ and a Siegel  triangulation $\bDelta(f_n)$ respecting $\fH_n$ such that
\begin{itemize}
\item  $\bDelta(f_n)$ has a wall $\bPi(f_n)$ approximating $\partial Z_\str$;
\item $\bDelta(f_{n-1})$ and $\fH_{n-1}$ are full lifts of $\bDelta(f_n)$ and $\fH_{n}$.
\end{itemize}

Moreover, for a given closed disk $\bD\subset \bH^0$ the flower $\fH_n$ with $n\ll0$  can be constructed in such a way that $\bD$ projects via $\intr\left(\bS^{\#}_{n}\right)\simeq V\setminus \gamma_1$ (see~\eqref{eq:dfn:Ssharp}) to a subset of $\fH_n^0$.
\end{thm}
\begin{proof}
Let us recall (see~\S\ref{ss:AttrFlow}) that a local flower $H_{0}$ was chosen sufficiently small such that $H_{0}\subset V\setminus \gamma_1$, possibly up to a slight rotation of $\gamma_1$ in a small neighborhood of $\alpha$. We denote by $\bDelta_0^\new$ the triangulation obtained from $\bDelta_0$ by this slight adjustment of $\gamma_1$.  By Lemma~\ref{lem:gam1 rotate}, the triangulation $\bDelta_0^\new$ admits a full lift $\bDelta_{-n}^\new$ to the dynamical plane of $f_n$ for all $n\le0$. Since $H_0$ is respected by $\bDelta_0^\new$, the flower $H_0$ also admits a full lift $H_n$ to the dynamical plane of $f_n$ such that $H_n$ is respected by $\bDelta^\new_{-n}$.

\subsubsection{Valuable petals} Recall that $\pp_n/\qq_n$ denotes the rotation number of $f_n$. A \emph{valuable petal} $\fH^j_n$ of $f_n$ is an open connected set attached to $\alpha$ such that
\begin{itemize}
\item $f_n^{\qq_n}$ extends analytically from a neighborhood of $\alpha$ to $f_n^{\qq_n}\colon \fH^j_n\to \fH^j_n$; (in particular, $\fH^j_n$ is  $f^{\qq_n}_n$-invariant)
\item  $f_n^{\qq_n}\colon \fH^j_n\to \fH^j_n$ has a critical point; and
\item all points in $\fH^j_n$ are attracted to $\alpha$ under $f_n^{\qq_n}$.
\end{itemize}

\begin{claim3}
\label{cl3:ValPetalExists}
There is an $n\ll 0$ such that $f_n$ has a valuable petal $\fH^0_n$ containing the critical value $0$ such that $\fH^0_n = H_n^0\cup D$, where $H_n^0$ is a petal of $H_n$ and $D$ is a small neighborhood of $c_1$ containing the projection of $\bD$ via~\eqref{eq:dfn:Ssharp}. Moreover, there is an $M>0$ such that $f^{\qq_nM}_n(\fH^0_n)\subset H_n$.
\end{claim3}
\begin{proof}
In the dynamical plane of $\bF_0$ consider the petal $\bH^0\ni 0$. Recall from~\S\ref{ss:Enumer of bH} that $\bH^\#_n$ denotes the lift of $H_n$ to the dynamical plane of $\bF^{\#}_n$. If $n\ll 0$ is sufficiently big, then $\bH^0$ contains a unique connected component of $\bH^\#_n$, call it $(\bH_n^\#)^0$. Note also that $(\bH_n^\#)^0= (\bH_m^\#)^0$ for all sufficiently big $n,m\ll 0$, see Proposition~\ref{prop:ClassOfPerComp}.

Enlarge $\bD$ to a bigger closed disk $\bD\subset \bH^0$ such that
\begin{itemize}
 \item $(\bH_n^\#)^0\cup \bD$ is forward invariant under the first return map $\bbf_{0,-}^\rr\circ\bbf_{0,+}^\ss$, see Lemma~\ref{lem: wH is forw invar}; and
\item $\bbf_{0,-}^\rr\circ\bbf_{0,+}^\ss\left((\bH_n^\#)^0\cup \bD\right)\ni 0$.
\end{itemize}

Since $\bD$ is compact, we have $\bD\subset \bS_n^{\#}$ for all sufficiently big $n\ll 0$. For such $n$ we can project $\bD$ to the dynamical plane of $f_n$; we denote this projection by $D\ni c_1$. By construction, $D\cup H_{n}^0$ is $f^{\qq_n}_n$-invariant: $f_n^{\qq_n}\colon H_{n}^0 \to H_{n}^0 $ has an analytic extension to $f_n^{\qq_n}\colon D\cup H_{n}^0\to D\cup H_{n}^0$. For $n\ll 0$, the disk $D$ is a small neighborhood of $c_1$.  
\end{proof}

For $n\ll0$, we enumerate petals of $H_n$ counterclockwise so that $H_n^0\subset \fH_n^0$. Choose a big $K$ (we will specify $K$ in~\S\ref{ss:InnerRays}). For $k\in \{0,1,\dots, K\}$ we define $D_k$ to be the image of $D_0=D$ under $f_n^{\qq_n k}$, and for $k\in\{-K,-K+1,\dots, -1\}$ we define $D_k$ to be the lift of $D_0$ along the orbit of $f_n^{-\qq_nk}\colon H_n^{\qq_n k}\to H_n^{0}$. Then
\begin{equation}
\label{eq:defn:val pet}
\fH^{\qq_n k}_n\coloneqq H_n^{\qq_n k}\cup D_k;
\end{equation}
 is a valuable petal extending $H_n^{\qq_n k}$ for all $k\in \{-K,\dots, K\}$. For $n\ll0$, all $\fH_n^{\qq_n k}$  are in a small neighborhood of $\overline Z_\str$.

\subsubsection{Walls respecting $H_n$} 
\label{sss:N Walls}
Set $N\coloneqq M+3$, where $M$ is defined in Claim~\ref{cl3:ValPetalExists}. Let us consider the dynamical plane of $f_0$. In a small neighborhood of $\alpha$ we can choose a univalent $(N+1)\qq_0$-wall $A_0$ \emph{respecting} $H_0$ in the following way:
\begin{itemize}
\item[(a)] $\alpha$ is in the bounded component $O_0$ of $\C\setminus A_0$ while the critical point and the critical value of $f_0$ are in the unbounded component of $\C\setminus A_0$; 
\item[(b)]  each petal $H_{0}^{i}$ intersects $A_0$ at a connected set;
\end{itemize}
and by enlarging $H_0$, we can also guarantee: 
\begin{itemize}
\item[(c)] $H_0$ contains all $z\in A_0\cup O_0$ with forward orbits in $A_0\cup O_0$.
\end{itemize}

We can also assume that the intersection of $A_0$ with each triangle of $\bDelta_0^\new$ is a closed topological rectangle. Lifting these rectangles to the dynamical plane of $f_n$ and spreading around them, we obtain a \emph{full lift} $A_n$ of $A_0$. Then $A_n$ is a univalent $N\qq_n$-wall (see Lemma~\ref{lem:ap:full lift of wall}) enclosing an open topological disk $O_n\ni \alpha$ such that $A_n$ respects $H_n$ as above (see (a)--(c)). Naturally, $A_n$ consists of closed topological rectangles: each rectangle is in a certain triangle of $\bDelta^\new_n$.

\begin{claim3}
\label{cl3:WallAppr partial Z}
For $n\ll0$, the wall $A_n$ approximates $\partial Z_\str$ (compare to Lemma \ref{lem:SiegWall}, Part~\eqref{Claim:Axil:on expansion:5}): $\partial Z_\str$ is a concatenation of arcs $J_0J_1\dots J_{m-1}$ such that $J_i$ is close to the $i$-th rectangle of $A_n$ counting counterclockwise.
\end{claim3}
\begin{proof}
By Theorem~\ref{thm:ComPsConj}, it is sufficient to prove such statement in the dynamical plane of $f_\str$: if $A_0$ is an annulus bounded by two equipotentials of $Z_\str$, then a full lift $A_n$ approximates $\partial Z_\str$ for a big $n$. Since the anti-renormalization change of variables for $f_\str$ is conjugate to $z\to z^{t}$ with $t<1$, the claim follows.
\end{proof}

Consider the dynamical plane of $f_\str$. Recall that $f_\str\mid \overline Z_\str$ is a homeomorphism. For $k\in \Z$, we define \[c_k\coloneqq \left(f_\str\mid \overline Z_\str\right)^k(c_0).\] Consider now the dynamical plane of $f_n$. For $k\in \{-K,-K+1,\dots, K\}$, we define $c_k(f_n)\in f_n^{k}\{ c_0\}$ to be the closest point to $c_k(f_\str)$. The point $c_k(f_n)$ is well defined as long as $f_n$ is in a small neighborhood of $f_\str$.

\begin{claim3}
\label{cl3:fH are small} For $k\in \{-K,-K+1,\dots, K\}$, we have
\begin{itemize}
\item $c_{k-1}(f_n)\in \fH_n^{\qq_n k}$; and
\item $ \fH_n^{\qq_n k}\setminus O_n$ is in a small neighborhood of $c_{k-1}$ 
\end{itemize}
\end{claim3}
\begin{proof}
The first statement follows from $c_{k-1}(f_n)\in D_k\subset  \fH_n^{\qq_n k}$, see~\eqref{eq:defn:val pet}. The second statement follows from the improvement of the domain.
\end{proof}

\begin{claim3}
\label{cl3:contr of rep petals}
Let $P$ be a connected component of $O_n\setminus H_n$. Then $f_n^{\qq_n i}\mid P$ is univalent for all $i\in \{1,\dots, N\}.$ Moreover, 
\[f_n^{\qq_ni}(P)\subset f_n^{\qq_nj}(P) \sp \text{ for all }i<j \text{ in }\{0,1,\dots, N\}.\]
\end{claim3}
\begin{proof}
The first claim follows from the assertion that $A_n$ is an $N\qq_n$-wall. The second claim follows from (c).
\end{proof}

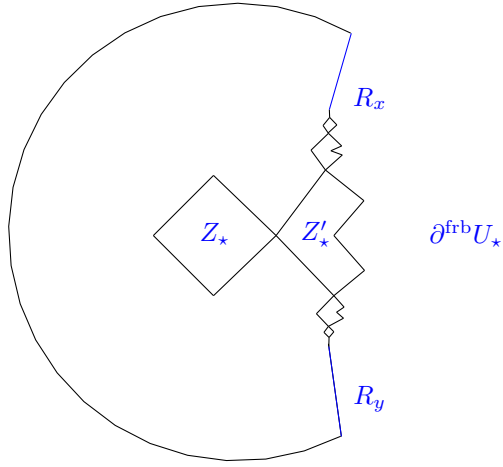
\begin{figure}[t!]
\usetikzlibrary{arrows}
\pagestyle{empty}
\begin{tikzpicture}
\begin{scope}[scale =2]

\begin{scope}[shift={(0.66,0.088)},scale =0.2]
\draw (-2.,0.)-- (0.,2.);
\draw (0.,2.)-- (2.08,0);
\draw (2.08,0)-- (0.,-2.);
\draw (0.,-2.)-- (-2.,0.);
\draw (2.08,0)-- (3.72,2.18);
\draw (3.72,2.18)-- (5.,1.16);
\draw (5.,1.16)-- (4.,0.);
\draw (4.,0.)-- (5.02,-1.16);
\draw (5.02,-1.16)-- (4.,-2.);
\draw (4.,-2.)-- (2.08,0);
\draw (3.72,2.18)-- (3.24,2.84);
\draw (3.24,2.84)-- (3.82,3.4);
\draw (3.82,3.4)-- (4.26,2.98);
\draw (4.26,2.98)-- (3.9,2.82);
\draw (3.9,2.82)-- (4.28,2.68);
\draw (4.28,2.68)-- (3.72,2.18);
\draw (4.,-2.)-- (3.42,-2.6);
\draw (3.42,-2.6)-- (3.82,-3.02);
\draw (4.34,-2.38)-- (4.,-2.);
\draw (3.82,-3.02)-- (4.32,-2.75);
\draw (4.32,-2.75)-- (4.09,-2.55);
\draw (4.09,-2.55)-- (4.34,-2.38);
\draw (3.82,-3.02)-- (3.67,-3.22);
\draw (3.67,-3.22)-- (3.84,-3.39);
\draw (3.84,-3.39)-- (4.,-3.2);
\draw (4.,-3.2)-- (3.82,-3.02);
\draw (3.82,3.4)-- (3.64,3.66);
\draw (3.64,3.66)-- (3.86,3.93);
\draw (3.86,3.93)-- (4.1,3.68);
\draw (4.1,3.68)-- (3.81064,3.41352);
\draw (3.86,3.93)-- (3.85,4.19);

\draw[blue]  (4.58,6.728)-- (3.85,4.19);
\draw (4.25,-6.682)-- (3.82,-3.6);
\draw[blue] (4.25,-6.682)-- (3.8319859383222643,-3.6859085160679514);
\draw (3.8319859383222643,-3.6859085160679514)-- (3.84,-3.39);
\end{scope}

\draw [shift={(0.8200819142056478,0.11146554573543112)}] plot[domain=1.0540413778662818:5.179936858384536,variable=\t]({1.*7.610224188310985*cos(\t r)/5+0.*7.610224188310985*sin(\t r)/5},{0.*7.610224188310985*cos(\t r)/5+1.*7.610224188310985*sin(\t r)/5});

\draw[blue] (0.67,0.1) node{$Z_\str$};
\draw[blue] (1.34,0.1) node{$Z'_\str$};
\draw[blue] (2.34,0.1) node{$\partial ^\frb U_\str$};
\draw[blue] (1.7, 1) node {$R_x$};
\draw[blue] (1.7,-1) node {$R_y$};
\end{scope}
\end{tikzpicture}
\caption{Separation of $\partial ^\frb U_\str$. Co-Siegel disk $Z'_\str$ together with its iterated lifts form two periodic bubble chains landing at periodic points $x$ and $y$. The bubble chains together with external rays $R_x$ and $R_y$  separate $\partial ^\frb U_\str$ from $\alpha$.}
 \label{Fig:SiegPacm:SepForbSector}
\end{figure}

\subsubsection{Julia rays in {$\partial \Jul_{\str}$}}
\label{ss:InnerRays}
Consider the dynamical plane of $f_\str: U_{\str}\to V$. By Theorem~\ref{thm:LocConnOfJul for StandPacm}, we can choose (see Figure~\ref{Fig:SiegPacm:SepForbSector}) two periodic points $x,y \in \Jul_{\str}$ together with two periodic external rays $R_x,R_y$ landing at $x$, $y$ and two periodic bubble chains $B_x,B_y$ landing at $x,y$ so that $x$ and $y$ are close to $\partial^\frb U_\str$ and $R_x\cup B_x\cup B_y\cup R_y$ separates $\partial ^\frb U_\str$ from $c_1$ as well as from all the remaining points in the forward orbits of $x,y$. Let $p$ be a common period of $x,y$. Set $K\coloneqq 4p$.

A \emph{Julia ray} $J$ of $\Jul_{\str}$ is a simple arc in $\Jul_{\str}$ starting at a point in $\partial Z_{\str}$.

\begin{claim3}
\label{clm3:IxIy}
There are Julia rays $J_x\subset B_x$ and $J_y\subset B_y$ such that $J_x$ and $J_y$ start at the critical point $c_0$ and land at $x$ and $y$ respectively. Moreover, $J_x$ and $J_y$ are periodic with period $p$: the rays $J_x$ and $J_y$ decompose as concatenations $J_{x}^{1}J_{x}^{2}J_{x}^{3}\dots$ and $J_{y}^{1}J_{y}^{2}J_{y}^{3}\dots $ such that $f^p_\str$ maps $J_{x}^{k}$ and $J_{y}^{k}$ to $J_{x}^{k-1}$ and $J_{y}^{k-1}$ respectively.
 \end{claim3}
\begin{proof}
Write $B_x=(Z_1, Z_2,\dots )$; since $x$ is close to $\partial ^\frb U_\str$ we see that $Z_1=\overline Z'_\str$. Since $x$ is periodic with period $p$, there is an $a>0$ such that $f^p$ maps $Z_{a+i}$ to $Z_i$ for all $i$.

Let $J_{x}^{1}\subset \Jul_\str$ be a simple arc in $\partial Z_1 \cup \partial Z_2 \cup \dots \cup \partial Z_{a}$ connecting the critical point $c_0$ to the point where $\partial Z_{a+1}$ is attached to $\partial Z_{a}$. We inductively define $J_{x}^{j}$ to be the iterated lift of $J_{x}^{j-1}$ such that $J_{x}^{j}$ starts where $J_{x}^{j-1}$ terminates. This constructs $J_x=J_{x}^{1}J_{x}^{2}J_{x}^{3}\dots$; similarly $J_y=J_{y}^{1}J_{y}^{2}J_{y}^{3}\dots $ is constructed.\end{proof}

\subsubsection{Julia rays for $f_n$}
\label{ss:AlmInnRays}
Recall that in Claim~\ref{clm3:IxIy} we specified Julia rays $J_x(f_\str)$ and $J_y(f_\str)$. Since $f_0$ is sufficiently close to $f_\str$, the periodic points $x,y$ exist in the dynamical plane of $f_0$ and are close to those of $f_\str$. For $n\ll0$ let us now construct \emph{Julia rays} $J_x(f_n)=J_{x}^{1}J_{x}^{2}J_{x}^{3}\dots $ and  $J_y(f_n)=J_{y}^{1}J_{y}^{2}J_{y}^{3}\dots $ such that
\begin{enumerate}
\item $f^p_n$ maps $J_{x}^{k}$ to $J_{x}^{k-1}$ and  $J_{y}^{k}$ to $J_{y}^{k-1}$ (compare with Claim~\ref{clm3:IxIy});
\item $J_{x}^{k}(f_n)$ and  $J_{y}^{k}(f_n)$ are in small neighborhoods of $J_{x}^{k}(f_*)$ and  $J_{y}^{k}(f_*)$ respectively;
\item for $z\in J_{x}^{1}\cup J_{x}^{2}\cup J_{y}^{1}\cup  J_{y}^{2}$ there is a $q\le 2p$ such that either $f_n^{q}(z)\in O_n$ or $\displaystyle f_n^{q}(z)\in  \bigcup_{|k|\le 2p}\fH_{n}^{ k\qq_n} $. In the former case we can assume that $f_n^\ell(z)\not\in A_n\cup O_n$ for $\ell\in \{0,1,\dots, q-1\}$.\label{pr:3:ss:AlmInnRays}
\end{enumerate}

\begin{proof}[Construction of $J_{x}$ and $J_y$]
We will use notations from the proof of Claim~\ref{clm3:IxIy}. By stability of periodic points, $x,y$ exist for $f_n$ and are close to $x(f_\str), y(f_\str)$. The curve $J_{x}^{1}$ is a simple arc in $\partial Z_1 \cup \partial Z_2 \cup \dots \cup \partial Z_{a}$. We split $J_1$ as the concatenation $\ell_1\cup \ell_2 \dots\cup \ell_{a}$ with $\ell_j=J_{x}^{1}\cap \partial Z_j$. Let $f_\str^{d(j)}$ be the smallest iterate mapping $ Z_j$ to $\overline Z_\str$. Since $J_x\subset \Jul_\str$, the curve \[\widetilde \ell_j \coloneqq f_\str^{d(j)}(\ell_j)\] is a simple arc in $\partial Z_\str$ connecting $c_1$ and a certain $c_{t(j)}$.

 Using Claims~\ref{cl3:WallAppr partial Z} and~\ref{cl3:fH are small}, we approximate each $\widetilde \ell_j (f_\str)$ by a curve $\widetilde \ell_j (f_n)$ within $O_n \cup \fH_{n}^{ t(j)+1}\cup\fH_{n}^{ 0}$. Lifting $\widetilde \ell_j (f_n)$ along the branch of $f_n^{d(j)}$ that is close to $f_\str^{d(j)}\mid   \ell_j (f_\str)$, we construct $\ell_j(f_n)$ that is close to $\ell_j(f_\str)$. Assembling all $\ell_j$, we construct $J_{x}^{1}(f_n)$. By continuity, pulling back $J_{x}^{1}(f_n)$ we construct finitely many $J_{x}^{k}(f_n)$ approximating $J_{x}^{k}(f_\str)$ such that the remaining curves $J_{x}^{k}(f_\str)$ are within the linearization domain of $x$. Taking pullbacks within the linearization domain of $x$, we construct a ray  $J_x(f_n)$ landing at $x$. Similarly,  $J_y$ is constructed. Property (3) follows from $|t(j)|\le p$.
\end{proof}

\subsubsection{Blocking $\partial ^\frb U_n$}
\label{sss:prf:thm:GlobParabFlower}

Recall from Claim~\eqref{cl3:ValPetalExists} that $f^{\qq_n M}_n(\fH^0_n)\subset H^0_n$. For $t\in \{M,M-1,M-2,\dots, 0\}$ we set $H_n^{(t)}$ to be the forward $f_n$-orbit of $f^{\qq_n k}_n(\fH^0_n)$.

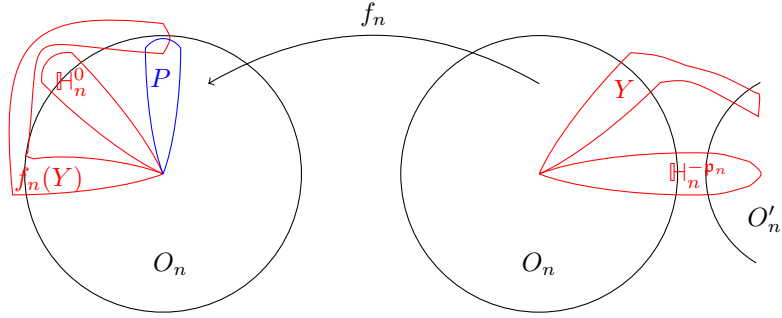
\begin{figure}

\centering{\begin{tikzpicture}[red]

\begin{scope}[shift={(0,0)},rotate=0,scale =0.4]

\draw[black] (0,-3) node {$O_n$};
\draw[black] (7.5,-1.5) node {$O'_n$};
 \draw[black]  (0,3) edge[->,bend right] node [above] {$f_{n}$}  (-11,3);

\begin{scope}[shift={(9,0)}]
\draw[black] plot[domain=2.067609495318839:4.1705147732538323,variable=\t]({1.*3.4616181187415807*cos(\t r)+0.*3.4616181187415807*sin(\t r)},{0.*3.4616181187415807*cos(\t r)+1.*3.4616181187415807*sin(\t r)});
\end{scope}

\draw[black] (0.,0.) circle (4.6);

\draw (0,0) .. controls (1,0.4) and (3,0.7).. 
(5,0.7) .. controls (6,0.7)..
(7, 0.4).. controls (7.5,0).. 
(7, -0.4).. controls (6,-0.7) ..
(5,-0.7) .. controls (3,-0.7) and (1,-0.4).. (0,0);

\draw  (5.3,0) node {$\fH_{n}^{-\pp_n}$};
\end{scope}

\begin{scope}[shift={(0,0)},rotate=45,scale =0.4]

\draw (0,0) .. controls (1,0.4) and (3,0.7).. 
(5,0.7) .. controls (5.6,0)..
(6, -0.9).. controls (6.6,-1.9)..
(7, -3.2).. controls (7.1,-3.3).. 
(6.5, -3.8).. controls (5.8,-1.3).. 
(5,-0.7) .. controls (3,-0.7) and (1,-0.4).. (0,0);
\draw  (4,-0) node {$Y$};
\end{scope}

\begin{scope}[shift={(-5,0)},rotate=135]

\begin{scope}[shift={(0,0)},rotate=0,scale =0.4]

\draw[black] (0.,0.) circle (4.6);

\draw (0,0) .. controls (1,0.4) and (3,0.7).. (5,0.7) .. controls (5.5,0.3) and (5.5,-0.3).. (5,-0.7) .. controls (3,-0.7) and (1,-0.4).. (0,0);
\draw  (4.3,0) node {$\fH_{n}^{0}$};
\end{scope}

\begin{scope}[shift={(0,0)},rotate=45,scale =0.4]

\draw (0,0) .. controls (1,0.4) and (3,0.7).. 
(5,0.7).. controls (5.5,-4.5) and (4.5,-5.5)..
(0,-5).. controls (-0.3,-4.5)..
(0,-4).. controls (5,-4.3) and (4,-5.3)..
(4.5,-0.7) .. controls (4.6,-0.6)..
(4.3,-0.5) .. controls (3,-0.7) and (1,-0.4).. (0,0);
\draw  (3.8,0.2) node {$f_n(Y)$};
\end{scope}

\begin{scope}[blue,shift={(0,0)},rotate=-45,scale =0.335]
\draw (0,0) .. controls (1,0.4) and (3,0.7).. (5,0.7) .. controls (5.5,0.3) and (5.5,-0.3).. (5,-0.7) .. controls (3,-0.7) and (1,-0.4).. (0,0);
\draw  (3.8,0.1) node {$P$};
\end{scope}
\end{scope}

\draw[black] (-4.9,-1.2) node {$O_n$};
\end{tikzpicture}}
\caption{Illustraction to the proof of Claim~\ref{cl3:Nr4}. If $Y$ intersects $O'_n$, then applying $f_n$ we obtain that $P\cup f_n(Y)$ encloses $\fH_{n}^{0}$. Since $P$ is surrounded by the wall $A_n$, the set $f^{\qq_n M}_n\left( P\cup f_n(Y)\right)$ also encloses $ \fH_{n}^{0}$. Then $ f^{\qq_n}_n\mid f^{\qq_n M}_n\left( P\cup f_n(Y)\right)$ has degree one while $f^{\qq_n}_n\mid \fH_{n}^{0}$ has degree $2$; this is a contradiction.} \label{Fig:TopConstr}
\end{figure}

\begin{claim3}
\label{cl3:Nr4}
The flower $H_{n}^{(t)}$ does not intersect $\partial^\frb U_n$ for all $t\in \{M,\dots, 0\}$.
\end{claim3} 
\noindent As a consequence, $H_n$ extends to a required $\fH_n\coloneqq H_{n}^{(0)}$ for $n\ll 0$.
\begin{proof}
Recall that valuable petals $\fH_{n}^{ k\pp_n}\subset U_n$ with $|k| \le K$ are already constructed. Set
\begin{equation}
\label{eq:defn:H'n}
H'^{(t)}_{n}\coloneqq H^{(t)}_{n}\setminus \bigcup_{|k|\le K} \fH_{n}^{ k\pp_n}.
\end{equation}
Let us show that $H'^{(t)}_{n}$ does not hit $R_x\cup J_x\cup J_y\cup R_y$; this would imply that $H^{(t)}_n$ does not intersect $\partial ^\frb U_n$. Suppose converse; since $H'^{(t)}_{n}$ does not intersect $R_x\cup R_y$, we can consider the first moment $t$ (i.e.~$t$ is the closest to $M$) when $H'^{(t)}_{n}$ hits $J_x\cup J_y$. Denote by $X$ a petal of $H'^{(t)}_{n}$ intersecting $J_x\cup J_y$. Choose $z\in X \cap \left(J_x\cup J_y \right)$; we can assume that $z\in J_{x}^{1}\cup J_{x}^{2}\cup J_{y}^{1}\cup  J_{y}^{2}$, otherwise $t$ is not the first moment when $H'^{(t)}_{n}$ hits $J_x\cup J_y$. By~Property~\eqref{pr:3:ss:AlmInnRays} from \S\ref{ss:AlmInnRays}, there is a $q\le 2p$ such that either $f_n^{q}(z)\in O_n$ or $f_n^{q}(z)\in  \bigcup_{|k|\le 2p}\fH_{n}^{ k\qq_n} $. The latter would imply that $X$ is a petal in $\bigcup_{|k|\le 4p}\fH_{n}^{ k\qq_n}$; this contradicts to~\eqref{eq:defn:H'n}. Therefore, $f_n^{q}(z)\in O_n$.

Write \[O'_n\coloneqq f_{n}^{-1}(O_n)\setminus (A_n\cup O_n)\ni f_n^{q-1}(z)\] and  set $Y\coloneqq f^{q-1}(X)$. We have $O'_n\cap Y\ni  f_n^{q-1}(z)$,
see Figure~\ref{Fig:TopConstr}. Since $\fH^{-\pp_n}_n$ contains a critical point, we see that $f_n^q(z)$ is within a connected components $P$ of $ O_n\setminus (H_n\cup \{\alpha\})$ and, moreover, $P\cup f_n(Y)$ surrounds $\fH^{0}_n$.

 Let us apply $f_n^{\qq_n M}$ to $f_n(Y)\cup P$. By Claim~\ref{cl3:contr of rep petals} (recall that $N>M+1$, see \S\ref{sss:N Walls}), we have $f_n^{\qq_n M}(P)\subset (A_n\cup O_n)\setminus H_n$ and $f_n^{\qq_n M}(P)$ does not contain a critical point of $f_n^{\qq_n}$. On the other hand, $f_n^{\qq_n M+1}(Y)$ does not contain a critical point of $f^{\qq_n}_n$ as a subset of $H_n$. Note that $f^{\qq_n M}_n\left( P\cup f_n(Y)\right)$ still surrounds $\fH_n^0$. This is a contradiction: $f_n^{\qq_n}\mid f^{\qq_n M}_n\left( P\cup f_n(Y)\right)$ has degree one while $f_n^{\qq_n}\mid \fH_n^0$ has degree $2$.
\end{proof}

\subsubsection{Siegel triangulation}
It remains to construct a Siegel triangulation $\bDelta(f_n)$ respecting $\fH_n$ for $n\ll0$. In the dynamical plane of $f_n$, let us choose a curve $\ell_1\subset V$ connecting $\partial V$ to $\alpha$ such that $\ell_1$ enters $U_n$ in $O'_n$, then reaches $\partial \fH_{n}^{-\pp_n}$, then travels to $\alpha$ within $\partial  \fH_{n}^{-\pp_n}$. We can assume that $\ell_1\setminus O_n$ is disjoint from $\gamma_1\setminus O_n$. Observe that $\ell_1$ is liftable to the dynamical planes $f_m$ for all $m\le n.$ Indeed, $\ell_1\cap \partial  \fH_{n}^{-\pp_n}$ is liftable because so is $\partial  \fH_{n}^{-\pp_n}$, while  $\ell_1\setminus  \partial  \fH_{n}^{-\pp_n}$ is liftable because it is disjoint from $\gamma_1$.  

Let us slightly perturb $\ell_1$ so that the new $\ell_1$ is disjoint from $\fH_n$. Define $\ell_0$ to the preimage of $\ell_1$ connecting $\partial U_n$ to $\alpha$. Then $\ell_1\cup \ell_0$ splits $U_n$ into two closed sectors; they form the triangulation denoted by $\bDelta(f_n)$. We can assume that $\ell_1$ was chosen so that $\ell_1\setminus O_n$ and $\ell_0\setminus O_n$ are connected. We define the wall $\bPi(f_n)$ to be the closures of two connected components of $U_n\setminus (O_n\cup \ell_0\cup \ell_1).$ 

For $m\le n$ we define $\bDelta(f_m)$ and $\bPi(f_m)$ to be the full lifts of  $\bDelta(f_n)$ and $\bPi(f_n)$. Then $\bDelta(f_m)$ is a required triangulation for $m\ll n$.
\end{proof}

\section{Hyperbolicity Theorem} 
\label{s:HypThm}

Recall that by $\lambda_\str$ we denote the multiplier of the $\alpha$-fixed point of $f_\str$. For $\lambda$ close to $\lambda_\str$ set 
\[\FF(\lambda):= \{f\in \WW^u \mid \text{the multiplier of $\alpha$ is $\lambda$}\}\]
the analytic submanifold of $\WW^u$ parametrized by fixing the multiplier at $\alpha$. Then $\FF(\lambda)$ forms a foliation of a neighborhood of $f_\str$.

\subsection{Holomorphic motion of $P(\bF_{0})$.}

Let $\UU\subset \WW^u$ be a small neighborhood of $f_\str$ such that every $f\in \UU$ has a maximal prepacmen, see Theorem~\ref{thm:MaxCommPair:long}.

\begin{lem}[Holomorphic motion of the critical orbits]
\label{lem:HolMotOfP}
For every $\pp/\qq$, the set \[\bigcup_{n\le 0}\orb_0(\bF^{\#}_{n}) \] moves holomorphically with $f_0\in \FF(\ee (\pp/\qq))\cap \UU$. 
\end{lem}
\noindent Recall from Lemma~\ref{lem:CritOrbPasses0} that $P(\bF_{0})\subset \bigcup_{n\le 0}\orb_0(\bF^{\#}_{n}) $ thus $P(\bF_{0})$ also moves holomorphically with $f_0\in \FF(\ee (\pp/\qq))\cap \UU$.

\begin{proof} 
By Corollary~\ref{cor:contr of 0 orb}, points in $\orb_0(\bF^{\#}_{n})$ do not collide with each other when $f_0\in \FF(\ee (\pp/\qq))\cap \UU$ is deformed. This gives a holomorphic motion of  $\orb_0(\bF_{0})\subset \orb_0(\bF^{\#}_{1})\subset  \orb_0(\bF^{\#}_{2})\subset \dots$ and we can take the union.
\end{proof}

Let $\UU'\subset \UU$ be a neighborhood of $f_\str$ such that every non-empty $\FF(\lambda)\cap \UU'$ has radius at least three times less than those of $\FF(\lambda)\cap \UU$.

\begin{cor}[Extended holomorphic motions]
\label{cor:UnifMotOfP}
	For $f_0\in \FF(\ee (\pp/\qq))\cap \UU'$ there is a holomorphic motion $\tau(f_0)$ of $\widehat \C$ such that $\tau(f_0)$ is equivariant (with the dynamics of $(\bF^{\#}_{n})_n$) on 
 \[\bigcup_{n\le 0}\orb_0(\bF^{\#}_{n}) .\] 
\end{cor}
\begin{proof}
Follows by applying the $\lambda$-lemma to the holomorphic motion from Lemma~\ref{lem:HolMotOfP}.
\end{proof}

\begin{cor}[Passing to the limit of holomorphic motions]
\label{cor:UnifMotOfP:2}
For $f_0\in \FF(\lambda_\str)\cap \UU'$ there is a holomorphic motion $\tau(f_0)$ of $\widehat \C$ such that $\tau(f_0)$ is equivariant  on 
 \[\bigcup_{n\le 0}\orb_0(\bF^{\#}_{n}) .\] 
\end{cor}
\begin{proof}
Choose a sequence $\pp_n/\qq_n$ such that $\ee (\pp_n/\qq_n)\to \ee(\theta_\str)$. By passing to the limit in Corollary~\ref{cor:UnifMotOfP} we obtain the required property.
\end{proof}
 
\begin{cor}
\label{cor:WWu has dim 1}
The dimension of $\FF(\lambda_\str)$ is $0$.
\end{cor}
\begin{proof}
Suppose the dimension of $\FF(\lambda_\str)$ is greater than $0$. Consider the space $\FF(\lambda_\str)\cap \UU'$. By Corollary~\ref{cor:UnifMotOfP:2} the set $\overline {P(\bF_{0})}\subset \overline{\bigcup_{n\le 0}\orb_0(\bF^{\#}_{n})}$ moves holomorphically with $f_0\in \FF(\lambda_\str)\cap \UU'$. Projecting this holomorphic motion to the dynamical plane of $f_0$, we obtain a holomorphic motion of the post-critical set of $f_0\in \FF(\lambda_\str)\cap \UU'$. Therefore, there is a small neighborhood of $f_\str$ in $\FF(\lambda_\str)\cap \UU'$ consisting of Siegel maps. But all such maps must be in the stable manifold of $f_\str$ by Theorem~\ref{thm:ExpConv}. 
\end{proof}

\subsection{The exponential convergence}
\label{ss:ExpConv}
The following theorem follows from~\cite[Theorem 8.1]{McM3}. 
\begin{thm}
\label{thm:ExpConv}
Suppose that a pacman $f\in \BB$ is Siegel of the same rotation number as $f_\str$ such that $f$ is sufficiently close to $f_\str$. Then $\RR^n f$ converges exponentially fast to $f_\str$. 
\end{thm}

\begin{rem}
The proof of~\cite[Theorem 8.1]{McM3} is based on a ``deep point argument''. Alternatively, the exponential convergence follows from a variation of the Schwarz lemma following the lines of~ \cites{L:FCT,AL-horseshoe}.
\end{rem}

\subsection{The Hyperbolicity Theorem}
\begin{thm}[Hyperbolicity of $\RR$]
\label{thm:RR is hyper}
The renormalization operator $\RR\colon \BB\to \BB$ is hyperbolic at $f_\str$ with one-dimensional unstable manifold $\WW^u$ and codimension-one stable manifold $\WW^s$.

 In a small neighborhood of $f_\str$ the stable manifold $\WW^s$ coincides with the set of pacmen in $\BB$ that have the same multiplier at the $\alpha$-fixed point as $f_\str$. Every pacman in $\WW^s$ is Siegel.
 
  In a small neighborhood of $f_\str$ the unstable manifold $\WW^u$ is parametrized by the multipliers of the $\alpha$-fixed points of $f\in \WW^u$.
\end{thm}
\begin{proof}
It was already shown in Corollary~\ref{cor:WWu has dim 1} that the dimension of $\WW^u$ is one. Let us show that $\WW^s$ has codimension one. Denote by $\BB^{*}$ the submanifold of $\BB$ consisting of all the pacmen with the same multiplier at the $\alpha$-fixed point as $f_\str$. Then $\RR$ naturally restricts to $\RR\colon \BB^{*}\to \BB^{*}$. Consider the derivative $\Diff(\RR\mid \BB^* )$; by Corollary~\ref{cor:WWu has dim 1}  the spectrum of $\Diff(\RR\mid \BB^* )$ is within the closed unit disk. Suppose that the spectrum of $\Diff(\BB^* )$  intersects the unit circle. By \cite[Small orbits theorem]{L:FCT} $\RR\mid \BB^*$ has a small slow orbit: there is an $f\in \BB^*$ such that $f$ is infinitely many times renormalizable but 
\[\lim_{n \to +\infty} \frac 1 n \log \norm{\RR^n f)} =0.\]
Moreover, it can be assumed that $\{\RR^n f\}_{n\ge 0}$ is in a sufficiently small neighborhood of $f_\str$. By Corollary~\ref{cor:inf ren are hybr conj}, $f$ is a Siegel pacman and by Theorem~\ref{thm:ExpConv}, $\RR^n f$ converges exponentially fast to $f_\str$. This is a contradiction. Therefore, the spectrum of $\RR$ is compactly contained in the unit disk, all of the pacmen in $\BB^*$ are infinitely renormalizable and thus are Siegel (Corollary~\ref{cor:inf ren are hybr conj}). The submanifold $\BB^*$ coincides with $\WW^s$ in a small neighborhood of $f_\str$. 
\end{proof}

\subsection{Control of Siegel disks}
The following lemma follows from \cite[Theorem 8.1]{McM3} combined with Theorem~\ref{thm:HybrCongSiegMaps} and Lemma~\ref{lem:SiegPacmExist}.
\begin{lem}
\label{lem:R_2 of Sieg map}
Every Siegel map $f$ has a pacman renormalization $\RR_2 f$ such that $\RR_2 f$ is in $ \BB$ and is sufficiently close to $f_\str$. \qed
\end{lem}

We say a holomorphic map $f\colon U\to V$ is \emph{locally Siegel} if it has a distinguished Siegel fixed point. The following corollary follows from Theorem~\ref{thm:RR is hyper} combined with Lemma~\ref{lem:R_2 of Sieg map}

\begin{cor}
\label{cor:Sieg disk depends cont}
Let $f\colon U\to W$ be a Siegel map with rotation number $\theta\in \Theta_\per$ and let $N(f)$ be a small Banach neighborhood of $f$. Then every locally Siegel map $g\in N(f)$ with rotation number $\theta$ is a Siegel map. Moreover, the Siegel disk $\overline Z_g$ depends continuously on $g$. \qed
\end{cor}

\section{Scaling Theorem}
\label{s:ScalinThm}
In this section we prove a refined version of Theorem~\ref{scaling thm}. Consider  $\theta_\str\in \Theta_\per$ and let $f$ be a Siegel map with rotation number $\theta_\str$. Let $\UU\ni f$ be a small Banach neighborhood of $f$ and let $\WW\subset \UU$ be a one-dimensional slice containing $f$ such that $\WW$ is transverse to the hybrid class of $f$; i.e. in a small neighborhood of $f\in \WW$ all maps have different multipliers at their $\alpha$-fixed points.

\begin{figure}
\begin{tikzpicture}
\node at (0,0){\includegraphics[scale=0.86]{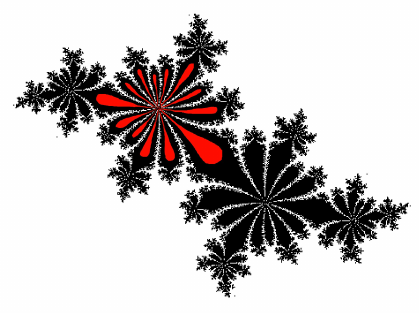}};
\node at (6,0){\includegraphics[scale=0.725]{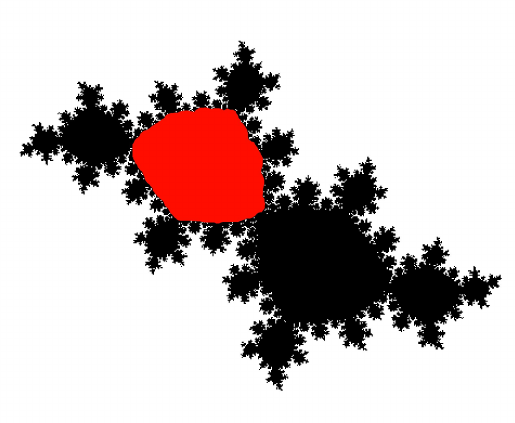}};

\end{tikzpicture}

\caption{A satellite valuable flower (red) of the $5/13$ Rabbit approximates the golden Siegel disk (also red).}
\label{Fig:SattValFlow}
\end{figure}

We say a map $g\in \UU$ is \emph{satellite} if it has a satellite valuable flower: 
\begin{defn}[Satellite valuable flowers]
\label{dfn:SatValFlow}
 A \emph{satellite valuable flower} of $g$ is an open forward invariant set $\fH$ such that (see Figure~\ref{Fig:SattValFlow})
\begin{itemize}
\item[(A)] $\fH\cup \{\alpha(g)\}$ is connected;
\item[(B)] $\fH$ has $\qq$ connected components $\fH^{0}, \fH^{1}\dots ,\fH^{\qq-1}$, called \emph{petals}, enumerated counterclockwise at $\alpha$; every $\fH^{i}$ is an open topological disk with a single access to $\alpha$;
\item[(C)] $g(\fH^{i})\subset \fH^{i + \pp}$, where $\pp$ is coprime to $\qq$;
\item[(D)] there is an attracting periodic cycle $\gamma=(\gamma_0,\gamma_1,\dots, \gamma_{\qq-1})$ with $\gamma_i\in \fH^i$ attracting all points in $\fH$;
\item[(E)] $\fH^{-\pp}$ contains the critical point of $g$.
\end{itemize}
The number $\pp/\qq$ is called the \emph{combinatorial rotation number} of $\fH$. The \emph{multiplier} of $\fH$ is the multiplier of $\gamma$. 
\end{defn}

For convenience, let us say that a parabolic valuable flower (see Definition~\ref{dfn:ParValFlow}) with rotation number $\pp/\qq$ is a satellite valuable flower with rotation number $\pp/\qq$ and multiplier $1$.

By Lemma~\ref{lem:eq:S3:R_prm}, $\RR$ acts on the rotation numbers of indifferent pacmen as $\cRRc^\kk$ for a certain $\kk\ge 1$; see also Remark~\ref{lem:min ren per}. 
\begin{thm}
\label{thm:ScalThm}
Suppose a sequence $(\pp_n/\qq_n)_{n=0}^{-\infty}$ converges to $\theta_\str$ so that $\cRRc^\kk (\pp_n/\qq_n)= \pp_{n+1}/\qq_{n+1}$. Fix $\lambda_1\in \Disk$ and a small neighborhood of $\overline Z_f$. Then there is a continuous path $\lambda_t\in  \Disk$ with $ t\in (0,1]$ emerging from $1=\lambda_0$ such that for every sufficiently big $n\ll0$ there is a unique path $g_{n,t}\in \WW$, where $t\in[0,1]$, with the following properties
\begin{itemize}
\item $g_{n,t}$ has a satellite valuable flower $\fH_{n,t}$ with rotation number $\pp_n/\qq_n$ and multiplier $\lambda_t$;
\item all $\fH_{n,t}$ are in the given small neighborhood of $\overline Z_f$ and depend continuously on $t$; and
\item  $\dist(f, g_{n,t})  \sim \left((\cRRc^\kk)' (\theta_\str) \right)^n$ for every $t$.
\end{itemize}
\end{thm}
\noindent Note that the path $g_{n,t}$ starts at a unique parabolic map in $\WW$ with rotation number $\pp_n/\qq_n$. 

\subsection{Proof of Theorem~\ref{thm:ScalThm}}The proof is split into short subsections. Consider a pacman hyperbolic renormalization operator $\RR\colon \BB\dashrightarrow \BB$ around a fixed point $f_\str=\RR(f_\str)$ with rotation number $\theta_\str$.  As before, $\WW^u$ denotes the unstable manifold of $\RR$ at $f_\str$.

\subsubsection{Perturbation of parabolic pacmen}
\label{sss:ScThm:PertParPacm}
By shifting the sequence $(\pp_n/\qq_n)_n$ we can assume that $\pp_0/\qq_0$ is close to $\theta_\str$. Then there is a unique parabolic pacman $f_0\in \WW^u$ with rotation number $\pp_0/\qq_0$. Then $f_n \coloneqq \RR^{n}f_0,\sp n\le 0$ has rotation number $\pp_n/\qq_n$. By Theorem~\ref{thm:GlobParabFlower} and possibly by further shifting $(\pp_n/\qq_n)_n$, we can assume that: 
\begin{itemize}
\item each $f_n$ has a valuable flower $\fH(f_n)$ at the $\alpha$-fixed;
\item each $f_n$ has a triangulation $\bDelta (f_n)$ respecting $\fH(f_n)$: every petal of $\fH(f_n)$ is within a triangle of $\bDelta(f_n)$; 
\item $\bDelta (f_n)$ has a wall $\bPi(f_n)$ approximating $\partial Z_\str$;
\item $\bDelta (f_n)$ and $\fH(f_n)$ are the full lifts of $\bDelta (f_{n+1})$ and $\fH(f_{n+1})$.
\end{itemize}

 Let $g_0\in \WW^u$ be a slight perturbation of $f_0$ that splits $\alpha$ into a repelling fixed point $\alpha$ and an attracting cycle $\gamma(g_0)$ such that $\alpha$ is on the boundary of the immediate attracting basin of $ \gamma(g_0)$. Then $\bDelta (f_0)$, $\bPi(g_0)$, $\fH(f_0)$ are perturbed to $\bDelta (g_0)$, $\bPi(g_0)$, $\fH(g_0)$ such that all points in $\fH(g_0)$ are attracted by $ \gamma(g_0)$. We can assume that the perturbation is sufficiently small such that $\bPi (g_0)$ still approximates $\partial Z_\str$. By Lemma~\ref{lem:SiegTriangLifting}, there are full lifts $\bDelta (g_n)$,  $\fH(g_n)$ of $\bDelta (g_0)$, $\fH(g_0)$.
 
As before, we denote by $\bF_n$ and $\bG_n$ the maximal prepacmen of $f_n$ and $g_n$ and we denote $\bG^\#_n$ the the rescaled version of $\bG_n$ such that $\bG_0=\bG_{0}^{\#}$ is an iteration of $\bG^{\#}_n$. Recall from \S\ref{ss:AttrFlow}~that $\fH(f_0)$ admits a global extension $\bH(\bF_0)$ in the dynamical plane of $\bF_0$. Similarly, we now define the maximal extension $\bH(\bG_n)$ of  $\fH(g_n)$.

 Each $\fH(g_n)$ lifts to the dynamical plane of $\bG^{\#}_n$; denote by $\bH (g_0)$ the lift of  $\fH(g_0)$.  Similar to~\eqref{eq: dom f_0 pm}, we set 
\[\bH(\bG_0)\coloneqq \bigcup_{a,b\in \Z} (\bbg_{0,-})^a\circ (\bbg_{0,+})^{b} \left(\bH(g_0) \right)\]
 to be the full orbit of $\bH(g_0)$. The same argument as in the proof of Lemma~\ref{lem: wH is forw invar} shows that $\bH(\bG_0)$ is fully invariant and is within $\Dom \bG_{0,-} \cap \Dom \bG_{0,+} $. 
 
  Denote by $\bH^\per(\bG_0)$ the union of periodic components of $\bH(\bG_0)$. The same argument as in the proof of Proposition~\ref{prop:ClassOfPerComp} shows:
 \begin{prop}[Parameterization of $\bH^\per(\bG_0)$]
The connected components of $\bH^\per(\bG_0)$ 
are uniquely enumerated as $(\bH^i)_{i\in \Z}$ such that $\bH^0\ni 0$ and such that the actions of $\bbg^{\#0}_{n,\pm}$ on $(\bH^i)_{i\in \Z}$ are given (up to interchanging $\bbg_{n,-}$ and $\bbg_{n,+}$) by
\begin{equation}
 \bbg^{\#}_{n,-}(\bH^i) = ( \bH^{i-\pp_n}) \text{ and } \bbg^{\#}_{n,+}(\bH^i) = \bH^{i+\qq_n-\pp_n}. \qed
\end{equation}
\end{prop}

 \subsubsection{QC-deformation of $g_n$}
 \label{ss:ScalThm:QC-deform}
Suppose first that $\lambda_1\not= 0$. Denote by $\lambda_0$ the multiplier of $ \gamma(g_0)$. Let $\bbg_{0,-}^\rr\circ \bbg_{0,+}^\ss\colon \bH^0(\bG_0)\to \bH^0(\bG_0)$ be the first return map (compare with Lemma~\ref{lem: wH is forw invar}). There is a semiconjugacy $\bbh\colon \bH^0(\bG_0) \to \C$ from $\bbg_{0,-}^\rr\circ \bbg_{0,+}^\ss$ to the linear map $z\to \lambda_0 z$. Choose a continuous path of qc maps $\tau_t\colon \C\to \C$ with $t\in  [0,1]$ such that $\tau_0=\id$ and  $\tau_t$ conjugates $z\to \lambda_0z$ to $z\to \lambda_t z$.

Lifting $\tau_t$ under $\bbh$ and spreading dynamically the associated Beltrami form, we obtain a qc map $\btau_t\colon \C\to \C$ conjugating $\bG_0$ to a maximal prepacman $\bG_{0,t}$; similarly $\btau_t$ conjugates $\bG^\#_{n,t}$ to a maximal prepacman $\bG^\#_{n,t}$ for $n\le 0$.

Define now $\tau_{n,t}$ to be the projection of $\btau_t$ to the dynamical plane of $g_n$ via $\intr \bS^{\#}_{k} \simeq V\setminus \gamma_1$ (see~\eqref{eq:dfn:Ssharp}); we normalize $\tau_{n,t}$ to preserve $\alpha(g_n)$ and $c_1(g_n)$. By compactness of qc-maps, there is a small $T>0$ such that all $g_{n,t}$ are in $\BB$ for $t\le T$. For $m\le0$ consider the sequence $\RR^{-n+m}(g_{n,t})$. All pacmen in this sequence are qc-conjugate with uniform dilatation. Moreover, the conjugacies preserve the critical value and the $\alpha$-fixed point because of the normalization for renormalization change of variables, see~\S\ref{ss:Pacm}. By compactness of qc-maps, $\RR^{-n+m}(g_{n,t})$ has an accumulated point $q_{m,t}\in \BB$, and moreover, we can assume that $\RR q_{m,t}=q_{m+1,t}$; i.e.~$q_{m,t}\in \WW^u$ and $q_{m,t}$ tends to $f_\str$ as $m$ tends to $-\infty$. We define $\bDelta(q_{n,t}),\bPi(q_{n,t}), \fH(q_{n,t})$ to be the images of  $\bDelta(g_{n,t}),\bPi(g_{n,t}), \fH(g_{n,t})$ via the qc-conjugacy from $g_{n,t}$ to $q_{n,t}$. By improvement of the domain, $\bDelta(q_{n,t})$ is in a small neighborhood of $\overline Z_\str$ and $\bPi(q_{n,t})$ approximates $\partial Z_\str$ for $n\ll0$. By shifting the sequence $(\pp_n/\qq_n)_n$ we can assume that this already occurs for $n=0$. We can now repeat the above argument and construct $q_{n,t}$ for $t\in [T,2T]$. After finitely many repetitions, we construct $q_{n,t}$ for all $t$ in $[0,1]$.

 \subsubsection{QC-surgery towards the center}
 \label{ss:ScalThm:QC-surj}
Suppose now $\lambda_1=0$.  In this case we apply a qc-surgery. As in~\S\ref{ss:ScalThm:QC-deform} we denote by  $\lambda_0$ the multiplier of $ \gamma(g_0)$. 
 
 Consider the first return map  \[\bbw_0\coloneqq\bbg_{0,-}^\rr\circ \bbg_{0,+}^\ss\colon \bH^0(\bG_0)\to \bH^0(\bG_0).\] It has a unique attracting fixed point $\bgamma^0$ and a unique critical value at $0$. Thus $\bbw_0$ has also a unique critical point. We can choose a small disk $\bD$ around $\bgamma^0$ such that 
\begin{itemize}
 \item   $0\in   \bbw_0(\bD)\Subset \bD$;
 \item $\bbw_0\colon  \bH^0(\bG_0)\setminus \bD \to   \bH^0(\bG_0)\setminus \bbw_0(\bD)$ is a $2$-to-$1$ covering map.
\end{itemize}
By Theorem~\ref{thm:GlobParabFlower}, we can project $\bD$ to a disk within $\fH(g_0)$.    
We claim that there is a continuous path of qc maps $\btau'_t\colon \bH^0(\bG_0)\to \bH^0(\bG_0)$ 
 and a continuous path $\bbw_t\colon \bH^0(\bG_0)\to \bH^0(\bG_0)$ such that
 \begin{itemize}
 \item $\btau'_t$ is equivariant on $\bH^0(\bG_0)\setminus \bD$;  
 \item $\bbw_t$ has a unique critical value at $0$ and a unique attracting fixed point at $\bgamma_{0,t}$;
 \item  $\bgamma_{0,1}=0$; i.e.~$0$ is supperattracting fixed point of $\bbw_1$.
 \end{itemize}
Indeed, it is sufficient to construct $\bbw_t\mid \bD$ and $\btau'_t \mid \bD$ equivariant on $\partial \bD$; pulling back the Beltrami  differential of $\btau'_t \mid \bD$ via the covering map $\bbw_0\mid  \bH^0(\bG_0)\setminus \bD$ gives the  Beltrami  differential for $\btau'_t \mid \bH^0(\bG_0)$.

Applying  $\bG_0$, we spread dynamically the Beltrami form of $\btau'$ to obtain a global qc map $\btau_t\colon \C\to \C$ which is unique up to affine rescaling. Spreading dynamically the surgery, we obtain a continuous path of maximal prepacmen $\bG^\#_{n,t}$. Define now $\tau_{n,t}$ to be the projection of $\btau_t$ to the dynamical plane of $g_n$ via $\intr\bS^{\#}_{n}\simeq V\setminus \gamma_1$;  similarly, $g_{n,t}$ is the projection of $\bG^\#_{n,t}$. The argument now continues in the same way as in~\S\ref{ss:ScalThm:QC-deform}.

\subsubsection{Lamination around $f_\str$} 
In~~\S\ref{sss:ScThm:PertParPacm},~\S
\ref{ss:ScalThm:QC-deform}, \S\ref{ss:ScalThm:QC-surj} we constructed continuous paths $q_{n,t}\in \WW^u,\sp n\ll0$ with $\RR(q_{n,t})=q_{n+1,t}$ so that each $q_{n,t}$ has a valuable flower $\fH(q_{n,t})$ with multiplier $\lambda_t$, where $\lambda_{0}=1$. Moreover, $\fH(q_{n,t})$ is within a triangulation $\bDelta (q_n)$ respecting $\fH(q_{n,t})$ such that the wall $\bPi(f_n)$ approximate $\partial Z_\str$.

For a big $m\ll 0$,  we define $\Fol_{m,t}$ to be the set of all pacmen close to $q_{m,t}$ such that the multiplier of $\gamma(q_{m,t})$ is $\lambda_t$. Locally $(\Fol_{m,t})_{t}$ is a codimension-one lamination of $\BB$. Since $\Fol_{m,t}$ is in a small neighborhood of $q_{m,t}$, every pacman $g\in\Fol_{m,t}$ has a valuable flower $\fH(g)$ and a triangulation $\bDelta (g) $ respecting $\fH(g)$ such that  $\bDelta (g) $ and $\fH(g)$ depend  continuously on $g$. The wall $\bPi (g) $ approximates $\partial Z_\str$.

For $n\le m$, we define
\[\Fol_{n,t}\coloneqq \{g\in \BB\mid \RR^{m-n}(g)\in \Fol_{m,t}\}.\] 
Since $\RR$ is hyperbolic,
 \begin{equation}
\label{eq:ScThm:bfol} 
\bFol\coloneqq \{\Fol_{n,t}\}_{n,t}\cup \{\WW^s\}
\end{equation} forms a codimension-one lamination in a neighborhood of $f_\str$. A pacman $g\in \Fol_{n,t}$ has $\fH(g)$  and $\bDelta (g) $ satisfying the same conditions as above. In particular, all the pacmen in $\Fol_{n,t}$ are hybrid conjugate in neighborhoods of their valuable flowers.  

\subsubsection{Scaling}
 By Corollary~\ref{cor:SP embeds in SM}, the Siegel map $f$ can be renormalized to a pacman. By Lemma~\ref{lem:R_2 of Sieg map} we can assume that the renormalization of $f$ is within a small neighborhood of $f_\str$. This allows us to define an analytic renormalization operator $\RR_2\colon\UU\dashrightarrow \BB$ from a small neighborhood of $f$ to a small neighborhood of $f_\str$. Since maps in $\WW$ have different multipliers, the image of $\WW$ under $\RR_2$ is transverse to the lamination $\bFol$, see~\eqref{eq:ScThm:bfol}. 
 
 We define $f_{n,t}$ to be the unique intersection of $\Fol_{n,t}$ with the image of $\WW$ under $\RR_2$, and we define $g_{n,t}\in \WW$ to be the preimage of $f_{n,t}$ via $\RR_2$. Since $\bPi(f_{n,t})$ approximates $\partial Z_\str$, the triangulation $\bDelta (f_{n,t})$ and the valuable flower $\fH(f_{n,t})$ have full lifts $\bDelta(g_{n,t})$ and $\fH(g_{n,t})$, see Lemma~\ref{lem:SiegTriangLifting}. Since the holonomy along $\bFol$ is asymptotically conformal \cite[Appendix 2, The $\lambda$-lemma (quasi- conformality)]{L:FCT}, we obtain the scaling result for $g_{n,t}$.
 \subsubsection{Uniqueness of $g_{n,t}$}

Recall (Theorem~\ref{thm:RR is hyper}) that $\WW^u$ is parametrized by the multipliers of the $\alpha$-fixed points. Therefore, parabolic pacmen with rotation numbers $\pp_n/\qq_n, \sp n\ll0$ are unique. As a consequence the paths of satellite pacmen emerging from these parabolic pacmen are unique. 
Similarly, parabolic maps $g_{n,0}\in W$ with rotation numbers $\pp_n/\qq_n$ are unique; thus the paths $g_{n,t}$ are unique. 
\qed

\appendix
\section{Sector renormalizations of a rotation}
\label{ss:ap:SecRen}
Consider $\theta\in \R/\Z$ and let \[\Lbb_\theta\colon \colon \overline{\Disk}\to \overline{\Disk},\sp z\to \ee(\theta) z\] 
be the corresponding rotation of the closed unit disk by angle $\theta$. 
\subsection{Prime renormalization of a rotation}

Assume that $\theta\not=0$ and consider a closed internal ray $\Ibb$ of $\overline \Disk$. A \emph{fundamental sector} $\Sby\subset \overline \Disk$ of $\Lbb_\theta$ is the smallest closed sector bounded by $\Ibb$ and $\Lbb_\theta (\Ibb)$. If $\theta=1/2$, then $\Ibb\cup \Lbb_\theta (\Ibb)$ is a diameter and both sectors of $\overline \Disk$ bounded by $\Ibb\cup \Lbb_\theta (\Ibb)$ are fundamental. The angle $\omega$ at the vertex of $\Sby$ is $\theta$ if $\theta\in [0,1/2]$ or $1-\theta$ if $1-\theta\in [0,1/2]$.

A fundamental sector is defined uniquely up to rotation; let us first rotate it such that $1\in \overline \Disk\setminus \Sby$.  Set $\Sby_-\coloneqq \Lbb^{-1}_\theta(\Sby)$ and set $\Sby_+$ to be the closure of $\Disk\setminus (\Sby\cup \Sby_-)$, see Figure~\ref{Fg:RenOfRotDisc}. Then
\begin{equation}
\label{eq:ap:FRM:DeletSect}
(\Lbb_\theta\mid \Sby_+,\sp  \Lbb^2_\theta\mid \Sby_-)
\end{equation}is the first return of points in $\Sby_-\cup \Sby_+$ back to $\Sby_-\cup \Sby_+$. The \emph{prime renormalization of $\Lbb_\theta$} is the rotation $\Lbb_{\cRRc(\theta)}\colon \overline \Disk\to \overline \Disk $ obtained from~\eqref{eq:ap:FRM:DeletSect} by applying the gluing map 
\[\psi_{\prm}\colon \Sby_-\cup \Sby_+\to \Disk,\sp\sp z\to z^{1/(1-\omega)}.\]

\begin{lem}
\label{lem:ActOfPrimRenorm}
We have 
\begin{equation}
\label{eq:R_prm}
\cRRc(\theta)  = \begin{cases} \frac{\theta}{1-\theta}& \mbox{if }0\le \theta \le \frac{1}{2}, \\
\frac{2\theta-1}{\theta} & \mbox{if }\frac{1}{2}\le \theta\le 1.
\end{cases}
\end{equation}

Present $\theta$ using continued fractions in the following ways
\[\theta=[0;a_1,a_2,\dots]= 1-[0;b_1,b_2,\dots].\]
with $a_i,b_i\in \N_{>0}$.
Then 
\[\cRRc([0;a_1,a_2,\dots])=\begin{cases}
[0;a_1-1,a_2,\dots]& \text{ if }a_1>1, \\
1 - [0;a_2,a_3\dots]& \text{ if }a_1=1,
\end{cases}\]
and 
\[\cRRc(1-[0;b_1,b_2,\dots])=\begin{cases}
1-[0;b_1-1,b_2,\dots]& \text{ if }b_1>1, \\
  [0;b_2,b_3\dots]& \text{ if }b_1=1.
\end{cases}\]
\end{lem}  
\noindent As a consequence, $\theta$ is periodic under $\cRRc$ if and only if there is a $\theta'$ with periodic continued fraction expansion such that $\theta=\cRRc^n(\theta')$ for some $n\ge 0$.
\begin{proof}
Follows by routine calculations. If $\theta\in [0,1/2]$, then projecting $z\to \ee(\theta) z$ by $\psi_\prm$ we obtain
\[z\to \left(\ee(\theta) z ^{1-\theta}\right)^{1/1-\theta}=\ee\left(\frac{\theta}{1-\theta}\right)z.\]
If $\theta\in [1/2,1]$, then projecting $z\to \ee(\theta-1) z$ by $\psi_\prm$ we obtain
\[z\to \left(\ee(\theta-1) z ^{\theta}\right)^{1/\theta}=\ee\left(\frac{\theta-1}{\theta}\right)z.\]
Observe that $\frac{\theta-1}{\theta}\in [-1,0]$; adding $+1$ we obtain  
$\frac{2\theta-1}{\theta}$.

Write \[\theta= \frac{1}{a_1+[0;a_2,a_3,\dots]}.\] and observe that $\theta \in [0,1/2]$ if and only if $a_1>1$. (With the exception $\theta=[0;1,1]$.) If $a_1>1$, then
\[\frac{\theta}{1-\theta}=\frac{1}{a_1+[0;a_2,a_3,\dots] -1}=\cRRc(\theta).\] If $a_1=1$, then
\[\frac{2\theta-1}{\theta}=2-a_1-[0;a_2,a_3,\dots] =\cRRc(\theta).\]
Similarly $\cRRc(1-[0;b_1,b_2,\dots])$ is verified.
\end{proof}

\begin{figure}[t!]
\begin{tikzpicture}[  scale=1.3]

  \begin{scope}[shift={(0,0)}, scale =0.5]

\draw(0.,0.) circle (4.cm);

\draw[red ,shift={(0.,0.)}, fill=red, fill opacity=0.2]
 plot[domain=2.557797598845188:3.8669922442899227,variable=\t]({1.*4.*cos(\t r)+0.*4.*sin(\t r)},{0.*4.*cos(\t r)+1.*4.*sin(\t r)});

\draw[draw opacity=0, red, fill, fill opacity=0.2] (0.,0.)-- (-3.337507402085898,2.204777617135533)
  -- (-2.9929375178397235,-2.653737932484554)--(0.,0.);
  \draw[red ] (0.,0.)-- (-3.337507402085898,2.204777617135533)
   (-2.9929375178397235,-2.653737932484554)--(0.,0.);
    \draw[blue] 
   (1.26, 3.8)--(0,0);
   \draw (-0.42, 2.44) node {$\Sby_-$};
   \draw (-2, 0) node {$\Sby$}; 
    \draw (1.5, -1.5) node {$\Sby_+$}; 
    \draw (-0.62, 2.14) edge[->,bend right] node[left]{$\Lbb_\theta$}(-1.8,0.3);
    \draw (-2,-0.7) node {(delete)};
    \draw  (-1.8,-0.9) edge[->,bend right] node[left]{$\Lbb_\theta$}(-0.32, -2.5);
     \draw (1.94, 0.68) edge[->,bend right] node[right]{$\Lbb_\theta$}(-0., 2.24);XX
\end{scope}

 \begin{scope}[ shift={(5,0)}, scale =0.5, rotate around={180:(0,0)}]
\draw(0.,0.) circle (4.cm);

\begin{scope}[rotate around={7:(0,0)}]
\draw[blue ,shift={(0.,0.)}, fill=blue, fill opacity=0.2]
 plot[domain=2.557797598845188:3.8669922442899227,variable=\t]({1.*4.*cos(\t r)+0.*4.*sin(\t r)},{0.*4.*cos(\t r)+1.*4.*sin(\t r)});

\draw[draw opacity=0, blue, fill, fill opacity=0.2] (0.,0.)-- (-3.337507402085898,2.204777617135533)
  -- (-2.9929375178397235,-2.653737932484554)--(0.,0.);
    \draw[blue ] (0.,0.)-- (-3.337507402085898,2.204777617135533)
   (-2.9929375178397235,-2.653737932484554)--(0.,0.);
\end{scope}  
  
  \node at (-4.3, 0){$1$};

   \draw[blue] (0,0)--(-4, -0);
   \draw (-3, -1.3) node {$\Sbb_+$};
   \draw (-2.92, 0.86) node {$\Sbb_-$};
    \draw (2,-0.7) node {\LARGE (delete)};
      \draw (-0.62, 2.14) edge[->,bend right] node[right]{$\Lbb_\theta$}(-1.8,0.3);

    \draw  (-1.8,-0.9) edge[->,bend right] node[right]{$\Lbb_\theta$}(-0.32, -2.5);
     \draw (1.94, 0.68) edge[->,bend right] node[left]{$\Lbb_\theta$}(-0., 2.24);
     \draw[fill]
     (0.22, -2.6) circle (0.05cm)
     (0.9, -2.5) circle (0.05cm)
     (1.6, -2.1) circle (0.05cm);
\end{scope}    
    
\end{tikzpicture}
\caption{Left: the prime renormalization deletes a fundamental sector $\Sby$ and projects $(\Lbb^2_\theta \mid \Sby_-~,\sp \Lbb_\theta \mid \Sby_+)$ to a new rotation. Right: $\left(\Lbb^{\qq+1}_\theta\mid \Sbb_-~,\sp \Lbb^{\qq}_\theta \mid \Sbb_+\right)$ is the first return map to a fundamental sector $\Sby=\Sbb_-\cup\Sbb_+$}.
\label{Fg:RenOfRotDisc}
\end{figure}

\subsection{Sector renormalization}
\label{ss:RenRotUnitDisc}
A \emph{sector renormalization} $\RR$ of $\Lbb_\theta$ is 
\begin{itemize}
\item a \emph{renormalization} sector $\Sbb$ presented as a union of two subsectors  $\Sbb_-\cup \Sbb_+$ normalized so that $1\in \Sbb_-\cap \Sbb_+$;
\item  a pair of iterates, called a sector \emph{pre-renormalization}, 
\begin{equation}
\label{eq:app:FirstReturnToSector}
\left(\Lbb_\theta ^\aa\mid \Sbb_- , \sp\sp \Lbb_\theta ^\bb\mid \Sbb_+\right)
\end{equation}
realizing the first return of points in $\Sbb_-\cup \Sbb_+$ back to $\Sbb$; and 
\item the gluing map 
\[\psi\colon \Sbb_-\cup \Sbb_+\to \overline \Disk,\sp\sp z\to z^{1/\omega},\]
projecting~\eqref{eq:app:FirstReturnToSector} to a new rotation $\Lbb_\mu$,
where $\omega$ is the angle of $\Sbb$ at $0$.
\end{itemize}
We write $\RR \Lbb_\theta =\Lbb_\mu$, and we call $\aa$ and $\bb$ the \emph{renormalization return times}. We allow one of the sectors $\Sbb_\pm$ to degenerate, but not both. Note that the assumption $1\in \Sbb_-\cap \Sbb_+$; can always be achieved using rotation.

The prime renormalization is an example of a sector renormalization.

Suppose two sector renormalizations $\RR_1(\Lbb_\theta)=\Lbb_\mu$ and $\RR_2(\Lbb_\mu)=\Lbb_\nu$ are given. The \emph{composition} $\RR_2\circ \RR_1(\Lbb_\theta)=\Lbb_\nu$ is obtained by lifting the pre-renormalization of $\RR_2$ to the dynamical plane of $\Lbb_\theta$.

\begin{lem}
\label{lem:SectRen is prime power}
A sector renormalization is an iteration of the prime renormalization.
\end{lem}
\begin{proof}
Suppose $\RR$ is a sector renormalization with renormalization return times $\aa$ and $\bb$ as above. Proceed by induction on $\aa+\bb$. If $\aa+\bb=3$, then $\RR$ is the prime renormalization. Otherwise, we project the pre-renormalization of $\RR$ to the dynamical plane of $\RRc(\Lbb_\theta)$ and obtain the new sector renormalization $\RR'$ of $\RRc(\Lbb_\theta)$ so that 
\[\RR' \circ \RRc(\Lbb_\theta)=\RR(\Lbb_\theta).\]
The renormalization return times $a',b'$ of $\RR'$ satisfy $a'+b'< a+b$.
\end{proof}

Consider again the fundamental sector $\Sby$ bounded by $\Ibb$ and $\Lbb_\theta(\Ibb)$. There is a unique $\aa>0$
 such that $\Lbb^{-\aa}(\Ibb)\subset \Sby$. Up to rotation, we can assume that $\Lbb^{-\aa}(\Ibb)$ lands at $1$. We define $\Sbb_+$ to be the subsector of $\Sby$ bounded $\Ibb$ and $\Lbb^{-\aa}(\Ibb)$ and we define  $\Sbb_-$ to be the subsector of $\Sby$ bounded $\Lbb(\Ibb)$ and $\Lbb^{-\aa}(\Ibb)$. Then
 \[\left(\Lbb_\theta ^\aa\mid \Sbb_- , \sp\sp \Lbb_\theta ^{\aa+1}\mid \Sbb_+\right)\]
is a sector pre-renormalization, called the \emph{first return to the fundamental sector}, see Figure~\ref{Fg:RenOfRotDisc}. We denote by $\RR_{\fast}$ the associated sector renormalization and we write $\mu=\cRR_\fast(\theta)$ if $\RR_{\fast}(\Lbb_\theta)=\Lbb_\mu$. 

By Lemma~\ref{lem:SectRen is prime power}, for every $\theta\not=0$ there is a unique $\nn(\theta)$ such that $\cRR_\fast(\theta)=\cRRc^{\nn(\theta)}(\theta)$. We note that if $\theta\in\{1/m, 1-1/m\}$ with $m>1$, then $\nn(\theta)=m-1$. (In this case the sector $\Sbb_-$ is degenerate.)

\subsection{Renormalization triangulation}
Given a sector pre-renormalization~\eqref{eq:app:FirstReturnToSector}, the set of sectors \[\bigcup_{i=0}^{\aa-1} \Lbb_\theta (\Sbb_-)\bigcup_{i=0}^{\bb-1} \Lbb_\theta(\Sbb_+) \]
is called a \emph{renormalization triangulation} of $\Disk$. Alternatively, consider the associated renormalization $\Lbb_{\mu}=\RR(\Lbb_\theta)$. The internal rays towards $1$ and $\Lbb_{\mu}(1)$ split $\Disk$ into two closed sectors $\Sbz_0$ and $\Sbz_{1}$. We call $\{\Sbz_-,  \Sbz_{+}\}$ the \emph{basic triangulation of $\Lbb_\mu$}. Lifting the sectors $\Sbz_-,  \Sbz_{+}$ via the gluing map, and spreading them dynamically we obtain the renormalization triangulation. We also say that  the renormalization triangulation is the \emph{full lift} of the basic triangulation.   

Let $\Theta_N$ be the set of angles $\theta$ such that $\theta= [0;a_1,a_2,\dots]$ with $|a_i| \le N$ or $\theta= 1-[0;a_1,a_2,\dots]$ with $|a_i| \le N$. By Lemmas~\ref{lem:ActOfPrimRenorm} and~\ref{lem:SectRen is prime power}, the set $\Theta_N$ is invariant under any sector renormalization.

\begin{lem}
\label{lem:ap:CompTiling}
For every $N$ there is a $t>1$ with the following property. Consider the renormalization triangulation associated with some sector renormalization of $\Lbb_\theta$, where $\theta\in \Theta_N$. Then any two triangles have comparable angles at $0$: the ratio of the angles is between $1/t$ and $t$.
\end{lem}
\begin{proof}
There is a neighborhood $U$ of $1$ such that for all $\theta\in \Theta_N$ we have $\Lbb_\theta(1)\not\in U$. Therefore, both sectors in the basic triangulation have comparable angles at $0$ uniformly on $\theta\in \Theta_N$. Since a renormalization triangulation is the full lift of a basic triangulation, the lemma is proven.   
\end{proof}

\subsection{Periodic case}
\label{sss:ap:SectRen:PerCase}
It follows from Lemmas~\ref{lem:ActOfPrimRenorm} and~\ref{lem:SectRen is prime power} that $\Lbb_\theta$ is a fixed point of some sector renormalization if and only if $\theta\in  \Theta_\per$. Suppose  $\theta\in \Theta_\per$ and choose a sector renormalization $\RR_1$ such that $\RR_1(\Lbb_\theta)=\Lbb_\theta$. Write $\RR_n\coloneqq \RR_1^{n}$ and denote by $\aa_n,\bb_n$ and $\psi_n$ the renormalization return times and the gluing map of $\RR_n$. Then $\psi_n=\psi_1^n$ and there is a matrix $\mathbb M$ with positive entries such that
\begin{equation}
 \label{eq:RernReturns}\left(\begin{matrix}
         \aa_n \\ \bb_n
        \end{matrix}\right) =\mathbb M^{n-1} \left(\begin{matrix}
         \aa_1 \\ \bb_1
        \end{matrix}\right). 
\end{equation}
As a consequence, $\aa_n, \bb_n$ have exponentially growth with the same exponent. 

We also note that 
\begin{equation}
\label{eq:app:aa1 bb1}
\aa_1,\bb_1\ge 2
\end{equation}
because $\RR_1=\RRc^t$ with $t>1$.

\section{Lifting of curves under anti-renormalization}
\label{s:SectRenAndAntiren}

In this appendix we give a sufficient condition for liftability of arcs under a sector  anti-renormalization. This implies that the sector anti-renormalization is robust with respect to a particular choice of cutting arcs, see Theorem~\ref{thm:RobAntiRen}.

\subsection{Robustness of anti-renormalization}
\label{ss:A1}

Consider a closed pointed topological disk $(W, 0)$ and let $U,V$ be two closed topological subdisks of $W$ such that $0\in \intr(U\cap V)$. 
A homeomorphism $f\colon U\to V$ fixing $0$ is called a \emph{partial homeomorphism} of $(W,0)$ and is denoted by $f\colon W\dashrightarrow W$ or
$f\colon (W,0)\dashrightarrow (W,0)$. If $U=V=W$, then $f$ is an actual self-homeomorphism of $(W,0)$.

\subsubsection{Leaves over \boldmath$f\colon (W,0)\dashrightarrow (W,0)$}

Let $\gamma_0, \gamma_1$ be two simple arcs connecting $0$ to points in  $\partial W$ such that $\gamma_0$ and $\gamma_1$ are disjoint except for $0$ and such that $\gamma_1$ is the image of $\gamma_0$ in the following sense: $\gamma'_0\coloneqq \gamma_0\cap U$ and  $\gamma'_1\coloneqq \gamma_1\cap V$ are simple closed curves such that $f$ maps $\gamma'_0$ to $\gamma'_1$. Such pair $\gamma_0, \gamma_1$ is called \emph{dividing}. Then $\gamma_0 \cup \gamma_1$ splits $W$ into two closed sectors $\A$ and $\B$ denoted so that $\intr \A, \gamma_1,\intr \B,\gamma_0$ are clockwise oriented around $0$, see the left-hand of~Figure~\ref{Fig:f and its 1 3 anti ren}. We say that $\gamma_0=\ell(\A)=\rho(\B)$ is the \emph{left boundary of $\A$} and the \emph{right boundary of $\B$} and we say that $\gamma_1=\rho(\A)=\ell(\B)$ is the \emph{right boundary of $\A$} and \emph{the left boundary of $\B$}. 

Let $X,Y$ be topological spaces and let $g\colon X\dashrightarrow Y$ be a partially defined continuous map. We define
\[X\sqcup_g Y\coloneqq  X\sqcup Y/ (\Dom g \ni x\sim g(x)\in \Im g).\] 

Consider a (finite or infinite) sequence $(S_k)_k$, where each $S_k$ is a copy of either $\A$ or $\B$. Define the partial map $g_k\colon \rho(S_k)\dashrightarrow \ell(S_{k+1})$ by
\begin{equation}
\label{eq:deff:g_k}
    g_k\coloneqq \begin{cases}
    \id \colon \gamma'_1\to \gamma'_1& \text{ if } (S_k,S_{k+1})\cong (\A,\B),\\
        \id \colon \gamma'_0\to \gamma'_0& \text{ if } (S_k,S_{k+1})\cong (\B,\A),\\
        f^{-1} \colon \gamma'_1\to \gamma'_0& \text{ if } (S_k,S_{k+1})\cong (\A,\A),\\
          f \colon \gamma'_0\to \gamma'_1& \text{ if } (S_k,S_{k+1})\cong (\B,\B).
    \end{cases}
\end{equation} The \emph{dynamical gluing of  $(S_k)_k$} is \[\dots S_{k-1} \sqcup_{g_{k-1}} S_k \sqcup_{g_k} S_{k+1}\sqcup_{g_{k+1}} \dots.\] The \emph{jump $\iota(k)$ from $S_{k} $ to $S_{k+1}$} is set to be $0,0,-1,1$ if $(S_k,S_{k+1})$ is a copy of $(\A,\B),$ $(\B,\A),$ $(\A,\A),$ $(\B,\B)$ respectively.

  For a sequence $\seq=(a_i)_{i\in I}$ we denote by $\seq[i]$ the $i$-th element in $\seq$; i.e.~$\seq[i]=a_i$.
\begin{defn}[Leaves of $f\colon W\dashrightarrow W$]
\label{defn:leaf over f}
Suppose $\seq \in \{\A,\B\}^{\Z}$. Set $W_{\seq}[i]$ to be a copy of the closed sector $\seq[i]$. The \emph{leaf $W_{\seq}$} is the surface obtained by the dynamical gluing of $(W_{\seq}[i])_{i\in \Z}$.

The \emph{projection} $\pi\colon W_{\seq} \to W$ maps each $\intr W_{\seq}[i] \cup \rho(W_\seq[i])$ to $\intr \seq[i]\cup \rho(\seq[i])$. By $\pi^{-1}_{\seq,i}\colon \intr \seq[i]\cup \rho(\seq[i]) \to \intr W_{\seq}[i] \cup \rho(W_\seq[i])$ we denote the corresponding inverse branch.
\end{defn}

 Note that if $\seq[i]=\seq[i+1]$, then $\pi$ is discontinuous at $W_\seq[i]\cap W_\seq[i+1]$. As $z$ approaches $W_\seq[i]\cap W_\seq[i+1]$ from $\intr W_\seq[i]$, respectively $\intr W_\seq[i+1]$, its image $\pi(z)$ approaches $\rho(\seq[i])$, respectively $\ell(\seq[i+1])\not = \rho(\seq[i])$.

For every $\seq$, there is a unique point $\widetilde 0\in W_\seq$ such that $\pi(\widetilde 0)=0$. By construction, $W_\seq\setminus \{\widetilde 0\}$ is topologically a closed half-plane.
 
For $J\subseteq \Z$ we write $\displaystyle W_\seq[J]= \bigcup _{j\in J} W_\seq[j]$. To keep notation simple, we write $W_\seq[\ge i] =W_\seq[\{k\mid k\ge i\}]$ and similar for ``$>$'',``$\le$'', ``$<$''.

\subsubsection{Lifts of curves}
\label{sss:lifts of curves}
Let $\alpha\colon [0,1]\to W\setminus \{0\}$ be a curve in $W$. A \emph{lift of $\alpha$ to $W_{\seq}$} is a curve $\widetilde \alpha\colon [0,1]\to W_{\seq}$ such that
\begin{itemize}
\item for every $t\in[0,1]$, there is an $n(t)\in \Z$ such that $\pi(\widetilde \alpha(t))=f^ {n(t)}(\alpha(t))$;
\item $n(0)=0$;  
\item $n(t)$ is constant for all $t$ for which $\widetilde \alpha(t)$ is within some $\intr W_{\seq}[i] \cup \rho(W_\seq[i])$; and 
\item if $\widetilde \alpha(t')\in \intr W_{\seq}[i]$ while $\widetilde\alpha(t)\in \intr W_{\seq}[i+1]$, then $n(t)-n(t')$ is equal to the jump from $W_{\seq}[i]$ to $W_{\seq}[i+1]$.
\end{itemize}
In other words, whenever $\alpha$ crosses the boundary of $\seq[i]$, the lift of $\alpha$ is adjusted to respect the dynamical gluing. Similarly is defined a lift of a curve parametrized by $[0,1)$. Note that $\pi(\widetilde \alpha)$ is, in general, discontinuous.

For every curve $\alpha$ as above, there is at most one lift  of $\alpha$ starting at a given preimage of $\alpha(0)$ under $\pi\colon W_\seq\to W$. It is easy to see that there is an $\varepsilon>0$ such that all lifts (specified by the starting points) of $\alpha\colon [0,\varepsilon]\to W$ exist, and thus unique. The main question we address is the existence of the global lifts.

If $\alpha\colon [0,1)\to W\setminus \{0\}$ is such that $\alpha(1)=\lim _{t\to 1}\alpha (t)=0$, then we say that a lift $\widetilde \alpha$ of $\alpha$ \emph{lands at $\widetilde 0$} if $\pi(\widetilde \alpha (t))\to 0$ as $t\to 1$.

\begin{figure}
\centering{\begin{tikzpicture}
\draw[red] (0,0) circle (2.5cm);

 \draw  (0,0) -- (2.5,0);
 \node at (2.15,0.15) {$\gamma_0$};
\draw[rotate around={155:(0,0)}] (0,0) -- (2.5,0);
\draw (-1, 1)node {$\gamma_1=f(\gamma_0)$};

\draw (0.7,1.8) node {$\B$};
\draw (0,-1.25) node {$\A$}; 
 

\begin{scope}[shift={(6,0)}]

\draw[red] (0,0) circle (2.5cm);

\draw  (0,0) -- (2.5,0); 
\draw[rotate around={90:(0,0)}] (0,0) -- (2.5,0);
\draw[rotate around={225:(0,0)}] (0,0) -- (2.5,0);
 
\draw (1.1,1.1) node {$\B$};
\draw (0.8,-1.65) node {$\A$}; 

\draw (-1.75,0.7) node {$\A$}; 

\draw[blue] (0.6,0.8) edge[<-, bend right] node[above]{$f$} (-0.8,0.5); 
\draw[blue] (-0.9,0.1) edge[<-, bend right] node[left]{$\id$} (0.1,-1); 
 
\draw[blue] (0.7,-1.) edge[<-, bend right] node[right]{$f$} (0.9,0.6); 
 
\node at (2.15,0.15) {$\gamma_0$};
\node at (2.15,-0.2) {$\gamma_0$};

\node at (0.2,2.15) {$\gamma_1$};
\node at (-0.15,2.15) {$\gamma_1$};

\node at (-1.65,-1.4) {$\gamma_0$};
\node at (-1.3,-1.65) {$\gamma_1$};

\end{scope}

\end{tikzpicture}}
\caption{Left: a homeomorphism $f\colon W\to W$ and a diving pair $\gamma_0,\gamma_1$. Right: the $1/3$ anti-renormalization of $f$ (with respect to the clockwise orientation).}
\label{Fig:f and its 1 3 anti ren}
\end{figure}
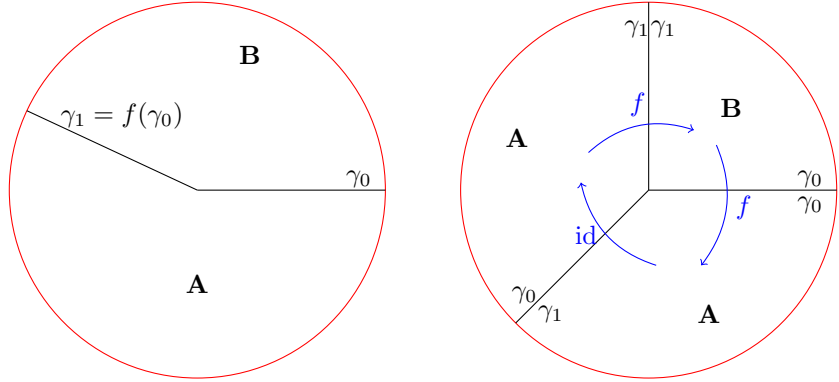

\subsubsection{Anti-renormalizations}
\label{sss:anti renorm} We will now show that for every $\pp/\qq$ there is a unique anti-renormalization with rotation number $\pp/\qq$.
\begin{lem}
\label{lem:seq:p:q}
For every $\qq\in \N_{>2}$ and every $\pp\in \{1,2,\dots ,\qq-1\}$ coprime with $\qq$ there exists a unique $\qq$-periodic sequence $\seq\in \{\A,\B\}^\Z$ such that
\begin{itemize}
\item $\seq[0]=\A$ and $\seq[-1]=\B$; 
\item for every $j\in \Z$ with $(j\mod \qq)\not\in \{-\pp ,-\pp-1\}$, we have $\seq[j+\pp]=\seq[j]$;
\item $\seq[-\pp-1]=\A$ and $\seq[-\pp]=\B$.
\end{itemize}
\end{lem}
\begin{proof} 
 Since $\pp$ and $\qq$ are coprime, there are unique $\aa,\bb\in\{ 1,2,\dots,\qq-1\}$ such that 
\begin{align*}
\pp \aa=-1\sp \mod\sp\qq, \\
\pp \bb=1\sp \mod\sp\qq.
\end{align*}
Note that $\aa+\bb=\qq$. We have:
\begin{itemize}
\item $\seq[i\pp +j\qq]=\A$ for all $i\in \{0,1,\dots, \aa-1\}$ and all $j\in \Z$; and
\item $\seq[-1+i\pp +j\qq]=\B$ for all $i\in \{0,\dots, \bb-1\}$ and all $j\in \Z$.
\end{itemize}
\end{proof}
The numbers $\aa, \bb$ appearing in the proof of Lemma~\ref{lem:seq:p:q} are called the \emph{renormalization return times}.

For a sequence $\seq$ as in Lemma~\ref{lem:seq:p:q}, let 
\[\seq/\qq \in \{\A,\B\}^{\Z/\qq\Z},\sp\sp (\seq/\qq)[i] \coloneqq \seq[i+\Z\qq]\]
be the quotient sequence, and let $W_{\seq/\qq}$ be the quotient of the leaf $W_{\seq}$ by identifying each $W_\seq[k]$ with $W_\seq[k+\qq]$. We denote by $\pi\colon W_{\seq/\qq}\to W$ the natural projection. Then the \emph{$\pp/\qq$-anti-renormalization $f_{-1}\colon W_{\seq/\qq}\dashrightarrow W_{\seq/\qq}$}  is defined as follows (see Figure~\ref{Fig:f and its 1 3 anti ren}):
\begin{itemize}
\item for every $j\not\in \{-\pp-1,-\pp\}$, the map $f_{-1}\colon W_{\seq/\qq}[j] \to W_{\seq/\qq}[j+\pp]$ is the natural isomorphism;
\item the map $f_{-1}\colon W_{\seq/\qq}[-\pp-1,-\pp]\dashrightarrow W_{\seq/\qq}[-1,0]$ is ${f\colon W\setminus \gamma_0\dashrightarrow  W\setminus \gamma_1}$.
\end{itemize}
Note that $\pp/\qq$ is the clockwise rotation number.

By construction, $(f_{-1}^\aa\mid W_{\seq/\qq}[0],\sp{ f_{-1}^\bb\mid W_{\seq/\qq}[-1])}$ is the first return 
of $f_{-1}$ back to $W_{\seq/\qq}[-1,0]$. After appropriate gluing of arcs in $\partial \left(W_{\seq/\qq}[-1,0]\right)$, the map ${(f_{-1}^\aa\mid W_{\seq/\qq}[0],\sp{ f_{-1}^\bb\mid W_{\seq/\qq}[-1])}}$ is $f\colon W\dashrightarrow W$.

Denote by 
\begin{equation}
\label{eq:ap:gammas seq qq} 
\gamma_0^{\seq/\qq},\sp\sp \gamma_p^{\seq/\qq}
\end{equation}
 the left boundaries of $W_{\seq/\qq}[0]$ and $W_{\seq/\qq}[p]$ respectively. Then $\gamma_0^{\seq/\qq}, \gamma_p^{\seq/\qq}$ is a dividing pair for   $f_{-1}\colon W_{\seq/\qq}\dashrightarrow W_{\seq/\qq}$ and the anti-renormalization procedure can be iterated.

Let $\beta$ be a curve in $W$ and let $\widetilde \beta$ be a lift of $\beta$ to $W_\seq$. The image of $\widetilde \beta$ in $W_{\seq/\qq}\simeq W_\seq/\sim$ is called a \emph{lift of $\beta$ to $W_{\seq/\qq}$}. For example, $\gamma_0^{\seq/\qq}$ is a lift of $\gamma_0$.

\subsubsection{Prime anti-renormalization}
\label{sss:PrimeAntiRen}
 The $1/3$ and $2/3$-anti-renormalizations are called \emph{prime}. It is easy to check that 
\begin{itemize}
\item if $\pp/\qq=1/3$, then $\seq/\qq=(\A,\A,\B)$;
\item if $\pp/\qq=2/3$, then $\seq/\qq=(\A,\B,\B)$.
\end{itemize}

\begin{lem}[Compare with Lemma~\ref{lem:SectRen is prime power}]
\label{lem:AntiRen is prime power}
Any anti-renormalization is an iteration of prime anti-renormalizations.
\end{lem}
\begin{proof}
We proceed by induction on $\qq$. Assume $\qq>3$, define $\pp'/\qq'\coloneqq \cRRc(\pp/\qq)$ (see~\eqref{eq:R_prm}), and observe that $\qq'<\qq$.
\begin{itemize}
\item  If $0<\pp<\qq/2$, then the $\pp/\qq$-anti-renormalization is the $1/3$-anti-renormalization of the $\pp'/\qq'$-anti-renormalization.
\item  If $\qq/2<\pp<\qq$, then the $\pp/\qq$ anti-renormalization is the $2/3$-anti-renormalization of the $\pp'/\qq'$-anti-renormalization.
\end{itemize}
\end{proof}

Denote by $\seqo\coloneqq (\A,\B)^{\Z}$ the sequence in $\{\A,\B\}^{\Z}$ with even entries equal to $\A$ and odd entries equal to $\B$. Simplifying notations, we write $W_{\seqo}=W_{\oseq}$.

Suppose that $f\colon W\to W$ is a homeomorphism. In this case anti-renormalizations of $f$ can be defined canonically (i.e.~independent of the choice of $\gamma_0,\gamma_1$) as follows. Observe first that
\begin{equation}
\label{eq:wWtoW0}
\pi\colon W_{\oseq}\setminus \{\widetilde 0\} \to W\setminus \{0\}
\end{equation}
is a universal cover. Let $( \wgamma_i^\oseq \subset W_\oseq)_{i\in \Z}$ be all the lifts of $\gamma_0$ and $\gamma_1$ enumerated from left to right such that $W_\oseq[i]$ is between $\wgamma_{i}^\oseq$ and $\wgamma_{i+1}^\oseq$; in particular, $\wgamma_i^\oseq$ is a lift of $\gamma_{i \mod 2}$. Let \[f_-\ ,\   f_+ \colon  W_\oseq \to  W_\oseq\] 
 be the lifts of $f\colon  W \to  W$ specified so that \[f_-(\wgamma_{2 i}^\oseq)=\wgamma_{2 i-1}^\oseq\sp \text{ and }\sp f_+(\wgamma_{2 i}^\oseq)=\wgamma_{2 i+ 1}^\oseq. \]
Observe that $f_-$ and $f_+$ commute. Write 
\[\tau \coloneqq f_-^{-1}\circ f_+ \colon  W_\oseq \to  W_\oseq; \]then $\tau\mid (W_{\oseq}\setminus \{\widetilde 0\})$ is a deck transformation of $\widetilde W$ and we can rewrite~\eqref{eq:wWtoW0} as \[W\setminus \{0\} \simeq \left(W_{\oseq}\setminus \{\widetilde 0\}\right)/  \langle \tau \rangle.\] We also write $W \simeq W_\oseq/  \langle \tau \rangle$.

\begin{lem}[The $1/3$-anti-renormalization]
\label{lem:1 3 AntiRen}
Suppose $f\colon W\to W$ is a self-homeomorphism. Set
\begin{align*}
f_{-1,+}\coloneqq& f_+, \\
f_{-1,-}\coloneqq&\tau^{-1}=f_{-}\circ f_+^{-1},\\
\tau_{-1}\coloneqq&f_{-1,-}^{-1}\circ f_{-1,+}= f^{-1}_{-}\circ f_+^2,
\end{align*}
Then  $\tau_{-1}$ acts properly discontinuously on $W_{\oseq}\setminus \{\widetilde 0\}$. We view \[\left(W_{\oseq}\setminus \{\widetilde 0\}\right)/\langle\tau_{-1} \rangle\] as a punctured closed topological disk and we view $ W_\oseq /\langle \tau_{-1} \rangle$ as a closed topological disk.

Let $f_{-1}\colon W_{-1}\to W_{-1}$ be the $1/3$-anti-renormalization of $f$. Then $f_{-1}$ is  conjugate to 
 \[ f_{-1,-}^{}/\langle \tau_{-1} \rangle = f_{-1,+}/\langle \tau_{-1} \rangle\colon W_\oseq /\langle \tau_{-1} \rangle\to W_\oseq /\langle \tau_{-1} \rangle  \]
 by the conjugacy \[h\colon W_{-1}\to W_\oseq [0,1,2]/\langle\tau_{-1}\rangle \] mapping 
 \[W_{-1}[0],\sp W_{-1}[1],\sp  W_{-1}[2]\]
respectively to  \[ {W_\oseq [2]/\langle \tau_{-1}\rangle},  \sp {W_\oseq[0]/\langle \tau_{-1}\rangle}, \sp{W_\oseq[1]/\langle \tau_{-1}\rangle}\]  which are copies of $\A,\A,\B$.
\end{lem}

\begin{figure}
\centering{\begin{tikzpicture}
\draw (2,3)--(2,-3); 
\draw (4,3)--(4,-3); 
\draw (6,3)--(6,-3); 
\draw (8,3)--(8,-3);
\draw (10,3)--(10,-3);
\draw (12,3)--(12,-3);

\draw (3,2.1) node{$W_\oseq[-1]$}; 
\draw (5,2.1) node{$W_\oseq[0]$}; 
\draw (7,2.1) node{$ W_\oseq[1]$}; 
\draw (9,2.1) node{$ W_\oseq[2]$}; 
\draw (11,2.1) node{$ W_\oseq[3]$}; 

\draw[blue] (5,0.5) edge [bend left,<-]node[above]{$f_{-1,-}$} (9,0.6);

\draw[blue] (4.5,-1) edge [bend left,->]node[above ]{$f_{-1,+}$} (6.5,-1);
\draw[blue] (6.5,-2.2) edge [bend left,->]node[above ]{$f_{-1,+}$} (8.5,-2.2);

\draw[red] (1.8,3)--(12.1,3);

\draw (2.35,2.6) node{$ \wgamma^\oseq_{-1}$}; 
\draw (4.3,2.6) node{$ \wgamma^\oseq_0$}; 
\draw (6.3,2.6)  node{$ \wgamma^\oseq_1$}; 
\draw (8.3,2.6) node{$ \wgamma^\oseq_2$};
\draw (10.3,2.6)  node{$ \wgamma^\oseq_3$};

\draw[blue] (7.8,3) .. controls (7.8,1.5)  .. (8.1,0)
..controls (8.3, -1.5)..(8.3,-3);
\draw[blue] (8.7,0) node{$ f_+(\wgamma^\oseq_1)$};


\end{tikzpicture}}
\caption{Illustration to Lemma~\ref{lem:AntiRen is prime power}: ${f_{-1,-}\colon W_\oseq[2]\to W_\oseq[0]}$ and $f_{-1,+}\colon W_\oseq[0,1]\to W_\oseq[1,2]$ become the $1/3$-anti-renormalization of $f$ after gluing $\wgamma^\oseq_0$ and $\wgamma^\oseq_3$; see also Figure~\ref{Fig:f and its 1 3 anti ren}.} 
\label{Fig:FundReg:F}
\end{figure}

\begin{proof}
Clearly, $W_\oseq[0,1,2]$ is a fundamental domain for $\tau_{-1}$. It is easy to see (see Figure~\ref{Fig:FundReg:F}) that $h$ identifies
\begin{itemize}
\item $f_{-1} \colon W_{-1} [0]\to W_{-1} [1]$ (which is $\id \colon \A \to \A$)  with \[f_{-1,-}=\tau^{-1}\colon W_\oseq[2]\to W_\oseq[0].\]

\item $f_{-1} \colon W_{-1} [1,2]\to W_{-1} [2,0]$   (which is $f\colon W\setminus \gamma_0\to W\setminus \gamma_1$) with
\[ f_{-1,+}\colon W_\oseq[0,1]\to W_\oseq[1,2] .\] 
\end{itemize}
\end{proof}

\begin{rem}
\label{rem:lem:1 3 AntiRen} 

The proof of Lemma~\ref{lem:1 3 AntiRen} shows also that 
$h$ is uniquely characterized by the following properties:
\begin{itemize}
\item $h$ maps $\gamma_0^{\seq/3}$ to $\wgamma_2^\oseq /\langle\tau\rangle$ (see~\eqref{eq:ap:gammas seq qq});
\item if $\ell\subset W\setminus \{0\}$ is a curve starting at $\gamma_0$ and $\widetilde \ell\subset W_{-1} $ is the unique lift of $\ell$ starting at some point of $\gamma_0^{\seq/3}$, then $h(\widetilde \ell /\langle\tau_{-1} \rangle)\subset W_\oseq$ is the unique lift of $\ell$ starting at $\wgamma_2^\oseq$.   
\end{itemize}
\end{rem}

Similar to Lemma~\ref{lem:1 3 AntiRen} we have
\begin{lem}[The $2/3$-anti-renormalization]
\label{lem:2 3 AntiRen} Suppose $f\colon W\to W$ is a self-homeomorphism.
Set
\begin{align*}
f_{-1,+}\coloneqq& \tau =f_+\circ f_-^{-1}, \\
f_{-1,-}\coloneqq&f_{-},\\
\tau_{-1}\coloneqq&f_{-1,-}^{-1}\circ f_{-1,+}= f^{-2}_{-}\circ f_+,
\end{align*}
Then  $\tau_{-1}$ acts properly discontinuously on $W_{\oseq}\setminus \{\widetilde 0\}$. We view \[\left(W_{\oseq}\setminus \{\widetilde 0\}\right)/\langle\tau_{-1} \rangle\] as a punctured closed topological disk and we view $ W_\oseq /\langle \tau_{-1} \rangle$ as a closed topological disk.

Let $f_{-1}\colon W_{-1}\to W_{-1}$ be the $2/3$-anti-renormalization of $f$. Then $f_{-1}$ is  conjugate to 
 \[ f_{-1,-}^{}/\langle \tau_{-1} \rangle = f_{-1,+}/\langle \tau_{-1} \rangle\colon W_\oseq /\langle \tau_{-1} \rangle\to W_\oseq /\langle \tau_{-1} \rangle  \]
 by the conjugacy \[h\colon W_{-1}\to W_\oseq [-1,0,1]/\langle\tau_{-1}\rangle \] mapping 
 \[W_{-1}[0],\sp W_{-1}[1],\sp  W_{-1}[2]\] respectively to \[ {W_\oseq [0]/\langle \tau_{-1}\rangle},  \sp {W_\oseq[1]/\langle \tau_{-1}\rangle}, \sp{W_\oseq[-1]/\langle \tau_{-1}\rangle}\] which are the copies of $\A,\B,\B$.
\qed
\end{lem}
\subsubsection{Fences}
\label{sss:fences} Consider again a partial homeomorphism $f\colon W\dashrightarrow W$ and let $\seq$ be an anti-renormalization sequence from Lemma~\ref{lem:seq:p:q}. We view $W$ as a subset of $\C$.

A \emph{fence} is a simple closed curve $\wall\subset \Dom f\cap \Im f$ such that
\begin{itemize}
\item $0$ is in the bounded component $\inn$ of $\C\setminus \wall$; and
\item $\wall$ intersects $\gamma_0$ at a single point $x$ and  $\wall$ intersects $\gamma_1$ at $f(x)$.
\end{itemize}

Let $f_{-1}\colon W_{\seq/\qq}\dashrightarrow W_{\seq/\qq}$ be an anti-renormalization of $f_{-1}$ as in~\S\ref{sss:anti renorm}. We denote by $\wall_\seq$ the lift of $\wall$ to $W_\seq$ and we denote by $\wall_{\seq/\qq}$ the projection of $\wall_\seq$ to $W_{\seq/\qq}$.

\begin{lem}
\label{lem:lift of  fence}
The curve $\wall_{\seq/\qq}$ is again a fence respecting $\gamma_0^{\seq/\qq}, \gamma_p^{\seq/\qq}$; see~\eqref{eq:ap:gammas seq qq}.
\end{lem}
\begin{proof}
Every $\wall_{\seq/\qq}\cap W_{\seq/\qq}[i]$ is an arc connecting a point on the left boundary of $W_{\seq/\qq}[i]$ to a point on the right boundary of $W_{\seq/\qq}[i]$. Moreover,  $\wall_{\seq/\qq}\cap W_{\seq/\qq}[i]$ meets  $\wall_{\seq/\qq}\cap W_{\seq/\qq}[i+1]$ because $g_k\colon \rho(S_i)\dashrightarrow \ell(S_{i+1})$ (see~\eqref{eq:deff:g_k}) respects the intersection of $\wall$ with $\gamma_0,\gamma_1$.
\end{proof}

\subsubsection{Robustness of anti-renormalization}
\begin{thm}
\label{thm:RobAntiRen}
Let $f\colon W\dashrightarrow W$ be a partial homeomorphism, $\gamma_0,\gamma_1\subset W$ be a dividing pair of arcs, $Q\subset \Dom f$ be a fence respecting $\gamma_0,\gamma_1$ and enclosing $\inn\ni 0$, and let \[f_{-1}\colon W_{-1}\dashrightarrow W_{-1}\] be the $\pp/\qq$-anti-renormalization of $f$; see~\S\ref{sss:anti renorm}. 

Assume that $\gamma^\new_0,\gamma^\new_1$ is another pair of dividing arcs such that  ${\gamma^\new_0\setminus \inn},{\gamma^\new_1\setminus \inn}$ coincides with $\gamma_0\setminus \inn,\gamma_1\setminus \inn$. Denote by \[f_{-1,\new}\colon W_{-1,\new }\dashrightarrow W_{-1,\new }\] the $\pp/\qq$-anti-renormalization of $f$ relative to the pair $\gamma^\new_0,\gamma_1^\new$. Then $f_{-1}$ and $f_{-1,\new}$ are naturally conjugate by $h\colon W_{-1}\to W_{-1,\new }$ uniquely specified by the following properties:
\begin{enumerate}
\item $\pi\circ h(z)=\pi(z)$ for every $z\in W_{-1}\setminus\inn_{-1}$, where $\inn_{-1}$ is the topological disk enclosed by $\wall_{-1}$, see Lemma~\ref{lem:lift of  fence}; and\label{eq:1:thm:RobAntiRen}
\item if $\widetilde \beta\subset W_{-1}$ is a lift of a curve $\beta\subset W$, then $h(\widetilde \beta)$ is a lift of $\beta$ to $W_{-1,\new}$.\label{eq:2:thm:RobAntiRen}
\end{enumerate}
\end{thm}
\begin{proof}
Since the pair ${\gamma^\new_0\setminus \inn},{\gamma^\new_1\setminus \inn}$ coincides with $\gamma_0\setminus \inn,\gamma_1\setminus \inn$, Condition~\eqref{eq:1:thm:RobAntiRen} uniquely specifies $h\mid W_{-1}\setminus\inn_{-1}$. 

Let us now extend $f\colon W\dashrightarrow W$ to a homeomorphism $f\colon W\to W$ mapping $\gamma_0$ to $\gamma_1$. The extension changes $f_{-1}\mid W_{-1}\setminus \inn_{-1}$ and $f_{-1, \new}\mid W_{-1,\new }\setminus \inn_{-1,\new }$ but does not affect $f_{-1}\mid  \inn_{-1}$ and $f_{-1, \new}\mid  \inn_{-1,\new }$. Therefore, it is sufficient to prove the theorem under the assumption that $f\colon W\to W$ is a homeomorphism.

Since every anti-renormalization is an iteration of prime anti-renormalizations (see~Lemma~\ref{lem:AntiRen is prime power}), we can further assume that $f_{-1}$ and $f_{-1,\new}$ are prime anti-renormalizations. By Lemmas~\ref{lem:1 3 AntiRen} and~\ref{lem:2 3 AntiRen} both $f_{-1}$ and $f_{-1,\new}$ are naturally conjugate to 
 \[ f_{-1,-}^{}/\langle \tau_{-1} \rangle = f_{-1,+}/\langle \tau_{-1} \rangle\colon W_\oseq /\langle \tau_{-1} \rangle\to W_\oseq /\langle \tau_{-1} \rangle , \] 
 which is independent on the choice of $\gamma_0,\gamma_1$. It remains to observe that the conjugacy between $f_{-1}$ and $f_{-1,\new}$ satisfies Condition~\ref{eq:2:thm:RobAntiRen} -- see Remark~\ref{rem:lem:1 3 AntiRen}. 
\end{proof}
\begin{cor}[Lifting condition]
\label{cor:LiftCond}
The curves $\gamma_{0,\new}$ and $\gamma_{1,\new}$ have unique lifts
\begin{equation}
\label{eq:cor:LiftCond}
h^{-1}\left(\gamma_0^{\seq/\qq,\new}\right),\sp\sp h^{-1}\left(\gamma_p^{\seq/\qq,\new}\right)\subset W_{-1}
\end{equation}
(see~\eqref{eq:ap:gammas seq qq}) such that the pair~\eqref{eq:cor:LiftCond} coincide with $\gamma_0^{\seq/\qq}, \gamma_p^{\seq/\qq}$ in $W_{-1}\setminus \inn$. Moreover,~\eqref{eq:cor:LiftCond} is a dividing pair.\qed
\end{cor}

\subsection{Lifting theorem}
\label{ss:ap:LiftThm}
In this subsection we give a sufficient condition for liftability of arcs to leaves. This gives an alternative ``manual'' proof of Theorem~\ref{thm:RobAntiRen}. Theorem~\ref{thm:AntiRen:Lifts} is not used directly elsewhere in the paper.

As before, we consider a partial homeomorphism $f\colon (W,0)\dashrightarrow (W,0)$ and we assume that $\gamma_0,\gamma_1$ is a dividing pair of curves. We also view $W$ as a subset of $\C$. 

A \emph{wall around $0$ respecting $\gamma_0,\gamma_1$} is either a closed annulus or a simple closed curve $\wall\subset U\cap V$ such that
\begin{enumerate}
\item $\C\setminus \wall$ has two connected components. Moreover, denoting by $\inn$ the bounded component of $\C\setminus \wall$, we have $0\in \inn$.
\item $\gamma_0\cap \wall$ and $\gamma_1\cap \wall$ are connected.
\item if $x\in \inn$, then $f^{\pm 1}(x)\in \wall\cup \inn$.
\end{enumerate}
In other words, points in $W$ do not jump over $\wall$ under the iteration of $f$. If $\wall$ is a simple closed curve, then $f$ restricts to an actual homeomorphism $f\colon \inn\to \inn$.

\begin{rem}
\label{rem:wall has a fence}
Note that a wall contains a fence, see~\S\ref{sss:fences}. Therefore, in the statement of Theorem~\ref{thm:RobAntiRen} we can replace a fence with a wall.
\end{rem}

For a sequence $\seq\in \{\A,\B\}^\Z$ we denote by $\wall_\seq$ and $\inn_\seq$ the closures of the preimages of $\wall$ and $\inn$ under $\pi\colon W_\seq\to W$.  We denote by $\wall_\seq[i]$ and $\inn_\seq[i]$ the intersections of $\wall_\seq$ and $\inn_\seq$ with $W_\seq[i]$. 

\begin{lem}
The set $\wall_\seq$ is connected. The closure of the connected component of $W_\seq\setminus \wall_\seq$ containing $\widetilde 0$ is $\inn_\seq$.
\end{lem}
\begin{proof}
Follows from the definition: since points in $\inn$ do not jump over $\wall$ every  $\wall_\seq[i]$ intersects $\wall_\seq[i+1]$, therefore $\wall_\seq$ is connected and the claim follows.  
\end{proof}

Suppose $f_{-1}\colon W_{\seq/\qq}\to W_{\seq/\qq}$ is an anti-renormalization of $f$ and suppose $W$ has a wall $\wall$ (respected by $\gamma_0,\gamma_1,f$) enclosing $\inn$. The image of $\wall_\seq$ in $W_{\seq/\qq}$ is called the \emph{full lift $\wall_{\seq/\qq}$ of} $\wall$. Similarly, we denote by $\inn_{\seq/\qq}$ the image of $\inn_\seq$ in  $W_{\seq/\qq}$. We say that $\wall$ is an \emph{$N$-wall} if it take at least $N$ iterates of $f^{\pm1}$ for points in $\inn$ to cross $\wall$. It follows by definition that:
\begin{lem}
\label{lem:ap:full lift of wall}
If $\wall$ is an \emph{$N$-wall}, then  $\wall_{\seq/\qq}$ is an $(N-1)\min \{\aa,\bb\}$-wall.\qed
\end{lem}

Since for a periodic combinatorics $\min \{\aa,\bb\}\ge 2$ (see~\eqref{eq:app:aa1 bb1}), we have: 
\begin{cor}
\label{cor:lift of 2wall is 2wall}
Suppose $f_{-1}\colon W_{\seq/\qq}\to W_{\seq/\qq}$ is an anti-renormalization of $f$ associated with a periodic combinatorics, see~\S\ref{sss:ap:SectRen:PerCase}. Then a lift of a $2$-wall  (respected by $\gamma_0,\gamma_1,f$) is again a $2$-wall. \qed
\end{cor}

Let  $\beta_0$, $\beta_1\coloneqq f(\beta_0)\subset W$ be two simple curves ending at $0$ such that they are disjoint away from $0$. We say that $\beta_0,\beta_1$ \emph{respect} $(\wall,\gamma_0,\gamma_1)$ if
\begin{enumerate}
\item $\beta_0$, $\beta_1$ start outside of $\inn\cup \wall$;
\item $\beta_0$, $\beta_1$ do not intersect $(\gamma_0\cup \gamma_1)\setminus \inn$; and
\item $\beta_0\cap \wall$ and $\beta_1\cap \wall$ are connected subset of different connected components of $\wall\setminus (\gamma_0\cap \gamma_1)$. 
\end{enumerate}
If we think that the components of $\wall\setminus (\gamma_0\cap \gamma_1)$ are gates of the wall $\wall$, then (3) says that $\beta_0,\beta_1$ enter $\inn$ through different gates.

\begin{rem}
\label{rem:on dividing curves}
We can slightly relax Conditions (2) and (3) to allow $\beta_0,\beta_1$ to touch (but not cross-intersect) $\gamma_0,\gamma_1$ in $W\setminus \inn$. For example, we can allow $\beta_0\setminus \inn=\gamma_0\setminus\inn$ and $\beta_1\setminus \inn=\gamma_1\setminus \inn$. 
\end{rem}

We say that a sequence $\seq$ is \emph{mixed} if $(\A,\B)$ or $(\B,\A)$ appears infinitely many times in both $\seq[\ge 0]$ and $\seq[\le 0]$.

\begin{thm}[Lifting of curves]
\label{thm:AntiRen:Lifts}
Let $f\colon W\dashrightarrow W$ be a partial homeomorphism, let $\wall\subset W$ be a wall respecting $\gamma_0,\gamma_1$, and let $\beta_0, \beta_1=f(\beta_0)$ be a pair of curves respecting $(\wall,\gamma_0,\gamma_1)$.  Let $\seq \in \{\A,\B\}^{\Z}$ be a sequence of a mixed type. Then all lifts of $\beta_0,\beta_1$ in $W_\seq$ exist, are pairwise disjoint, and land at $\wzero$.
\end{thm}
\begin{proof}

We split the proof into short subsections.

\subsubsection{Notations and Conventions}
As in~\S\ref{sss:PrimeAntiRen} we denote by $\seqo\coloneqq (\A,\B)^{\Z}$ the sequence in $\{\A,\B\}^{\Z}$ with even entries equal to $\A$ and odd entries equal to $\B$. Simplifying notations, we write $W_{\seqo}=W_{\oseq}$, $\inn_{\seqo}=\inn_{\oseq}$, and $\wall_{\seqo}=\wall_{\oseq}$. We note that $\pi\colon \inn_{\oseq }\setminus \{\widetilde 0\} \to \inn\setminus \{0\}$ and  $\pi\colon \wall_{\oseq }\to \wall$ are universal coverings. However, $\pi \colon W_{\oseq }\setminus \{\widetilde 0\} \to W\setminus \{0\}$ needs not be a covering map: the sectors of $W_{\oseq }$ are glued through $\gamma'_i$ and not through $\gamma_i$.

Denote by $U_{\oseq }$ and $V_{\oseq }$ the preimages of $U$ and $V$ under $\pi\colon W_{\oseq}\to W$. The map $f\colon U\setminus \{0\} \to V\setminus\{0\}$ admits a lift $\widetilde f\colon U_{\oseq }\setminus\{\widetilde 0\} \to V_{\oseq } \setminus\{\widetilde 0\}$ unique up to the action of the group of decks transformation, which is isomorphic to $\Z$. 
We always set $\widetilde f(\widetilde 0)\coloneqq \widetilde 0$ -- this is a continuous extension. We often write $\widetilde f$ as a partial homeomorphism $ W_{\oseq }\dashrightarrow W_{\oseq }$, and call it a \emph{lift} of $f\colon W\dashrightarrow W$.  

We specify two lifts  $\widetilde f_-~,\widetilde f_+\colon W_{\oseq }\dashrightarrow W_{\oseq }$ of $f$ as follows
\begin{itemize}
\item $\widetilde f_-$ maps $\rho(W_{\oseq }[1])$ (which is a copy of $\gamma_0$) to $\ell(W_{\oseq }[1])$; 
\item $\widetilde f_+$ maps $\rho(W_{\oseq }[1])$ to $\rho(W_{\oseq }[2])$. 
\end{itemize}
To simplify notation, we omit the tilde: $f_-=\widetilde f_-$ and $f_+=\widetilde f_+$. Note that $f_-$, $f_+^{-1}$ move points slightly to the left, while $f_+, f_-^{-1}$ move points slightly to the right.

We make the following assumptions. We assume that $\beta_0$ starts and thus crosses the wall in $\A$ while $\beta_1$ starts and thus crosses the wall in $\B$. We also assume that $\seq[0]=\A$. All other cases are completely analogous. 

We parametrize all the lifts of $\beta_0,\beta_1$ in $W_\seq$ by starting points: for $i\in\Z$ we denote by $\widetilde \beta_i=\widetilde \beta_i^\seq$ the lift of $\beta_{0}$ (if $\seq[i]=\A$) or of $\beta_1$  (if $\seq[i]=\B$) starting in $W_{\seq}[i]$. Recall that every lift $\widetilde \beta_i$ exists locally around its starting point. We will show that $\wbeta_0$ exists and lands at $\widetilde 0$, by a completely analogous argument all $\widetilde \beta_i$ exist and land at $\widetilde 0$.

Similarly we parametrize all the lifts of $\beta_0,\beta_1$ in $W_\oseq$ by starting points: we denote by $\wbeta^{\oseq}_i$ the lift of $\beta_{0}$ (if $i$ is even) or of $\beta_1$ (if $i$ is odd) starting at a point in $W_{\oseq }[i]$. Since 
$\pi\colon \inn_{\oseq } \setminus\{\widetilde 0\}\to \inn\setminus \{0\}$ is a universal cover, all $\wbeta^{\oseq}_i$ exist, pairwise disjoint, and land at $\widetilde 0$. 

We also write $\wgamma_i =\rho(W_\seq[i-1])$ and $\wgamma^{\oseq}_i =\rho(W_{\oseq }[i-1])$. (By construction, $\wgamma^{\oseq}_i$ is a lift of $\gamma_{0}$ or of $\gamma_1$ under $\pi\colon W_{\oseq }\to W$.)

\subsubsection{Example: clockwise spiraling}
\label{sss:ap:ExmClockw}
\begin{figure}
\centering{\begin{tikzpicture}
\draw[red] (0,0) circle (2.5cm);
\draw[red] (0,0) circle (2.4cm);
\draw[red] (-2,-1.9) node {$\wall$};

 \draw  (0,0) -- (2.7,0);
 \node at (2.15,0.15) {$\gamma_0$};
\draw[rotate around={155:(0,0)}] (0,0) -- (2.7,0);
\draw (-1.8, 1)node {$\gamma_1$};

\draw[blue] (1.8,1) node {$\beta_1$};
\draw[blue] (0.5,-1.9) node {$\beta_0$};
 
\draw (0,1.5) node {$\B$};
\draw (0,-1.25) node {$\A$};

\draw [rotate around={185:(0,0)},blue,domain=0:10,variable=\t,smooth,samples=100]
        plot ({3\t r}: {0.025*\t*\t});
        \draw [blue,domain=0:10,variable=\t,smooth,samples=100]
        plot ({3\t r}: {0.025*\t*\t});

\draw (4,3)--(4,-3); 
\draw (6,3)--(6,-3); 
\draw (8,3)--(8,-3);
\draw (10,3)--(10,-3);

\draw (5,3) node{$ \A$}; 
\draw (7,3) node{$ \B$}; 
\draw (9,3) node{$ \A$}; 

\draw[red] (3.5,2.6)--(10.1,2.6);
\draw[red] (3.5,2.8)--(10.1,2.8);

\draw (3.7,2.3) node{$ \wgamma^\oseq_0$}; 
\draw (6.3,2.3)  node{$ \wgamma^\oseq_1$}; 
\draw (8.3,2.3) node{$ \wgamma^\oseq_2$};
\draw (10.3,2.2)  node{$ \wgamma^\oseq_3$};
\draw[red] (10.4,2.8)  node{$ \wall_\oseq$};

\begin{scope}[shift={(0,0.2)}]

\draw[blue] (4.35,2.6)--(6,2);
\draw[blue,dashed] (6,2)--(10,0.7);
\draw[blue] (4.5,2.1) node{$ \wbeta_0^{\oseq}$};
\draw[blue] (6.4,1.5) node{$ \wbeta'$};

 \begin{scope}[shift={(0,-1.5)}]
 \draw[blue] (4,2.7)--(6,2);
\draw[blue,dashed] (6,2)--(10,0.7);
\draw[blue] (4.6,1.9) node{$f^{-1}_+( \wbeta')$};
\draw[blue] (6.4,1.5) node{$ \wbeta''$};
\end{scope}

 \end{scope}

 \begin{scope}[shift={(0,-2.8)}]
 \draw[blue] (4,2.7)--(8,1.33);
\draw[blue,dashed] (8,1.33)--(10,0.7);
\draw[blue] (4.6,1.9)  node{$f^{-1}_+( \wbeta'')$};
\draw[blue] (8.6,1.4) node{$ \wbeta^{(3)}$};
 \end{scope}

  \begin{scope}[shift={(2,-4.5)}]
 \draw[blue] (4,2.7)--(6,2);
\draw[blue] (6,2)--(7,1.65);
\draw[blue] (4.8,1.9) node{$f_-( \wbeta^{(3)})$};
 \end{scope}

\end{tikzpicture}}
\caption{Left: $\beta_0$ and $\beta_1$ spiral clockwise around $0$. Right: the curve $\wbeta^{(k+1)}$ is $f_{\pm}^{\mp 1 }(\wbeta^{(k)})$ truncated by an appropriate $\wgamma^\oseq_\ell$.}
\label{Fig:BetasGammas}
\end{figure}

\begin{figure}[p]

{\begin{tikzpicture}

\begin{scope}[scale=1.6]

 \node at (2.15,0.15) {$\gamma_0$};

\draw (-1.9, 1.1)node {$\gamma_1$};
\draw[red] (0,0) circle (2.5cm);
\draw[red] (0,0) circle (2.4cm);
\draw[red] (-2,-1.9) node {$\wall$};

 \draw  (0,0) -- (2.7,0);
\draw[rotate around={155:(0,0)}] (0,0) -- (2.7,0);

\draw (0,2) node {$\B$};
\draw (0,-.8) node {$\A$}; 

\draw[blue] (2,-1.9) node {$\ell_0$};

\draw [rotate around={185:(0,0)},blue,domain=7.17:10,variable=\t,smooth,samples=100]
        plot ({3\t r}: {0.25*\t});

\begin{scope}[scale=0.8]  
\draw [rotate around={15:(0,0)}, blue,domain=2.445:6.02,variable=\t,smooth,samples=50]
        plot ({\t r}: {0.9*sqrt{\t}});
        \begin{scope}[scale=0.8]
\draw [ rotate around={15:(0,0)},
   blue,domain=2.445:6.02,variable=\t,smooth,samples=50]
        plot ({\t r}: {0.7*sqrt{\t}});        
\draw [rotate around={15:(0,0)}, blue,domain=4.37:2.445*4,variable=\t,smooth,samples=50]
        plot ({(\t/2-2.445) r}: {0.35*sqrt{\t}});
\draw [rotate around={15:(0,0)}, blue,domain=4.37:2.445*4,variable=\t,smooth,samples=50]
        plot ({(\t/2-2.445) r}: {0.15*sqrt{\t}});                
\draw [rotate around={15:(0,0)}, blue,domain=2.445:6.02,variable=\t,smooth,samples=50]
        plot ({\t r}: {0.128*sqrt{\t}});        

 \draw[blue] (1.51,-0.18) node {$\ell_2$};
 \draw[blue] (-0.45,0) node {$\ell_3$};       
        \end{scope}
  \draw[blue] (2.25,0.15) node {$\ell_1$};        
\end{scope}

\end{scope}

\begin{scope}[shift={(-6.5,-9)}]
\draw (2,2.9)--(2,-3); 
\draw (4,3)--(4,-3); 
\draw (6,3.1)--(6,-3); 
\draw (8,3.1)--(8,-3);
\draw (10,3)--(10,-3);
\draw (12,3)--(12,-3);
\draw (1.5,2.9) node{$ \B$}; 
\draw (3,2.9) node{$ \A$}; 
\draw (5,3) node{$ \A$}; 
\draw (7,3.1) node{$ \A$}; 
\draw (9,3.1) node{$ \B$}; 
\draw (11,3) node{$ \B$}; 
\draw (12.5,3) node{$ \A$}; 
\draw[red] (12.9,2.5)  node{$ \wall_\seq$};

\draw[blue] (2.7,2.6) -- (12.5,-2.5); 

\node[blue] at (2.7,2) {$\pi^{-1}_{\seq,0}(\ell_0)$};
\node[blue] at (4.7,1) {$\pi^{-1}_{\seq,1}(\ell_1)$};
\node[blue] at (6.7,0) {$\pi^{-1}_{\seq,2}(\ell_2)$};
\node[blue] at (8.7,-1.05) {$\pi^{-1}_{\seq,3}(\ell_2)$};
\node[blue] at (10.7,-2.1) {$\pi^{-1}_{\seq,4}(\ell_3)$};
\node[blue] at (12.7,-3.) {$\pi^{-1}_{\seq,5}(\ell_3)$};

\draw[red] (1.5,2.4)--(4,2.4);
\draw[red] (1.5,2.6)--(4,2.6);
\draw[red] (4,2.5)--(6,2.5);
\draw[red] (4,2.7)--(6,2.7);
\draw[red] (6,2.6)--(10,2.6);
\draw[red] (6,2.8)--(10,2.8);
\draw[red] (10,2.5)--(12.5,2.5);
\draw[red] (10,2.7)--(12.5,2.7);

\end{scope}
\end{tikzpicture}}

\caption{Top: the curves $\ell_i$ are disjoint and within $\inn\cap \wall$. Bottom: construction of $\wbeta_0$ as the concatentaion of appropriate lifts of $\ell_i$.} \label{Fig:BetasGammas2}
\vspace{128in}
\end{figure}
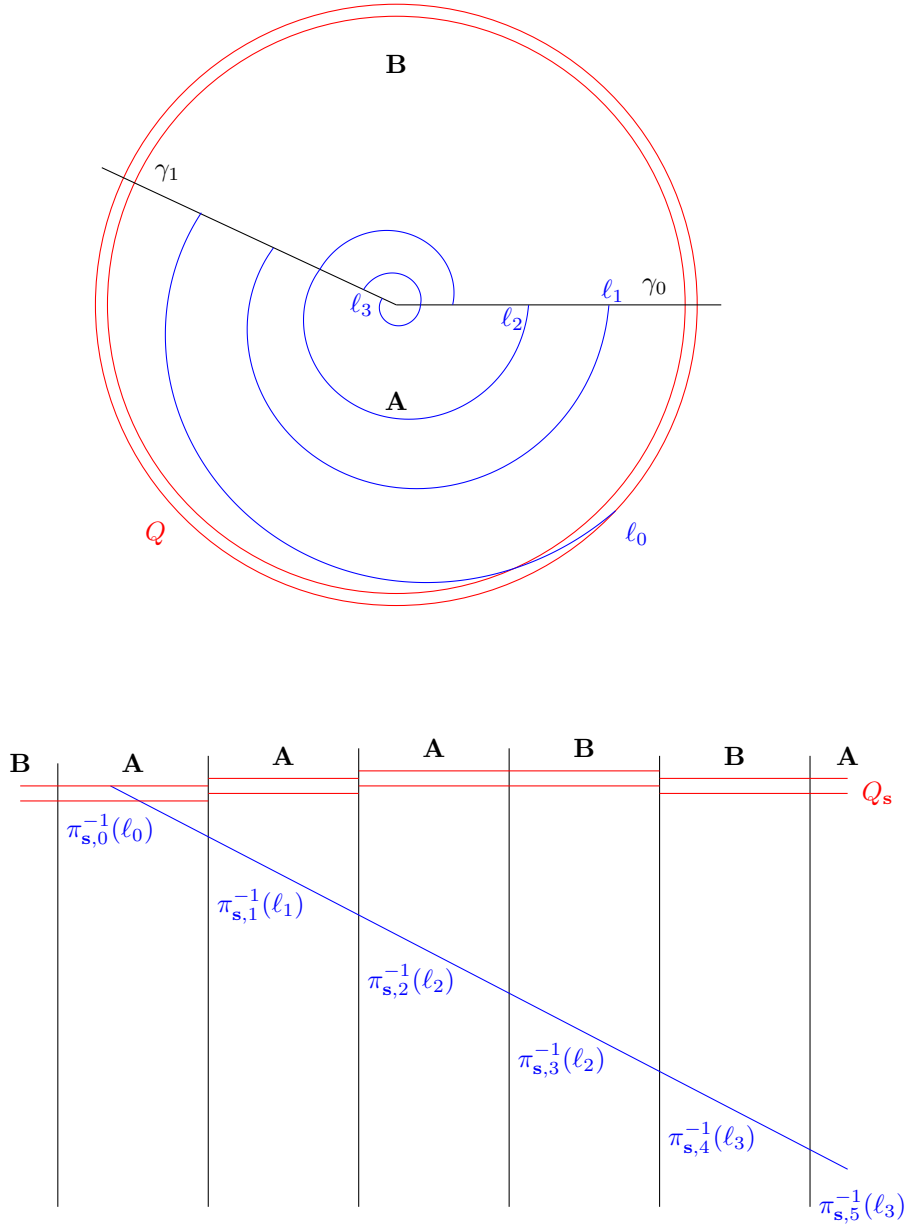 

Let us illustrate the idea of the proof in the case when $\beta_0$ and $\beta_1$ spiral clockwise around $0$, see the left-hand of Figure~\ref{Fig:BetasGammas}. Take the following $5$-periodic sequence: \[\seq[5k,5k+1,5k+2,5k+3,5k+4]=(\A,\A,\A,\B,\B) \sp\sp \forall k\in \Z.\]

Let us inductively construct the curves $\wbeta^{(k)}$, as it is illustrated on the right part of Figure~\ref{Fig:BetasGammas}:
\begin{itemize}
\item $\wbeta'$ is the sub-curve of $\wbeta^\oseq_0$ (a lift of $\beta_0$) on the right of $\wgamma^\oseq_1$. Since $\seq[0,1]=(\A,\A)$ consider $f_+^{-1}(\wbeta')$. (The curve $\wbeta'$ is in the domain of $f_+^{-1}$ because $\wbeta'$ is below the wall.)
\item $\wbeta''$ is the sub-curve of $f_+^{-1}(\wbeta')$ on the right of $\wgamma^\oseq_1$. Since $\seq[1,2]=(\A,\A)$ consider $f_+^{-1}(\wbeta'')$. (The curve $\wbeta''$ is in the domain of $f_+^{-1}$ because $\wbeta''$ is below $\wbeta'$.)
\item $\wbeta^{(3)}$ is the  sub-curve of $f_+^{-1}(\wbeta'')$ on the right of $\wgamma^\oseq_2$ -- the subindex is $2$ because $\seq[2,3]=(\A,\B)$ but $\seq[3,4]=(\B,\B)$.
\item Since $\seq[3,4]=(\B,\B)$ we consider next $f_-(\wbeta^{(3)})$.  (The curve $\wbeta^{(3)}$ is in the domain of $f_-$ because  $\wbeta^{(3)}$ is below $\wbeta''$. )
\item $\wbeta^{(4)}$ is the subcurve of $f_-(\wbeta^{(3)})$ on the right of $\wgamma^\oseq_3$ -- the subindex is $3$ because $\seq[4,5]=(\B,\A)$ but $\seq[5,6]=(\A,\A)$. 
\item Since $\seq[5,6]=(\A,\A)$, we consider next $f_+^{-}(\wbeta^{(4)})$.
\end{itemize} 
The construction continues by periodicity.

Define now ${\widetilde \ell_0\coloneqq \wbeta^\oseq_0\setminus \wbeta'}$, $\sp {\widetilde \ell_1\coloneqq f_+^{-1}(\wbeta')\setminus \wbeta''},$ $\sp{\widetilde \ell_2\coloneqq f_+^{-1}(\wbeta'')\setminus \wbeta^{(3)}},$\newline ${\widetilde \ell_3\coloneqq f_-(\wbeta^{(3)})\setminus \wbeta^{(4)},\dots},$  and $\ell_i\coloneqq \pi(\widetilde \ell_i)$, see the upper part of Figure~\ref{Fig:BetasGammas2}. Then the curve $\wbeta_0$ is the concatenation (see the bottom part of Figure~\ref{Fig:BetasGammas2}) of
\begin{itemize}
\item $\pi_{\seq ,0}^{-1}(\ell_0)$; followed by
\item $ \pi_{\seq,1}^{-1}(\ell_1)$ -- because $f^{-1}$ maps the end point of $\ell_0$ to the starting point of $\ell_1$ (recall that $\seq[0,1]=(\A,\A)$); followed by 
\item $\pi_{\seq,2}^{-1}(\ell_2\cap \A)$ -- because $f^{-1}$ maps the end point of $\ell_1$ to the starting point of $\ell_2$; followed by
\item $\pi_{\seq,3}^{-1}(\ell_2\cap \B)$ -- because $\seq[2,3]=(\A,\B)$; followed by  
\item $\pi_{\seq,4}^{-1}(\ell_3\cap \B)$ -- because $f$ maps the end point of $\ell_2$ to the starting point of $\ell_3$; followed by
\item  $\pi_{\seq,4}^{-1}(\ell_3\cap \A)$ -- because $\seq[4,5]=(\B,\A)$.
\end{itemize} 
The construction continues by periodicity.

 Note that $\ell_1\subset \inn\cup \wall$ ($\ell_1$ can intersect $\wall$) and $\ell_1$ is disjoint from $\ell_0$. The curve $\ell_2\cap \A$ is separated by $\ell_1$ from $\wall$ while $\ell_2\cap \B$ is separated by the continuation of $\ell_1$ from $\wall$. By induction, all $\ell_i$ are well defined.

Let $\inn_\oseq[\kappa(j)]\cup \wall_\oseq[\kappa(j)]$ be the strip where $\widetilde \ell_j$ starts. Since $\seq$ is mixing, we have $\kappa(j)\to +\infty$. Therefore, every $z\in \wbeta^\oseq_0$ eventually escapes to a certain $\widetilde \ell_j$ under the iteration of \[f_+^{-1},\sp f_+^{-1},\sp f_-,\sp f_+^{-1},\sp f_+^{-1},\sp f_-,\dots\] This shows that $\wbeta_0$ is a complete lift of $\beta_0\setminus \{0\}$. Since the curves $\widetilde \ell_j$ tend to the right and they are all below $\wbeta^\oseq_0$, the curves  $\widetilde \ell_j$ tend to $\widetilde 0$. Therefore, $\ell_i$ tend to $0$; i.e.~$\wbeta_0$ lands at $\widetilde 0$.

\subsubsection{Example: no mixing condition} Let us illustrate that Theorem~\ref{thm:AntiRen:Lifts} fails if $\seq$ is not mixing. Suppose $\seq=(\dots,\A,\A,\A ,\dots)$. Choose $f\colon W\dashrightarrow W$
as it shown on Figure~\ref{Fig:ViolatingMixing}: the curves $\gamma_i=f^i(\gamma_0)$ accumulate at $\gamma_\infty$ and the sector between $\gamma_0$ and $\gamma_{\infty}$ (counting clockwise) is the concatenation \[\A\cup f(\A)\cup f^{2}(\A)\cup \dots \eqqcolon {\mathbf X}\]
Then only $\beta_0\cap {\mathbf X}$ is liftable to $W_{\seq}$.

\begin{figure}
\centering{\begin{tikzpicture}[scale=1.4]
\begin{scope}
\draw[red] (0,0) circle (2.5cm);
\draw[red] (0,0) circle (2.4cm);
\draw[red] (-2,-1.9) node {$\wall$};

 \draw  (0,0) -- (2.7,0);
 \node at (2.15,0.15) {$\gamma_0$};

\begin{scope}[rotate around={60:(0,0)}]
\draw[rotate around={155:(0,0)}] (0,0) -- (2.7,0);
\draw (-1.8, 1)node {$\gamma_1$};
\end{scope}

\draw[rotate around={155:(0,0)}] (0,0) -- (2.7,0);
\draw (-1.8, 1)node {$\gamma_2$};

\begin{scope}[rotate around={-40:(0,0)}]
\draw[rotate around={155:(0,0)}] (0,0) -- (2.7,0);
\draw (-1.8, 1)node {$\gamma_3$};
\end{scope}

\begin{scope}[rotate around={-70:(0,0)}]
\draw[rotate around={155:(0,0)}] (0,0) -- (2.7,0);
\draw (-1.6, 1)node {$\gamma_\infty$};
\end{scope}

\draw[blue] (0.5,-1.9) node {$\beta_0$};
 
\draw (-0.2,1.5) node {$\dots$};
\draw (0,-1.25) node {$\A$}; 
 
 \draw (-1.05,-0.1) node {$f(\A)$}; 
 
 \draw (-1.25,1.4) node {$f^2(\A)$};  
 
\draw [rotate around={185:(0,0)}, blue,domain=5.95:10,variable=\t,smooth,samples=100]
        plot ({3\t r}: {0.025*\t*\t});
\draw [rotate around={185:(0,0)}, blue,dashed,domain=0:5.95,variable=\t,smooth,samples=100]
        plot ({3\t r}: {0.025*\t*\t});        
\end{scope}    
\end{tikzpicture}}    
\caption{Only $\beta_0\cap (\A\cup f(\A)\cup f^2(\A)\cup\dots)$ is liftable to $W_{\seq}$ if $\seq=(\dots,\A,\A,\A ,\dots)$.}   
   \label{Fig:ViolatingMixing}
\end{figure}

 \subsubsection{Combinatorics of jumps}
 \label{ss:CombOfJumps}
We now adapt the argument form~\S\ref{sss:ap:ExmClockw} to a possibility that $\wbeta_0$ oscillates at $\widetilde 0$. We start by introducing additional notations.

We define the following quantities. Recall first that for $j\in \Z$ the jump is defined by 
\[ \iota(j)\coloneqq \begin{cases}
0& \text{ if }\seq[j-1,j]\in  \{(\A,\B), (\B,\A)\}, \\
1& \text{ if }\seq[j-1,j]= (\B,\B),\\
-1& \text{ if }\seq[j-1,j]= (\A,\A).
\end{cases}\]
For $j>0$ define 
  \begin{equation}
  \label{eq:dfn:numukappa+}
    \begin{aligned}
      \nu(j)= & \#\{k\in \{1,\dots, j\} \mid \iota(k)=1\},\\
      \mu(j)= & \#\{k\in \{1,\dots, j\} \mid \iota(k)=-1\}, \\
      \kappa(j) =& \#\{k\in \{1,\dots, j\} \mid \iota(k)=0\},
    \end{aligned}
  \end{equation}
 while for $j<0$ define 
 \begin{equation}
   \label{eq:dfn:numukappa-}
   \begin{aligned}
     \nu(j)= &- \#\{k\in \{j+1,\dots, 0\} \mid \iota(k)=1\},\\
     \mu(j)= &- \#\{k\in \{j+1,\dots, 0\} \mid \iota(k)=-1\}, \\
     \kappa(j) =& -\#\{k\in \{j+1,\dots, 0\} \mid \iota(k)=0\}.
   \end{aligned}
 \end{equation}
 In particular, $\mu(j)+\nu(j)+\kappa(j)=j$ for all $j\not=0$. We also write $\mu(0)=\nu(0)=\kappa(0)=0$.
 
 For $i<j$ we define the \emph{jump from $W_\seq[i]$ to $W_\seq[j]$} to be the sum of jumps from $W_\seq[i+k]$ to $W_\seq[i+k+1]$ with $k$ ranging from $0$ to $j-i-1$. The \emph{jump from $W_\seq[j]$ to $W_\seq[i]$} is defined to be the negative of the jump from $W_\seq[i]$ to $W_\seq[j]$. It follows from definitions:
 \begin{claim2}
 \label{cl2:jump is nu min mu}
 The jump from $W_\seq [0]$ to $W_\seq [k]$ is $\nu(k)-\mu(k)$.\qed
 \end{claim2}

\begin{claim2}
\label{cl2:lifting form}
Suppose that the lift $\widetilde \beta_0$ exists for all $t\in [0,\bar t]$. If $\widetilde \beta_0(t)\in {\intr W_\seq [k]\cup \rho(W_\seq [k])}$, then
\begin{equation}
\label{eq:MainId}
 \widetilde \beta_0(t) = \pi^ {-1}_{\seq,k}\left(f^ {\nu(k)-\mu(k)}\circ \beta_0(t)\right)= \pi^ {-1}_{\seq,k}\circ \pi \left(f^ {\nu(k)}_-\circ f_+^ {-\mu(k)}\circ \wbeta^{\oseq}_0(t)\right)
\end{equation}
and all maps in this equation are well defined. 
\end{claim2}

\begin{proof}
If $\widetilde \beta_0(t)\in  {\intr W_\seq [k]\cup \rho(W_\seq [k])}$, then by definition of the lift of a curve and by Claim~\ref{cl2:jump is nu min mu} we have
$\pi(\widetilde \beta_0(t))= f^{\nu(k)-\mu(k)}(\beta_0(t))$. This is the first equality in~\eqref{eq:MainId}. The second equality holds because $f^ {\nu(k)}_-\circ f_+^ {-\mu(k)}$ is a lift of $f^ {\nu(k)-\mu(k)}$.
\end{proof}

Since $\seq$ is of mixed type, we obviously have:
\begin{claim2}
\label{cl2:kappa tens inft}
 If $j \to \pm \infty$, then $\kappa(j)\to \pm\infty$ respectively.\qed
\end{claim2}
 
\subsubsection{Basic dynamical properties} For $X\subset W$ and $n\in \Z$, we write \[f^n(X)=f^n\left( X\cap \Dom f^n\right).\]

\begin{claim2}
\label{cl2:A}
For $n\in \Z$ we have
\[f^n(\gamma_0)\cap f^{n+1}(\gamma_0)=\{0\}\sp \text{ and }\sp f^n(\beta_0)\cap f^{n+1}(\beta_0)=\{0\}.\]
\end{claim2}
\begin{proof}
Follows from $\gamma_0\cap \gamma_1=\{0\}= \beta_0\cap \beta_1$ and the assumption that $f$ is a partial homeomorphism. 
\end{proof}

\begin{claim2}
\label{cl2:B} The curve $\wgamma^{\oseq}_{0}\cap \inn_{\oseq }$ is in the domains of $f_{\pm}^{\pm1}$ (for any choice of ``$+$'' and ``$-$''). Moreover, we have:
\begin{itemize}
\item[(1)] $ f_+^{-1}\left(\wgamma^{\oseq}_{1}\cap \inn_{\oseq } \right) \subset \wgamma^\oseq_0,$
\item[(2)]   $ f_-\left(\wgamma^{\oseq}_{1}\cap \inn_{\oseq } \right)  \subset W_\oseq[-1,0]\cup \inn_\oseq[-1,0],$
\item[(3)] $  f_+\left(\wgamma^{\oseq}_{0}\cap \inn_{\oseq } \right)  \subset \wgamma^\oseq_1,$ 
\item[(4)]  $f_-^{-1}\left(\wgamma^{\oseq}_{0}\cap \inn_{\oseq } \right)  \subset W_\oseq[0,1]\cup \inn_\oseq[0,1].$ 
\end{itemize}
\end{claim2}
\begin{proof}
Recall that the lifts $f_+$ and $f_-$ are specified so that $f_+^{-1}$ and $f_-$ move points slightly to the left while $f_+$ and $f_-^{-1}$ move points slightly to the right. Therefore, 
\begin{itemize}
\item (1) follows from $f^{-1}(\gamma_1)\subset \gamma_0$;
\item (2) follows from $f^{-1}(\gamma_0)\cap \gamma_0=\{0\}$, see Claim~\ref{cl2:A}.
\item (3) follows from $f(\gamma_0)\subset \gamma_1$;
\item (4) follows from $f(\gamma_1)\cap \gamma_1=\{0\}$, see Claim~\ref{cl2:A}.
\end{itemize}
\end{proof}

\subsubsection{Gulfs $D_{>0}$ and $D_{<0}$}
Let us define the \emph{gulf} $D_{> 0}$ to be the closed region in $\inn_{\oseq }[> 0]$ located on the right of $\wgamma^{\oseq}_1=\rho(W_{\oseq }[0])$ and on the left of $\wbeta^{\oseq}_0$. We recall that both $\wgamma^{\oseq}_1$ (and similarly $\wbeta^{\oseq}_0$) decomposes $\inn_\oseq$ into two connected  components; thus $D_{> 0}$  is well defined.  Similarly, the \emph{gulf} $D_{<0}$ is the closed region in $\inn_{\oseq }[<0]$ located on the left of $\wgamma^{\oseq}_0$ and on the right of $ \wbeta^{\oseq}_0$.

\begin{claim2}
\label{cl2:4possOfD}
Both $D_{>0}$ and $D_{<0}$ are in the domains of $f_\pm^{\pm 1}$  (for any choice of ``$+$'' and ``$-$'').
 Moreover, we have:
\begin{itemize}
\item[(1)] $f_+^{-1} \left(D_{>0}\right) \subset D_{>0}\cup \inn_{\oseq }[0]\cup \wall_{\oseq }[0],$
\item[(2)]   $f_- \left(D_{>0}\right)\subset D_{>0}\cup \inn_{\oseq }[-1,0]\cup \wall_{\oseq }[-1,0] ,$
\item[(3)] $ f_+ \left(D_{<0}\right)\subset D_{<0}\cup \inn_{\oseq }[0]\cup \wall_{\oseq }[0],$ 
\item[(4)]  $f_-^{-1} \left(D_{>0}\right) \subset D_{<0}\cup \inn_{\oseq }[0,1]\cup \wall_{\oseq }[0,1] .$ 
\end{itemize}
\end{claim2}
\begin{proof}
The first claim follows from $D_{>0}\cup D_{<0}\subset \inn_\oseq$. Statements (1)--(4) follow from Statements (1)--(4) of Claim~\ref{cl2:B} respectively. Indeed, by definition, $f_+^{-1}(D_{>0})$ is bounded by $f_+^{-1}(\wgamma^{\oseq}_1 \cap D_{>0} )\subset \wgamma^{\oseq}_{0}$ and by $f_+^{-1}(\wbeta^{\oseq}_0\cap D_{>0})$; this implies (1). Other Statements are analogous. 
\end{proof}

\begin{figure}
\centering{\begin{tikzpicture}

\draw (-4,2.6)--(-4,-3);
\draw (-2,2.6)--(-2,-3);
\draw (0,2.6)--(0,-3); 
\draw (2,2.6)--(2,-3); 
\draw (4,2.6)--(4,-3); 
\draw (6,2.6)--(6,-3); 
\draw (8,2.6)--(8,-3);
\draw[red] (-4.2,2.1)--(8.2,2.1);
\draw[red] (-4.2,2.6)--(8.2,2.6);

\draw (-2.3,2.3) node{$ \wgamma^\oseq_{-1}$}; 
\draw (-0.3,2.3) node{$ \wgamma^\oseq_0$}; 
\draw (1.7,2.3)  node{$ \wgamma^\oseq_1$}; 
\draw (3.7,2.3) node{$ \wgamma^\oseq_2$};
\draw (5.7,2.3)  node{$ \wgamma^\oseq_3$};

\draw (1,1.5) node{$\inn_{\oseq}[0]$}; 
\draw (-1,1.5) node{$\inn_{\oseq}[-1]$}; 
\draw (3,1.5) node{$\inn_{\oseq}[1]$};

\draw[blue] plot[smooth,tension=1]
  coordinates{(1,2.6) (1.3,2) (2,1.5)}; 

\draw[ draw =blue,fill=blue, fill opacity=0.2] plot[smooth,tension=1]
  coordinates{(2,1.5) (3.5,1)(5,0.5) (5,0.1) 
  (3.5,-0.2)(2,-0.1)};
  
  \draw[blue] plot[smooth,tension=1]
  coordinates{(2,-0.1)(1,-0.05)(0,-0.1)}; 
  
  \draw[ fill=blue,draw =blue, fill opacity=0.2] plot[smooth,tension=1]
  coordinates{(0,-0.1) (-2.2,-0.5) (-3.2,-1)   (-2.2,-1.5)(0,-2.2)};

\draw[blue] plot[smooth,tension=1]
  coordinates{(0,-2.2) (1,-2.3) (2,-2.3)};

  \draw[ draw=blue] plot[smooth,tension=1]
  coordinates{(2,-2.3) (3.9,-2.5) (5.9,-2.8)(8,-3)};
  \draw[ fill=blue,draw opacity=0, fill opacity=0.2] plot[smooth,tension=1]
  coordinates{(2,-2.3) (3.9,-2.5) (5.9,-2.8)(8,-3)(2,-3)};

  \draw[blue] (5.5,0.2)  node{$\wbeta^\oseq_0$};
\draw[blue]  (2,1.5) --  (2,-0.1);
\draw[blue] (3.2,0.5) node {$D_{>0}$};
\draw[blue] (0,-0.1)--(0,-2);

\draw[blue] (-1.4,-1) node{$D_{<0}$};

\draw[blue]   (2,-2.3) -- (2,-3);
\draw[blue] (3,-2.7) node{$D_{>0}$};


\begin{scope}[shift={(0,-6.5)}]
\draw (-4,2.6)--(-4,-3);
\draw (-2,2.6)--(-2,-3);
\draw (0,2.6)--(0,-3); 
\draw (2,2.6)--(2,-3); 
\draw (4,2.6)--(4,-3); 
\draw (6,2.6)--(6,-3); 
\draw (8,2.6)--(8,-3);
\draw[red] (-4.2,2.1)--(8.2,2.1);
\draw[red] (-4.2,2.6)--(8.2,2.6);

\draw (-2.3,2.3) node{$ \wgamma^\oseq_{-1}$}; 
\draw (-0.3,2.3) node{$ \wgamma^\oseq_0$}; 
\draw (1.7,2.3)  node{$ \wgamma^\oseq_1$}; 
\draw (3.7,2.3) node{$ \wgamma^\oseq_2$};
\draw (5.7,2.3)  node{$ \wgamma^\oseq_3$};

\draw (1,1.725) node{$\inn_{\oseq}[0]$}; 
\draw (-1,1.5) node{$\inn_{\oseq}[-1]$}; 
\draw (3,1.5) node{$\inn_{\oseq}[1]$}; 


\draw[blue] plot[smooth,tension=1]
  coordinates{(1,2.6) (1.3,2) (2,1.5)}; 

\draw[ draw =blue, ] plot[smooth,tension=1]
  coordinates{(2,1.5) (3.5,1)(5,0.5) (5,0.1) 
  (3.5,-0.2)(2,-0.1)};
  
  \draw[blue] plot[smooth,tension=1]
  coordinates{(2,-0.1)(1,-0.05)(0,-0.1)}; 
  
  \draw[ draw =blue,] plot[smooth,tension=1]
  coordinates{(0,-0.1) (-2.2,-0.5) (-3.2,-1)   (-2.2,-1.5)(0,-2.2)};

\draw[blue] plot[smooth,tension=1]
  coordinates{(0,-2.2) (1,-2.3) (2,-2.3)};

  \draw[ draw=blue] plot[smooth,tension=1]
  coordinates{(2,-2.3) (3.9,-2.5) (5.9,-2.8)(8,-3)};

  \draw[blue] (5.5,0.2)  node{$\wbeta^\oseq_0$};


\begin{scope}[shift={(-2,0.2)}]
\draw[green,fill=green, fill opacity=0.2] plot[smooth,tension=1]
coordinates{ (2,1.5) (3.5,1)(4,0.85)(5,0.5) (5,0.1)
  (3.5,0)(2,0.2)};
  \draw[green](4,0.85) -- (4,0);
  \draw[green](2,1.5)-- (2,0.2);
  \draw[green](2.85,0.5) node {$T_{1}$};
  \draw[green](4.55,0.4) node {$T_{>1}$};
\end{scope}

\begin{scope}[shift={(2,0)}]
\draw[green,fill=green, fill opacity=0.2] plot[smooth,tension=1]
coordinates{(0,-0.3) (-2.2,-0.5) (-3,-1) 
  (-2.2,-1.5)(0,-2)};
  \draw[green] (-2,-0.45)-- (-2,-1.55);
  \draw[green] (0,-0.3)--(0,-2);
  \draw[green] (-1,-1) node{$T_{-1}$};
  \draw[green] (-2.5,-1) node{$T_{<-1}$};
  \end{scope}
  
\end{scope}

\begin{scope}[shift={(0,-13)}]
\draw (-4,2.6)--(-4,-3);
\draw (-2,2.6)--(-2,-3);
\draw (0,2.6)--(0,-3); 
\draw (2,2.6)--(2,-3); 
\draw (4,2.6)--(4,-3); 
\draw (6,2.6)--(6,-3); 
\draw (8,2.6)--(8,-3);
\draw[red] (-4.2,2.1)--(8.2,2.1);
\draw[red] (-4.2,2.6)--(8.2,2.6);

\draw (-2.3,2.3) node{$ \wgamma^\oseq_{-1}$}; 
\draw (-0.3,2.3) node{$ \wgamma^\oseq_0$}; 
\draw (1.7,2.3)  node{$ \wgamma^\oseq_1$}; 
\draw (3.7,2.3) node{$ \wgamma^\oseq_2$};
\draw (5.7,2.3)  node{$ \wgamma^\oseq_3$};

\draw (1,1.5) node{$\inn_{\oseq}[0]$}; 
\draw (-1,1.5) node{$\inn_{\oseq}[-1]$}; 
\draw (3,1.5) node{$\inn_{\oseq}[1]$};

\draw[blue] plot[smooth,tension=1]
  coordinates{(1,2.6) (1.3,2) (2,1.5)}; 

\draw[ draw =blue,fill=blue, fill opacity=0.2] plot[smooth,tension=1]
  coordinates{(2,1.5) (3.5,1)(5,0.5) (5,0.1) 
  (3.5,-0.2)(2,-0.1)};
  
  \draw[blue] plot[smooth,tension=1]
  coordinates{(2,-0.1)(1,-0.05)(0,-0.1)}; 
  
  \draw[ fill=blue,draw =blue, fill opacity=0.2] plot[smooth,tension=1]
  coordinates{(0,-0.1) (-2.2,-0.5) (-3.2,-1)   (-2.2,-1.5)(0,-2.2)};

\draw[blue] plot[smooth,tension=1]
  coordinates{(0,-2.2) (1,-2.3) (2,-2.3)};

  \draw[ draw=blue] plot[smooth,tension=1]
  coordinates{(2,-2.3) (3.9,-2.5) (5.9,-2.8)(8,-3)};
  \draw[ fill=blue,draw opacity=0, fill opacity=0.2] plot[smooth,tension=1]
  coordinates{(2,-2.3) (3.9,-2.5) (5.9,-2.8)(8,-3)(2,-3)};

  \draw[blue] (5.5,0.2)  node{$\wbeta^\oseq_0$};
\draw[blue]  (2,1.5) --  (2,-0.1);

\draw[blue] (0,-0.1)--(0,-2);
\draw[blue] (-1.3,-1) node {$D_{-1}$};

\draw[blue] (-2.2,-0.5)--(-2.2,-1.5);
\draw[blue] (-2.65,-1) node{$D_{<-1}$};

\draw[blue] (2.8,0.5) node {$D_1$};
\draw[blue] (3.5,1) --  (3.5,-0.2);
\draw[blue] (4.3,0.3) node {$D_{>1}$};

\draw[blue]   (2,-2.3) -- (2,-3);
\draw[blue] (2.5,-2.7) node{$D_{1}$};
  \draw[blue]   (3.9,-2.5) -- (3.9,-3);
\draw[blue] (4.3,-2.8) node{$D_{>1}$};
\end{scope}
\end{tikzpicture}}
\caption{Top: closed regions $D_{>0}$ and $D_{<0}$. Middle: assuming that  $\seq[-1,0,1]=(\A,\A,\A)$, we have $f^{-1}_+(D_{>0})=T_{>0}=T_1\cup T_{>1}$ and $f_-(D_{<0})=T_{<0}=T_{-1}\cup T_{<-1}$. Bottom: decompositions $D_{1}\cup D_{>1}=D_{>0}$ and $D_{-1}\cup D_{<-1}= D_{<0}$.}\label{Fig:ConstrD_i}
\vspace{128in}
\end{figure}
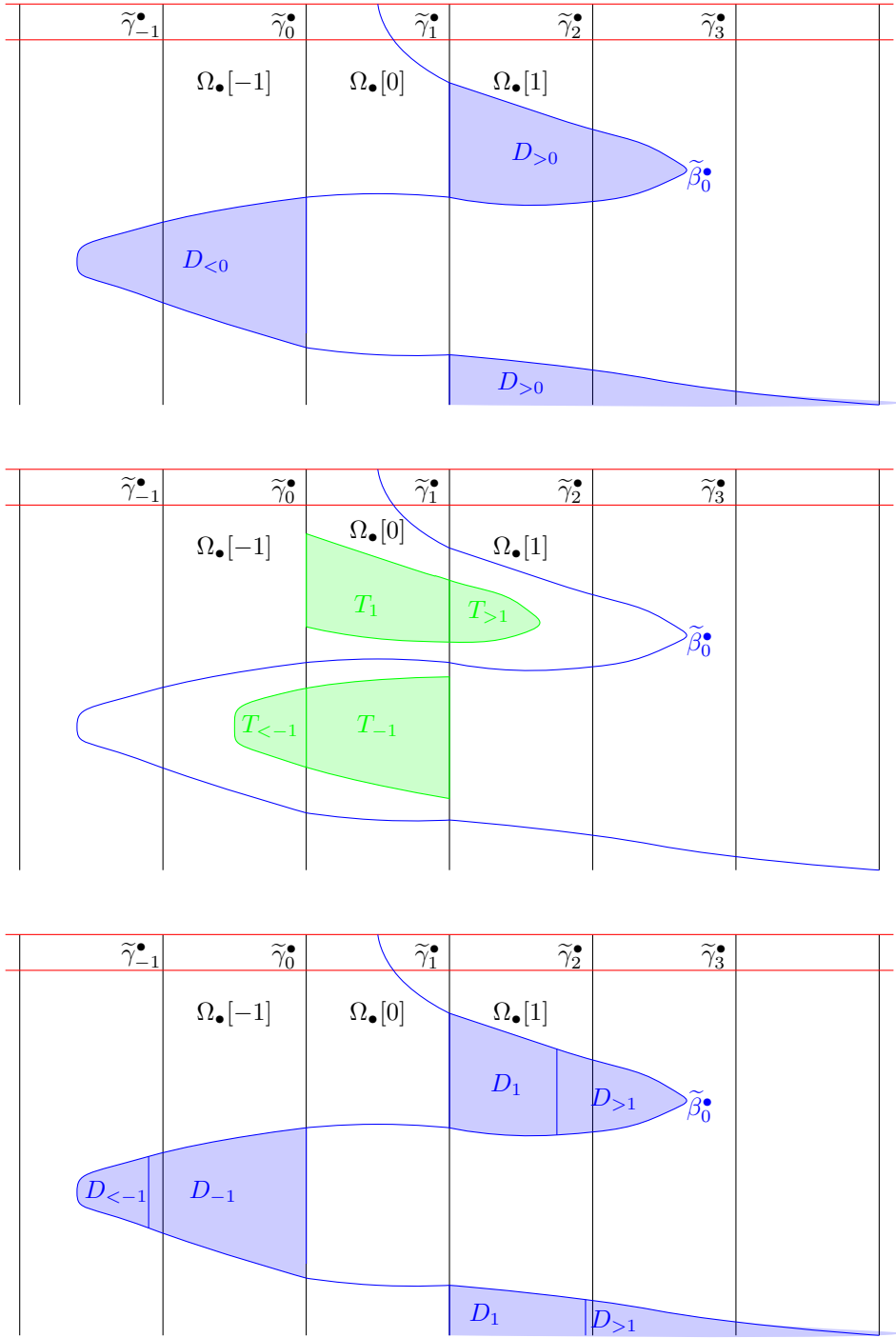

\subsubsection{Channels $D_{k},D_{-k},T_{k},T_{-k}$} Let us now apply inductively $f_{+}^ {-\mu(k)}\circ   f_{-}^ {\nu(k)}$ to $D_{>0}$; on each step we define $D_k$ as the set of points that escape to $W_\oseq [\kappa(k)]$, see Figure~\ref{Fig:ConstrD_i} and its caption. 

Consider first the case $k=1$. If $\seq[1]=\A$, then $\nu(1)=0$, $\sp \mu(1)=1$, and $\kappa(1)=0$. In this case (see Figure~\ref{Fig:ConstrD_i}) we define 
\begin{align*}
T_{>0}\coloneqq f_+^{-1}(D_{>0},)\\
T_1\coloneqq T_{>0}\cap W[0],\\
 D_1\coloneqq  f_+(T_1),\\
T_{>1}\coloneqq \overline{T_{>0}\setminus T_1},\\
D_{>1}\coloneqq  f_+(T_{>1}).
\end{align*}

If $\seq[1]=\B$, then $\nu(1)=0$, $\sp \mu(k)=0$, and $\kappa(1)=1$. In this case we define
\begin{align*}
T_{>0}\coloneqq D_{>0},\\
T_1=D_1\coloneqq T_{>0}\cap \inn[1],\\
T_{>1}=D_{>1}\coloneqq \overline{T_{>0}\setminus T_1}.
\end{align*} 

In general, for $k>0$ we set inductively 
\begin{align*}
T_k\coloneqq  &   f_{+}^ {-\mu(k)}\circ   f_{-}^ {\nu(k)} (D_{>k-1}) \cap W_{\oseq }[\kappa(k)],\\
D_{k}\coloneqq  & f_{+}^ {\mu(k)}\circ   f_{-}^ {-\nu(k)} \left(T_k \right),\\
T_{>k+1} \coloneqq& \overline{T_{>k}\setminus T_{k+1}},\\
D_{>k+1} \coloneqq& \overline{D_{>k}\setminus D_{k+1}},
\end{align*}
and similarly, for $k<0$, we set
\begin{align*}
T_{-k}\coloneqq  &   f_{+}^ {-\mu(-k)}\circ   f_{-}^ {\nu(-k)} (D_{<-k+1}) \cap W_{\oseq }[\kappa(-k)],\\
D_{-k}\coloneqq  & f_{+}^ {\mu(-k)}\circ   f_{-}^ {-\nu(-k)} \left(T_{-k} \right),\\
T_{<-k-1} \coloneqq& \overline{T_{<-k}\setminus T_{-k-1}},\\
D_{<-k-1} \coloneqq& \overline{D_{<-k}\setminus D_{-k-1}}.
\end{align*}
The case $\seq[-1,0,1]=(\A,\A,\A)$ is in Figure~\ref{Fig:ConstrD_i}. We call $D_{>k}, D_{<-k}$ \emph{gulfs}, and we say that $T_k, D_k$ are \emph{channels}. The channels $T_k$ play the role of the curves $\widetilde \ell_k$ from~\S\ref{sss:ap:ExmClockw}. 

By an easy induction $f_{+}^ {-\mu(k)}\circ   f_{-}^ {\nu(k)}(D_{>k})  \subset D_{>0}$ for $k>0$; thus by Claim~\ref{cl2:4possOfD} the gulf $D_{>k}$ is in the domain of $f_{+}^ {-\mu(k+1)}\circ   f_{-}^ {\nu(k+1)}$. Similarly, $D_{<-k}$ is in the domain of $f_{+}^ {-\mu(-k-1)}\circ   f_{-}^ {\nu(-k-1)}$.

For $k\not= 0$ write $\ell(T_k )\coloneqq \ell(W_{\oseq }[\kappa(k)])\cap T_k$ and $\rho(T_k )\coloneqq \rho(W_{\oseq }[\kappa(k)]) \cap T_k$.

\begin{claim2}
\label{cl2:splt of D0}
The gulfs $D_{>0}$ and $D_{<0}$ are the unions $D_{1}\cup D_{2} \cup D_{3}\cup \dots $ and $D_{-1}\cup D_{-2} \cup D_{-3}\cup \dots $ respectively. Moreover, $D_i\cap D_{j}=\emptyset$ if $|i-j|\ge 0$. Write $\delta\coloneqq D_k\cap D_{k+1}$ for $k\not\in \{-1,0\}$. Then 
\begin{align*} 
f_{+}^ {-\mu(k)}\circ   f_{-}^ {\nu(k)}(\delta) &\subset \rho (T_k)\subset \wgamma^\oseq_{\kappa(k)+1},\\ 
f_{+}^ {-\mu(k+1)}\circ   f_{-}^ {\nu(k+1)}(\delta)& \subset \ell (T_{k+1}) \subset \wgamma^\oseq_{\kappa(k+1)}.
\end{align*} 
\end{claim2}

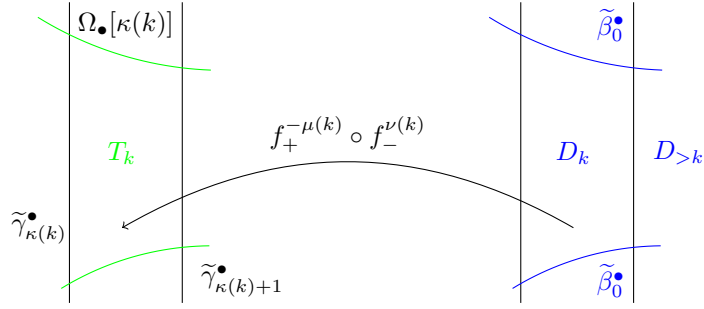
\begin{figure}
\centering{\begin{tikzpicture}

\begin{scope}[shift={(0,0)}]
\draw (-2.5,-2.)-- (-2.5,2.);
\draw (-1.,-2.)-- (-1.,2.);
\draw[green] [shift={(-0.44,5.8)}] plot[domain=4.157393253046103:4.674432535196376,variable=\t]({1.*4.706421145626473*cos(\t r)+0.*4.706421145626473*sin(\t r)},{0.*4.706421145626473*cos(\t r)+1.*4.706421145626473*sin(\t r)});
\draw[green] [shift={(-0.6,-5.08)}] plot[domain=1.5812126167258664:2.12139123904915,variable=\t]({1.*3.840208327682236*cos(\t r)+0.*3.840208327682236*sin(\t r)},{0.*3.840208327682236*cos(\t r)+1.*3.840208327682236*sin(\t r)});

\draw (-1.75,1.7) node {$\inn_\oseq[\kappa(k)]$};
\draw[green] (-1.8,0) node {$T_k$};
\draw (-2.9,-1) node {$\wgamma^\oseq_{\kappa(k)}$};
\draw (-0.2,-1.7) node {$\wgamma^\oseq_{\kappa(k)+1}$};
\end{scope}

\draw (-1.8,-1) edge[<-,bend left] node [above] {$f_{+}^ {-\mu(k)}\circ   f_{-}^ {\nu(k)}$} (4.2,-1);

\begin{scope}[shift={(6,0)}]
\draw[] (-2.5,-2.)-- (-2.5,2.);
\draw[] (-1.,-2.)-- (-1.,2.);
\draw[blue] [shift={(-0.44,5.8)}] plot[domain=4.157393253046103:4.674432535196376,variable=\t]({1.*4.706421145626473*cos(\t r)+0.*4.706421145626473*sin(\t r)},{0.*4.706421145626473*cos(\t r)+1.*4.706421145626473*sin(\t r)});
\draw[blue] [shift={(-0.6,-5.08)}] plot[domain=1.5812126167258664:2.12139123904915,variable=\t]({1.*3.840208327682236*cos(\t r)+0.*3.840208327682236*sin(\t r)},{0.*3.840208327682236*cos(\t r)+1.*3.840208327682236*sin(\t r)});

\draw[blue] (-1.8,0) node {$D_k$};
\draw[blue] (-0.4,0) node {$D_{>k}$};
\draw[blue] (-1.3,1.7) node {$\wbeta^\oseq_{0}$};
\draw[blue] (-1.3,-1.7) node {$\wbeta^\oseq_{0}$};
\end{scope}
\end{tikzpicture}}
\caption{The region $D_k$ is the set of points in $D_{>0}$ that escape to $\inn_\oseq [\kappa(k)]$ under $f_{+}^ {-\mu(k)}\circ   f_{-}^ {\nu(k)}$. The region $D_{>k}$ is the set of points in $D_{>0}$ on the right of $D_k$.}\label{Fig:DkTk}
\end{figure}
\begin{proof}
We will verify the claim for $D_{>0}$; the case of $D_{<0}$ is similar.

 By induction, if for $z\in D_{>0}$ the point $f_{+}^ {-\mu(k)}\circ   f_{-}^ {\nu(k)}(z)$ is on the right of $W_\oseq[\kappa(k)]$, then $f_{+}^ {-\mu(k+1)}\circ   f_{-}^ {\nu(k+1)}(z)$ is ether on the right of $W_\oseq[\kappa(k+1)]$ or ${f_{+}^ {-\mu(k+1)}\circ   f_{-}^ {\nu(k+1)}(z)\in W_\oseq[\kappa(k+1)]}$. Thus points in $D_{>k}$ do not jump over $T_{k+1}$ under one iteration. Recall that $f_+^{-1}$ and $f_-$ move points to the left. By Claim~\ref{cl2:kappa tens inft}, $\kappa(k)\to +\infty$. Therefore, every point in $D_{>0}$ eventually escapes to some $T_k\subset W_\oseq[\kappa(k)].$ Let us now show that a point in $D_{>0}$ escapes to at most two (neighboring) $T_k$th.

For $k>0$ the channel $T_k= f_{+}^ {-\mu(k)}\circ   f_{-}^ {\nu(k)} (D_k) $ is on the left of $f_{+}^ {-\mu(k)}\circ   f_{-}^ {\nu(k)} (\wbeta^{\oseq}_0)$, on the left of $\wgamma^{\oseq}_{\kappa(k)+1}$, and on the right of
 $\wgamma^{\oseq}_{\kappa(k)}$. On the other hand,  $f_{+}^ {-\mu(k)}\circ   f_{-}^ {\nu(k)} (D_{>k})$ is on the right of $\wgamma^{\oseq}_{\kappa(k)+1}$. It is now easy to see that \[T_k\cap f_{+}^ {-\mu(k)}\circ   f_{-}^ {\nu(k)} (D_{>k}) = T_k\cap f_{+}^ {-\mu(k)}\circ   f_{-}^ {\nu(k)} (D_{k+1})=f_{+}^ {-\mu(k)}\circ   f_{-}^ {\nu(k)} (\delta)\subset \wgamma^\oseq_{\kappa(k)+1}.\]
This proves that a point in $D_{>0}$ escapes to at most two $T_k$th; it also verifies the first identity. Similarly, the second identity is verified.
\end{proof}

\subsubsection{The natural map from $D_{<0}\cup \inn_{\oseq }[0]\cup \wall_{\oseq }  \cup D_{>0}$ to $W_\seq$}For convenience, let us extend by continuity the map $\pi^{-1}_{\seq,i}\colon \intr \seq[i]\cup \rho(\seq[i]) \to \intr W_{\seq}[i] \cup \rho(W_\seq[i])$ to $\pi^{-1}_{\seq,i}\colon  \seq[i] \to  W_{\seq}[i]$.  

We define the map $\theta_0\colon \inn_{\oseq }[0]\cup \wall_{\oseq }[0]\to \inn_\seq[0]\cup \wall_\seq[0]$ to be $\pi^{-1}_{\seq,0}\circ \pi$. For $k\not=0$ we define the map $\theta_k\colon D_k\to W_\seq[k]$ as $ f_{+}^ {-\mu(k)}\circ   f_{-}^ {\nu(k)}\colon D_k\to T_k$, followed by ${\pi\colon T_k\to \seq[k]\subset W}$, and followed by $\pi^{-1}_{\seq, k}\colon \seq[k]\to W_\seq[k]$.   Combining $\theta_k$ and $\theta_0$ we obtain the map  \[\theta \colon D_{<0}\cup \inn_{\oseq }[0]\cup \wall_{\oseq }[0]  \cup D_{>0}\to W_\seq\] such that $\theta\mid D_k=\theta_k$ and $\theta\mid \inn_\oseq[0] \cup \wall_\seq [0]=\theta _0$. We note that there is no ambiguity on $D_{k}\cap D_{k+1}$:

\begin{claim2}
\label{cl2: nat map}
The map $\theta$ is a homeomorphism on its image. Furthermore, for every curve \[\alpha \colon [0,1]\to D_{<0}\cup \inn_{\oseq }[0]\cup \wall_{\oseq }[0]  \cup D_{>0}\] starting in $ \inn_{\oseq }[0]\cup \wall_{\oseq }[0]$, the lift of  $\pi(\alpha)$ starting in $\inn_{\seq}[0]\cup \wall_{\seq}[0]$ exists and is equal to $\theta(\alpha)$.
\end{claim2}
\begin{proof}
It is routine to check that $\theta_k$ and $ \theta_j$ agree on $\Dom \theta_k \cap\Dom \theta_{j} =\emptyset$. Indeed, if $|k-j|>1$, then $\Dom \theta_k \cap\Dom \theta_{j} =\emptyset$ by Claim~\ref{cl2:splt of D0}. If $j=k+1$, then writing $\delta\coloneqq \Dom \theta_k \cap\Dom \theta_{k+1}$, we check (see again Claim~\ref{cl2:splt of D0}) that $\theta_k\mid \delta =\theta_{k+1}\mid \delta$ and $\theta_k(\delta)\subset W_\seq[k]\cap W_\seq[k+1]$. Since  $\theta_k$ and $ \theta_j$  have disjoint images away from $\Dom \theta_k \cap\Dom \theta_{j}$, we obtain that $\theta$ is a homeomorphism.

By the definition, $\Im(\theta_k)\subset W_\seq[k]$ and $f_{+}^ {\mu(k)}\circ   f_{-}^ {-\nu(k)}$ transfers $D_k$ to $W_\seq[k]$ -- this is  the correct number of iterations, see Claim~\ref{cl2:lifting form} and~\eqref{eq:MainId}. Therefore, $\theta(\alpha)$ is the lift of $\pi(\alpha)$ starting in $\wall\oseq[0]\cup\inn\oseq[0]$.

\end{proof}
\noindent As a corollary, we obtain
\begin{claim2}
All $ \wbeta_m$ exist.
\end{claim2}
\begin{proof}
By Claim~\ref{cl2: nat map},  $ \wbeta_0=\theta(\wbeta^\oseq_0)$ exists. By a similar argument all $\widetilde \beta_m$ exist; let us give a brief sketch. For every $\wbeta^\oseq_m$  define $D'_{>m}$ to be the closed region on the left of $\wbeta^\oseq_m$ and on the right of $\wgamma^\oseq_{m+1}$, define  $D'_{<m}$ to be the closed region on the right of $\wbeta^\oseq_m$ and on the left of $\wgamma^\oseq_m$. As in \S\ref{ss:CombOfJumps} specify the quantities $\nu(k),\mu(k),\kappa(k)$ with the only difference is that the count in~\eqref{eq:dfn:numukappa+} starts from $m+1$ instead of $1$ (and similar in~\eqref{eq:dfn:numukappa-}). By the same argument as for $\wbeta_0$ we construct 
 \[\theta' \colon D_{<m}\cup \inn_{\oseq }[m]\cup \wall_{\oseq }[m]  \cup D_{>m}\to W_\seq\] 
 such that Claim~\ref{cl2: nat map} (with necessary adjustments) holds for $\wbeta_m$.
\end{proof}

\begin{claim2}
\label{cl2: beta land at 0}
All $\widetilde \beta_i$ land at $\widetilde 0$. 
\end{claim2}
\begin{proof}
Let us show that $\widetilde \beta_0$ lands at $\widetilde 0$; other cases are completely analogous. By Claim~\ref{cl2: nat map} we have $ \wbeta_0=\theta(\wbeta_0^{\oseq})$. Parametrize $\beta_0$ as $\beta_0\colon [0,1]\to W$ with $\beta_0(1)=0$.

Choose a big $M>0$. Since $\theta\mid \Dom \theta \cap W_\oseq[-M,-M+1,\dots, M]$ is continuous we have
\begin{itemize}
\item if $\wbeta^{\oseq}_0(t_n)\in W_{\oseq}[ -M, -M+1, \dots, M]$ and $t_n\to1-0$, then $\pi(\theta(\wbeta^{\oseq}_0(t_n))\to 0$.
\end{itemize}
It remains to show that if $t_n\to 1-0$ such that $\wbeta^{\oseq}_0(t)\in W_{\oseq}[>M_n]\cup W_{\oseq}[<-M_n]$ with $M_n\to +\infty$, then $\pi(\theta(\wbeta^{\oseq}_0(t_n))\to 0$.

Write $\wbeta^{\oseq}_0(t_n)\in D_{k(n)}$; then $k(n)\to \pm \infty$. By Claim~\ref{cl2:kappa tens inft} $\kappa\circ k(n)\to \pm \infty$. Recall that $T_{\kappa\circ k(n)}\subset \inn_\oseq[\kappa\circ k(n)]$ (see Figure~\ref{Fig:DkTk}) and that $T_{\kappa\circ k(n)}\subset \inn_\oseq[\kappa\circ k(n)]$ is separated by $\wbeta_0^\oseq$ from $\wall_\oseq[\kappa\circ k(n)]$.

Since $\pi(\wbeta_0^\oseq\cap  W_\oseq[\kappa\circ k(n)])\to 0$, we obtain that $\pi(T_{\kappa\circ k(n)})\to 0$. Since ${\pi(T_{\kappa\circ k(n)})\ni \pi(\theta(\wbeta^{\oseq}_0(t_n))}$, we obtain  $\pi(\theta(\wbeta^{\oseq}_0(t_n))\to 0$.
\end{proof}

\begin{claim2}
All lifts of  $\beta_0,\beta_1$ are pairwise disjoint. 
\end{claim2}
\begin{proof}
By Claim~\ref{cl2: beta land at 0} all  $\widetilde \beta_i$ land at $\widetilde 0$; therefore, every $\widetilde \beta_i$ disconnects $\inn_\seq$ into two connected components. It follows from Claims~\ref{cl2:lifting form} and~\ref{cl2:A} that  $\widetilde \beta_{i}, \widetilde \beta_{i+1}$ are disjoint. Since $\wbeta_{i-1}$ is on the left from  $\wbeta_{i}$ while $\wbeta_{i+1}$ is on the right from  $\wbeta_{i}$, we obtain that $\wbeta_{i-1}$ and $\wbeta_{i+1}$ are also disjoint.  Repeating the argument, we obtain that all $\widetilde \beta_{i}$ are pairwise disjoint.
\end{proof}

\end{proof}

\subsubsection{Proof of Theorem~\ref{thm:RobAntiRen}}

Let $D\subset W\setminus \{0\}$ be a topological disk. A \emph{lift} of $D$ to $W_\seq$ is defined in the same way as a lift of a curve, see~\S\ref{sss:lifts of curves}. Alternatively, a lift $\iota\colon D\to W_\seq$ of $D$ is characterized by the property that if $\alpha\subset D$ is a curve, then $\iota(\alpha)$ is the lift of $\alpha\subset D$ starting in $D$. A \emph{lift} of $D$ to $W_{\seq/\qq}$ is the projection of a lift of $D$ to $W_\seq$. 

Since ${\gamma^\new_0\setminus \inn},{\gamma^\new_1\setminus \inn}$ coincide with $\gamma_0\setminus \inn,\gamma_1\setminus \inn$, Condition~\eqref{eq:1:thm:RobAntiRen} uniquely specifies $h\mid W_{\seq/\qq}\setminus\inn_{\seq/\qq}$. 

Since the pair $\gamma^\new_0,\gamma^\new_1$ is dividing, it splits $W$ into two closed sectors, call them $\A_\new$ and $\B_\new$ specified so that $\A_\new\setminus \inn=\A\setminus \inn$ and $\B_\new\setminus \inn=\B\setminus \inn$. We need to show that all lifts of $\A_\new$ and $\B_\new$ to $W_{\seq/\qq}$ exist. By Theorem~\ref{thm:AntiRen:Lifts} (see also Remark~\ref{rem:on dividing curves}), all lifts of $\gamma^\new_0,\gamma^\new_1$ to $W_\seq$ exist, pairwise disjoint, and land at $\widetilde 0$. Projecting to $W_{\seq/\qq}$, we obtain that all lifts of $\gamma^\new_0,\gamma^\new_1$ to $W_{\seq/\qq}$ exist, pairwise disjoint, and land at $0$. The lifts of $\gamma^\new_0,\gamma^\new_1$ split $W_{\seq/\qq}$ into $\qq$ closed sectors; each of them is a lift of either $\A_\new$ or $\B_\new$. Mapping this lifts of $\A_\new$ or $\B_\new$ to the corresponding sectors of $W_{\seq/\qq,\new}$, we obtain a required $h$. 

Since a lift of a curve (if it exists) is uniquely specified by a starting point, the conjugacy $h$ is unique.
\qed

\begin{rem} Anti-renormalization can easily be defined for a partial branched covering $f_0\colon (W,0)\dashrightarrow (W,0)$ of any degree. In this case it is natural to assume that $\gamma_0$ does not contain a critical point of $f$. To apply Theorem~\ref{thm:RobAntiRen}, it is sufficient to assume that there is a \emph{univalent} fence $\wall$ (respected by $\gamma_0,\gamma_1,f$) enclosing $\inn$ such that $f\mid \wall\cup \inn$ has degree one. The anti-renormalization is robust with respect to a replacement $\gamma_0,\gamma_1$ with a new pair $\gamma^\new_0,\gamma^\new_1$ as above.
\end{rem}

\section{The molecule conjecture}
\label{S:ApMolec}
Let us denote by $\Mol$ the main molecule of the Mandelbrot set; i.e.~$\Mol$ is the smallest closed subset of $\MM$ containing the main hyperbolic component as well as all hyperbolic components obtained from the main component via parabolic bifurcations; see~\cites{DH:Orsay,L-book} for the background on the Mandelbrot set. In this appendix we write $f_c(z)=z^2+c$.

\subsection{Branner-Douady maps}
\label{ss:ap:BranDouad}

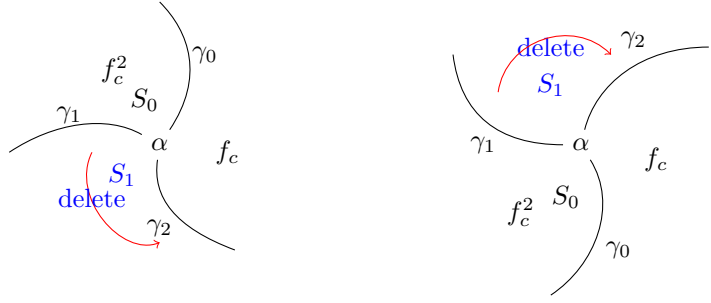
\begin{figure}[t!]
\[
\begin{tikzpicture}

\node (v1) at (-1.8,1) {$\alpha$};

\draw (v1) .. controls (-1.2,1.9) and (-1.4,2.5) .. (-1.8,2.9);

\draw (v1) .. controls (-2.5,1.4) and (-3.2,1.3) .. (-3.8,0.9);
\draw (v1) .. controls (-1.9,0.3) and (-1.6,-0.1) .. (-0.8,-0.4);

\node (v_2) at (3.8,1) {$\alpha$};
\draw (v_2) .. controls (2.7,1) and (2.2,1.4) .. (2.1,2.2);
\draw (v_2) .. controls (4.3,0.2) and (4,-0.6) .. (3.4,-1);
\draw (v_2) .. controls (4,1.8) and (4.6,2.3) .. (5.5,2.3);


\draw [red,->] (-2.7,0.9) .. controls (-3,0.3) and (-2.3,-0.5) .. (-1.8,-0.3);
\node[blue] at (-2.7,0.3) {delete};
\node[blue] at (3.4,2.3) {delete};

\draw[red,<-] (4.2,2.2) .. controls (3.7,2.7) and (2.8,2.4) .. (2.7,1.7);
\node at (4.3,-0.4) {$\gamma_0$};
\node at (-1.2,2.2) {$\gamma_0$};
\node at (2.5,1) {$\gamma_1$};
\node at (-3,1.4) {$\gamma_1$};
\node at (-1.8,-0.1) {$\gamma_2$};
\node at (4.5,2.4) {$\gamma_2$};
\node at (-2.4,2) {$f_c^2$};
\node at (3,0.1) {$f_c^2$};

\node[blue] at (-2.3,0.6) {$S_1$};
\node[blue] at (3.4,1.8) {$S_1$};
\node at (-2,1.6) {$S_0$};
\node at (3.6,0.3) {$S_0$};
\node at (-0.9,0.9) {$f_c$};
\node at (4.8,0.8) {$f_c$};
\end{tikzpicture}\]
\caption{Possible local dynamics at the $\alpha$-fixed point.}
\label{Fg:LocalPict}
 \end{figure}

Let us denote by $\LL_{ \pp/\qq}$ the primary $\pp/ \qq$-limb of the Mandelbrot set and let us denote by $\MM_{ \pp/ \qq}\subset \LL_{ \pp /\qq}$ the $\pp/\qq$-satellite small copy of $\MM.$  We also write $\LL_{0/1}=\MM_{0/1}=\MM$.

In \cite{BD} Branner and Douady constructed a partial surjective continuous map $\mRRc:\LL_{1/3}\dashrightarrow \LL_{1/2}$ such that its inverse $\mRRc^{-1}\colon \LL_{1/2} \to L_{1/3}$ is an embedding. This construction could be easily generalized to a continuous map ${\mRRc:\LL_{\pp/\qq}\dashrightarrow \LL_{\mRRc(\pp/\qq)}}$, where (compare to~\eqref{eq:R_prm})
\[\cRRc\left(\pp/\qq\right) = \begin{cases} \frac{\pp}{\qq-\pp}& \mbox{if }0< \frac \pp \qq \le \frac{1}{2} \\[1em]
\frac{2\pp-\qq}{\pp} & \mbox{if }\frac{1}{2}\le \frac{\pp}{\qq}< 1,
\end{cases}\]
as follows. Recall that $c\in \LL_{ \pp /\qq}$ if and only if in the dynamical plane of $f_c$ there are exactly $\qq$ external rays landing at the $\alpha$-fixed point and the rotation number of these rays is $\pp/\qq$; i.e.~if $\gamma$ is a ray landing at $\alpha$, then there are $\pp-1$ rays landing at $\alpha$ between $\gamma$ and $f_c(\gamma)$ counting counterclockwise. 

Choose an external ray $\gamma_0$ landing at $\alpha$ in the dynamical plane of $f_c$ with $c\in \LL_{ \pp/ \qq}$. Define $\gamma_1=f_c(\gamma_0)$ and $\gamma_2=f_c(\gamma_1)$. Denote by $S_0$ the open sector between $\gamma_0$ and $\gamma_1$ not containing $\gamma_2$, see Figure~\ref{Fg:LocalPict}. Similarly, let $S_1$ be the open sector between $\gamma_1$ and $\gamma_2$ not containing $\gamma_0$. We assume that $\gamma_0$ is chosen such that $S_1$ does not contain the critical value, thus $S_1$ has two conformal lifts, one of them is $S_0$, we denote by $S'_0$ the other. If $S_1\supset S'_0$, then replace $S'_0$ by its unique lift in $\C\setminus S_1$.

Let us delete $S_1$, glue $\gamma_1$ and $\gamma_2$ dynamically $\gamma_1\ni x\sim f(x)\in\gamma_2$, and iterate $f_c$ twice on $S_0$. We obtain a new map denoted by $\bar f_c\colon \C\setminus S'_0\to \C$. 
The \emph{filled-in Julia set} $\overline K_c$ of $\bar f_c$ is the set of points with bounded orbits that do not escape to $S'_0$. The set $\overline K_c$ is connected if and only if $0$ does not escape to $S'_0$; in this case the new local dynamics of $\bar f_c$ at $\alpha$ has rotation number $\cRRc(\pp/\qq)$ and, moreover, $\bar f_c$ is hybrid equivalent to a quadratic polynomial $f_{\mRRc(c)}$ with $c\in \LL_{\cRRc(\pp/\qq)}$. This defines the map $\mRRc:\LL_{\pp/\qq}\dashrightarrow \LL_{\RRc(\pp/\qq)}$. 

 In general, $\mRRc:\LL_{\pp/\qq}\dashrightarrow \LL_{\cRRc(\pp/\qq)}$ depends on the choice of $\gamma_0$. However, if $c\in \MM_{\pp/\qq}$, then $\mRRc(c)\in \MM_{\cRRc(\pp/\qq)}$ and $\mRRc:\MM_{\pp/\qq}\to \MM_{\cRRc(\pp/\qq)}$ coincides with the canonical homeomorphism between small copies of the Mandelbrot set.

\begin{rem}
 The Branner--Douady surjery has also been studied by Riedl~\cite{R}; he showed, in particular, that every dyadic Misiurewicz parameter is connected through a simple arc (vein) in the Mandelbrot set to the origin.
\end{rem}

\subsection{The molecule and the fast molecule maps}
Denote by $\HH$ the main hyperbolic component of $\MM$. Recall that a parameter $c\in \partial \HH$ is parametrized by the multiplier $\ee(\theta(c))$ of its non-repelling  fixed point. We define \emph{the molecule map} $\mRRc\colon \MM\dashrightarrow \MM$ such that 
\begin{itemize}
\item $\mRRc\colon \LL_{\pp/\qq}\dashrightarrow \LL_{\cRRc(\pp/\qq)}$ is the Branner--Douady renormalization map for $\pp/\qq\neq 0/1$ and for some choice of $\gamma_0$; and
\item if $c\in \partial \HH$, then $\mRRc(c)$ is so that
\[\theta (\mRRc(c))  = \begin{cases} \frac{\theta(c)}{1-\theta(c)}& \mbox{if }0\le \theta(c) \le \frac{1}{2}, \\[1em]
\frac{2\theta(c)-1}{\theta(c)} & \mbox{if }\frac{1}{2}\le \theta(c)\le 1.
\end{cases}\] 
\end{itemize}

Siegel parameters of periodic type are exactly periodic points of $\mRRc\mid \partial \HH$ (Lemma~\ref{lem:SectRen is prime power}). Furthermore, for a satellite copy of the Mandelbrot set $\MM_s$, there is an $n\ge 1$ such that $\mRRc^n\colon \MM_s\to \MM$ is the Douady-Hubbard straightening map.

The map $\mRRc\colon \MM\dashrightarrow \MM$ is combinatorially modeled by $Q(z)\coloneqq z(z+1)^2$, see Figure~\ref{Fg:Molec}. The latter map has a unique parabolic fixed point as $0$. The attracting basin of $0$ contains exactly one critical point of $Q$. The second critical point is a preimage of $0$. Denote by $F$ the invariant Fatou component of $Q$. We can extend $\mRRc$ to $ \HH$ so that $\mRRc\mid \overline \HH$ is conjugate, say by $\pi$, to $Q\mid \overline F$. Then $\pi$ extends uniquely to a monotone continuous map $\pi:\Mol \to K_Q$ semi-conjugating $\mRRc\mid  \Mol$ and $Q\mid K_Q$, where $K_Q$ is the filled-in Julia set of $Q$:

\[
\begin{tikzpicture}[description/.style={fill=white,inner sep=2pt}]
    \matrix (m) [matrix of math nodes, row sep=3em,
    column sep=4.5em, text height=1.5ex, text depth=0.25ex]
    { \Mol& \Mol\\
      K_Q& K_Q\\};    \path[->,font=\scriptsize]
    (m-1-1) edge node[description] {$\mRRc$} (m-1-2) edge node[description] {$\pi$} (m-2-1)
    (m-1-2) edge node[description] {$\pi$} (m-2-2)
    (m-2-1) edge node[description] {$Q$} (m-2-2);
\end{tikzpicture}
\]

If the MLC-conjecture holds, then $\pi$ is a homeomorphism.

\begin{figure}[t!]
\includegraphics[width=6cm]{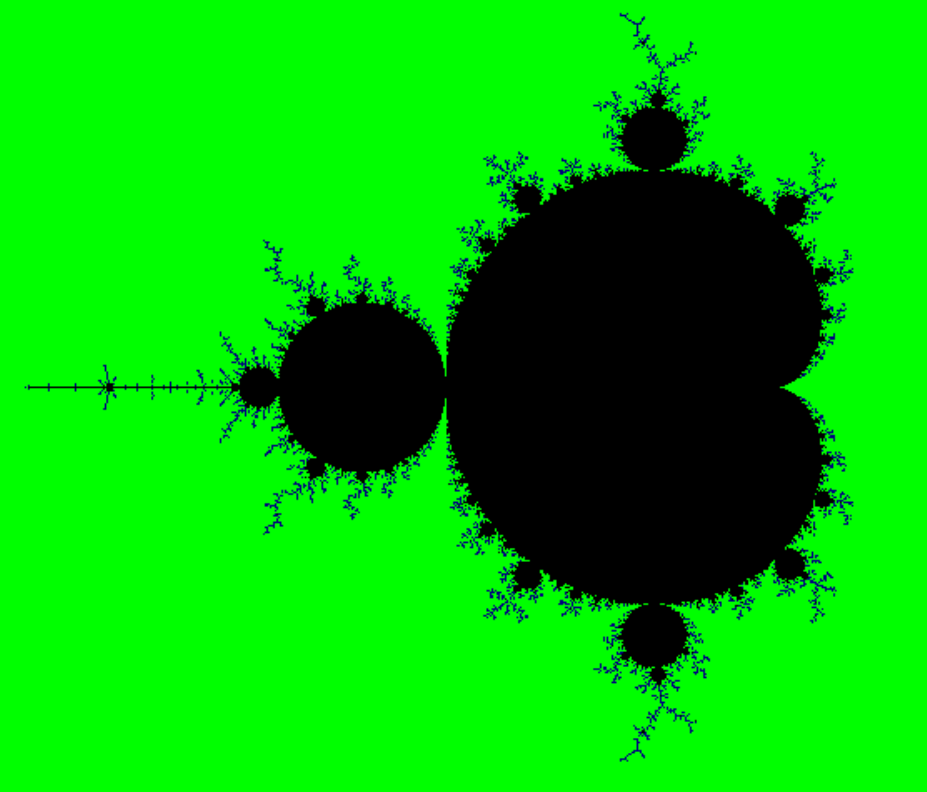}
\includegraphics[width=6cm]{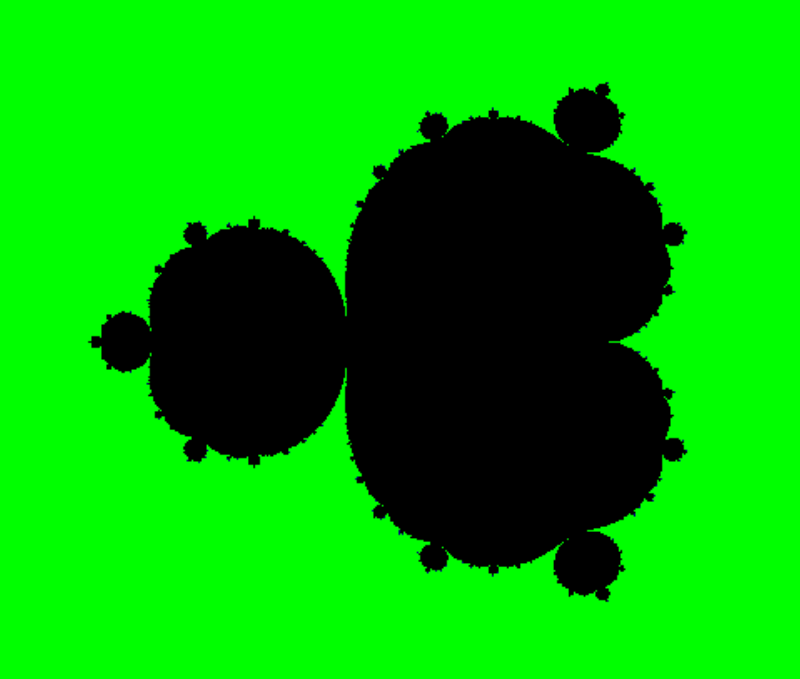}
\caption{Left: the Mandelbrot set. Right: the filled Julia set of $Q(z)=z(z+1)^2$.}
\label{Fg:Molec}
\end{figure}

For every $c\in \partial \Delta \setminus\{\text{cusp}\}$ define $\nn(c)\coloneqq \nn(\theta_c)$, where $\theta_c$ is the rotation number of $f_c$ and $\nn(\theta)$ is specified by $\cRR_\fast(\theta)=\cRRc^{\nn(\theta)}(\theta)$, see~\S\ref{ss:RenRotUnitDisc}. For every $c\in \LL_{\pp/\qq}$ define $\nn(c)\coloneqq \nn(c_{\pp/\qq})$, where $c_{\pp/\qq}$ is the root of $\LL_{\pp/\qq}$. The \emph{fast Molecule map} is a partial map  on $\MM$ defined by
\[\mRR_{\fast}(c)= \mRRc^{\nn(c)}(c).\]
The restriction $\mRR_{\fast}\mid \partial \Mol\setminus\{\text{cusp}\}$ is continuous but it does not extend continuously to the cusp: $\RR_{\fast}(\partial \MM_{1/n})=\partial \MM$.

\subsection{The Molecule Conjecture}Given a renormalization operator $\RR\colon \BB\dashrightarrow \BB$, its \emph{renormalization horseshoe} is the set of points in $\BB$ with bi-infinite pre-compact orbits. We conjecture that there is a pacman renormalization operator $\RR_\fast\colon \BB_\Mol\to \BB_\Mol$ with the following properties. The operator $\RR_\fast$ is hyperbolic and piecewise analytic with one-dimensional unstable direction such that its renormalization horseshoe $ \RR_\fast\colon \Hors_\Mol\to\Hors_\Mol$ is compact and combinatorially associated with $\RR_\fast\mid \Mol\setminus \{\text{cusp}\}$ as follows.

There is a  continuous surjective map $\rho\colon \Hors_{\Mol}\to \Mol$ that is a semi-conjugacy away from the cusp:
\[
\begin{tikzpicture}[description/.style={fill=white,inner sep=2pt}]
    \matrix (m) [matrix of math nodes, row sep=3em,
    column sep=5.5em, text height=1.5ex, text depth=0.25ex]
    {  \Hors_{\Mol}\setminus \rho^{-1}(\text {cusp})&  \Hors_{\Mol}\setminus \rho^{-1}(\text {cusp})\\
      \partial\Mol\setminus \{\text{cusp}\}&  \partial\Mol\setminus \{\text{cusp}\}\\};    \path[->,font=\scriptsize]
    (m-1-1) edge node[description] {$\RR_\fast$} (m-1-2) edge node[description] {$\rho$} (m-2-1)
    (m-1-2) edge node[description] {$\rho$} (m-2-2)
    (m-2-1) edge node[description] {$\mRR_\fast$} (m-2-2);
\end{tikzpicture}.\]
Denote by $\partial^\irr \Mol$ the set of non-parabolic parameters in $\partial \Mol$. Conjecturally, $\RR_\fast\mid  \Hors_{\Mol}$ is the natural extension of ${\mRR_\fast\mid \partial \Mol\setminus \{\text{cusp}\}}$ compactified by adding limits to parabolic parameters at all possible directions. Such construction is known as a parabolic enrichment, see~\cites{La,D-discont}.

The space $\BB_\Mol$ has a codimension-one stable lamination $(\Fol^s_c)_{c\in \Mol}$ such that all pacmen in $\Fol^s_c$ are hybrid conjugate to $f_c$ in neighborhoods of their ``mother hedgehogs'', see~\S\ref{ss:ap:conj:UppSemic}. For every $f\in \Hors_\Mol$, the leaf $\Fol^s_{\rho(f)}$ is a stable manifold of $\RR_\fast$ at $f$. The unstable manifold of $\RR_\fast$ at $f$ is parametrized by a parabolic enrichment of a neighborhood of $\rho(f)$. Locally, $\RR_\fast$  can be factorize as an iterate of $\RRc\colon \BB_\Mol\to \BB_\Mol$; however the latter operator has parabolic behavior at $\rho^{-1}(\text{cusp})$.

The Molecule Conjecture contains both Theorem~\ref{thm:RR is hyper} (for periodic type parameters from $\partial \HH$) and the Inou-Shishikura theory \cite{IS} (for high type parameters from $\partial \HH$).  It also implies the local connectivity of the Mandelbrot set for all parameters on the main (and thus any) molecule.

\subsection{Conjecture on the upper semicontinuity of the mother hedgehog}
\label{ss:ap:conj:UppSemic}
A closely related conjecture is the upper semicontinuity of the mother hedgehog. For a non-parabolic parameter $c\in \partial \HH$, consider the closed Siegel disk $\overline Z_c$ of $f_c$; if $f_c$ has a Cremer point, then $\overline Z_c\coloneqq \{\alpha\}$. If $\overline Z_c$ contains a critical point, then we set $H_c\coloneqq \overline Z_c$. Otherwise, $f_c$ has a \emph{hedgehog} (see~\cite{PM}): a compact closed connected filled-in forward invariant set $H'\supsetneqq \overline Z_c$ such that $f_c\colon H'\to H'$ is a homeomorphism. We define $H_c$ to be the \emph{mother hedgehog} (see~\cite{Ch}): the closure of the union of all of the hedgehogs of $f_c$.

Recall that the filled-in Julia set $K_g$ of a polynomial depends upper semicontinuously on $g$. Thinking of $H_c$ as an indifferent-dynamical analogue of $K_g$, we conjecture:
\begin{conj}
\label{conj:ap:UppSemiCont}
The mother hedgehog $H_c$ depends upper semicontinuously on $c$. 
\end{conj}
For bounded type parameters (i.e.~when $H_c$ is a Siegel quasidisk) Conjecture~\ref{conj:ap:UppSemiCont} follows from the continuity of the Douady-Ghys surgery.

Conjecture~\ref{conj:ap:UppSemiCont} can be adjusted for parabolic parameters $c\in \HH$ as follows. Let $A_c$ be the immediate attracting basin of the parabolic fixed point $\alpha$. Then there is a choice of a valuable flower $H_c$ with $\overline H_c\subset  A_c\cup \{\alpha\}$ such that $H_c$ depends upper semicontinuously on $c\in \partial \HH$. For example, $H_c$ is the union of all limiting mother hedgehogs for perturbations of $f_c$. 

Similarly, Conjecture~\ref{conj:ap:UppSemiCont} can be adjusted for  all parameters in $\partial \Mol$. Our result on the control of the valuable flower (see Theorem~\ref{thm:ScalThm}) can be thought as a partial case of this general conjecture.

Conjecture~\ref{conj:ap:UppSemiCont} and its generalizations describe in a convenient way how an attracting fixed point bifurcates into repelling. An important consequence is control of the post-critical set: if a perturbation of $f_c$ is within $\Mol$, then the new post-critical set is within a small neighborhood of $H_c$. A statement of this sort (for parabolic parameters approximating a Siegel polynomial) was proven by Buff and Ch\'eritat, see~\cite[Corollary 4]{BC}. This was a necessary ingredient in constructing a Julia set with positive measure.

\end{document}